\documentclass[reqno]{amsart}

\usepackage{amsmath,amsthm,amssymb,amsfonts}
\usepackage{mathrsfs}
\usepackage{enumerate}
\usepackage{multirow} 
\usepackage{bm}
\usepackage{bbm}
\usepackage{enumitem}
\usepackage{longtable}
\usepackage{here}
\usepackage{placeins} 
\DeclareMathOperator*{\tr}{tr}

\newtheoremstyle{mystyle}
  {}
  {}
  {\normalfont}
  { }
  {\bfseries}
  {}
  {10pt}
  { }
\theoremstyle{mystyle}
\newtheorem{thm}{Theorem}[section]
\newtheorem{lemma}{Lemma}[section]
\newtheorem{col}{Corollary}[section]
\newtheorem{prop}{Proposition}[section]
\newtheorem{rmk}{Remark}[section]

\newcommand{\argsup}{\mathop{\rm argsup}\limits}

\usepackage[dvipdfmx]{graphicx}
\usepackage{color}
\usepackage{hyperref}

\usepackage[foot]{amsaddr}
\usepackage[top=3cm,bottom=3cm,left=2.5cm,right=2.5cm,marginparwidth=1.75cm]{geometry}

\allowdisplaybreaks[1]

\usepackage{comment}

\numberwithin{equation}{section} 

\usepackage[authoryear,round]{natbib}

\title[Adaptive inference for jump diffusion processes]{Quasi-Likelihood Ratio Test for Jump-Diffusion Processes Based on Adaptive Maximum Likelihood Inference}
\author[N Nishikawa]{Hiromasa Nishikawa $^{1}$}
\author[T Kawai]{Tetsuya Kawai $^{2}$}
\author[M Uchida]{Masayuki Uchida $^{1,3}$}
\address{ $^{1}$ Graduate School of Engineering Science, Osaka University}
\address{$^{2}$Toyota Motor Corporation, Toyota, Japan}
\address{$^{3}$Center for Mathematical Modeling and Data Science (MMDS), Osaka University and JST CREST}

\begin{document}

\begin{abstract}
In this paper, we consider parameter estimation and quasi-likelihood ratio tests for multidimensional jump-diffusion processes defined by stochastic differential equations. In general, simultaneous estimation faces challenges such as an increase of computational time for optimization and instability of estimation accuracy as the dimensionality of parameters grows.
To address these issues, we propose an adaptive quasi-log likelihood function based on the joint quasi-log likelihood function introduced by \citet{Shimizu-Yoshida_JP,Shimizu-Yoshida} and \citet{Ogihara-Yoshida}. We then show that the resulting adaptive estimators possess consistency and asymptotic normality. Furthermore, we extend the joint quasi-log likelihood function proposed by \citet{Shimizu-Yoshida_JP,Shimizu-Yoshida} and \citet{Ogihara-Yoshida} and construct a test statistic using the proposed adaptive estimators.
We prove that the proposed test statistic converges in distribution to a $\chi^2$-distribution under the null hypothesis and that the associated test is consistent. Finally, we conduct numerical simulations using a specific jump-diffusion process model to examine the asymptotic behavior of the proposed adaptive estimators and test statistics.
\end{abstract}

\maketitle


\section{Introduction}
Given a filtered probability space $(\Omega,\mathcal{F},(\mathcal{F}_t)_{t\geq 0},P)$, let $X=(X_t)_{t\geq 0}$ be a $d$-dimensional c\'{a}dl\'{a}g $(\mathcal{F}_t)$-adapted process satisfying the following stochastic differential equation:
\begin{equation}\label{jump-diff:model}
\begin{cases}
dX_t = b(X_{t-}, \beta) dt + a(X_{t-}, \alpha)dW_t + \displaystyle\int_{E}c(X_{t-}, z, \beta)p(dt,dz), \quad t \in [0,T], \\
X_0 = x_0,
\end{cases}
\end{equation}
where $x_0$ is a $d$-dimensional random variable, $W_{t}$ is an $s$-dimensional standard $(\mathcal{F}_t)$-Brownian motion,
$E=\mathbb{R}^d\setminus\{0\}$, and $p(dt,dz)$ is a Poisson random measure on $\mathbb{R}_+\times E$ with compensator $q^{\beta}(dt,dz)=\mathbb{E}_\beta[p(dt,dz)]$. 
We set $q^{\beta}(dt,dz)=f_\beta(z)dzdt$ and $f_\beta(z)=\lambda(\beta)F_\beta(z)$, where $\lambda(\beta)$ is a positive function of $\beta$ and $F_\beta(z)$ is a probability density function.
We assume for any $t\ge0$, $\sigma(W_u-W_t; \ u\ge t)$, $\mathcal{F}_t$ and $\sigma\left(p(A\cap((t,\infty)\times E)\right); \ A\subset\mathbb{R}^d\times E \text{ is a Borel set})$ are independent.
Let $\alpha\in\Theta_{\alpha}\subset\mathbb{R}^{p},\ \beta \in \Theta_{\beta} \subset \mathbb{R}^{q}$, $\theta =(\alpha,\beta)$, $\Theta := \Theta_{\alpha} \times \Theta_{\beta}$ be a compact and convex parameter space. Moreover, $a:{\mathbb{R}}^{d} \times {\Theta}_{\alpha} \to {\mathbb{R}}^{d}\otimes {\mathbb{R}}^{s}$, $b:{ \mathbb{R}}^{d} \times {\Theta}_{\beta} \to {\mathbb{R}}^{d}$
and $c:{ \mathbb{R}}^{d} \times E\times{\Theta}_{\beta} \to {\mathbb{R}}^{d}$ are known except for the parameter $\theta$, and the true parameter $\theta_0=(\alpha_0,\beta_0)$ belongs to $\mathrm{Int}(\Theta)$. The data are discrete observations $(X_{t^n_i})_{0\le i \le n}$, where $t^n_i = i h_n$ for $i=0,1,\ldots,n$, and the discretization step $h_n$ satisfies $h_n \to 0$ and $nh_n=T \to \infty\ \mathrm{as}\ n\to\infty$. Moreover, we will assume $n h_n^{1+\delta} \to 0$ for some $\delta\in(0,1)$ later.
In this setting, we consider the problem of estimating the unknown parameters $\theta=(\alpha,\beta)$ and the hypothesis testing problem in an ergodic jump-diffusion process model based on discrete observation data.
Jump-diffusion process models are used in various applications, including the modeling of option prices in financial markets. Therefore, statistical analysis of jump-diffusion process models is important.
As a prior study on the simultaneous estimation of parameters in an ergodic jump-diffusion process model based on discrete observations, \citet{Shimizu-Yoshida_JP,Shimizu-Yoshida} established the consistency and asymptotic normality of the quasi-maximum likelihood estimator under the conditions $h_n\to 0$, $nh_n\to \infty$ and $nh_n^2\to 0$.
In quasi-likelihood analysis for jump-diffusion process models, it is necessary to allocate the increments of the data to either the continuous part or the jump part of the quasi-likelihood function, which requires distinguishing whether an increment includes a jump or not. To address this issue, \citet{Shimizu-Yoshida_JP,Shimizu-Yoshida} proposed a threshold-based filtering method, enabling asymptotic identification of jumps. Moreover, \citet{Ogihara-Yoshida} relaxed the regularity conditions imposed by \citet{Shimizu-Yoshida_JP,Shimizu-Yoshida} and demonstrated the convergence of moments. Additionally, \citet{Masuda} provided conditions under which multidimensional jump-diffusion process models satisfy the ergodicity property, allowing the estimation of parameters for L\'{e}vy-Ornstein-Uhlenbeck (L\'{e}vy-OU) processes. \citet{Ogihara-Uehara} established the local asymptotic normality of ergodic jump-diffusion process models based on discrete observation data, thereby proving the asymptotic efficiency of the quasi-maximum likelihood estimator proposed by \citet{Shimizu-Yoshida_JP,Shimizu-Yoshida} and \citet{Ogihara-Yoshida}.
It is generally known that joint estimation suffers from computational inefficiencies, such as increased optimization time and instability of estimation accuracy, as the dimensionality of the parameters increases. To address this issue, we propose an adaptive quasi-log likelihood function based on the simultaneous quasi-log likelihood function introduced in \citet{Ogihara-Yoshida}.  Optimizing the diffusion parameter $\alpha$ separately from the drift and jump parameters $\beta$, we aim to improve the computational efficiency and stability. In particular, whereas \citet{Shimizu-Yoshida_JP,Shimizu-Yoshida} and \citet{Ogihara-Yoshida} used a single threshold for distinguishing the continuous and jump components, our proposed method utilizes three thresholds for estimation, leading to improved estimation accuracy. In this paper, we demonstrate the consistency and asymptotic normality of the adaptive quasi-maximum likelihood estimator derived from the proposed adaptive quasi-log likelihood function.
Other prior studies on jump-diffusion process models include \citet{Mancini}, who proposed a consistent estimator for the volatility parameter in non-ergodic jump-diffusion process models, and \citet{Gloter}, who, by focusing on the estimation of the drift parameter in ergodic jump-diffusion process models, relaxed the balance conditions related to the sampling frequency $h_n$. Moreover, \citet{Inatsugu-Yoshida} proposed a highly accurate estimation method for the diffusion term in non-ergodic jump-diffusion process models using a Global Jump Filter.

Next, as an application of the constructed estimator, we consider a quasi-likelihood ratio test for the unknown parameter $\theta=(\alpha,\beta)$ in an ergodic jump-diffusion process model. In the quasi-likelihood ratio test based on the joint quasi-log likelihood function of \citet{Shimizu-Yoshida_JP,Shimizu-Yoshida} and \citet{Ogihara-Yoshida}, the simultaneous estimator is used in the construction of the test statistic. However, this approach suffers from issues such as an increase of optimization time and instability of estimation accuracy.
To address these issues, in this paper, we construct a quasi-likelihood ratio test statistic using a modified simultaneous quasi-log likelihood function, in which the number of threshold parameters used in the construction of the simultaneous quasi-log likelihood function in \citet{Ogihara-Yoshida} is expanded from one to two, along with the proposed adaptive quasi-maximum likelihood estimator. We also discuss its asymptotic properties. This approach improves the computational efficiency and stabilizes numerical calculations.
In particular, for the test statistic, we introduce five thresholds in the construction of the adaptive estimator and the quasi-likelihood ratio, allowing for further improvements in testing accuracy.
Furthermore, studies on adaptive testing methods have been conducted not only for jump-diffusion process models but also for other stochastic processes. For example, adaptive testing methods for ergodic diffusion process models have been discussed in \citet{Kitagawa-Uchida}, \citet{Nakakita-Uchida}, and \citet{Kawai-Uchida}.

This paper is organized as follows.
In Section \ref{sec4.2:m}, we provide definitions of notation and assumptions.
In Section \ref{sec4.3:m}, we propose joint and adaptive quasi-log likelihood functions based on \citet{Shimizu-Yoshida_JP,Shimizu-Yoshida}, \citet{Ogihara-Yoshida}, and discuss the asymptotic properties of the estimator derived from it.
In Section \ref{sec:test}, we construct a test statistic using the results of Section \ref{sec4.3:m} and describe its asymptotic properties.
In Section \ref{sec:sim}, we conduct numerical simulations for the estimators and test statistics proposed in Sections \ref{sec4.3:m} and \ref{sec:test}, using a specific L\'{e}vy-OU process model.
Finally, in Section \ref{sec:proof}, we provide proofs of the theorems established in this paper.


\section{Notation and assumptions}\label{sec4.2:m}
Let us introduce some notation.
\begin{enumerate}
\item[1.] We set the true value of $\lambda(\beta)$ by $\lambda_0=\lambda(\beta_0)=\int_E f_{\beta_0}(z)dz$.

\item[2.] For a vector $\kappa=(\kappa_1,\ldots,\kappa_l)^\top$, $\partial_{\kappa_i}:=\frac{\partial}{\partial {\kappa_i}}$,\  $\partial^2_{\kappa_i}:=\frac{\partial^2}{\partial {\kappa_i^2}}$, $\partial^2_{\kappa_i\kappa_j}:=\frac{\partial^2}{\partial {\kappa_i}\partial {\kappa_j}}$, $\partial_\kappa:=(\partial_{\kappa_1},\ldots,\partial_{\kappa_l})^\top$ and $\partial_\kappa^2:=(\partial^2_{\kappa_i\kappa_j})_{1\le i,j\le l}$, where $\top$ stands for the transpose.

\item[3.] For a function $g$ defined on $\mathbb{R}^d\times\Theta$, $g_{i-1}(\theta)$ denotes the value $g(X_{t_{i-1}^n},\theta)$.
If $g$ is a vector or a matrix function, then we express its components with upper index. For example, if $g$ is a vector, then its $k$-component is $g^{(k)}$, and if $g$ is a matrix, then its $(k,l)$-component is $g^{(k,l)}$.

\item[4.] Let $\mathcal{F}_{i-1}^n:=\mathcal{F}_{t_{i-1}^n}, \quad\Delta X_i^n:=X_{t_i^n}-X_{t_{i-1}^n},\quad\Delta X_t:=X_t-X_{t-}, \quad\bar{X}_{i,n}(\beta):=\Delta X_i^n-h_nb_{i-1}(\beta)$, $\quad S(x,\alpha):=a(x,\alpha)a(x,\alpha)^\top$.

\item[5.] For a matrix $A$, we define that $|A|=\sqrt{\mathrm{tr}(AA^\top)}$.

\item[6.] We often use the notation $C$ (resp. $C_k$) as a general positive constant (resp. depending on the index $k$), therefore we sometimes use the same character for different constants from line to line without specially mentioning.

\item[7.] Let $u_n$ be a real valued sequence. $R:\Theta\times\mathbb{R}\times\mathbb{R}^d\to\mathbb{R}$ denotes a function for which there exists a constant $C>0$ such that for any
$(\theta,x,n)\in\Theta\times\mathbb{R}^d\times\mathbb{N}$,
\begin{align*}
|R(\theta,u_n,x)|\le u_nC(1+|x|)^C,
\end{align*}
and we set $\tilde{R}(\theta,u_n,x):=1-R(\theta,u_n,x)$.

\item[8.] If we write $X$, then it means the solution to \eqref{jump-diff:model} with $\theta=\theta_0$. 

\item[9.]  The symbols $\overset{P}{\longrightarrow}$ and $\overset{d}{\longrightarrow}$ stand for
convergence in probability and convergence in distribution, respectively.

\end{enumerate}

We make the following assumptions to obtain main results.

\begin{enumerate}
\item[{[\bf A1]}] There exist a constant $C>0$ and a function $\zeta(z)$ of polynomial growth in $z$ such that for all $x,y \in\mathbb{R}^d$,
\begin{align*}
&|a(x,\alpha_0)-a(y,\alpha_0)|+|b(x,\beta_0)-b(y,\beta_0)|\le C|x-y|,\\
&|c(x,z,\beta_0)-c(x,z,\beta_0)|\le\zeta(z)|x-y|,\quad |c(x,z,\beta_0)|\le\zeta(z)(1+|x|).
\end{align*}

\item[{[\bf A2]}] The jump diffusion process $X$ is ergodic with its invariant measure $\pi(dx)$: For any $\pi$-integrable function
$f$, it holds that
\begin{align*}
\frac{1}{T}\int_{0}^{T}{f(X_t)dt}\overset{P}{\longrightarrow}\int f(x)\pi(dx)
\end{align*}
as $T\to\infty$. Moreover, we assume the stationarity of $X$ for simplicity.

\item[{[\bf A3]}] For any $p\ge 1$, $\displaystyle  \sup_{t\ge0} E_{\theta}[|X_{t}|^p]<\infty$.

\item[{[\bf A4]}] For each $\alpha$ and $\beta$, the derivatives $\partial_{x}^k a(x,\alpha)$ and $\partial_{x}^k b(x,\beta)$ $(k=0,1,2,3,4)$ exist on $\mathbb{R}^d$ and they are continuous in
$x$. Moreover, for each fixed $x$, the derivatives
$\partial_{\alpha}^l a(x,\alpha)$ and $\partial_{\beta}^l b(x,\beta)$ $(l=1,2,3)$ exist, which are continuous on
$\Theta_\alpha$ and $\Theta_\beta$, respectively. Furthermore, $a$, $b$ and their derivatives are of at most polynomial growth in $x$ uniformly in $\theta$: 
\begin{align*}
|\partial_{x}^k a(x,\alpha)|,\ |\partial_{x}^k b(x,\beta)|,\ |\partial_{\alpha}^l a(x,\alpha)|,\ |\partial_{\beta}^l b(x,\beta)|\le C(1+|x|)^C\quad(x\in\mathbb{R}^d,\ \theta\in\Theta),
\end{align*}
for $k=0,1,2,3,4$ and $l=1,2,3$.

\item[{[\bf A5]}] There exists constants $r>0$ and $K>0$ such that $f_{\beta_0}(z)\boldsymbol{1}_{\{|z|\le r\}}\le K|z|^{1-d}$. Moreover, for any $p\ge1$, 
\begin{align*}
\sup_{\beta\in\Theta_\beta}\int_E |z|^p f_\beta(z)dz<\infty.
\end{align*}

\item[{[\bf A6]}] For each $(\beta,x)\in\Theta_\beta\times\mathbb{R}^d$, the mapping $z\mapsto y=c(x,z,\beta)$ is an injection from $E$ to $E$ and has an inverse
$z=c^{-1}(x,y,\beta)$ from the image of $c$ onto $E$, which is differentiable with respect to $y$. Furthermore, the set $B:=\mathrm{Im}(c(x,\cdot,\beta))=\{y\in E;\  ^\exists x\in E \text{ s.t. } y=c(x,z,\beta)\}\in\mathbb{R}^d$ is open and independent of
$(x,\beta)$, and the set $\{(x,y)\in\mathbb{R}^d\times E; x\in\mathbb{R}^d,\ y\in B\}$ is a Borel set.
Moreover, we set
\begin{align*}
\Psi_\beta(y,x)=f_{\beta}(c^{-1}(x,y,\beta))J(x,y,\beta)\quad(x\in\mathbb{R}^d,\ y\in B,\ \beta\in\Theta_\beta),
\end{align*}
where $J(x,y,\beta)$ is the absolute value of the Jacobian of $c^{-1}(x,y,\beta)$, and the set $A=\{y\in B; \Psi_\beta(y,x)\neq0\}$ does not depend on $(x,\beta)$.

\item[{[\bf A7]}] There exist positive constants $c_0>0$ and $r_1>0$ such that 
\begin{align*}
|y|\ge c_0|c^{-1}(x,y,\beta_0)|,\quad\left(x\in \mathbb{R}^d,\ y\in B\cap\{y; |y|\le r_1\}\right).
\end{align*}

\item[{[\bf A8]}]$\displaystyle \inf_{x, \alpha} \det(S(x, \alpha ))>0$.

\item[{[\bf A9]}] 
$\det S(x,\alpha) =\det S(x, \alpha_{0} ) \ \text{for} \ \text{a.s.}\  \text{all} \ x \Longrightarrow \alpha = \alpha_{0}.$

$b(x,\beta) = b(x, \beta_{0} )\  \text{and} \ \Psi_\beta(x,y)=\Psi_{\beta_0}(x,y) \ \text{for} \ \text{a.s.}\  \text{all} \ (x,y) \Longrightarrow \beta = \beta_{0}.$

\item[{[\bf A10]}] The function $\Psi_\beta(y,x)$ is differentiable with respect to $x$ and $y$, and three times 
continuously differentiable with respect to $\beta$.
Moreover, for $x\in\mathbb{R}^d$,
\begin{align*}
&\int_{B}\sup_{\beta\in\Theta_\beta}\left|\partial_{{\beta}}^k\Psi_{\beta}(y,x)\right|dy\le C(1+|x|)^C\quad(k=0,1,2,3),\\
&\int_{B}\sup_{\beta\in\Theta_\beta}\left|\partial_x\partial_{{\beta}}^l\Psi_{\beta}(y,x)\right|dy\le C(1+|x|)^C\quad(l=0,1,2),\\
&\int_{A}\sup_{\beta\in\Theta_\beta}\left|\partial_{{\beta}}^k\log\Psi_{\beta}(y,x)\right|\Psi_{\beta_0}(y,x)dy\le C(1+|x|)^C\quad(k=0,1,2,3).
\end{align*}

\item[{[\bf A11]}] There exists some $dy$-integrable function $L(y,\theta)$, which does not depend on $x$, such that
\[\left|\partial_x\left(\partial_\theta^l\log\Psi_\beta(y,x)\varphi_n(x,y)\Psi_{\beta_0}(y,x)\right)\right|\leq L(y,\theta),\ (x\in\mathbb{R}^d,\ y\in B,\ \theta\in\Theta),\]
for $l=0,1,2$.
\end{enumerate}

Let $I(\theta;\theta_0)$ be a $(p+q)\times (p+q)$-matrix such that 
\begin{align*}
I(\theta;\theta_0)=\begin{pmatrix}
I_a(\alpha;\alpha_0)&O\\
O&I_{b,c}(\theta;\theta_0)
\end{pmatrix},
\end{align*}
where
\begin{align*}
I_a^{(i,j)}(\alpha;\alpha_0)&=\frac{1}{2}\int \left(\mathrm{tr}\left[\partial^2_{\alpha_i\alpha_j}S^{-1}(x,\alpha) S(x,\alpha_0)\right]+\partial^2_{\alpha_i\alpha_j}\log\det S(x,\alpha) \right)\pi(dx)\quad(1\le i,j\le p),\\
I_{b,c}^{(i,j)}(\theta;\theta_0)&=I_{b}^{(i,j)}(\theta;\theta_0)+I_{c}^{(i,j)}(\beta;\beta_0)\quad(1\le i,j\le q),\notag\\
I_b^{(i,j)}(\theta;\theta_0)&=\int\left(\left(\partial^2_{\beta_i\beta_j}b(x,\beta)\right)^\top S^{-1}(x,\alpha)\left(b(x,\beta)-b(x,\beta_0)\right)+\left(\partial_{{\beta_i}}b(x,\beta)\right)^\top S^{-1}(x,\alpha)\partial_{{\beta_j}}b (x,\beta)\right)\pi(dx),\\
I_c^{(i,j)}(\beta;\beta_0)&=\iint_{A} \left(\partial^2_{\beta_i\beta_j}\Psi_\beta(y,x)-\left(\partial^2_{\beta_i\beta_j}\log\Psi_\beta(y,x)\right)\Psi_{\beta_0}(y,x)\right)dy\pi(dx).
\end{align*}
In particular, we have
\begin{align}
I_a^{(i,j)}(\alpha_0;\alpha_0)&=\frac{1}{2}\int \mathrm{tr}\left[S^{-1}\left(\partial_{{\alpha_i}}S\right) S^{-1}\left(\partial_{{\alpha_j}}S\right)\right](x,\alpha_0)\pi(dx)\quad(1\le i,j\le p),\label{I_a:jump}\\
I_{b,c}^{(i,j)}(\theta_0;\theta_0)&=I_{b}^{(i,j)}(\theta_0)+I_{c}^{(i,j)}(\beta_0)\quad(1\le i,j\le q),\notag\\
I_b^{(i,j)}(\theta_0;\theta_0)&=\int \left(\partial_{{\beta_i}}b(x,\beta_0)\right)^\top S^{-1}(x,\alpha_0)\partial_{{\beta_j}}b (x,\beta_0)\pi(dx),\label{I_b:jump}\\
I_c^{(i,j)}(\beta_0;\beta_0)&=\iint_{A} \frac{\partial_{{\beta_i}}\Psi_{\beta_0}\partial_{{\beta_j}}\Psi_{\beta_0}}{\Psi_{\beta_0}}(y,x)dy\pi(dx),\label{I_c:jump}
\end{align}
with $\theta=\theta_0$.

\begin{enumerate}
\item[{[\bf A12]}]
$I(\theta_0;\theta_0)$ is non-singular for $\theta_0\in\mathrm{Int}(\Theta)$. 
\end{enumerate}

Finally, let us introduce a truncation function $\varphi_n$ to ensure the integrability of quasi-log likelihood functions in the next section.

\begin{enumerate}
\item[{[\bf A13]}] 
At least one of the following two conditions holds true.
\begin{enumerate}
\item[(i)] For $k=0,1,2,3$ and $l,m=0,1$, there exists a constant $C>0$ such that
\begin{align*}
\left|\partial_{{x}}^m\partial_{{y}}^l\partial_{{\beta}}^k\log\Psi_\beta(y,x)\right|\le C(1+|y|)^C(1+|x|)^C\quad((x,y,\beta)\in\mathbb{R}^d\times E\times \Theta_\beta).
\end{align*}
Moreover, there exists a constant $C'>0$, which depends on $C$, such that
\begin{align*}
\int_{B}|y|^c\sup_{\beta\in\Theta_\beta}\left|\partial_x\Psi_\beta(y,x)\right|dy\leq C'(1+|x|)^{C'}.
\end{align*}
In this case, we define the truncation function as $\varphi_n\equiv1$.

\item[(ii)]  There exists a sequence of real valued Borel functions $\{\varphi_n(x,y)\}_{n\in\mathbb{N}}$ on $\mathbb{R}^d\times E$, possessing the following properties:
$0\le\varphi_n\le1,\quad\varphi_n\to1\quad \pi(dx)\times dy $ -a.s., and there exists a constant $M>0$ such that $\varphi_n(x,y)=0$ whenever $(x,y)\in D_n$, where $D_n=D_n(M)$ is the set defined by
\begin{small}
\begin{align*}
D_n&=\bigcup_{k=0}^3\bigcup_{l=0}^1\left\{(x,y)\in\mathbb{R}^d\times E;\ \sup_{\beta\in\Theta_\beta}\left|\partial_{{x}}^l\partial_{{\beta}}^k\log\Psi_\beta(y,x)\right|\ge\varepsilon_n^{-((k+l)\lor1)}\cdot M(1+|x|)^M\right\}\\
&\qquad\bigcup\bigcup_{k=0}^3\left\{(x,y)\in\mathbb{R}^d\times E;\ \sup_{\beta\in\Theta_\beta}\left|\partial_y\partial_{{\beta}}^k\log\Psi_\beta(y,x)\right|\ge\varepsilon_n^{-(k+1)}\cdot M(1+|x|)^M(1+|y|)^M\right\}.
\end{align*}
\end{small}
Moreover, $\varphi_n$ is differentiable with respect to $x$ and $y$.  $\partial_x\varphi_n$ and $\partial_y\varphi_n$ are continuous in $x$ and $y$, respectively. In addition,
\begin{align*}
\partial_x\varphi_n=\partial_y\varphi_n=0\quad\text{on }\ D_n,\quad\sup_{(x,y)\in\mathbb{R}^d\times E}|\partial_x\varphi_n|+\sup_{(x,y)\in\mathbb{R}^d\times E}|\partial_y\varphi_n|=O(\varepsilon_n^{-1}).\
\end{align*}
\end{enumerate} 
\end{enumerate}

Next, we introduce a real valued sequence $\varepsilon_n$ to ensure the asymptotic properties of the estimators for $\theta$, and set the balance conditions with $n, h_n$ and $\varepsilon_n$.
\begin{enumerate}
    \item[{[\bf B1]}]
    $\varepsilon_n\to0, nh_n\varepsilon_n^4\to\infty, h_n\varepsilon_n^{-8}\to0\ as\ n\to\infty$.
    \item[{[\bf B2]}]
    $\varepsilon_n\to0, nh_n\varepsilon_n^4\to\infty, nh_n^2\varepsilon_n^{-4}\to0\ as\ n\to\infty$.
    \item[{[\bf B3]}]
    There exists a constant $\delta\in(0,1)$ such that $nh_n^{1+\delta}\to 0\ as\ n\to\infty$.
\end{enumerate}
\begin{rmk}We can give some examples of the values $h_n,\varepsilon_n$ which fulfill the balance conditions {[\bf B1]} and {[\bf B2]}.
    Example 1 : $h_n=n^{-1/4}$, $\varepsilon_n=n^{-1/64}$.
    Example 2 : $h_n=n^{-2/3}$, $\varepsilon_n=n^{-1/16}$.
    In particular, Example 1 satisfies the condition {[\bf B1]}, but not {[\bf B2]}.
\end{rmk}
\begin{rmk}\label{balance}
    The condition {\bf{[B2]}} implies {\bf{[B1]}}, but not vice versa.
    This is because, under the condition {\bf{[B2]}}, $h_n\varepsilon_n^{-8}=(nh_n^2\varepsilon_n^{-4})/(nh_n\varepsilon_n^4)\to 0$, while Example 1 shows that  {\bf{[B1]}}$\Rightarrow${\bf{[B2]}} is not valid.
\end{rmk}
\begin{rmk}\label{balance2}
    The condition {\bf{[B1]}} implies $h_n\to 0,\ nh_n\to \infty$, and {\bf{[B2]}} implies $nh_n^2\to 0$.

    This is because, under the condition {\bf{[B1]}}, $h_n=(h_n\varepsilon_n^{-8})\cdot\varepsilon_n^8\to 0$, $nh_n=(nh_n\varepsilon_n^4)\cdot \varepsilon^{-4}\to\infty$, and also
    under the condition {\bf{[B2]}}, $nh_n^2=(nh_n\varepsilon_n^{-4})\cdot\varepsilon_n^4\to 0$.
\end{rmk}
According to Remark \ref{balance2}, the condition {\bf{[B2]}} implies more strict condition $nh_n^2\to 0$ for the consistency of the estimator for $\theta$. Therefore, we need the condition {\bf{[B1]}} for the proof of the consistency. And also, according to Remark \ref{balance}, the proposition, which holds under the condition {\bf{[B1]}}, holds under the condition {\bf{[B2]}}.


\section{Quasi-maximum likelihood estimation}\label{sec4.3:m}
\subsection{Joint and adaptive estimator}
Firstly, we introduce quasi-log likelihood function for 
joint estimation.
 Let $D_1$, $D_2$ be constants satisfying that $D_1,D_2>0$ and $0<\rho_1,\rho_2<1/2$, 
\begin{align}
&l_n(\theta):=\bar{l}_n(\theta)+\tilde{l}_n(\beta),\label{l:joint}\\
&\bar{l}_n(\theta):=-\frac{1}{2}\sum_{i=1}^{n}\left\{h_n^{-1}(\bar{X}_{i,n}(\beta))^\top S_{i-1}^{-1}(\alpha)\bar{X}_{i,n}(\beta)+\log\det S_{i-1}(\alpha)\right\}\boldsymbol{1}_{\{|\Delta X_i^n|\le D_1h_n^{\rho_1}\}},\label{l-bar:jump}\\
&\tilde{l}_n(\beta):=\sum_{i=1}^{n}\left(\log\Psi_\beta(\Delta X_i^n,X_{t_{i-1}^n})\right)\varphi_n(X_{t_{i-1}^n},\Delta X_i^n)\boldsymbol{1}_{\{|\Delta X_i^n|> D_2h_n^{\rho_2}\}}\notag\\
&\qquad\qquad\qquad-h_n\sum_{i=1}^{n}\int_{B}\Psi_\beta(y,X_{t_{i-1}^n})\varphi_n(X_{t_{i-1}^n},y)dy.\label{l-tilde:jump}
\end{align}
Secondly, we introduce quasi-log likelihood function for adaptive estimation.
Let constants $D_1,D_2,D_3>0$ and $0<\rho_1,\rho_2,\rho_3<1/2$, 
\begin{align}
&l_n^{(1)}(\alpha):=-\frac{1}{2}\sum_{i=1}^{n}\left\{h_n^{-1}(\Delta X_i^n)^\top S_{i-1}^{-1}(\alpha)\Delta X_i^n+\log\det S_{i-1}(\alpha)\right\}\boldsymbol{1}_{\{|\Delta X_i^n|\le D_1h_n^{\rho_1}\}},\label{l1:jump}\\
&l_n^{(2)}(\beta|\bar{\alpha}):=\bar{l}_n^{(2)}(\beta|\bar{\alpha})+\tilde{l}_n^{(2)}(\beta),\notag\\
&\bar{l}_n^{(2)}(\beta|\bar{\alpha}):=-\frac{1}{2h_n}\sum_{i=1}^{n}(\bar{X}_{i,n}(\beta))^\top S_{i-1}^{-1}(\bar{\alpha})\bar{X}_{i,n}(\beta)\boldsymbol{1}_{\{|\Delta X_i^n|\le D_3h_n^{\rho_3}\}},\label{l2-bar:jump}\\
&\tilde{l}_n^{(2)}(\beta):=\sum_{i=1}^{n}\left(\log\Psi_\beta(\Delta X_i^n,X_{t_{i-1}^n})\right)\varphi_n(X_{t_{i-1}^n},\Delta X_i^n)\boldsymbol{1}_{\{|\Delta X_i^n|> D_2h_n^{\rho_2}\}}\notag\\
&\qquad\qquad\qquad-h_n\sum_{i=1}^{n}\int_{B}\Psi_\beta(y,X_{t_{i-1}^n})\varphi_n(X_{t_{i-1}^n},y)dy.\label{l2-tilde:jump}
\end{align}
Our joint and adaptive quasi-log likelihood function functions are based on the quasi-log likelihood function in
\citet{Shimizu-Yoshida_JP,Shimizu-Yoshida} and \citet{Ogihara-Yoshida}. For joint estimation, we modified their quasi-log likelihood function by changing their filters into two sets: ${\{|\Delta X_i^n|\le D_1h_n^{\rho_1}\}}$ and ${\{|\Delta X_i^n|> D_2h_n^{\rho_2}\}}$.
For adaptive estimation, we divide their quasi-log likelihood function into the two step functions which enable us to optimize the parameters $\alpha$ and $\beta$ separately. Moreover, 
we distinguish the filters for each function, and three filters ${\{|\Delta X_i^n|\le D_1h_n^{\rho_1}\}}$, ${\{|\Delta X_i^n|\le D_3h_n^{\rho_3}\}}$ and ${\{|\Delta X_i^n|> D_2h_n^{\rho_2}\}}$ are adopted. 
we estimate both parameters $\alpha$ and $\beta$ more accurately since we choose thresholds from larger region than region 
By using these two filters for joint estimation and three filters for adaptive estimation,  
it can be expected that
we estimate both parameters $\alpha$ and $\beta$ more accurately since we choose thresholds from a larger region than that 
of estimation in \citet{Shimizu-Yoshida_JP,Shimizu-Yoshida} and \citet{Ogihara-Yoshida}.

Using joint and adaptive quasi-log likelihood functions above, we define our joint estimator $\hat{\theta}_n=(\hat{\alpha}_n,\hat{\beta}_n)$ and adaptive estimator $\check{\theta}_n=(\check{\alpha}_n,\check{\beta}_n)$ for $\theta=(\alpha,\beta)$ as follows:
\begin{align*}
\hat{\theta}_{n}&:=\mathrm{argmax}_{\theta\in\Theta}l_n(\theta),\\
\check{\alpha}_{n}:=\mathrm{argmax}_{\alpha\in\Theta_\alpha}&l_n^{(1)}(\alpha),\
\check{\beta}_{n}:=\mathrm{argmax}_{\beta\in\Theta_\beta}l_n^{(2)}(\beta|\check{\alpha}_{n}).
\end{align*}
First, we state the consistency for our estimators $\hat{\theta}_n$, $\check{\theta}_n$.
Then, we adopt the balance condition {[\bf B1]} for $\varepsilon_n$.
In order to obtain the consistency under {\bf[A13]}-(ii), we define the following set:
\begin{align*}
B_1(k):=\left\{\rho\in(0,\frac{1}{2});\ h_n^\rho\epsilon_n^{-k}\to 0\ as\ n\to\infty\right\}. 
\end{align*}
Our theorem and corollary for consistency of 
$\check{\theta}_n$ and $\hat{\theta}_n$
are the following. 
\begin{thm}\label{thm1-2:jump}Assume {\bf{[A1]}}-{\bf{[A11]}},{\bf{[B1]}}, and  either [$\textbf{C}_1 \textbf{1}$] or [$\textbf{C}_1 \textbf{2}$], which are the following:
\begin{enumerate}
\item[{[$\textbf{C}_1 \textbf{1}$]}]  
Fulfill $\rho_1,\rho_2,\rho_3\in(0,\frac{1}{2})$ and {\bf{[A13]}}-(i). 
\item[{[$\textbf{C}_1 \textbf{2}$]}]  
Fulfill $\rho_1,\rho_3\in(0,\frac{1}{2})$, $\rho_2\in B_1(1)$ and {\bf{[A13]}}-(ii). 
\end{enumerate}
Then, $\check{\theta}_n \overset{P}{\longrightarrow} \theta_0$.
\end{thm}
\begin{col}\label{thm1-1:jump}
Assume {\bf{[A1]}}-{\bf{[A11]}}, {\bf{[B1]}}, and either {[$\textbf{D}_1 \textbf{1}$]} or {[$\textbf{D}_1 \textbf{2}$]}, which are the following:
\begin{enumerate}
\item[{[$\textbf{D}_1 \textbf{1}$]}]  
Fulfill $\rho_1,\rho_2\in(0,\frac{1}{2})$ and {\bf{[A13]}}-(i). 
\item[{[$\textbf{D}_1 \textbf{2}$]}]   
Fulfill $\rho_1\in(0,\frac{1}{2})$, $\rho_2\in B_1(1)$ and {\bf{[A13]}}-(ii).
\end{enumerate}
Then, $\hat{\theta}_n \overset{P}{\longrightarrow} \theta_0$.
\end{col}
\begin{rmk}\label{notempty:thm1}
Under the condition {\bf{[B1]}}, 
$B_1(1)\neq \emptyset$.
\end{rmk}
Next, let us state the asymptotic normality for our estimators $\hat{\theta}_n$, $\check{\theta}_n$. In order to obtain this, we need to assume that the balance condition {\bf [B3]}, and  for $\varepsilon_n$, we assume the balance condition {\bf{[B2]}}.
In addition, to ensure asymptotic normality under the {\bf[A13]}-(ii), we define the following sets:
\begin{align*}
    &B_2:= \left\{\rho\in(0,\frac{1}{2});\ nh_n^{1+2\rho}\epsilon_n^{-2}\to 0\ as\ n\to\infty\right\},\\
    &B_3(\delta):= \left\{\rho\in(0,\frac{1}{2});\ h_n^{3\rho-\frac{1}{2}}\epsilon_n^{-1}\to 0\ as\ n\to\infty,\ \rho\geq \frac{1+\delta}{6},\ \rho>\frac{1}{5}\right\}. 
\end{align*}
Our theorem and corollary for asymptotic normality are the following.
\begin{thm}\label{thm2-2:jump}
Assume {\bf{[A1]}}-{\bf{[A12]}}, {\bf{[B2]}}, {\bf{[B3]}}, and either [$\textbf{C}_2 \textbf{1}$] or [$\textbf{C}_2 \textbf{2}$], which are the following:
\begin{enumerate}
\item[{[$\textbf{C}_2 \textbf{1}$]}] 
Fulfill $\rho_1\in(\frac{1}{5},\frac{1}{2})\cap[\frac{1+\delta}{6},\frac{1}{2})$, $\rho_2\in[\frac{\delta}{2},\frac{1}{2})$, $\rho_3\in[\frac{\delta}{4},\frac{1}{2})$ and
{\bf{[A13]}}-(i). 
\item[{[$\textbf{C}_2 \textbf{2}$]}]  
Fulfill $\rho_1\in B_3(\delta)$, $\rho_2\in B_2$, $\rho_3\in [(\frac{\delta}{4}\wedge\frac{1}{16}),\frac{1}{2})$ and {\bf{[A13]}}-(ii). 
\end{enumerate}
Then,  
$(\sqrt{n}(\check{\alpha}_n-\alpha_0), \sqrt{nh_n}(\check{\beta}_n-\beta_0))\overset{d}{\to}N(0,I(\theta_0;\theta_0)^{-1})$.
\end{thm}

\begin{col}\label{thm2-1:jump}
Assume {\bf{[A1]}}-{\bf{[A12]}}, {\bf{[B2]}}, {\bf{[B3]}}, and either [$\textbf{D}_2 \textbf{1}$] or [$\textbf{D}_2 \textbf{2}$], which are the following:
\begin{enumerate}
\item[{[$\textbf{D}_2 \textbf{1}$]}]
Fulfill $\rho_1\in(\frac{1}{5},\frac{1}{2})\cap[\frac{1+\delta}{6},\frac{1}{2})$, $\rho_2\in[\frac{\delta}{2},\frac{1}{2})$ and {\bf{[A13]}}-(i). 
\item[{[$\textbf{D}_2 \textbf{1}$]}]
Fulfill $\rho_1\in B_3(\delta)$, $\rho_2\in B_2$ and {\bf{[A13]}}-(ii). 
\end{enumerate}
Then,  
$(\sqrt{n}(\hat{\alpha}_n-\alpha_0), \sqrt{nh_n}(\hat{\beta}_n-\beta_0))\overset{d}{\to}N(0,I(\theta_0;\theta_0)^{-1})$.
\end{col}
\begin{rmk}\label{notempty:thm2}
Under the conditions {\bf[B2]}, {\bf{[B3]}}, $B_2\neq\emptyset,\ B_3(\delta)\neq\emptyset$.
\end{rmk}
\begin{rmk}\label{subset:thm2}
Under the conditions {\bf[B2]}, {\bf{[B3]}}, $B_2\subset B_1(3)$. 
\end{rmk}

\begin{rmk}\label{efficient:jump}
The proposed adaptive estimator $\hat{\theta}_n$ has asymptotic efficiency, see  \citet{Ogihara-Uehara}.
\end{rmk}

\begin{rmk}\label{thmcomp:jump}
Let us compare our results with those of the previous studies, \citet{Shimizu-Yoshida_JP,Shimizu-Yoshida} and \citet{Ogihara-Yoshida})
\begin{enumerate}
\item[(i)]  Compared with the assumption {[\bf A6]} in \citet{Shimizu-Yoshida_JP,Shimizu-Yoshida}, our assumption {[\bf A5]} allows general jump densities such as a normal distribution. On the other hand, the condition $nh_n^2\to 0$ is needed for the proof of the asymptotic normality in \citet{Shimizu-Yoshida_JP,Shimizu-Yoshida}. We need a little more strict condition $nh_n^{1+\delta}\to 0$. Moreover, our conditions regarding $\rho$ are stronger than their conditions. 
In particular, the sets $B_1(k),B_2$ and $B_3(\delta)$, which may restrict the range of $\rho$, are unique to our setting. 
If we assume {[\bf A1]}-{[\bf A11]} in \citet{Shimizu-Yoshida_JP,Shimizu-Yoshida} instead of ours, then we can prove Theorem \ref{thm1-2:jump}, Theorem \ref{thm2-2:jump}, Corollary \ref{thm1-1:jump} and Corollary \ref{thm2-1:jump} without $B_1(k),B_2$ and $B_3(\delta)$. If the constant $\gamma$, which satisfies assumption {[\bf A6]} in \citet{Shimizu-Yoshida_JP,Shimizu-Yoshida}, 
is large, then the range of their $\rho$ may become larger than ours.  
\item[(ii)] Compared with the conditions for $\varepsilon_n$ in \citet{Ogihara-Yoshida}, our conditions for $\varepsilon_n$ are mild.
Moreover, our argument is more general than theirs regarding the range of choices for $\rho$ and $\delta$. In actual, we impose the condition $nh_n^{1+\delta}\to 0$ for $\delta\in(0,1)$, while they assume $n^{-3/5}\leq h_n\leq n^{-4/7}$. This means that in their setting, $\frac{2}{3}<\delta<\frac{3}{4}$, and our range of $\rho$ is larger than theirs. The reason for this is that their study considers the convergence of moments, 
while our aim is to show the convergence in distribution of the proposed estimators.
\end{enumerate}
\end{rmk}

\subsection{The case in which the drift and jump parameter are split independently}\label{sec4.4:m}

We discuss a model which is expressed as follows:
\begin{equation}\label{jump-diff:model2}
\begin{cases}
dX_t = b(X_{t-},\beta) dt + a(X_{t-},\alpha)dW_t+\int_{E} c(X_{t-},z,\gamma)p(dt,dz) \quad t\in[0,T],\\
X_0 = x_0,
\end{cases}
\end{equation}
where $\alpha\in\Theta_\alpha\subset\mathbb{R}^p$, $\beta\in\Theta_\beta\subset\mathbb{R}^q$, $\gamma\in\Theta_\gamma\subset\mathbb{R}^r$, the parameter $\beta$ is independent of $\gamma$ and the jump density $f$ is written by $f_\gamma(z)=\lambda(\gamma)F_\gamma(z)$.
For this setting, we introduce an adaptive estimators for $\theta=(\alpha,\beta,\gamma)$.
By using quasi-log likelihood functions\eqref{l1:jump}-\eqref{l2-tilde:jump}, our adaptive estimators are defined as follows:
\begin{equation*}
\check{\alpha}_{n}:=\mathrm{argmax}_{\alpha\in\Theta_\alpha}l_n^{(1)}(\alpha), \check{\beta}_{n}:=\mathrm{argmax}_{\beta\in\Theta_\beta}\bar{l}_n^{(2)}(\beta|\check{\alpha}_{n}),\ \check{\gamma}_{n}:=\mathrm{argmax}_{\gamma\in\Theta_\gamma}\tilde{l}_n^{(2)}(\gamma),\ \check{\theta}_n=(\check{\alpha}_{n},\check{\beta}_{n},\check{\gamma}_{n}).
\end{equation*}
For the order of calculations, we note the following two points. {{First}} $\check{\beta}_n$ must be calculated after the calculation for $\check{\alpha}_n$. {{Second}} we can calculate $\check{\gamma}_n$ in any order because the quasi-log likelihood function for $\check{\gamma}_n$ is independent of $(\check{\alpha}_n$, $\check{\beta}_n)$.

Let true parameter values $\theta_0=(\alpha_0,\beta_0,\gamma_0)$ and 
$J(\theta;\theta_0)$ be a $(p+q+r)\times (p+q+r)$-matrix  such that
\begin{align*}
J(\theta;\theta_0)=\begin{pmatrix}
I_a(\alpha;\alpha_0)&0&0\\
0&I_{b}((\alpha,\beta);(\alpha_0,\beta_0))&0\\
0&0&I_{c}(\gamma;\gamma_0)
\end{pmatrix},
\end{align*}
where $I_a(\alpha;\alpha_0)$, $I_b((\alpha,\beta);(\alpha_0,\beta_0))$ and $I_c(\gamma;\gamma_0)$ are defined by \eqref{I_a:jump}-\eqref{I_c:jump}. In order to obtain asymptotic properties for this model, we fix the assumptions {\bf{[A5]}}-{\bf{[A7]}}, {\bf{[A9]}}-{\bf{[A13]}}.
Except for {\bf{[A9]}} and {\bf{[A12]}}, 
we replace $\beta$ with $\gamma$ in these assumptions.
For {\bf{[A9]}} and {\bf{[A12]}},
we set the following new assumptions:
\begin{enumerate}
\item[{[\bf A9']}] 
$\det S(x,\alpha) =\det S(x, \alpha_{0} ) \ \text{for} \ \text{a.s.}\ \text{all} \ x \Longrightarrow \alpha = \alpha_{0}.$

$b(x,\beta) = b(x, \beta_{0} )\ \text{for} \ \text{a.s.}\ \text{all} \ \ (x,y) \Longrightarrow \beta = \beta_{0}.$

$\Psi_\gamma(y,x)=\Psi_{\gamma_0}(y,x)  \ \text{for} \ \text{a.s.}\ \text{all} \  \ (x,y) \Longrightarrow \gamma = \gamma_{0}.$
\item[{\bf{[A12']}}]
$J(\theta_0;\theta_0)$ is non-singular for $\theta_0\in\mathrm{Int}(\Theta)$.
\end{enumerate}
Under these fixed assumptions, the following consistency and asymptotic normality are hold.
\begin{thm}\label{thm3:jump}
Assume {\bf{[A1]}}-{\bf{[A8]}}, {\bf{[A9']}}, {\bf{[A10]}}, {\bf{[A11]}}, {\bf{[B1]}}, and either [$\textbf{C}_1\textbf{1}$] or  [$\textbf{C}_1\textbf{2}$] of Theorem \ref{thm1-2:jump}.
Then, $\check{\theta}_n\to\theta_0$.
\end{thm}

\begin{thm}\label{thm4:jump}
Assume {\bf{[A1]}}-{\bf{[A8]}}, {\bf{[A9']}}, {\bf{[A10]}}, {\bf{[A11]}}, {\bf{[A12']}}, {\bf{[B2]}}, {\bf{[B3]}}, and either [$\textbf{C}_2\textbf{1}$] or [$\textbf{C}_2\textbf{2}$] of Theorem \ref{thm2-2:jump}.
Then, $(\sqrt{n}(\check{\alpha}_n-\alpha_0), \sqrt{nh_n}(\check{\beta}_n-\beta_0),\sqrt{nh_n}(\check{\gamma}_n-\gamma_0))\overset{d}{\to}N(0,J(\theta_0;\theta_0)^{-1})$.
\end{thm}
In a similar way to the proof of 
Theorems \ref{thm1-2:jump} and \ref{thm2-2:jump}, we can prove 
Theorems \ref{thm3:jump} and \ref{thm4:jump}, respectively. 
By dividing the argument for $l_n^{(2)}$ into that for $\bar{l}_n^{(2)}$ and $\tilde{l}_n^{(2)}$, in particular, we can show the above statements. Therefore, we omit detailed proofs.


\section{Quasi-likelihood ratio test}\label{sec:test}
We consider a statistical hypothesis testing problem for model (\ref{jump-diff:model}) as follows. 
Let $k$ and $l$ be known integer values.
\begin{align}\label{test:hypo}
\begin{cases}
H_0:\alpha^{(1)}=\cdots=\alpha^{(k)}=0,\beta^{(1)}=\cdots=\beta^{(l)}=0, \\
H_1:\text{not}\ H_0,
\end{cases}
\end{align}
where $1\leq k\leq p$, $1\leq l\leq q$. 
We set $\Theta_0=\{\theta\in\Theta\ |\ \theta \ \text{satisfies $H_0$} \}$, 
$\Theta_{\alpha_0}=\{ \alpha\in\Theta_\alpha\ |\  \alpha \  \text{satisfies $H_0$} \}$ and
$\Theta_{\beta_0}=\{\beta\in\Theta_\beta\ |\ \beta \ \text{satisfies $H_0$} \}$. 
We assume that $\Theta_0$, $\Theta_{\alpha_0}$ and $\Theta_{\beta_0}$ are compact convex sets.
More general cases, in which the null hypothesis is expressed as the form $H_0':\ g_1(\alpha)=0,\ldots,g_k(\alpha)=0\ \mathrm{and}\ h_1(\beta)=0,\ldots,h_l(\beta)=0$ with some smooth real valued functions $g_1,\ldots,g_k$ and $h_1,\ldots,h_l$, can be put into $H_0$ by a reparametrization. Let $\tilde{\theta}_n$ and $\tilde{\theta}^*_n$ be estimators on $\Theta$ and $\Theta_0$, respectively.
Then, we define the quasi-likelihood ratio test statistics $\Lambda_n$ with the joint quasi-log likelihood function $l_n$ defined by \eqref{l:joint} as follows:
\begin{align}\label{test:stat}
\Lambda_n(\tilde{\theta}_n,\tilde{\theta}^*_n)=-2(l_n(\tilde{\theta}^*_n)-l_n(\tilde{\theta}_n)),
\end{align}
and we define $\hat{\theta}_n^*=(\hat{\alpha}_n^*,\hat{\beta}_n^*)$ and $\check{\theta}_n^*=(\check{\alpha}_n^*,\check{\beta}_n^*)$ as follows:
\begin{align*}
\hat{\theta}_{n}^*&:=\mathrm{argmax}_{\theta\in\Theta_0}l_n(\theta),\\
\check{\alpha}_{n}^*:=\mathrm{argmax}_{\alpha\in\Theta_{\alpha_0}}&l_n^{(1)}(\alpha),\
\check{\beta}_{n}^*:=\mathrm{argmax}_{\beta\in\Theta_{\beta_0}}l_n^{(2)}(\beta|\check{\alpha}_{n}^*).
\end{align*}

\subsection{Asymptotic distribution of test statistics}
We state the asymptotic distribution of the test statistics $\Lambda_n$ under $H_0$. We make the following assumption to obtain this.
\begin{enumerate}
\item[{[{\bf T1}]}] Let $\hat{\theta}_n$ and $\hat{\theta}_n^*$ be the joint quasi-maximum likelihood estimators on $\Theta$ and $\Theta_0$, respectively, and $\tilde{\theta}_n,\tilde{\theta}_n^*$ be the estimators on $\Theta$ and $\Theta_0$, respectively. For all $\theta\in\Theta$, it holds that $D_n^\frac{1}{2}(\hat{\theta}_n-\tilde{\theta}_n)=o_p(1)$, and that $ D_n^\frac{1}{2}(\hat{\theta}^{*}_n-\tilde{\theta}^*_n)=o_p(1)$ under $H_0$.
\end{enumerate}
Our theorem regarding asymptotic distribution of test statistics is the following:
\begin{thm}\label{thm5:H0test}
Assume {\bf{[A1]}}-{\bf{[A12]}}, {\bf{[B2]}}, {\bf{[B3]}}, {[{\bf T1}]}, and either {[$\textbf{D}_2 \textbf{1}$]} or {[$\textbf{D}_2 \textbf{2}$]} of Corollary \ref{thm2-1:jump}. 
Then, under $H_0$,
$\Lambda_n(\tilde{\theta}_n,\tilde{\theta}^*_n)\overset{d}{\to}\chi_{k+l}^2$.
\end{thm}
From the view point of numerical analysis, the simultaneous estimators are quite unstable when the dimension of the parameter space is large. On the other hand, the adaptive estimators have good behaviors.
The following proposition shows our adaptive estimators can be applied to Theorem \ref{thm5:H0test}.
\begin{prop}\label{prop:test}
Assume {\bf{[A1]}}-{\bf{[A12]}}, {\bf{[B2]}}, {\bf{[B3]}}, and either ``{[$\textbf{C}_2 \textbf{1}$]} of Theorem \ref{thm2-2:jump} and {[$\textbf{D}_2 \textbf{1}$]} of Corollary \ref{thm2-1:jump}'' or ``{[$\textbf{C}_2 \textbf{2}$]} of Theorem \ref{thm2-2:jump} and {[$\textbf{D}_2 \textbf{2}$]} of Corollary \ref{thm2-1:jump}''.
Then, the adaptive estimator $(\tilde{\theta}_n,\tilde{\theta}_n^*)=(\check{\theta}_n,\check{\theta}_n^*)$ satisfies [{\bf T1}]. 
\end{prop}
\begin{rmk}
Proposition \ref{prop:test} shows that we can choose up to five thresholds to compose 
the test statistics $\Lambda_n$; in five thresholds, three thresholds are for adaptive estimators, and two thresholds are for joint quasi-log likelihood function. 
\end{rmk}
\begin{col}\label{col:H0test}
Assume {\bf{[A1]}}-{\bf{[A12]}}, {\bf{[B2]}}, {\bf{[B3]}}, and either ``{[$\textbf{C}_2 \textbf{1}$]} of Theorem \ref{thm2-2:jump} and {[$\textbf{D}_2 \textbf{1}$]} of Corollary \ref{thm2-1:jump}'' or ``{[$\textbf{C}_2 \textbf{2}$]} of Theorem \ref{thm2-2:jump} and {[$\textbf{D}_2 \textbf{2}$]} of Corollary \ref{thm2-1:jump}''.
Then, under $H_0$,
$\Lambda_n(\check{\theta}_n,\check{\theta}^*_n)\overset{d}{\to}\chi_{k+l}^2$.
\end{col}
By Theorem \ref{thm5:H0test} and Proposition \ref{prop:test}, the proof of Corollary \ref{col:H0test} is obvious.
\subsection{Consistency of test}
Next, we consider alternative hypothesis $H_1$.
 Let $\theta_1=(\alpha_1,\beta_1)$, which is the true parameter under $H_1$,  and $\pi^*$ be invariant probability measure for $\theta_1$.  
 We define $\theta^*$ as follows:
\begin{align*}
 \theta^*=(\alpha^*,\beta^*), \ \ \
\alpha^*=\argsup_{\alpha\in\Theta_{\alpha_0}}U_1^*(\alpha,\alpha_1),\ \ \ \beta^*=\argsup_{\beta\in\Theta_{\beta_0}}V_{\beta_1}^*(\alpha^*,\beta),
\end{align*}
where
\begin{align}
U_1^*(\alpha,\alpha_1)&:=-\frac{1}{2}\int_{}^{}{\left\{\mathrm{tr}\left(S^{-1}(x,\alpha)S(x,\alpha_1)\right)+\log\det S(x,\alpha)\right\}}\pi^*(dx),\\
\bar{U}_{\beta_1}^{(2)^*}(\alpha,\beta)&:= -\frac{1}{2}\int{(b(x,\beta)-b(x,\beta_1))^\top S^{-1}(x,\alpha)(b(x,\beta)-b(x,\beta_1))}\pi^*(dx),\\
\tilde{U}_{\beta_1}^{(2)^*}(\beta)&:=\iint_A{\left\{(\log\Psi_\beta (y,x))\Psi_{\beta_1} (y,x)-\Psi_\beta (y,x)\right\}}dy\pi^*(dx),\\
V_{\beta_1}^*(\alpha,\beta)&:=\bar{U}_{\beta_1}^{(2)^*}(\alpha,\beta)+\tilde{U}_{\beta_1}^{(2)^*}(\beta)-\tilde{U}_{\beta_1}^{(2)^*}(\beta_1).
\end{align}
We make the following assumptions to obtain consistency of the test:
\begin{itemize}
\item[{\bf [E1]}]
\begin{enumerate}
\item[ (i)]
For any $\varepsilon>0$, 
\begin{align*}
\sup_{\{\alpha\in\Theta_{\alpha_0}:|\alpha-\alpha^*|\geq\varepsilon\}}(U_1^*(\alpha,\alpha_1)-U_1^*(\alpha^*,\alpha_1))<0.
\end{align*}
\item[ (ii)]
For any $\varepsilon>0$, 
\begin{align*}
\sup_{\{\beta\in\Theta_{\beta_0}:|\beta-\beta^*|\geq\varepsilon\}}(V_{\beta_1}^*(\alpha^*,\beta)-V_{\beta_1}^*(\alpha^*,\beta^*))<0.
\end{align*}
\end{enumerate}
\item[{\bf [E2]}]
For any $\theta\in\Theta$, $I(\theta;\theta_1)$ is non-singular.
\end{itemize}
\begin{enumerate}
\item[{[{\bf T2}]}] Let $\hat{\theta}_n$ be the joint quasi-log likelihood estimator on $\Theta$, and $\tilde{\theta}_n$ and $\tilde{\theta}_n^*$ be the estimators on $\Theta$ and $\Theta_0$, respectively. For all $\theta\in\Theta$, it holds that $D_n^\frac{1}{2}(\hat{\theta}_n-\tilde{\theta}_n)=o_p(1)$, and that $\tilde{\theta}_n\overset{P}{\to}\theta_1$ and $\tilde{\theta}^*_n\overset{P}{\to}\theta^*$ under $H_1$.
\end{enumerate}
For $\varepsilon\in(0,1)$, $\chi_{k+l,\varepsilon}^2$ represents the upper $\varepsilon$ point of $\chi_{k+l}^2$.
Our theorem for consistent test is the following:
\begin{thm}\label{thm6:H1test}
Assume {\bf{[A1]}}-{\bf{[A12]}}, {\bf{[B2]}}, {\bf{[B3]}}, [{\bf E1}], [{\bf E2}], {[{\bf T2}]}, and either {[$\textbf{D}_2 \textbf{1}$]} or {[$\textbf{D}_2 \textbf{2}$]} of Corollary \ref{thm2-1:jump}.
Then, under $H_1$, $P(\Lambda_n(\tilde{\theta}_n,\tilde{\theta}^*_n)>\chi_{k+l,\epsilon}^2)\to 1$.
\end{thm}
The following proposition shows our adaptive estimators can be applied to Theorem \ref{thm6:H1test}.
\begin{prop}
Assume the assumption of Proposition \ref{prop:test}. Moreover, assume {\bf{[E1]}} and {\bf{[E2]}}. Then, the adaptive estimator $(\tilde{\theta}_n,\tilde{\theta}_n^*)=(\check{\theta}_n,\check{\theta}^*_n)$ satisfies [{\bf T2}].
\end{prop}\label{prop:H1test}
\begin{rmk}
Theorem \ref{thm6:H1test} and Proposition \ref{prop:H1test} show that 
 the test with proposed test statistics based on our adaptive estimator is consistent.
 
\end{rmk}
\begin{col}\label{col:H1test}
Assume {\bf{[A1]}}-{\bf{[A12]}}, {\bf{[B2]}}, {\bf{[B3]}}, [{\bf E1}], [{\bf E2}], and either ``{[$\textbf{C}_2 \textbf{1}$]} of Theorem \ref{thm2-2:jump} and {[$\textbf{D}_2 \textbf{1}$]} of Corollary \ref{thm2-1:jump}'' or ``{[$\textbf{C}_2 \textbf{2}$]} of Theorem \ref{thm2-2:jump} and {[$\textbf{D}_2 \textbf{2}$]} of Corollary \ref{thm2-1:jump}''. Then, under $H_1$, $P(\Lambda_n(\check{\theta}_n,\check{\theta}^*_n)>\chi_{r+l,\epsilon}^2)\to 1$.
\end{col}
By Theorem \ref{thm6:H1test} and Proposition \ref{prop:H1test}, the proof of Corollary \ref{col:H1test} is obvious.

\subsection{The case in which the drift and jump parameter are split independently}
We consider a statistical hypothesis testing problem for model (\ref{jump-diff:model2})  as follows.
Let $k$, $l$ and $m$ be known  integer values.
\begin{align}
\begin{cases}
H_0:\alpha^{(1)}=\cdots=\alpha^{(k)}=0,\beta^{(1)}=\cdots=\beta^{(l)}=0,\gamma^{(1)}=\cdots=\gamma^{(m)} =0, \\
H_1:\text{not}\ H_0,
\end{cases}
\end{align}
 where $1\leq k\leq p$, $1\leq l\leq q$. 
 Let $\Theta_0=\{\theta\in\Theta\ |\  \theta \ \text{satisfies $H_0$} \}$,
$\Theta_{\alpha_0}=\{\alpha\in\Theta_\alpha\ |\  \alpha \ \text{satisfies $H_0$} \}$,
$\Theta_{\beta_0}=\{\beta\in\Theta_\beta\ |\  \beta \ \text{satisfies $H_0$} \}$ and
$\Theta_{\gamma_0}=\{\gamma\in\Theta_\gamma\ |\  \gamma \ \text{satisfies $H_0$} \}$. 
 We suppose that  $\Theta_0$,  $\Theta_{\alpha_0}$, $\Theta_{\beta_0}$ and $\Theta_{\gamma_0}$
are compact convex sets. More general cases, in which the null hypothesis is expressed as the form $H_0':\ g_1(\alpha)=0,\ldots,g_k(\alpha)=0\ \mathrm{and}\ h_1(\beta)=0,\ldots,h_l(\beta)=0\ \mathrm{and}\ i_1(\gamma)=0,\ldots,i_m(\gamma)=0$ with the some smooth real valued functions $g_1,\ldots,g_k$, $h_1,\ldots,h_l$ and $i_1,\ldots,i_m$, can be put into $H_0$ by a  reparametrization.
Let $\tilde{\theta}_n$ and $\tilde{\theta}^*_n$ be estimators on $\Theta$ and $\Theta_0$, respectively. 
Then, we define the quasi-likelihood ratio test statistics $\Lambda_n$ with the joint quasi-log likelihood function $l_n(\theta)=\bar{l}_n(\alpha,\beta)+\tilde{l}_n(\gamma)$ defined by (\ref{l-bar:jump}) and (\ref{l-tilde:jump}) as follows:
\begin{align}
\Lambda(\tilde{\theta}_n,\tilde{\theta}^*_n)=-2(l_n(\tilde{\theta}^*_n)-l_n(\tilde{\theta}_n)),
\end{align}
and we define $\hat{\theta}_n^*$ and $\check{\theta}_n^*$ as follows:
\begin{align*}
\hat{\theta}_{n}^*&:=\mathrm{argmax}_{\theta\in\Theta_0}l_n(\theta),\\
\check{\alpha}_{n}^*:=\mathrm{argmax}_{\alpha\in\Theta_{\alpha_0}}l_n^{(1)}(\alpha),\
&\check{\beta}_{n}^*:=\mathrm{argmax}_{\beta\in\Theta_{\beta_0}}\bar{l}_n^{(2)}(\beta|\check{\alpha}_{n}),\ 
\check{\gamma}_{n}^*:=\mathrm{argmax}_{\gamma\in\Theta_{\gamma_0}}\tilde{l}_n^{(2)}(\gamma).
\end{align*}
Then, we have the following corollary.
\begin{col}\label{col:H0test'}
Assume {\bf{[A1]}}-{\bf{[A8]}}, {\bf {[A9']}}, {\bf {[A10]}}, {\bf {[A11]}}, {\bf {[A12']}}, {\bf{[B2]}}, {\bf{[B3]}}, and either ``{[$\textbf{C}_2 \textbf{1}$]} of Theorem \ref{thm2-2:jump} and {[$\textbf{D}_2 \textbf{1}$]} of Corollary \ref{thm2-1:jump}'' or ``{[$\textbf{C}_2 \textbf{2}$]} of Theorem \ref{thm2-2:jump} and {[$\textbf{D}_2 \textbf{2}$]} of Corollary \ref{thm2-1:jump}''.
Then, under $H_0$,
$\Lambda_n(\check{\theta}_n,\check{\theta}^*_n)\overset{d}{\to}\chi_{k+l}^2$.
\end{col}
Next, in order to discuss consistency of the test, 
 let $\theta_1=(\alpha_1,\beta_1,\gamma_1)$, which is the true parameter under $H_1$, 
and we make the following assumptions instead of assumptions {\bf {[E1]}} and {\bf {[E2]}}.
\begin{itemize}
\item[{\bf [E1']}]
\begin{enumerate}
\item[ (i)] 
For all $\varepsilon>0$, 
\begin{align*}
\sup_{\{\alpha\in\Theta_{\alpha_0}:|\alpha-\alpha^*|\geq\varepsilon\}}(U_1^*(\alpha,\alpha_1)-U_1^*(\alpha^*,\alpha_1))<0.
\end{align*}
\item[ (ii)]
For all $\varepsilon>0$, 
\begin{align*}
\sup_{\{\beta\in\Theta_{\beta_0}:|\beta-\beta^*|\geq\varepsilon\}}(\bar{U}_{\beta_1}^{(2)^*}(\alpha^*,\beta)-\bar{U}_{\beta_1}^{(2)^*}(\alpha^*,\beta^*))<0.
\end{align*}
\item[ (iii)] 
For all $\varepsilon>0$, 
\begin{align*}
\sup_{\{\gamma\in\Theta_{\gamma_0}:|\gamma-\gamma^*|\geq\varepsilon\}}(\tilde{U}_{\gamma_1}^{(2)^*}(\gamma)-\tilde{U}_{\gamma_1}^{(2)^*}(\gamma^*))<0.
\end{align*}
\end{enumerate}
\item[{\bf [E2']}]
For all $\theta\in\Theta$, $J(\theta)$ is non-singular.
\end{itemize}
Under the fixed assumptions, we have the following corollary.
\begin{col}\label{col:H1test'}
Assume {\bf{[A1]}}-{\bf{[A8]}}, {\bf {[A9']}}, {\bf {[A10]}}, {\bf {[A11]}}, {\bf {[A12']}}, {\bf{[B2]}}, {\bf{[B3]}}, [{\bf E1'}], [{\bf E2'}], and either ``{[$\textbf{C}_2 \textbf{1}$]} of Theorem \ref{thm2-2:jump} and {[$\textbf{D}_2 \textbf{1}$]} of Corollary \ref{thm2-1:jump}'' or ``{[$\textbf{C}_2 \textbf{2}$]} of Theorem \ref{thm2-2:jump} and {[$\textbf{D}_2 \textbf{2}$]} of Corollary \ref{thm2-1:jump}''. Then, under $H_1$, $P(\Lambda_n(\check{\theta}_n,\check{\theta}^*_n)>\chi_{r+l,\epsilon}^2)\to 1$.
\end{col}
In a similar way to the proof of  Corollaries \ref{col:H0test} and \ref{col:H1test}, 
we can prove  Corollaries \ref{col:H0test'} and \ref{col:H1test'}, respectively. By dividing the argument for $l_n^{(2)}$ into that for $\bar{l}_n^{(2)}$ and $\tilde{l}_n^{(2)}$, in particular, we can show the above statements. Therefore, we omit detailed proofs.


\section{An example and simulation study}\label{sec:sim}
Let $R_i^+>R_i^-\ (i=1,2,3,5),\ R_4^+,R_4^->0$.
We consider the following one-dimensional L\'{e}vy-OU model.
\begin{align}\label{sim:model}
dX_t=-\beta X_{t-}dt+\alpha dW_t+\int_E zp(dt,dz)\ \ t\in[0,T],
\end{align}
where the initial value $X_0$ follows the invariant probability measure $\pi$, the jump density $f_\gamma$ is written by
\begin{align*}
    f_\gamma(z)=\lambda\frac{1}{\sqrt{2\pi\sigma^2}}\exp\left\{-\frac{(z-\mu)^2}{2\sigma^2}\right\},\ \ \gamma=(\lambda,\mu,\sigma^2),
\end{align*}
and $(\alpha, \beta,\lambda,\mu,\sigma^2)\in[R_1^-,R_1^+]\times[R_2^-,R_2^+]\times[R_3^-,R_3^+]\times[-R_4^-,R_4^+]\times[R_5^-,R_5^+]$. 
Then
\begin{align*}
    \log\Psi_\gamma(y,x)=\log\lambda-\frac{1}{2}\log(2\pi\sigma^2)-\frac{(y-\mu)^2}{2\sigma^2}.
\end{align*}
We treat an adaptive estimation and test for $\theta=(\alpha,\beta,\gamma)=(\alpha, \beta,\lambda,\mu,\sigma^2)$.
Since the jump distribution is normal, we can show that model (\ref{sim:model}) satisfies {\bf [A1]}-{\bf [A8]}, {\bf [A9']}, {\bf [A10]}, {\bf [A11]}, {\bf [A12']} and {\bf [A13]}-(i). Therefore, we omit the truncation function $\varphi_n$
, and quasi-log likelihood functions are as follows:
\begin{align*}
    &l_n^{(1)}(\alpha)=-\frac{1}{2}\sum_{i=1}^n\{h_n^{-1}\alpha^{-2}(\Delta X_i^n)^2+2\log\alpha\}\boldsymbol{1}_{\{|\Delta X_i^n|\le D_1h_n^{\rho_1}\}},\\
    &\bar{l}_n^{(2)}(\beta|\alpha)=-\frac{1}{2\alpha^2h_n}\sum_{i=1}^n(\Delta X_i^n+\beta h_nX_{t_{i-1}^n})^2\boldsymbol{1}_{\{|\Delta X_i^n|\le D_3h_n^{\rho_3}\}},\\
    &\tilde{l}_n^{(2)}(\gamma)=\sum_{i=1}^n\log\Psi_\gamma(\Delta X_i^n,X_{t_{i-1}^n})\boldsymbol{1}_{\{|\Delta X_i^n|> D_2h_n^{\rho_2}\}}-\lambda nh_n.
\end{align*}
We set $n_1=\sum_{i=1}^n\boldsymbol{1}_{\{|\Delta X_i^n|\le D_1h_n^{\rho_1}\}}$ and $n_2=\sum_{i=1}^n\boldsymbol{1}_{\{|\Delta X_i^n|> D_2h_n^{\rho_2}\}}$. 
Then the adaptive estimator $(\check{\alpha}_n,\check{\beta}_n,\check{\lambda}_n,\check{\mu}_n,\check{\sigma}_n^2)$ can be calculated as
\begin{align*}
\check{\alpha}_{n}&=\sqrt{\frac{1}{n_1h_n}\sum_{i=1}^{n}{(\Delta X_i^n)^2}\boldsymbol{1}_{\{|\Delta X_i^n|\le D_1h_n^{\rho_1}\}}},\quad
\check{\beta}_{n}=-\frac{\sum_{i=1}^{n}{X_{t_{i-1}^n}\Delta X_i^n}\boldsymbol{1}_{\{|\Delta X_i^n|\le D_3h_n^{\rho_3}\}}}{h_n\sum_{i=1}^{n}{X_{t_{i-1}^n}^2}\boldsymbol{1}_{\{|\Delta X_i^n|\le D_3h_n^{\rho_3}\}}}, \\
\check{\lambda}_{n}&=\frac{n_2}{nh_n},\quad \check{\mu}_{n}=\frac{1}{n_2}\sum_{i=1}^{n}{\Delta X_i^n\boldsymbol{1}_{\{|\Delta X_i^n|> D_2h_n^{\rho_2}\}}},\quad
\check{\sigma}_{n}^2=\frac{1}{n_2}\sum_{i=1}^{n}{(\Delta X_i^n-\check{\mu}_n)^2\boldsymbol{1}_{\{|\Delta X_i^n|> D_2h_n^{\rho_2}\}}}.
\end{align*} 
In our simulation, we set $\theta_0=(2,2.5,6,0,4.5)$ and for simplicity, we took $D_1=D_2=D_3=1$. Note that values of $\check{\alpha}_n$ and $\check{\beta}_n$ depend on $\rho_1$ and $\rho_3$, respectively, and the values of $\check{\lambda}_n,\check{\mu}_n$ and $\check{\sigma}_n^2$ depend on $\rho_2$. Theorem \ref{thm2-2:jump} shows that the following convergence holds:
\begin{align*}
(\sqrt{n}(\check{\alpha}_n-\alpha_0), \sqrt{nh_n}(\check{\beta}_n-\beta_0),\sqrt{nh_n}(\check{\gamma}_n-\gamma_0))\overset{d}{\to}N(0,K^{-1}),
\end{align*}
where
\begin{align*}
K=\mathrm{diag}\left(\frac{2}{\alpha_0^2},\frac{\mu_2}{\alpha_0^2},\frac{1}{\lambda_0},\frac{\lambda_0}{\sigma_0^2},\frac{\lambda_0}{2\sigma_0^4}\right),\quad
\mu_2=\int x^2\pi(dx).
\end{align*}
Let $n=10^6$ and $h_n=n^{-2/3}$, which means $T=100$. For all $\varepsilon\in(0,\frac{1}{2})$, if we set $\delta=\frac{1}{2}+\varepsilon$, it holds that $nh_n^{1+\delta}\to0$. Hence, we can choose $\rho_1$ and $\rho_2$ from $(\frac{1}{4},\frac{1}{2})$, and take $\rho_3$ from $(\frac{1}{8},\frac{1}{2})$ for Theorem \ref{thm2-2:jump} to hold. In our simulation, we conducted the adaptive estimation with each values of $\rho_1,\rho_2$ and $\rho_3$. 
We generated 1000 independent sample paths for each setting and compared simulation results.
Table \ref{table1:jump} shows sample means of the simulated adaptive estimators when $\rho_1,\rho_2$ and $\rho_3$ vary from 0.255 to 0.3 in increments of 0.005. We know that it is most difficult to estimate the jump intensity and variance parameters $\lambda$ and $\sigma^2$ among the five parameters. In particular, only the cases around $\rho_2=0.26$ can estimate $\lambda$ and $\sigma^2$, precisely.
Next, we consider $\check{\alpha}_n$ and $\check{\beta}_n$ related to $\rho_1$ and $\rho_3$, respectively, and $\check{\lambda}_n$, $\check{\mu}_n$ and $\check{\sigma}_n^2$ related to $\rho_2$ as one group each. Figure \ref{fig:ada_qq} shows QQ-plot of the simulated adaptive estimators with $\rho$ which takes from $0.255$ to $0.285$.
Figure \ref{fig:ada_qq} illustrates that
the most suitable choices of $\rho_1$ and $\rho_2$ are $\rho_1=0.285$ and $\rho_2=0.26$. Moreover, while the behavior of $\check{\beta}_n$ is the most robust of all these estimators, taking into account the results in Table \ref{table1:jump}, it can be seen that $\rho_3=0.255$ is the most suitable choice. Since the joint estimation in \citet{Shimizu-Yoshida_JP,Shimizu-Yoshida} can only use one kind of $\rho$, this results imply that we should utilize our adaptive estimator for estimating each parameter more accurately.
While our adaptive estimator has better behavior than the joint estimator, 
it is still difficult to choose 
optimal thresholds 
$\rho_1,\rho_2$ and $\rho_3$ in practice. However, by the definition of $l_n^{(1)},\bar{l}_n^{(2)}$ and $\tilde{l}_n^{(2)}$, we have an insight that the thresholds can be determined in different ways for $\rho_1,\rho_2$ and $\rho_3$. In particular, we can consider the choosing problem of the threshold in the continuous part and the jump part, separately.
\begin{table}[htbp]
\centering
\caption{Sample mean (true parameter value) of the simulated adaptive estimators.} 
\label{table1:jump}
\renewcommand{\arraystretch}{1.3}
\scalebox{1}{
\begin{tabular}{l|ccccc}
$\rho_1,\rho_2,\rho_3$&$\alpha (2)$&$\beta (2.5)$&$\lambda (6)$&$\mu (0)$&$\sigma^2 (20.25)$\\
\hline\hline
0.255&2.00370&2.50007&   5.91793&-0.00090&20.51549\\0.26&2.00362&2.49988&   5.95749&-0.00087&20.37968\\0.265&2.00346&2.49949&   6.04572&-0.00087&20.08283\\0.27&2.00312&2.49868&   6.24211&-0.00087&19.45233\\0.275&2.00249&2.49714&   6.65118&-0.00085&18.25820\\0.28&2.00137&2.49435&   7.45195&-0.00074&16.30029\\0.285&1.99951&2.48973&   8.91656&-0.00062&13.62686\\0.29&1.99657&2.48237&  11.47515&-0.00048&10.59289\\0.295&1.99217&2.47140&  15.69133&-0.00029& 7.74991\\0.3&1.98589&2.45590&  22.31719&-0.00021& 5.45139
\end{tabular}
}
\end{table}

\begin{figure}[htbp]
\centering
\includegraphics[width=\columnwidth]{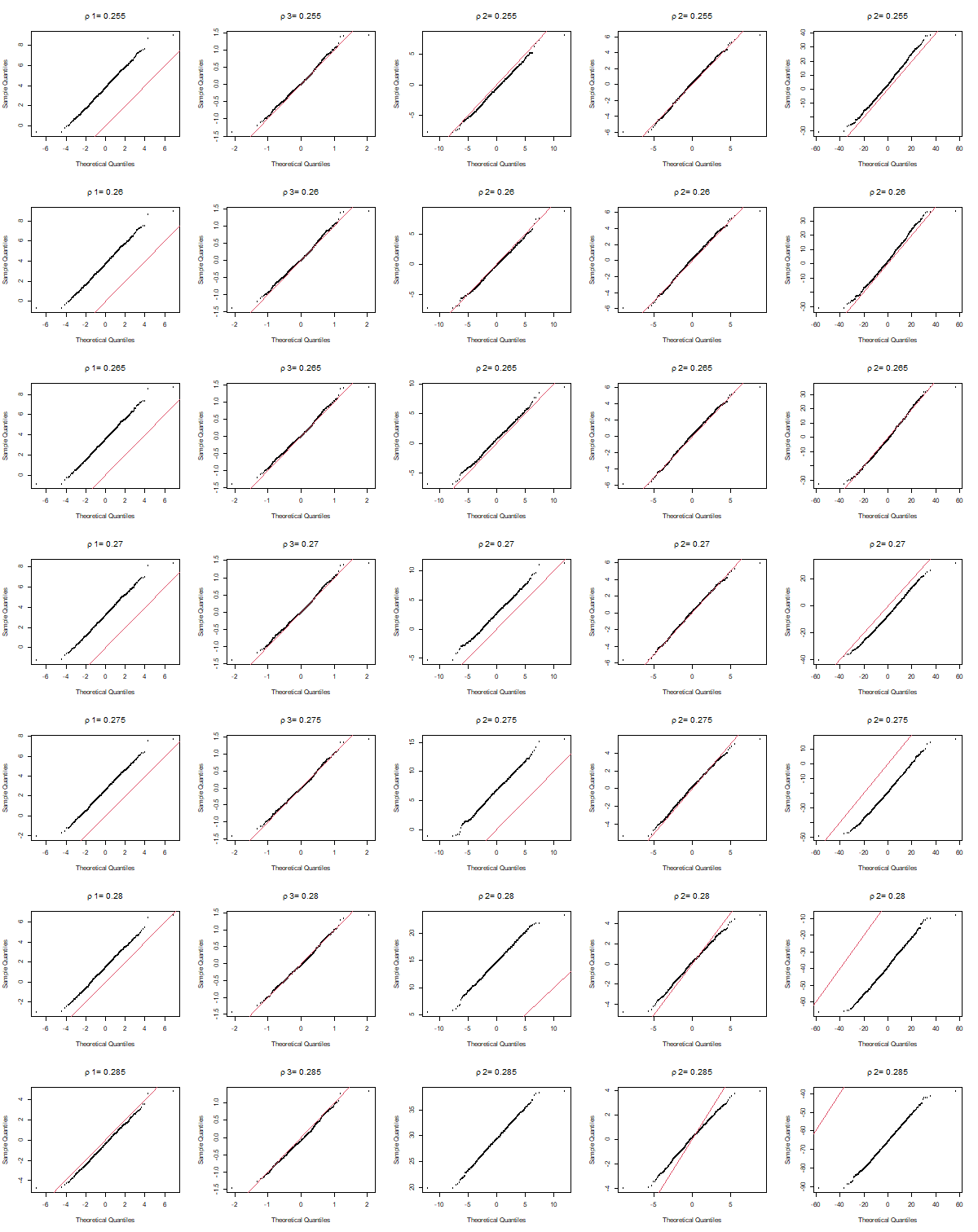}	 
\caption{QQ-plot of the simulated adaptive estimators. From left to right: the estimator for $\alpha,\beta,\lambda,\mu,\sigma^2$. From top to bottom: it is set that $\rho_1=\rho_2=\rho_3=0.255,0.26,0.265,0.27,0.275,0.28,0.285$. The solid line is $y=x$.
}
\label{fig:ada_qq}
\end{figure}

Next, we consider the adaptive test. 
The model setup is the same as estimation, and we set the hypothesis testing problem as follows:
\begin{align*}
    \begin{cases}
        H_0:\ \ \alpha=2,\ \beta=2.5,\ \lambda=6,\ \mu=0,\ \sigma^2=20.25, \\
        H_1:\ \ \mathrm{not}\ H_0.
    \end{cases}
\end{align*}
First, we simulate the asymptotic behavior of the adaptive test statistic under $H_0$. 
Joint quasi-log likelihood function, used for composing the adaptive test statistic, is as follows:
\begin{align*}
l_n(\theta)&=-\frac{1}{2}\sum_{i=1}^{n}\left\{h_n^{-1}\alpha^{-2}(\Delta X_i^n+\beta h_n X_{t_{i-1}^n})^2+2\log\alpha\right\}{{\boldsymbol{1}_{\left\{|\Delta X_i^n|\le \bar{D}_1h_n^{\bar{\rho}_1}\right\}}}}\\
&\qquad+\sum_{i=1}^{n}\log f_\gamma(\Delta X_i^n){{\boldsymbol{1}_{\left\{|\Delta X_i^n|> \bar{D}_2h_n^{\bar{\rho}_2}\right\}}}}-\lambda nh_n.
\end{align*}
Then we can calculate the adaptive test statistics as
\begin{align*}
\Lambda_n(\check{\theta}_n,\check{\theta}^*_n)&= -\frac{1}{{\check{\alpha}_n^2}h_n}\sum_{i=1}^n (\Delta X_i^n+h_nX_{t_{i-1}^n}{\check{\beta}_n})^2{{\boldsymbol{1}_{\left\{|\Delta X_i^n|\le \bar{D}_1h_n^{\bar{\rho}_1}\right\}}}}\\
&\quad +\frac{1}{{(\check{\alpha}_n^*)^2}h_n}\sum_{i=1}^n (\Delta X_i^n+h_nX_{t_{i-1}^n}{\check{\beta}_n^*})^2{{\boldsymbol{1}_{\left\{|\Delta X_i^n|\le \bar{D}_1h_n^{\bar{\rho}_1}\right\}}}}\\
&\quad 2{n_1}\log\frac{{\check{\alpha}_n}}{{\check{\alpha}_n^*}}+2{n_2}\log\frac{{\check{\lambda}_n}}{{\check{\lambda}_n^*}}-2({\check{\lambda}_n}-{\check{\lambda}_n^*})nh_n-{n_2}\log\frac{{\check{\sigma}_n^2}}{{\check{\sigma}_n^{*2}}}\\
&\quad -\frac{1}{{\check{\sigma}_n^2}}\sum_{i=1}^n(\Delta X_i^n-{\check{\mu}_n})^2{{\boldsymbol{1}_{\left\{|\Delta X_i^n|> \bar{D}_2h_n^{\bar{\rho}_2}\right\}}}}-\frac{1}{{\check{\sigma}_n^{*2}}}\sum_{i=1}^n(\Delta X_i^n-{\check{\mu}_n^*})^2{{\boldsymbol{1}_{\left\{|\Delta X_i^n|> \bar{D}_2h_n^{\bar{\rho}_2}\right\}}}}.
\end{align*}
Theorem \ref{thm5:H0test} shows that the following convergence holds:
\begin{align*}
    \Lambda_n(\check{\theta}_n,\check{\theta}_n^*)\overset{d}{\to}\chi^2_{5}.
\end{align*}
For simplicity, we set $\bar{D}_1=\bar{D}_2=1$. The adaptive estimators on the constrained parameter space $\Theta_0$ are given by $\check{\theta}_n^*=(2,2.5,6,0,20.25)$. In this simulation, we set  $n=10^6,h_n=n^{-2/3}$ again, and then $\bar{\rho}_1$ and $\bar{\rho}_2$ can be chosen from $(\frac{1}{4},\frac{1}{2})$. 
First, we fixed the adaptive estimators used for constructing the test statistic at $\rho_1=0.285,\rho_2=0.26$ and $\rho_3=0.255$, which yielded the best results in the simulation, and confirmed the behavior when we changed the values of $\bar{\rho}_1$ and $\bar{\rho}_2$. 
Figure \ref{fig:test_best_qq} shows QQ-plot of the simulated adaptive test statistics with $\rho_1$ and $\rho_2$ which takes from 0.255 to 0.285 and from 0.255 to 0.275, separately.
From Figure \ref{fig:test_best_qq}, it can be confirmed that the adaptive test statistic converges to its asymptotic distribution if we take five thresholds suitably. In particular, we see that the suitable ranges of $\bar{\rho}_1$ and $\bar{\rho}_2$ are $0.255\le \bar{\rho}_1\le 0.27$ and $0.255\le \bar{\rho}_2\le 0.265$, respectively. This implies that while there is no issue in aligning $\rho_3$ with $\bar{\rho}_1$, and $\rho_2$ with $\bar{\rho}_2$, we should exercise caution in aligning $\rho_1$ with $\bar{\rho}_1$. Figure \ref{fig:pattern1_qq} shows QQ-plot of the simulated adaptive test statistics with $\rho_3=0.255,\rho_2=0.26$, and with $\bar{\rho}_1=\rho_1$ and  $\bar{\rho}_2$ which takes from 0.255 to 0.285 and 0.255 to 0.275, respectively.
Figure \ref{fig:pattern1_qq} suggests that
$\rho_1$, which is related to estimation for the diffusion coefficient, and $\bar{\rho}_1$, which is contained in the continuous part of joint quasi-log likelihood function, should not be set to the same value. 
This error arises because the adaptive estimator does not maximize the joint quasi-log likelihood function $l_n$.
In actual, the joint estimators $\hat{\theta}_n=(\hat{\alpha}_n,\hat{\beta}_n,\hat{\lambda}_n,\hat{\mu}_n,\hat{\sigma}_n^2)$ are as follows: 
\begin{align*}
\hat{\alpha}_{n}&=\sqrt{\frac{1}{\bar{n}_1h_n}\sum_{i=1}^{n}{(\Delta X_i^n+\hat{\beta}_n h_n X_{t_{i-1}^n})^2}\boldsymbol{1}_{\{|\Delta X_i^n|\le \bar{D}_1h_n^{\bar{\rho}_1}\}}},\quad
\hat{\beta}_{n}=-\frac{\sum_{i=1}^{n}{X_{t_{i-1}^n}\Delta X_i^n}\boldsymbol{1}_{\{|\Delta X_i^n|\le \bar{D}_1h_n^{\bar{\rho}_1}\}}}{h_n\sum_{i=1}^{n}{X_{t_{i-1}^n}^2}\boldsymbol{1}_{\{|\Delta X_i^n|\le \bar{D}_1h_n^{\bar{\rho}_1}\}}},\\
\hat{\lambda}_{n}&=\frac{\bar{n}_2}{nh_n},\quad \hat{\mu}_{n}=\frac{1}{\bar{n}_2}\sum_{i=1}^{n}{\Delta X_i^n\boldsymbol{1}_{\{|\Delta X_i^n|> \bar{D}_2h_n^{\bar{\rho}_2}\}}},\quad
\hat{\sigma}_{n}^{2}=\frac{1}{\bar{n}_2}\sum_{i=1}^{n}{(\Delta X_i^n-\hat{\mu}_n)^2\boldsymbol{1}_{\{|\Delta X_i^n|> \bar{D}_2h_n^{\bar{\rho}_2}\}}},
\end{align*} 
where $\bar{n}_1=\sum_{i=1}^n\boldsymbol{1}_{\{|\Delta X_i^n|\le \bar{D}_1h_n^{\bar{\rho_1}}\}}$ and $\bar{n}_2=\sum_{i=1}^n\boldsymbol{1}_{\{|\Delta X_i^n|> \bar{D}_2h_n^{\bar{\rho}_2}\}}$.
Through the comparison of the adaptive and joint estimators, it is found that these estimators are identical, except for estimation of $\alpha$. Therefore, we take note that we should decide not only $\rho_1$ and $\bar{\rho}_1$ separately, but also the other thresholds.
The conclusion is that, in the model (\ref{sim:model}), for the construction of the adaptive test statistic, the threshold for estimating the diffusion coefficient and the threshold for the continuous component in the joint quasi-log likelihood function, which is used for the construction of quasi-likelihood ratio, should be determined separately. However, the remaining thresholds can be aligned.
Figure \ref{fig:pattern4_qq} shows QQ-plot of the simulated adaptive test statistic with $\rho_1=0.285$, and with $\bar{\rho}_1=\rho_3$ and $\bar{\rho}_2=\rho_2$, where $\bar{\rho}_1=\rho_3$ varies from 0.255 to 0.285 and $\bar{\rho}_2=\rho_2$ varies from 0.255 to 0.275, respectively. Since $0.255\le\bar{\rho}_1=\rho_3\le 0.27$ and $0.255\le \bar{\rho}_2=\rho_2\le 0.265$ from Figure \ref{fig:pattern4_qq}, it can be seen that by setting a separate threshold for $\rho_1$, it is fine to set $\bar{\rho}_1=\rho_3$ and $\bar{\rho}_2=\rho_2$.
In particular, by observing the diagonal elements, it is evident that there is no issue with aligning all the thresholds except for $\rho_1$.
Next, we check whether our adaptive test is consistent or not. We set null and alternative hypotheses as follows:
\begin{align*}
    \begin{cases}
        H_0':\ \ \alpha=2.01,\beta=2.5,\lambda=6,\mu=0,\sigma^2=20.25, \\
        H_1':\ \ \mathrm{not}\ H_0'.
    \end{cases}
\end{align*}
Then we generated 1000 independent sample paths for each setting under $H_1'$($\theta_1=(2,2.5,6,0,20.25)$), and for the construction of the test statistic, we fixed the thresholds used for estimation at $\rho_1=0.285,\rho_2=0.26$ and $\rho_3=0.255$, which yielded good results in the simulation. Moreover, we examined the number of rejections 
when we changed $\bar{\rho}_1$ and $\bar{\rho}_2$, which were used in the construction of the quasi-likelihood ratio.
Table \ref{table:test-consistent} shows the number of rejections 
when $\bar{\rho}_1$ and $\bar{\rho}_2$ change from 0.255 to 0.265. 
From this table, it is found that the power converges to 1.
\begin{figure}[htbp]
\centering
\includegraphics[width=\columnwidth]{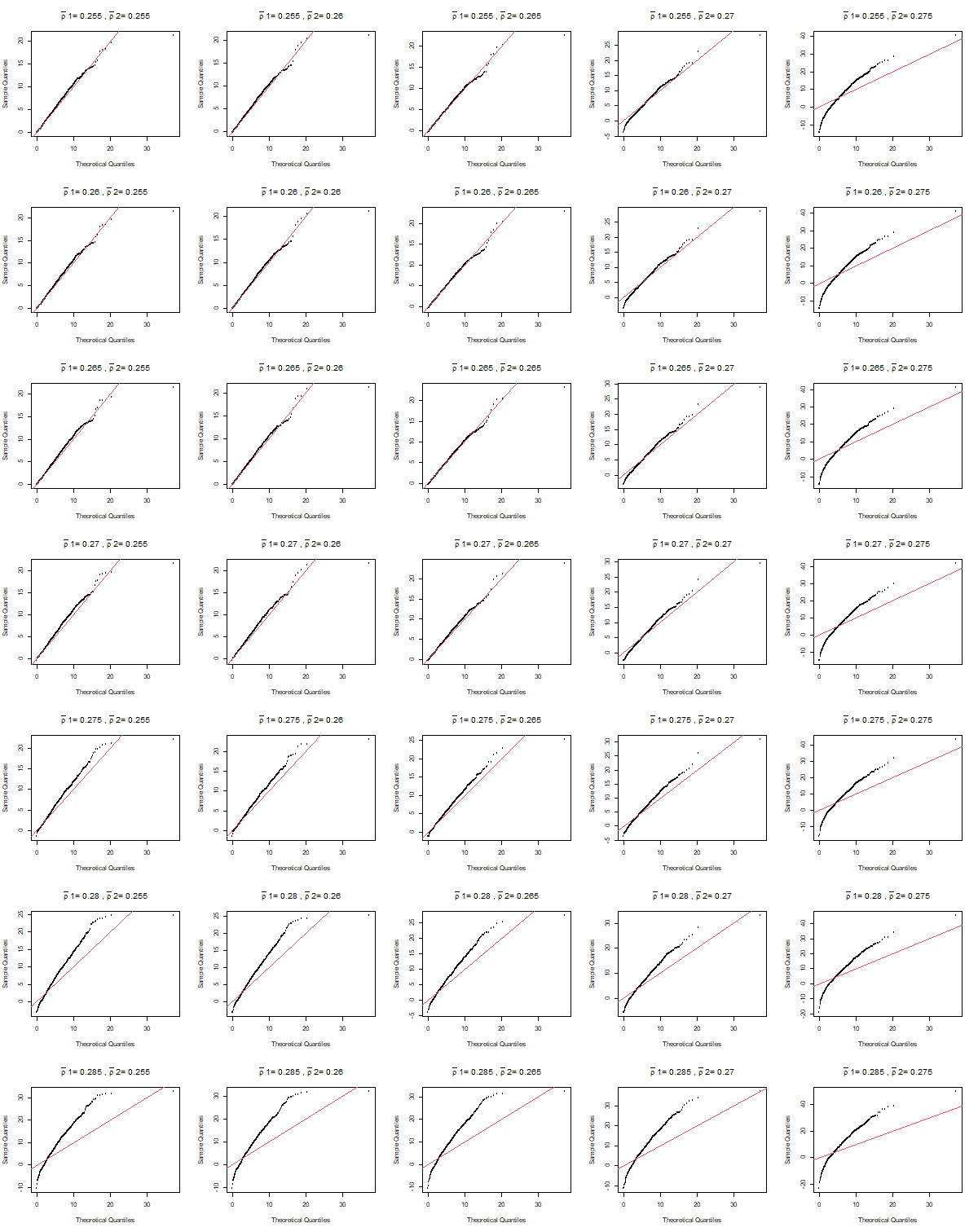}	
\caption{QQ-plot of the simulated adaptive test statistic with $\rho_1=0.285,\rho_2=0.26$, $\rho_3=0.255$, $\bar{\rho}_1$ and $\bar{\rho}_2$. From left to right: it is set that $\bar{\rho}_2=0.255,0.26,0.265,0.27,0.275$. From top to bottom: it is set that $\bar{\rho}_1=0.255,0.26,0.265,0.27,0.275,0.28,0.285$. The solid line is $y=x$.}
\label{fig:test_best_qq}
\end{figure}
\begin{figure}[htbp]
\centering
\includegraphics[width=\columnwidth]{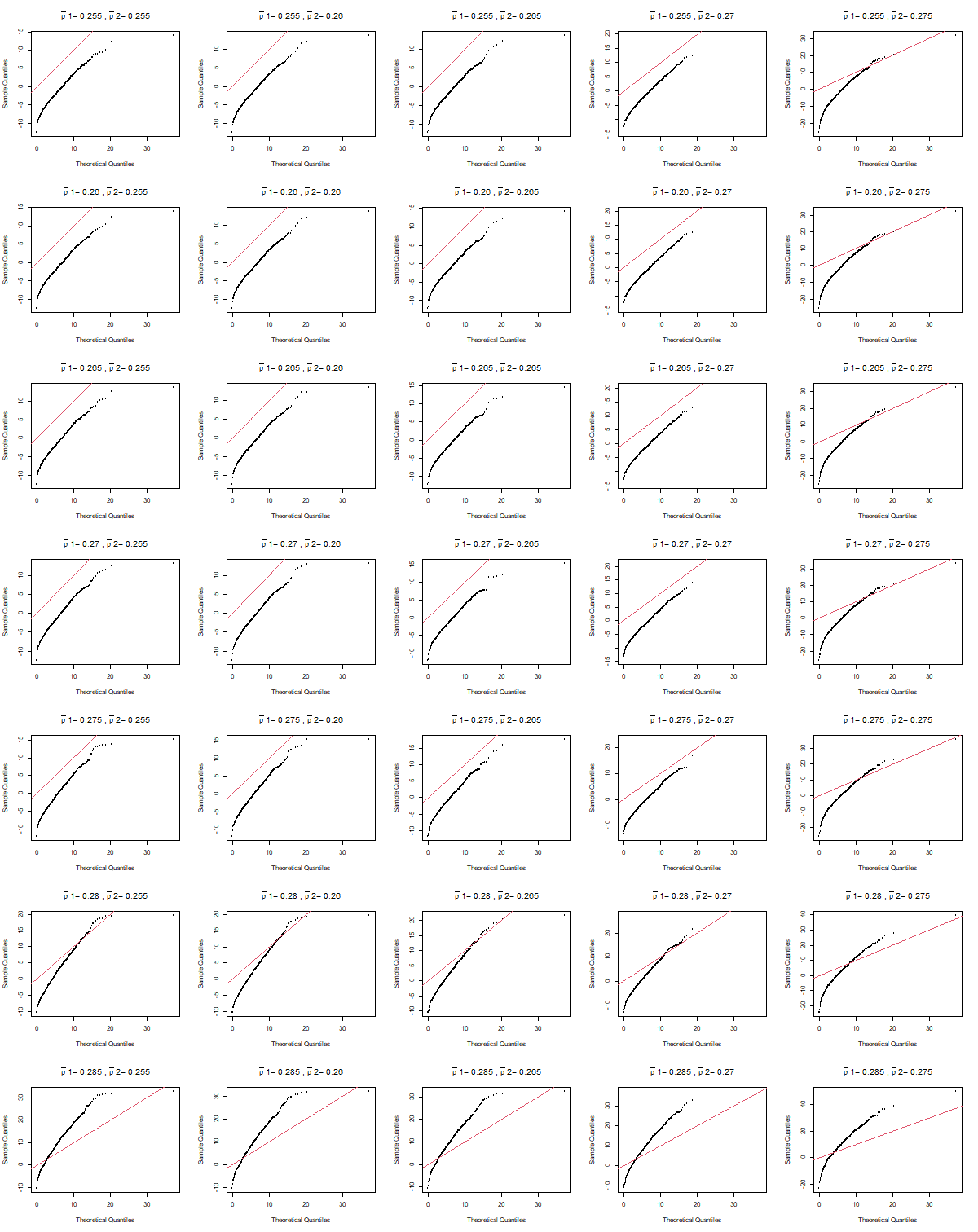}	
\caption{QQ-plot of the simulated adaptive test statistic with $\rho_2=0.26,\rho_3=0.255,\rho_1=\bar{\rho}_1$ and $\bar{\rho}_2$. From left to right: it is set that $\bar{\rho}_2=0.255,0.26,0.265,0.27,0.275$. From top to bottom: it is set that $\rho_1=\bar{\rho}_1=0.255,0.26,0.265,0.27,0.275,0.28,0.285$. The solid line is $y=x$.}
\label{fig:pattern1_qq}
\end{figure}
\begin{figure}[htbp]
\centering
\includegraphics[width=\columnwidth]{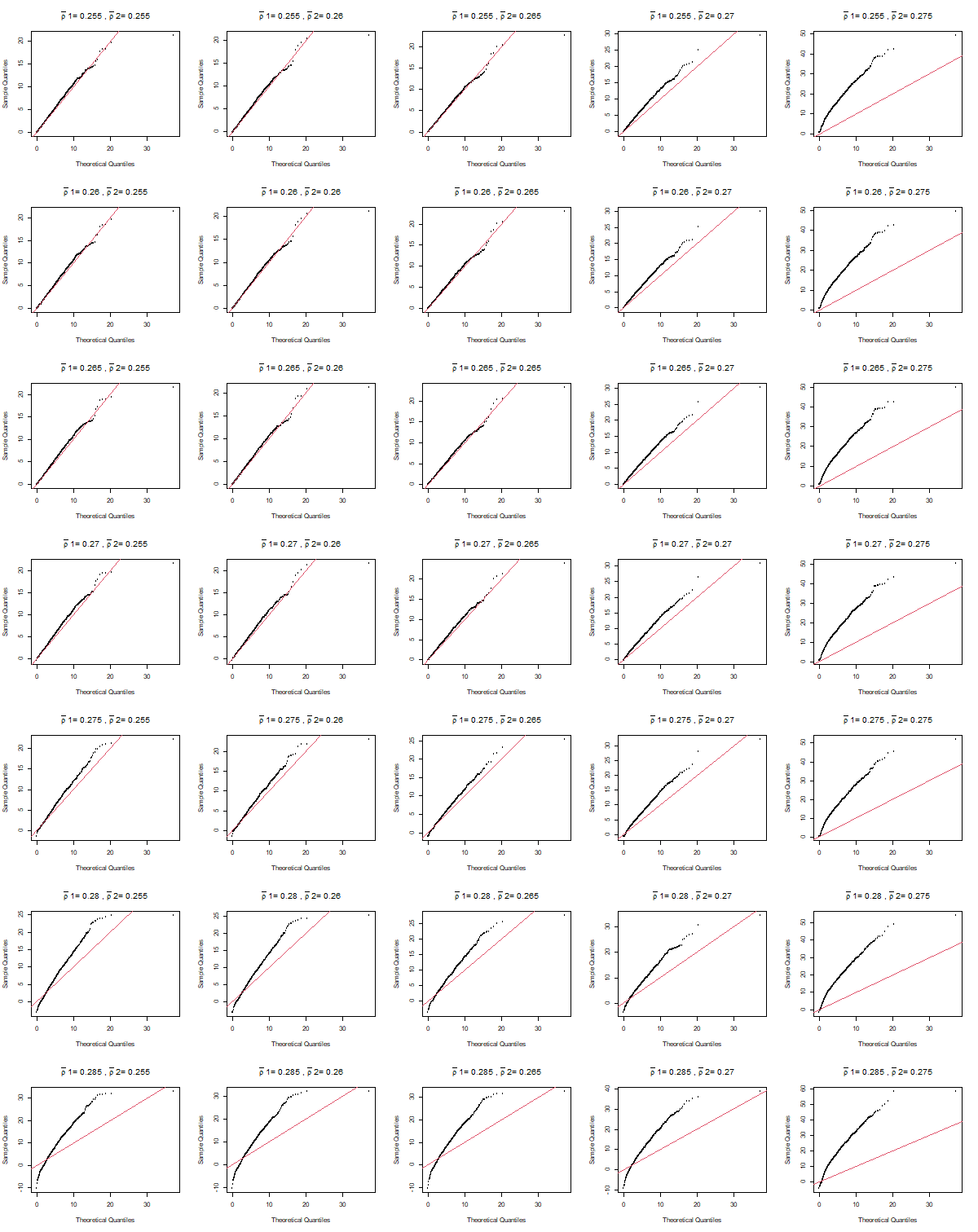}	
\caption{QQ-plot of the simulated adaptive test statistic with $\rho_1=0.285,\rho_2=\bar{\rho}_2$ and $\rho_3=\bar{\rho}_1$. From left to right: it is set that $\rho_2=\bar{\rho}_2=0.255,0.26,0.265,0.27,0.275$. From top to bottom: it is set that $\rho_3=\bar{\rho}_1=0.255,0.26,0.265,0.27,0.275,0.28,0.285$. The solid line is $y=x$.
}
\label{fig:pattern4_qq}
\end{figure}
\begin{table}[htbp]
\begin{center}
\caption{Rejection number with $\rho_1=0.285,\rho_2=0.26$, $\rho_3=0.255$, $\bar{\rho}_1$ and $\bar{\rho}_2$. From left to right: it is set that $\bar{\rho}_2=0.255,0.26,0.265$. From top to bottom: it is set that $\bar{\rho}_1=0.255,0.26,0.265$.}
\label{table:test-consistent}
\begin{tabular}{c||ccc}
${\bar{\rho}_1}$,${\bar{\rho}_2}$&0.255&0.26&0.265\\
\hline\hline
0.255&1.000&1.000&1.000\\
0.26&1.000&1.000&1.000\\
0.265&1.000&1.000&1.000\\
\end{tabular}
\end{center}
\end{table}
\FloatBarrier%


\section{Proofs}\label{sec:proof}
In this section, we sometimes omit the true parameter values without specially mentioning. For example, 
we  abbreviate $q^{\beta_0}(ds,dz)$ as $q(ds,dz)$ or $a(X_s,\alpha_0)$ as $a(X_s)$ and so on.
\subsection{Proofs of Chapter \ref{sec4.3:m}}
\subsubsection{Preliminary results}
We show some propositions and lemmas for proving theorems in Chapter \ref{sec4.3:m}.
\begin{prop}{(\textbf{\citet{Shimizu-Yoshida_JP,Shimizu-Yoshida}})}\label{prop1:jump}
Suppose {\bf{[A1]}}, {\bf{[A3]}} and {\bf{[A5]}}. For $k\ge 2$, $k\in\mathbb{N}$, $t_{i-1}^n\le t\le t_i^n$,
\begin{equation}\label{eq1:prop1:jump}
\mathbb{E}\left[|X_t-X_{t_{i-1}^n}|^k\ |\mathcal{F}_{i-1}^n\right]\le C_k|t-t_{i-1}^n|(1+|X_{t_{i-1}^n}|)^{C_k}.
\end{equation}
If $g$ is a function defined on $\mathbb{R}^d\times\Theta$ and is polynomial growth in $x$ uniformly in $\theta$, then, 
\begin{equation}\label{eq2:prop1:jump}
\mathbb{E}\left[|g(X_t,\theta)| \ |\mathcal{F}_{i-1}^n\right]\le C(1+|X_{t_{i-1}^n}|)^{C}.
\end{equation}
\end{prop}

\begin{rmk}\label{proofskip1:jump}
 Assumptions in Proposition \ref{prop1:jump} are slightly different from those in \citet{Shimizu-Yoshida_JP,Shimizu-Yoshida}. However, we can prove Proposition \ref{prop1:jump} in an analogous manner to 
 the proof in \citet{Shimizu-Yoshida_JP,Shimizu-Yoshida}. Moreover, for Lemma \ref{lem1:jump}, 
 Propositions \ref{prop3:jump}, \ref{prop5:jump}, \ref{prop6:jump}, we can verify that similar arguments hold.
\end{rmk}
We define the random times $\tau_i^n$ and $\eta_i^n$ as follows:
\begin{align*}
    \tau_i^n&:=\inf\{t\in[t_{i-1}^n,t_i^n)\ ;\ |\Delta X_t|>0\},\\
    \eta_i^n&:=\sup\{t\in[t_{i-1}^n,t_i^n)\ ;\ |\Delta X_t|>0\}.
\end{align*}
If the infimum or supremum on the right-hand side does not exist, we define the random times  equal to $t_{i}^n$.
The random times $\tau_i^n$ and $\eta_i^n$  denote the first and last jump time on $[t_{i-1}^n,t_i^n)$, respectively.  
\begin{lemma}{(\textbf{\citet{Shimizu-Yoshida_JP,Shimizu-Yoshida}})}\label{lem1:jump}
Suppose {\bf{[A1]}}, {\bf{[A3]}} and {\bf{[A5]}}. For $D>0$, $\rho\in[0,1/2)$ and any $p\ge 1$, 
\begin{align}
P\left(\sup_{t\in[t_{i-1}^n,\tau_i^n)}|X_t-X_{t_{i-1}^n}|>Dh_n^\rho\ |\mathcal{F}_{i-1}^n\right)&=R(\theta,h_n^p,X_{t_{i-1}^n}),\label{eq1:lem2:jump}\\
P\left(\sup_{t\in[\eta_{i}^n,t_i^n)}|X_{t_i^n}-X_{t}|>Dh_n^\rho\ |\mathcal{F}_{i-1}^n\right)&=R(\theta,h_n^p,X_{t_{i-1}^n}),\label{eq2:lem2:jump}
\end{align}
 where $\sup\emptyset=-\infty$ and each function $R$ does not depend on $i$.
\end{lemma}
Let $J_i^n=p\left((t_{i-1}^n,t_i^n]\times E\right)$, 
\begin{align*}
C_{i,0}^n(D,\rho)&=\left\{J_i^n=0, |\Delta X_i^n|\le Dh_n^\rho\right\},\quad D_{i,0}^n(D,\rho)=\left\{J_i^n=0, |\Delta X_i^n|> Dh_n^\rho\right\},\\
C_{i,1}^n(D,\rho)&=\left\{J_i^n=1, |\Delta X_i^n|\le Dh_n^\rho\right\},\quad D_{i,1}^n(D,\rho)=\left\{J_i^n=1, |\Delta X_i^n|> Dh_n^\rho\right\},\\
C_{i,2}^n(D,\rho)&=\left\{J_i^n\ge2, |\Delta X_i^n|\le Dh_n^\rho\right\},\quad D_{i,2}^n(D,\rho)=\left\{J_i^n\ge2, |\Delta X_i^n|> Dh_n^\rho\right\}.
\end{align*}
Then, we can express
\begin{align*}
\left\{|\Delta X_i^n|\le Dh_n^\rho\right\}=\bigcup_{j=0}^{2} {C_{i,j}^n(D,\rho)},\quad \left\{|\Delta X_i^n|> Dh_n^\rho\right\}=\bigcup_{j=0}^{2} {D_{i,j}^n(D,\rho)}.
\end{align*}

\begin{prop}\label{prop2:jump}
Suppose {\bf{[A1]}}, {\bf{[A3]}} and {\bf{[A5]}}-{\bf{[A7]}}. For $D,D_1,D_2>0$, $\rho,\rho_1,\rho_2\in(0,1/2)$, any $p\ge 1$ and sufficiently large $n$,
\begin{alignat*}{3}
&P(C_{i,0}^n(D,\rho)\ |\mathcal{F}_{i-1}^n)=\tilde{R}(\theta,h_n,X_{t_{i-1}^n}),\quad 
&&P(D_{i,0}^n(D,\rho)\ |\mathcal{F}_{i-1}^n)={R}(\theta,h_n^p,X_{t_{i-1}^n}),\\
&P(C_{i,1}^n(D,\rho)\ |\mathcal{F}_{i-1}^n)={R}(\theta,h_n^{1+\rho},X_{t_{i-1}^n}),\quad
&&P(D_{i,1}^n(D,\rho)\ |\mathcal{F}_{i-1}^n)=\lambda_0h_n\tilde{R}(\theta,h_n^\rho,X_{t_{i-1}^n}),\\
&P(C_{i,2}^n(D,\rho)\ |\mathcal{F}_{i-1}^n)\le\lambda_0^2h_n^2,\quad
&&P(D_{i,2}^n(D,\rho)\ |\mathcal{F}_{i-1}^n)\le\lambda_0^2h_n^2,\\
&P\left(\left\{|\Delta X_i^n|\le D_1h_n^{\rho_1}\right\}\cap\left\{|\Delta X_i^n| > D_2h_n^{\rho_2}\right\}\ |\mathcal{F}_{i-1}^n\right)&&=R(\theta,h_n^{1+\rho_1},X_{t_{i-1}^n}).
\end{alignat*}

\end{prop}
\begin{rmk}\label{proofcompare1:jump}
The statements of Proposition \ref{prop2:jump} generalize the results of Lemma 2 of \citet{Ogihara-Yoshida}. Our argument allows us more freedom in the choice of $\rho$ compared to theirs. 
\end{rmk}

\begin{proof}
Proofs except for the last statement are almost the same as those of Lemma 2.2 of \citet{Shimizu-Yoshida_JP,Shimizu-Yoshida}.
First, it is obvious that $P(C_{i,2}^n(D,\rho)\ |\mathcal{F}_{i-1}^n)\le\lambda_0^2h_n^2$ and $P(D_{i,2}^n(D,\rho)\ |\mathcal{F}_{i-1}^n)\le\lambda_0^2h_n^2$.
Next, on $C_{i,1}^n(D,\rho)$, we have
\begin{align*}
&P(C_{i,1}^n(D,\rho)\ |\mathcal{F}_{i-1}^n)\\
&=P\left(\left|(X_{t_{i}^n}-X_{\tau_{i}^n})+(X_{\tau_{i-}^n}-X_{t_{i-1}^n})+\Delta X_{\tau_{i}^n}\right|\le Dh_n^{\rho}\ ,\ J_i^n=1 \ |\mathcal{F}_{i-1}^n\right)\\
&\le P\left(\left|(X_{t_{i}^n}-X_{\tau_{i}^n})+(X_{\tau_{i-}^n}-X_{t_{i-1}^n})+\Delta X_{\tau_{i}^n}\right|\le Dh_n^{\rho}\ ,\ |\Delta Z_{\tau_{i}^n}|>\frac{3Dh_n^\rho}{c_0}\ ,\ J_i^n=1 \ |\mathcal{F}_{i-1}^n\right)\\
&\quad+P\left(|\Delta Z_{\tau_{i}^n}|\le\frac{3Dh_n^\rho}{c_0}\ ,\ J_i^n=1 \ |\mathcal{F}_{i-1}^n\right),
\end{align*}
where $c_0$ is the constant in condition {\bf[A7]} and $\Delta Z_{\tau_{i}^n}$ has density $F_{\beta_0}$ under $\mathcal{F}_{i-1}^n$. Under {\bf[A7]},
\begin{align*}
\left|(X_{t_{i}^n}-X_{\tau_{i}^n})+(X_{\tau_{i-}^n}-X_{t_{i-1}^n})+\Delta X_{\tau_{i}^n}\right|\le Dh_n^{\rho},\quad |\Delta Z_{\tau_{i}^n}|>\frac{3Dh_n^\rho}{c_0},
\end{align*} 
and $|\Delta X_{\tau_{i}^n}|$ is small enough, then it holds that 
\begin{align*}
\left|X_{t_{i}^n}-X_{\tau_{i}^n}\right|+\left|X_{\tau_{i-}^n}-X_{t_{i-1}^n}\right|&\ge |\Delta X_{\tau_{i}^n}|-Dh_n^\rho\\
&=\left|c(X_{\tau_{i}^n},\Delta Z_{\tau_i^n},\beta_0)\right|-Dh_n^\rho\\
&\ge c_0|\Delta Z_{\tau_{i}^n}|-Dh_n^\rho\\
&>2Dh_n^\rho.
\end{align*}
Therefore,  we see from Lemma \ref{lem1:jump} and {\bf[A5]} that for large $n$,   
\begin{align*}
&P(C_{i,1}^n(D,\rho)\ |\mathcal{F}_{i-1}^n)\\
&\le P\left(\sup_{t\in[\eta_{i}^n,t_i^n)}|X_{t_i^n}-X_{t}|+\sup_{t\in[t_{i-1}^n,\tau_i^n)}|X_t-X_{t_{i-1}^n}|>2Dh_n^\rho\ |\mathcal{F}_{i-1}^n\right)\\
&\quad+\lambda_0h_ne^{-\lambda_0h_n}\cdot\frac{1}{\lambda_0}\int_{|z|\le\frac{3Dh_n^\rho}{c_0}}f_{\beta_0}(z) dz\\
&\le P\left(\sup_{t\in[\eta_{i}^n,t_i^n)}|X_{t_i^n}-X_{t}|>Dh_n^\rho\ |\mathcal{F}_{i-1}^n\right)+P\left(\sup_{t\in[t_{i-1}^n,\tau_i^n)}|X_t-X_{t_{i-1}^n}|>Dh_n^\rho\ |\mathcal{F}_{i-1}^n\right)\\
&\quad+h_ne^{-\lambda_0h_n}\cdot K\int_{|z|\le\frac{3Dh_n^\rho}{c_0}}|z|^{1-d} dz\\
&=R(\theta,h_n^p,X_{t_{i-1}^n})+h_ne^{-\lambda_0h_n}\cdot CK\cdot\frac{3Dh_n^\rho}{c_0}\\
&=R(\theta,h_n^{1+\rho},X_{t_{i-1}^n})
\end{align*}
for $p\ge1+\rho$.
Hence it holds that
\begin{align*}
P(D_{i,1}^n(D,\rho)\ |\mathcal{F}_{i-1}^n)&=P(J_i^n=1\ |\mathcal{F}_{i-1}^n)-P(C_{i,1}^n(D,\rho)\ |\mathcal{F}_{i-1}^n)\\
&=\lambda_0h_ne^{-\lambda_0h_n}-R(\theta,h_n^{1+\rho},X_{t_{i-1}^n})\\
&=\lambda_0h_ne^{-\lambda_0h_n}\tilde{R}(\theta,h_n^\rho,X_{t_{i-1}^n})\\
&=\lambda_0h_n\tilde{R}(\theta,h_n^\rho,X_{t_{i-1}^n})+\lambda_0h_n(e^{-\lambda_0h_n}-1)\tilde{R}(\theta,h_n^\rho,X_{t_{i-1}^n})\\
&=\lambda_0h_n\tilde{R}(\theta,h_n^\rho,X_{t_{i-1}^n}).
\end{align*}
For $D_{i,0}^n(D,\rho)$, applying Lemma \ref{lem1:jump} again, we have 
\begin{align*}
P(D_{i,0}^n(D,\rho)\ |\mathcal{F}_{i-1}^n)=
P\left(\left|X_{\tau_{i}^n}-X_{t_{i-1}^n}\right|>Dh_n^\rho\ ,\ \tau_i^n=t_i^n\ |\mathcal{F}_{i-1}^n\right)=R(\theta,h_n^p,X_{t_{i-1}^n}).
\end{align*}
Therefore, 
\begin{align*}
P(C_{i,0}^n(D,\rho)\ |\mathcal{F}_{i-1}^n)&=P(J_i^n=0\ |\mathcal{F}_{i-1}^n)-P(D_{i,0}^n(D,\rho)\ |\mathcal{F}_{i-1}^n)\\
&=e^{-\lambda_0h_n}-R(\theta,h_n^{p},X_{t_{i-1}^n})\\
&=(e^{-\lambda_0h_n}-1)+\tilde{R}(\theta,h_n^{p},X_{t_{i-1}^n})\\
&=\tilde{R}(\theta,h_n,X_{t_{i-1}^n}).
\end{align*}
Finally, a simple computation yields that 
\begin{align*}
&P\left(\left\{|\Delta X_i^n|\le D_1h_n^{\rho_1}\right\}\cap\left\{|\Delta X_i^n| > D_2h_n^{\rho_2}\right\}\ |\mathcal{F}_{i-1}^n\right)\\
&=\sum_{j=0}^{\infty}{}P\left(\left\{|\Delta X_i^n|\le D_1h_n^{\rho_1}\right\}\cap\left\{|\Delta X_i^n| > D_2h_n^{\rho_2}\right\}\cap\{J_i^n=j\}\ |\mathcal{F}_{i-1}^n\right)\\
&\le\sum_{j=0}^{1}{}P\left(\left\{|\Delta X_i^n|\le D_1h_n^{\rho_1}\right\}\cap\left\{|\Delta X_i^n| > D_2h_n^{\rho_2}\right\}\cap\{J_i^n=j\}\ |\mathcal{F}_{i-1}^n\right)+P\left(J_i^n\ge2\ |\mathcal{F}_{i-1}^n\right)\\
&\le P(D_{i,0}^n(D_2,\rho_2)\ |\mathcal{F}_{i-1}^n)+P(C_{i,1}^n(D_1,\rho_1)\ |\mathcal{F}_{i-1}^n)+2\lambda_0^2h_n^2\\
&=R(\theta,h_n^p,X_{t_{i-1}^n})+R(\theta,h_n^{1+\rho},X_{t_{i-1}^n})+2\lambda_0^2h_n^2\\
&=R(\theta,h_n^{1+\rho},X_{t_{i-1}^n})
\end{align*}
for $p\ge1+\rho$. This completes the proof.
\end{proof}

\begin{prop}({\textbf{\citet{Shimizu-Yoshida_JP,Shimizu-Yoshida}}})\label{prop3:jump}
Suppose {\bf{[A1]}} and {\bf{[A3]}}-{\bf{[A7]}}. Then for $k_j=1,\ldots,d\ (j=1,2,3,4)$,
\begin{align}
\mathbb{E}\left[\bar{X}_{i,n}^{(k_1)}\boldsymbol{1}_{C_{i,0}^n(D,\rho)} \ |\mathcal{F}_{i-1}^n\right]&=R(\theta,h_n^2,X_{t_{i-1}^n}),\label{eq1:prop3:jump}\\
\mathbb{E}\left[\bar{X}_{i,n}^{(k_1)}\bar{X}_{i,n}^{(k_2)}\boldsymbol{1}_{C_{i,0}^n(D,\rho)} \ |\mathcal{F}_{i-1}^n\right]&=h_nS_{i-1}^{(k_1,k_2)}(\alpha_0)+R(\theta,h_n^2,X_{t_{i-1}^n}),\label{eq2:prop3:jump}\\
\mathbb{E}\left[\bar{X}_{i,n}^{(k_1)}\bar{X}_{i,n}^{(k_2)}\bar{X}_{i,n}^{(k_3)}\boldsymbol{1}_{C_{i,0}^n(D,\rho)} \ |\mathcal{F}_{i-1}^n\right]&=R(\theta,h_n^2,X_{t_{i-1}^n}),\label{eq3:prop3:jump}\\
\mathbb{E}\left[\bar{X}_{i,n}^{(k_1)}\bar{X}_{i,n}^{(k_2)}\bar{X}_{i,n}^{(k_3)}\bar{X}_{i,n}^{(k_4)}\boldsymbol{1}_{C_{i,0}^n(D,\rho)} \ |\mathcal{F}_{i-1}^n\right]&=h_n^2(S_{i-1}^{(k_1,k_2)}S_{i-1}^{(k_3,k_4)}+S_{i-1}^{(k_1,k_3)}S_{i-1}^{(k_2,k_4)}+S_{i-1}^{(k_1,k_4)}S_{i-1}^{(k_2,k_3)})(\alpha_0)\notag\\
&\qquad+R(\theta,h_n^3,X_{t_{i-1}^n}).\label{eq4:prop3:jump}
\end{align}
\end{prop}

\begin{prop}\label{prop4:jump}
Assume {\bf{[A1]}}-{\bf{[A3]}}, {\bf{[A5]}}-{\bf{[A7]}} and {\bf{[B1]}}. 
 Let $g^{(n)}$ : $\mathbb{R}^d\times\Theta\to\mathbb{R}$ be a function whcih satisfies the following conditions that
\begin{align*}
&|g^{(n)}(x,\theta)|\le C(1+|x|)^C,\quad |\partial_x g^{(n)}(x,\theta)|\le C\cdot \varepsilon_n^{-4}(1+|x|)^C,\quad
|\partial_\theta g^{(n)}(x,\theta)|\le C(1+|x|)^C,
\end{align*}
and that there exist  a function $g$ : $\mathbb{R}^d\times\Theta\to\mathbb{R}$ for each $\theta\in\Theta$ such that
\begin{align*}
g^{(n)}(x,\theta)\longrightarrow g(x,\theta)\quad \pi\text{-}a.s.\quad (n\to\infty).
\end{align*}
Then $g$ is a $\pi$-integrable function and the following types of convergence hold:
\begin{alignat*}{3}
&\text{(i)}\  &&\sup_{\theta\in\Theta}\left|\frac{1}{n}\sum_{i=1}^{n}{g_{i-1}^{(n)}(\theta)-\int{g(x,\theta)\pi(dx)}}\right|\overset{P}{\to}0\quad(n\to\infty),\\
&\text{(ii)}\  &&\sup_{\theta\in\Theta}\left|\frac{1}{n}\sum_{i=1}^{n}{g_{i-1}^{(n)}(\theta)\boldsymbol{1}_{\{|\Delta X_i^n|\le Dh_n^{\rho}\}}-\int{g(x,\theta)\pi(dx)}}\right|\overset{P}{\to}0\quad(n\to\infty),\\
&\text{(iii)}\  &&\sup_{\theta\in\Theta}\left|\frac{1}{nh_n}\sum_{i=1}^{n}{g_{i-1}^{(n)}(\theta)\boldsymbol{1}_{\{|\Delta X_i^n|> Dh_n^{\rho}\}}-\lambda_0\int{g(x,\theta)\pi(dx)}}\right|\overset{P}{\to}0\quad(n\to\infty).
\end{alignat*}
\end{prop}
\begin{rmk}\label{proofcompare2:jump}
The statements of Proposition \ref{prop4:jump} is similar to those of Proposition 3.3 of \citet{Shimizu-Yoshida_JP,Shimizu-Yoshida}. However, our balance conditions for $\varepsilon_n$ are milder than theirs.
\end{rmk}
\begin{rmk}
In the proof of Proposition \ref{prop4:jump}, the stationarity assumption in {\bf [A2]} can be relaxed as follows:
\begin{align*}
&g^{(n)}(x,\theta)\longrightarrow g(x,\theta)\quad \pi\text{-}a.s.\quad (n\to\infty),\\
&\sup_{s\ge 0}\mathbb{E}\left[|g^{(n)}(X_s,\theta)-g(X_s,\theta)|\right]\to 0.
\end{align*}
\end{rmk}
\begin{proof}
The proof is similar to that of Proposition 3.3 of \citet{Shimizu-Yoshida_JP,Shimizu-Yoshida}.
First, the uniform integrability of $g^{(n)}(x,\theta)$ leads to the $\pi$-integrability of $g(x,\theta)$. Next, let us prove that each convergence holds for any $\theta\in\Theta$.
We start with the proof of (i). 
For any $\varepsilon>0$, one has that 
\begin{align*}
&P\left(\left|\frac{1}{n}\sum_{i=1}^{n}{g_{i-1}^{(n)}(\theta)-\int{g(x,\theta)\pi(dx)}}\right|>\varepsilon\right)\\
&\le P\left(\left|\frac{1}{n}\sum_{i=1}^{n}{g_{i-1}^{(n)}(\theta)-\frac{1}{nh_n}\int_{0}^{nh_n}{g^{(n)}(X_s,\theta)ds}}\right|>\frac{\varepsilon}{3}\right)\\
&\quad+P\left(\left|\frac{1}{nh_n}\int_{0}^{nh_n}{g^{(n)}(X_s,\theta)ds-\frac{1}{nh_n}\int_{0}^{nh_n}{g(X_s,\theta)ds}}\right|>\frac{\varepsilon}{3}\right)\\
&\quad+P\left(\left|\frac{1}{nh_n}\int_{0}^{nh_n}{g(X_s,\theta)ds}-\int{g(x,\theta)\pi(dx)}\right|>\frac{\varepsilon}{3}\right).
\end{align*}
The third term on the right-hand side converges to 0 by the assumption of ergodicity.
Let us call the first and second terms $P_n^1$ and $P_n^2$, respectively.
Then,  we see from Taylor's theorem, Schwarz's inequality and Proposition \ref{prop1:jump} that  
\begin{align*}
P_n^1&\le\frac{3}{\varepsilon}\mathbb{E}\left[\left|\frac{1}{n}\sum_{i=1}^{n}{g_{i-1}^{(n)}(\theta)-\frac{1}{nh_n}\int_{0}^{nh_n}{g^{(n)}(X_s,\theta)ds}}\right|\right]\\
&\le\frac{3}{\varepsilon}\mathbb{E}\left[\frac{1}{nh_n}\sum_{i=1}^{n}\int_{t_{i-1}^n}^{t_i^n}\left|{g_{i-1}^{(n)}(\theta)-{g^{(n)}(X_s,\theta)}}\right|ds\right]\\
&=\frac{3}{nh_n\varepsilon}\sum_{i=1}^{n}\int_{t_{i-1}^n}^{t_i^n}\mathbb{E}\left[\left|{g_{i-1}^{(n)}(\theta)-{g^{(n)}(X_s,\theta)}}\right|\right]ds\\
&=\frac{3}{nh_n\varepsilon}\sum_{i=1}^{n}\int_{t_{i-1}^n}^{t_i^n}\mathbb{E}\left[\left|\int_{0}^{1}\partial_x{g^{(n)}(X_{t_{i-1}^n}+u(X_s-X_{t_{i-1}^n}),\theta)}du\right|\left|X_s-X_{t_{i-1}^n}\right|\right]ds\\
&\le\frac{3}{nh_n\varepsilon}\sum_{i=1}^{n}\int_{t_{i-1}^n}^{t_i^n}\mathbb{E}\left[\left|\int_{0}^{1}\partial_x{g^{(n)}(X_{t_{i-1}^n}+u(X_s-X_{t_{i-1}^n}),\theta)}du\right|^2\right]^\frac{1}{2}\mathbb{E}\left[\left|X_s-X_{t_{i-1}^n}\right|^2\right]^\frac{1}{2}ds\\
&\le \frac{C}{nh_n\varepsilon}\sum_{i=1}^{n}\int_{t_{i-1}^n}^{t_i^n}\mathbb{E}\left[\varepsilon_n^{-8}(|X_{t_{i-1}^n}|^2+|X_s-X_{t_{i-1}^n}|^2)\right]^\frac{1}{2}\mathbb{E}\left[\left|X_s-X_{t_{i-1}^n}\right|^2\right]^\frac{1}{2}ds\\
&\le\frac{C}{nh_n\varepsilon}\sum_{i=1}^{n}\int_{t_{i-1}^n}^{t_i^n}O(\sqrt{h_n}\varepsilon_n^{-4})ds\\
&=O(\sqrt{h_n}\varepsilon_n^{-4}).
\end{align*}
Since $h_n\varepsilon_n^{-8}\to 0$ under {\bf{[B1]}}, 
$P_n^1$ converges to 0.
For $P_n^2$,  it follows from stationarity and Lebesgue's convergence theorem that
\begin{align*}
P_n^2&\le\frac{3}{nh_n\varepsilon}\int_{0}^{nh_n}\mathbb{E}\left[\left|g^{(n)}(X_s,\theta)-g(X_s,\theta)\right|\right]ds\\
&=\frac{3}{nh_n\varepsilon}\int_{0}^{nh_n}\int\left|g^{(n)}(x,\theta)-g(x,\theta)\right|\pi(dx)ds\\
&=\frac{3}{\varepsilon}\int\left|g^{(n)}(x,\theta)-g(x,\theta)\right|\pi(dx)\\
&{\to}\ 0.
\end{align*}
For the pointwise convergence of (iii), by \citet{Genon-Catalot}, it is sufficient to show that
\begin{enumerate}
\item[(a)] $\displaystyle{\sum_{i=1}^{n}\mathbb{E}\left[\frac{1}{nh_n}g_{i-1}^{(n)}(\theta)\boldsymbol{1}_{\{|\Delta X_i^n|>Dh_n^\rho\}}\ |\mathcal{F}_{i-1}^n\right]\overset{P}{\longrightarrow}\lambda_0\int{g(x,\theta)\pi(dx)}}$,
\item[(b)]  $\displaystyle{\sum_{i=1}^{n}\mathbb{E}\left[\frac{1}{n^2h_n^2}\left(g_{i-1}^{(n)}(\theta)\right)^2\boldsymbol{1}_{\{|\Delta X_i^n|>Dh_n^\rho\}}\ |\mathcal{F}_{i-1}^n\right]\overset{P}{\longrightarrow} 0}$.
\end{enumerate}
{\bf Proof of (a).}  In an analogous manner to the proof of (i), we can calculate 
\begin{align*}
&P\left(\left|\sum_{i=1}^{n}\mathbb{E}\left[\frac{1}{nh_n}g_{i-1}^{(n)}(\theta)\boldsymbol{1}_{\{|\Delta X_i^n|>Dh_n^\rho\}}\ |\mathcal{F}_{i-1}^n\right]-\lambda_0\int g(x,\theta)\pi(dx)\right|>\varepsilon\right)\\
&\le P\left(\left|\sum_{i=1}^{n}\mathbb{E}\left[\frac{1}{nh_n}g_{i-1}^{(n)}(\theta)\boldsymbol{1}_{\{|\Delta X_i^n|>Dh_n^\rho\}}\ |\mathcal{F}_{i-1}^n\right]-\frac{\lambda_0}{nh_n}\int_{0}^{nh_n}{g^{(n)}(X_s,\theta)ds}\right|>\frac{\varepsilon}{3}\right)\\
&\quad+P\left(\left|\frac{\lambda_0}{nh_n}\int_{0}^{nh_n}{g^{(n)}(X_s,\theta)ds-\frac{\lambda_0}{nh_n}\int_{0}^{nh_n}{g(X_s,\theta)ds}}\right|>\frac{\varepsilon}{3}\right)\\
&\quad+P\left(\left|\frac{\lambda_0}{nh_n}\int_{0}^{nh_n}{g(X_s,\theta)ds}-\lambda_0\int{g(x,\theta)\pi(dx)}\right|>\frac{\varepsilon}{3}\right).
\end{align*}
By stationarity, Lebesgue's convergence theorem and ergodicity, the second and third  terms on the right-hand side converge to 0.
For the first term, it holds from Proposition \ref{prop2:jump} and the evaluation of $P_n^1$ that 
\begin{align*}
&P\left(\left|\sum_{i=1}^{n}\mathbb{E}\left[\frac{1}{nh_n}g_{i-1}^{(n)}(\theta)\boldsymbol{1}_{\{|\Delta X_i^n|>Dh_n^\rho\}}\ |\mathcal{F}_{i-1}^n\right]-\frac{\lambda_0}{nh_n}\int_{0}^{nh_n}{g^{(n)}(X_s,\theta)ds}\right|>\frac{\varepsilon}{3}\right)\\
&\le\frac{3}{\varepsilon}\mathbb{E}\left[\left|\sum_{i=1}^{n}\frac{1}{nh_n}g_{i-1}^{(n)}(\theta)P(|\Delta X_i^n|>Dh_n^\rho\ |\mathcal{F}_{i-1}^n)-\frac{\lambda_0}{nh_n}\int_{0}^{nh_n}{g^{(n)}(X_s,\theta)ds}\right|\right]\\
&\le\frac{3}{nh_n\varepsilon}\sum_{i=1}^n\int_{t_{i-1}^n}^{t_i^n}\mathbb{E}\left[\left|g_{i-1}^{(n)}(\theta)h_n^{-1}P(|\Delta X_i^n|>Dh_n^\rho\ |\mathcal{F}_{i-1}^n)-\lambda_0{g^{(n)}(X_s,\theta)}\right|\right]ds\\
&\le\frac{3}{nh_n\varepsilon}\sum_{i=1}^n\int_{t_{i-1}^n}^{t_i^n}\left\{\mathbb{E}\left[\left|g_{i-1}^{(n)}(\theta)h_n^{-1}P(|\Delta X_i^n|>Dh_n^\rho\ |\mathcal{F}_{i-1}^n)-\lambda_0{g_{i-1}^{(n)}(\theta)}\right|\right]\right.\\
&\qquad\qquad\quad\left.+\lambda_0\ \mathbb{E}\left[\left|{g_{i-1}^{(n)}(\theta)-{g^{(n)}(X_s,\theta)}}\right|\right]\right\}ds\\
&\le\frac{3}{n\varepsilon}\sum_{i=1}^{n}\mathbb{E}\left[\left(g_{i-1}^{(n)}(\theta)\right)^2\right]^{\frac{1}{2}}\mathbb{E}\left[\left(h_n^{-1}P(|\Delta X_i^n|>Dh_n^\rho\ |\mathcal{F}_{i-1}^n)-\lambda_0\right)^2\right]^{\frac{1}{2}}\\
&\quad+\lambda_0\cdot\frac{3}{nh_n\varepsilon}\sum_{i=1}^{n}\int_{t_{i-1}^n}^{t_i^n}\mathbb{E}\left[\left|{g_{i-1}^{(n)}(\theta)-{g^{(n)}(X_s,\theta)}}\right|\right]ds\\
&=O(h_n^\rho)+O(\sqrt{h_n}\varepsilon_n^{-4})\\
&\to 0.
\end{align*}
{\bf Proof of (b).} It follows from Schwarz's inequality that  
\begin{align*}
&P\left(\left|\frac{1}{n^2h_n^2}\sum_{i=1}^{n}\left(g_{i-1}^{(n)}(\theta)\right)^2P(|\Delta X_i^n|>Dh_n^\rho\ |\mathcal{F}_{i-1}^n)\right|>\varepsilon\right)\\
&\le\frac{1}{n^2h_n^2\varepsilon}\sum_{i=1}^{n}\mathbb{E}\left[\left(g_{i-1}^{(n)}(\theta)\right)^4\right]^\frac{1}{2}\mathbb{E}\left[P(|\Delta X_i^n|>Dh_n^\rho\ |\mathcal{F}_{i-1}^n)^2\right]^\frac{1}{2}\\
&=O\left(\frac{1}{nh_n}\right)\\
&\to 0.
\end{align*}
Hence, the pointwise convergence of (iii) holds.
We can easily deduce (ii) for each $\theta\in\Theta$ from (i) and (iii) since
\begin{align*}
\frac{1}{n}\sum_{i=1}^{n}g_{i-1}^{(n)}(\theta)\boldsymbol{1}_{\{|\Delta X_i^n|\le Dh_n^\rho\}}&=
\frac{1}{n}\sum_{i=1}^{n}g_{i-1}^{(n)}(\theta)-h_n\cdot\frac{1}{nh_n}\sum_{i=1}^{n}g_{i-1}^{(n)}(\theta)\boldsymbol{1}_{\{|\Delta X_i^n|> Dh_n^\rho\}}\\
&\overset{P}{\to}0.
\end{align*}
Finally, let us show the uniform convergence in $\theta$.
We only prove (i); the uniformly in (ii) can be shown similarly, and that in (iii) is proved by the same argument as the proof of more general Proposition \ref{prop7:jump}.  Hence, we omit the proof here.
Since
\begin{align*}
\sup_n\mathbb{E}\left[\sup_\theta\left|\partial_{{\theta}}\left( \frac{1}{n}\sum_{i=1}^{n}g_{i-1}^{(n)}(\theta)\right)\right|\right]&\le \sup_n\left(\frac{1}{n}\sum_{i=1}^{n}\mathbb{E}\left[\sup_\theta \left|\partial_{{\theta}}g_{i-1}^{(n)}(\theta)\right|\right]\right)\\
&\le C\sup_n\left(\frac{1}{n}\sum_{i=1}^{n}\mathbb{E}\left[(1+|X_{t_{i-1}^n}|)^C\right]\right)\\
&\le C \sup_{t\ge0} \mathbb{E}\left[1+|X_t|^C\right]\\
&<\infty,
\end{align*}
the uniform convergence for (i) holds.
\end{proof}

\begin{prop}{{(\textbf{\citet{Shimizu-Yoshida_JP,Shimizu-Yoshida}})}}\label{prop5:jump}
Assume {\bf{[A1]}}-{\bf{[A7]}} and $nh_n\to\infty$. 
 Suppose that  a function $g$ : $\mathbb{R}^d\times\Theta\to\mathbb{R}$
and its derivatives $\partial_\theta g$ and $\partial_x g$ are of polynomial growth uniformly in $\theta$:
\begin{align*}
|g(x,\theta)|,\ |\partial_\theta g(x,\theta)|,\  |\partial_x g(x,\theta)|\le C(1+|x|)^C\quad( ^\forall \theta\in\Theta).
\end{align*}
Then, for $k,l=1,2,\ldots,d$,
\begin{align*}
\sup_{\theta\in\Theta}\left|\frac{1}{nh_n}\sum_{i=1}^{n}{g_{i-1}(\theta)\bar{X}_{i,n}^{(k)}\bar{X}_{i,n}^{(l)}\boldsymbol{1}_{\{|\Delta X_i^n|\le Dh_n^{\rho}\}}-\lambda_0\int{g(x,\theta)S^{(k,l)}(x,\alpha_0)\pi(dx)}}\right|\overset{P}{\to}0\quad(n\to\infty).
\end{align*}
\end{prop}

\begin{prop}{{(\textbf{\citet{Shimizu-Yoshida_JP,Shimizu-Yoshida}})}}\label{prop6:jump}
Under the same assumptions as in Proposition \ref{prop5:jump}, for $k=1,2,\ldots,d$,
\begin{align*}
\sup_{\theta\in\Theta}\left|\frac{1}{nh_n}\sum_{i=1}^{n}{g_{i-1}(\theta)\bar{X}_{i,n}^{(k)}\boldsymbol{1}_{\{|\Delta X_i^n|\le Dh_n^{\rho}\}}}\right|\overset{P}{\to}0\quad(n\to\infty).
\end{align*}
\end{prop}

\begin{prop}\label{prop7:jump}
Assume {\bf{[A1]}}-{\bf{[A3]}}, {\bf{[A5]}}-{\bf{[A7]}} and {\bf{[B1]}}. Suppose, $g_n(\alpha,y,x):\Theta\times E\times\mathbb{R}^d\to\mathbb{R}$ satisfies that 
\begin{align*}
&|\partial_y\partial_\theta^m g_n(\theta,y,x)|\le C\cdot\varepsilon_n^{-k-1}(1+|y|)^C(1+|x|)^C\quad(m=0,1),\\
&\int_{B}^{}{\sup_{\theta\in\Theta}|\partial_\theta g_n(\theta,y,x)|\Psi_{\beta_0}(y,x)dy}\le C(1+|x|)^C,
\end{align*}
where $k$ is chosen from $\{1,2,3\}$, and at least one of the following two conditions holds true for $m=0,1$.
\begin{enumerate}
\item[{\bf[P1] :}] $|\partial_\theta^m g_n(\theta,y,x)|\le C(1+|y|)^C(1+|x|)^C$. 
\item[{\bf[P2] :}] $\rho\in B_1(k), \  
\ |\partial_\theta^m g_n(\theta,y,x)|\leq\begin{cases}    
C\cdot \varepsilon_n^{-k}(1+|x|)^C&\text{if}\ k =1,2, \\
C\cdot \varepsilon_n^{-(k+m-1)}(1+|x|)^C&\text{if}\ k =3.\\
\end{cases}$
\end{enumerate}
Moreover, suppose that there exist $\displaystyle G_n(\theta,x)=\int_{B}{g_n(\theta,y,x)\Psi_{\beta_0}(y,x)dy}$ and $g(\theta,y,x): \Theta\times E\times\mathbb{R}^d\to\mathbb{R}$, for all $(\theta,x)\in\Theta\times \mathbb{R}^d$, such that
\begin{align*}
&G_n(\theta,x)\longrightarrow\int_{E}^{}{g(\theta,y,x)\Psi_{\beta_0}(y,x)dy}\quad \pi\text{-}a.s.\quad((\theta,x)\in\Theta\times \mathbb{R}^d),\\
&|G_n(\theta,x)|^4\le C(1+|x|)^C,\quad
|\partial_xG_n(\theta,x)|\le C\cdot\varepsilon_n^{-k-1}(1+|x|)^C.
\end{align*}
Then, 
\begin{align*}
\sup_{\theta\in\Theta}\left|\frac{1}{nh_n}\sum_{i=1}^{n}{g_{n}(\theta,\Delta X_i^n,X_{t_{i-1}^n})\boldsymbol{1}_{\{|\Delta X_i^n|> Dh_n^{\rho}\}}}-\iint_{B}{g(\theta,y,x)\Psi_{\beta_0}(y,x)dy\pi(dx)}\right|\overset{P}{\to}0\quad(n\to\infty).
\end{align*}
\end{prop}

\begin{rmk}\label{proofcompare3:jump}
The statements of Proposition \ref{prop7:jump}  are similar to those of Proposition 3.6 of \citet{Shimizu-Yoshida_JP,Shimizu-Yoshida}. However, our balance conditions for $\varepsilon_n$ are milder than theirs. 
If the statements in this proposition  hold for $k=3$, then it is easy to show that  the statements for $k=1,2$ also hold. 
Thus, it is sufficient to prove the case of $k=3$. However, we show this proof for $k\in\{1,2,3\}$ since we utilize the proof of this proposition in the case of $k=1,2$ for proving Theorem \ref{thm2-2:jump}.
\end{rmk}
\begin{rmk}
 Remarks \ref{balance} and \ref{subset:thm2} show that Proposition \ref{prop7:jump} can be applied under the conditions of Theorem \ref{thm2-2:jump} and Corollary \ref{thm2-1:jump}.
\end{rmk}
\begin{rmk}
Under the additional assumptions {\bf {[A10]}}, {\bf {[A11]}} and {\bf {[A13]}}, the function
$g_n(\theta,y,x)=\partial_\beta^k\left(\log\Psi_\beta(y,x)\right)\varphi_n(x,y),\ (k=0,1,2)$ satisfies the conditions of Proposition \ref{prop7:jump}.
\end{rmk}
\begin{proof}
The proof is similar to that of Proposition 3.6 of \citet{Shimizu-Yoshida_JP,Shimizu-Yoshida}. We show the proof of the case of $k=3$ under {\bf [P2]}, lastly. Therefore, we start with the proof of the case of $k=1,2,3$ under {\bf [P1]} or 
 that of $k=1,2$ under {\bf [P2]}. 
Firstly, let us show the pointwise convergence.
For $p\in(0,\frac{1}{3})$, if we set $q=1+\frac{1}{p}$, then it holds from H\"{o}lder's inequality, Proposition \ref{prop2:jump} and $h_n\varepsilon_n^{-2k}\leq h_n\varepsilon_n^{-8}\to 0$ under {\bf[B1]} that
\begin{align}
&\sum_{j=0,2}P\left(\left|\frac{1}{nh_n}\sum_{i=1}^{n}g_n(\theta,\Delta X_i^n,X_{t_{i-1}^n})\boldsymbol{1}_{D_{i,j}^n(D,\rho)}\right|>\varepsilon\right)\nonumber\\
&\le \frac{1}{nh_n\varepsilon} \sum_{i=1}^{n}\sum_{j=0,2}\mathbb{E}\left[\left|g_n(\theta,\Delta X_i^n,X_{t_{i-1}^n})\boldsymbol{1}_{D_{i,j}^n(D,\rho)}\right|\right]\nonumber\\
&\le \frac{1}{nh_n\varepsilon} \sum_{i=1}^{n}\sum_{j=0,2}\mathbb{E}\left[\left|g_n(\theta,\Delta X_i^n,X_{t_{i-1}^n})\right|^q\right]^\frac{1}{q}P(D_{i,j}^n(D,\rho))^{\frac{1}{1+p}}\nonumber\\
&=\begin{cases}
O\left(h_n^\frac{1-p}{1+p}\right)& 
\text{(under {\bf [P1]})},\\
O\left(h_n^\frac{1-p}{1+p}\varepsilon_n^{-k}\right)=O\left(\sqrt{h_n}\varepsilon_n^{-k}\cdot h_n^\frac{1-3p}{2+2p}\right)& 
\text{(under {\bf [P2]}, $k=1,2$)}\label{p1}
\end{cases}\\
&=o(1).\nonumber
\end{align}
Hence we have the following decomposition:
\begin{align*}
&P\left(\left|\frac{1}{nh_n}\sum_{i=1}^{n}{g_{n}(\theta,\Delta X_i^n,X_{t_{i-1}^n})\boldsymbol{1}_{\{|\Delta X_i^n|> Dh_n^{\rho}\}}}-\iint_{B}{g(\theta,y,x)\Psi_{\beta_0}(y,x)dy\pi(dx)}\right|>\varepsilon\right)\\
&\le P\left(\left|\frac{1}{nh_n}\sum_{i=1}^{n}{g_{n}(\theta,\Delta X_i^n,X_{t_{i-1}^n})}\boldsymbol{1}_{D_{i,1}^n(D,\rho)}-\iint_{B}{g(\theta,y,x)\Psi_{\beta_0}(y,x)dy\pi(dx)}\right|>\varepsilon\right)\\
&\quad+\sum_{j=0,2}P\left(\left|\frac{1}{nh_n}\sum_{i=1}^{n}g_n(\theta,\Delta X_i^n,X_{t_{i-1}^n})\boldsymbol{1}_{D_{i,j}^n(D,\rho)}\right|>\varepsilon\right)\\
&\le \sum_{l=1}^{5}{I_l}+o(1),
\end{align*}
where 
\begin{small}
\begin{align*}
I_1&=P\left(\left|\frac{1}{nh_n}\sum_{i=1}^{n}{g_{n}(\theta,\Delta X_i^n,X_{t_{i-1}^n})}\boldsymbol{1}_{D_{i,1}^n(D,\rho)}-\frac{1}{nh_n}\sum_{i=1}^{n}{g_{n}(\theta,\Delta X_{\tau_i^n},X_{t_{i-1}^n})}\boldsymbol{1}_{D_{i,1}^n(D,\rho)}\right|>\frac{\varepsilon}{5}\right),\\
I_2&=P\left(\left|\frac{1}{nh_n}\sum_{i=1}^{n}{g_{n}(\theta,\Delta X_{\tau_i^n},X_{t_{i-1}^n})}\boldsymbol{1}_{D_{i,1}^n(D,\rho)}-\frac{1}{nh_n}\sum_{i=1}^{n}{g_{n}(\theta,\Delta X_{\tau_i^n},X_{t_{i-1}^n})}\boldsymbol{1}_{\{J_i^n=1\}}\right|>\frac{\varepsilon}{5}\right),\\
I_3&=P\left(\left|\frac{1}{nh_n}\sum_{i=1}^{n}{g_{n}(\theta,\Delta X_{\tau_i^n},X_{t_{i-1}^n})}\boldsymbol{1}_{\{J_i^n=1\}}-\frac{1}{nh_n}\sum_{i=1}^{n}\int_{t_{i-1}^n}^{t_i^n}\int{g_{n}(\theta,c_{i-1}(z,\beta_0),X_{t_{i-1}^n})}p(ds,dz)\right|>\frac{\varepsilon}{5}\right),\\
I_4&=P\left(\left|\frac{1}{nh_n}\sum_{i=1}^{n}\int_{t_{i-1}^n}^{t_i^n}\int{g_{n}(\theta,c_{i-1}(z,\beta_0),X_{t_{i-1}^n})}p(ds,dz)\right.\right.\\
&\qquad\qquad\qquad\qquad\qquad\qquad\qquad\quad\left.\left.-\frac{1}{nh_n}\sum_{i=1}^{n}\int_{t_{i-1}^n}^{t_i^n}\int{g_{n}(\theta,c_{i-1}(z,\beta_0),X_{t_{i-1}^n})}q^{\beta_0}(ds,dz)\right|>\frac{\varepsilon}{5}\right),\\
I_5&=P\left(\left|\frac{1}{nh_n}\sum_{i=1}^{n}\int_{t_{i-1}^n}^{t_i^n}\int{g_{n}(\theta,c_{i-1}(z,\beta_0),X_{t_{i-1}^n})}q^{\beta_0}(ds,dz)-\iint_{B}{g(\theta,y,x)\Psi_{\beta_0}(y,x)dy\pi(dx)}\right|>\frac{\varepsilon}{5}\right).
\end{align*}
\end{small}
Let us evaluate these terms. By Taylor's theorem and Schwarz's inequality,  one has 
\begin{align*}
I_1&\le \frac{5}{nh_n\varepsilon}\sum_{i=1}^{n}\mathbb{E}\left[\left|g_{n}(\theta,\Delta X_i^n,X_{t_{i-1}^n})-g_{n}(\theta,\Delta X_{\tau_i^n},X_{t_{i-1}^n})\right|\boldsymbol{1}_{D_{i,1}^n(D,\rho)}\right]\\
&\le \frac{5}{nh_n\varepsilon}\sum_{i=1}^{n}\mathbb{E}\left[\left|\int_{0}^{1}\partial_y g_n(\theta,\xi_i^n(\eta),X_{t_{i-1}^n})d\eta\right|\left|\Delta X_i^n-\Delta X_{\tau_i^n}\right|\boldsymbol{1}_{D_{i,1}^n(D,\rho)}\right]\\
&\qquad\qquad\qquad\qquad\quad\left(\xi_i^n:=\eta\Delta X_i^n+(1-\eta)\Delta X_{\tau_i^n}\right)\\
&\le \frac{5}{nh_n\varepsilon}\sum_{i=1}^{n}\mathbb{E}\left[\left|\int_{0}^{1}\partial_y g_n(\theta,\xi_i^n(\eta),X_{t_{i-1}^n})d\eta\right|\left(|X_i^n-X_{\tau_i^n}|+|X_{\tau_{i-}^n}-X_{t_{i-1}^n}|\right)\boldsymbol{1}_{\{J_i^n=1\}}\right]\\
&\le  \frac{C}{nh_n\varepsilon}\sum_{i=1}^{n}\mathbb{E}\left[\mathbb{E}\left[\left|\int_{0}^{1}\partial_y g_n(\theta,\xi_i^n(\eta),X_{t_{i-1}^n})d\eta\right|^2\ |J_i^n=1\right]^\frac{1}{2}\right.\\
&\qquad\qquad\qquad\quad\times\left.\mathbb{E}\left[|X_i^n-X_{\tau_i^n}|^2+|X_{\tau_{i-}^n}-X_{t_{i-1}^n}|^2\ |J_i^n=1\right]^\frac{1}{2}\boldsymbol{1}_{\{J_i^n=1\}}\right].
\end{align*}
  Let $\tilde{X}$ be the solution of the following stochastic differential equation under the set $\{J_i^n=1\}$:
 \begin{align*}
\tilde{X}_t-\tilde{X}_{t_{i-1}^n}=H_t+\int_{t_{i-1}^n}^{t}b(\tilde{X}_s)ds+\int_{t_{i-1}^n}^{t}a(\tilde{X}_s)dW_s,
\end{align*}
where $\tilde{X}_{t_{i-1}^n}=X_{t_{i-1}^n}$, $H_t=c(X_{u-},z)\boldsymbol{1}_{[u,t_i^n]}(t)$, $u$ is $[t_{i-1}^n,t_i^n]$-valued uniform random variable which is independent of $(W_t)_{t\ge0}$ and $J_i^n$, and $z$ is a random variable with density $F_{\beta_0}$ which is independent of $(W_t)_{t\ge0}$.
 We see from Burkholder-Davis-Gundy inequality that 
\begin{align*}
\mathbb{E}\left[|X_{\tau_{i-}^n}-X_{t_{i-1}^n}|^2\ |J_i^n=1\right]&=
\mathbb{E}\left[|\tilde{X}_{u-}-\tilde{X}_{t_{i-1}^n}|^2\ |J_i^n=1\right]\\
&=\mathbb{E}\left[|\tilde{X}_{u-}-\tilde{X}_{t_{i-1}^n}|^2\right]\\
&\le\mathbb{E}\left[\sup_{t\in[t_{i-1}^n,u-]}|\tilde{X}_{t}-\tilde{X}_{t_{i-1}^n}|^2\right] \\
&\le C\left(h_n^2\ \mathbb{E}\left[b(\tilde{X}_{t_{i-1}^n})\right]+\mathbb{E}\left[\int_{t_{i-1}^n}^{t_i^n}{a^2(\tilde{X}_s)}ds\right]\right)\\
&=O(h_n).
\end{align*}
In a similar way,  one has
\begin{align*}
\mathbb{E}\left[|X_{t_{i}^n}-X_{\tau_{i}^n}|^2\ |J_i^n=1\right]=O(h_n).
\end{align*}
Hence, 
\begin{align*}
I_1\le \frac{C}{nh_n\varepsilon}\sum_{i=1}^{n}{O\left(\sqrt{h_n}\varepsilon_n^{-k-1}\right)P(J_i^n=1)}=O\left(\sqrt{h_n}\varepsilon_n^{-k-1}\right)\to\ 0,
\end{align*}
since $h_n\varepsilon_n^{-2(k+1)}\leq h_n\varepsilon_n^{-8}\to 0$ under {\bf [B1]}.
For $I_2$, if the condition {\bf [P1]} holds, then there exist an integer $p>1$ and a sequence $u_n>0$ such that $h_n^\rho{u_n}\to0$ and $h_n{u_n}^p\to\infty$ as $n\to\infty$. For example, $u_n=h_n^{-\rho/2}$, $p=2\left(\left[\frac{1}{\rho}\right]+1\right)$.
Therefore,  we see that under {\bf [P1]}
\begin{align*}
I_2&\le \frac{5}{\varepsilon}\mathbb{E}\left[\left|\frac{1}{nh_n}\sum_{i=1}^{n}g_{n}(\theta,\Delta X_{\tau_i^n},X_{t_{i-1}^n})\right|\boldsymbol{1}_{C_{i,1}^n(D,\rho)}\right]\\
&\le\frac{5}{nh_n\varepsilon}\sum_{i=1}^{n}\mathbb{E}\left[\mathbb{E}\left[\left|g_{n}(\theta,\Delta X_{\tau_i^n},X_{t_{i-1}^n})\right|\boldsymbol{1}_{C_{i,1}^n(D,\rho)}\ |\mathcal{F}_{i-1}^n\right]\right]\\
&=\frac{5}{nh_n\varepsilon}\sum_{i=1}^{n}\mathbb{E}\left[\mathbb{E}\left[\left|g_{n}(\theta,\Delta X_{\tau_i^n},X_{t_{i-1}^n})\right|\boldsymbol{1}_{C_{i,1}^n(D,\rho)}\boldsymbol{1}_{\{|g_n(\theta,\Delta X_{\tau_{i}^n},X_{t_{i-1}^n})|>u_n\}}\ |\mathcal{F}_{i-1}^n\right]\right]\\
&\quad+\frac{5}{nh_n\varepsilon}\sum_{i=1}^{n}\mathbb{E}\left[\mathbb{E}\left[\left|g_{n}(\theta,\Delta X_{\tau_i^n},X_{t_{i-1}^n})\right|\boldsymbol{1}_{C_{i,1}^n(D,\rho)}\boldsymbol{1}_{\{|g_n(\theta,\Delta X_{\tau_{i}^n},X_{t_{i-1}^n})|\le u_n\}}\ |\mathcal{F}_{i-1}^n\right]\right]\\
&\le \frac{C}{nh_n}\sum_{i=1}^{n}u_n^{-p}\mathbb{E}\left[\left|g_{n}(\theta,\Delta X_{\tau_i^n},X_{t_{i-1}^n})\right|^{p+1}\right]+\frac{C}{nh_n}\sum_{i=1}^{n}u_nP(C_{i,1}^n(D,\rho))\\
&=O\left(\frac{1}{h_nu_n^p}\right)+O\left(h_n^\rho u_n\right)\\
&\to 0.
\end{align*}
On the other hand, since $\rho\in B_1(k)$,  it holds that under {\bf [P2]} with $k=1,2$, 
\begin{align}
I_2&\le \frac{5}{\varepsilon}\mathbb{E}\left[\left|\frac{1}{nh_n}\sum_{i=1}^{n}g_{n}(\theta,\Delta X_{\tau_i^n},X_{t_{i-1}^n})\right|\boldsymbol{1}_{C_{i,1}^n(D,\rho)}\right]\nonumber\\
&\le\frac{5}{nh_n\varepsilon}\sum_{i=1}^{n}\mathbb{E}\left[\left|g_{n}(\theta,\Delta X_{\tau_i^n},X_{t_{i-1}^n})\right|\boldsymbol{1}_{C_{i,1}^n(D,\rho)}\ |\mathcal{F}_{i-1}^n\right]\nonumber\\
&\le\frac{5\varepsilon_n^{-k}}{nh_n\varepsilon}\sum_{i=1}^{n}P(C_{i,1}^n(D,\rho))\nonumber\\
&=O\left(h_n^\rho\varepsilon_n^{-k}\right)\label{p2}\\
&\to 0.\nonumber
\end{align}
After all, $I_2$ tends to 0 for each conditions except for the case of $k=3$ under {\bf [P2]}.
We divide $I_3$ into the two terms:
\begin{small}
\begin{align*}
I_3&\le P\left(\left|\frac{1}{nh_n}\sum_{i=1}^{n}{g_{n}(\theta,\Delta X_{\tau_i^n},X_{t_{i-1}^n})}\boldsymbol{1}_{\{J_i^n=1\}}-\frac{1}{nh_n}\sum_{i=1}^{n}{g_{n}(\theta,c_{i-1}(\Delta Z_{\tau_i^n},\beta_0),X_{t_{i-1}^n})}\boldsymbol{1}_{\{J_i^n=1\}}\right|>\frac{\varepsilon}{10}\right)\\
&\quad+P\left(\left|\frac{1}{nh_n}\sum_{i=1}^{n}{g_{n}(\theta,c_{i-1}(\Delta Z_{\tau_i^n},\beta_0),X_{t_{i-1}^n})}\boldsymbol{1}_{\{J_i^n=1\}}\right.\right.\\
&\qquad\qquad\qquad\qquad\qquad\qquad\qquad\qquad\left.\left.-\frac{1}{nh_n}\sum_{i=1}^{n}\int_{t_{i-1}^n}^{t_i^n}\int{g_{n}(\theta,c_{i-1}(z,\beta_0),X_{t_{i-1}^n})}p(ds,dz)\right|>\frac{\varepsilon}{10}\right) \\
&=:I_3'+I_3'',
\end{align*}
\end{small}
where $I_3'=O\left(\sqrt{h_n}\varepsilon_n^{-k-1}\right)$ by the same argument as $I_1$. For $I_3''$, it holds from Schwarz's inequality and Proposition \ref{prop2:jump} that 
\begin{small}
\begin{align}
I_3''&\le\frac{10}{nh_n\varepsilon}\sum_{i=1}^{n}\mathbb{E}\left[\left|{g_{n}(\theta,c_{i-1}(\Delta Z_{\tau_i^n},\beta_0),X_{t_{i-1}^n})}\boldsymbol{1}_{\{J_i^n=1\}}-\int_{t_{i-1}^n}^{t_i^n}\int{g_{n}(\theta,c_{i-1}(z,\beta_0),X_{t_{i-1}^n})}p(ds,dz)\right|\right]\nonumber\\
&\le\frac{10}{nh_n\varepsilon}\sum_{i=1}^{n}\mathbb{E}\left[\left|\int_{t_{i-1}^n}^{t_i^n}\int{g_{n}(\theta,c_{i-1}(z,\beta_0),X_{t_{i-1}^n})}\boldsymbol{1}_{\{J_i^n\ge2\}}p(ds,dz)\right|\right]\nonumber\\
&\le\frac{10}{nh_n\varepsilon}\sum_{i=1}^{n}\mathbb{E}\left[\left|\int_{t_{i-1}^n}^{t_i^n}\int{g_{n}(\theta,c_{i-1}(z,\beta_0),X_{t_{i-1}^n})}p(ds,dz)\right|^2\right]^\frac{1}{2}P(J_i^n\ge2)^\frac{1}{2}\nonumber\\
&\le \frac{C}{n\varepsilon}\sum_{i=1}^{n}\mathbb{E}\left[\int_{t_{i-1}^n}^{t_i^n}\int{g_{n}^2(\theta,c_{i-1}(z,\beta_0),X_{t_{i-1}^n})}q^{\beta_0}(ds,dz)\right]^\frac{1}{2}\nonumber\\
&=\begin{cases}
O(\sqrt{h_n})& 
\text{(under {\bf [P1]}),}\\
O\left(\sqrt{h_n}\varepsilon_n^{-k}\right)& 
\text{(under {\bf [P2]}, $k=1,2$)}.\label{p3}
\end{cases}
\end{align}
\end{small}
Hence, $I_3=I_3'+I_3''=O\left(\sqrt{h_n}\varepsilon_n^{-k-1}\right)\to 0$.
Furthermore, 
\begin{align}
I_4&\le \frac{25}{\varepsilon^2}\mathbb{E}\left[\left|\frac{1}{nh_n}\sum_{i=1}^{n}\int_{t_{i-1}^n}^{t_i^n}\int{g_{n}(\theta,c_{i-1}(z,\beta_0),X_{t_{i-1}^n})}(p-q^{\beta_0})(ds,dz)\right|^2\right]\nonumber\\
&=\frac{25}{n^2h_n^2\varepsilon^2}\sum_{i=1}^{n}\mathbb{E}\left[\left|\int_{t_{i-1}^n}^{t_i^n}\int{g_{n}(\theta,c_{i-1}(z,\beta_0),X_{t_{i-1}^n})}(p-q^{\beta_0})(ds,dz)\right|^2\right]\nonumber\\
&\quad+\frac{50}{n^2h_n^2\varepsilon^2}\sum_{i<j}\mathbb{E}\left[\int_{t_{i-1}^n}^{t_i^n}\int{g_{n}(\theta,c_{i-1}(z,\beta_0),X_{t_{i-1}^n})}(p-q^{\beta_0})(ds,dz)\right.\nonumber\\
&\qquad\qquad\qquad\qquad\quad\times \left.\mathbb{E}\left[\int_{t_{j-1}^n}^{t_j^n}\int{g_{n}(\theta,c_{j-1}(z,\beta_0),X_{t_{j-1}^n})}(p-q^{\beta_0})(ds,dz)\ |\mathcal{F}_{j-1}^n\right]\right]\nonumber\\
&=\frac{25}{n^2h_n^2\varepsilon^2}\sum_{i=1}^{n}\mathbb{E}\left[\int_{t_{i-1}^n}^{t_i^n}\int{g_{n}^2(\theta,c_{i-1}(z,\beta_0),X_{t_{i-1}^n})}q^{\beta_0}(ds,dz)\right]\nonumber\\
&=\begin{cases}
O\left(\frac{1}{nh_n}\right)& 
\text{(under {\bf [P1]}),}\\
O\left(\frac{1}{nh_n\varepsilon_n^{2k}}\right)& 
\text{(under {\bf [P2]}, $k=1,2$)}\label{p4}
\end{cases}\\
&\to\ 0.\nonumber
\end{align}
On $I_5$, it is obvious that this converges to 0 from change of variables and Proposition \ref{prop4:jump}-(i).
Hence, the pointwise convergence holds.
Next, let us show the uniformly of convergence. We set
\begin{align*}
s_n(\theta)=\frac{1}{nh_n}\sum_{i=1}^{n}g_n(\theta,\Delta X_i^n,X_{t_{i-1}^n})\boldsymbol{1}_{\{|\Delta X_i^n|>Dh_n^\rho\}},
\end{align*}
and then it is sufficient to show the tightness of $\{s_n(\theta)\}$. 
 It follows from H\"{o}lder's inequality that  
\begin{align*}
\mathbb{E}\left[\sup_{\theta}|\partial_{{\theta}}s_n(\theta)|\right]&\le
\frac{1}{nh_n}\sum_{i=1}^{n}\sum_{j=0}^{2}\mathbb{E}\left[\sup_{\theta}|\partial_{{\theta}}g_n(\theta,\Delta X_i^n,X_{t_{i-1}^n})|\boldsymbol{1}_{D_{i,j}^n(D,\rho)}\right]\\
&=\frac{1}{nh_n}\sum_{i=1}^{n}\mathbb{E}\left[\sup_{\theta}|\partial_{{\theta}}g_n(\theta,\Delta X_i^n,X_{t_{i-1}^n})|\boldsymbol{1}_{D_{i,1}^n(D,\rho)}\right]\\
&\quad+\frac{1}{nh_n}\sum_{i=1}^{n}\sum_{j=0,2}\mathbb{E}\left[\sup_{\theta}|\partial_{{\theta}}g_n(\theta,\Delta X_i^n,X_{t_{i-1}^n})|\boldsymbol{1}_{D_{i,j}^n(D,\rho)}\right]\\
&=\frac{1}{nh_n}\sum_{i=1}^{n}\mathbb{E}\left[\sup_{\theta}|\partial_{{\theta}}g_n(\theta,\Delta X_i^n,X_{t_{i-1}^n})|\boldsymbol{1}_{D_{i,1}^n(D,\rho)}\right]+o(1).
\end{align*}
Since 
\begin{align*}
\int_{B}^{}{\sup_{\theta\in\Theta}|\partial_\theta g_n(\theta,y,x)|\Psi_{\beta_0}(y,x)dy}\le C(1+|x|)^C,
\end{align*}
if we show
\begin{small}
\begin{align*}
H:=\left|\mathbb{E}\left[\frac{1}{nh_n}\sum_{i=1}^{n}\sup_{\theta}|\partial_{{\theta}}g_n(\theta,\Delta X_i^n,X_{t_{i-1}^n})|\boldsymbol{1}_{D_{i,1}^n(D,\rho)}-\iint_B \sup_{\theta}|\partial_{{\theta}}g_n(\theta,y,x)|\Psi_{\beta_0}(y,x)dy\pi(dx)|\right]\right|=o(1),
\end{align*}
\end{small}
then it holds that $\mathbb{E}\left[\sup_{\theta}|\partial_{{\theta}}s_n(\theta)|\right]<\infty$, and we complete the proof of the tightness of $\{s_n(\theta)\}$. 
We calculate that $\displaystyle{H\le\sum_{l=1}^{5}{H_l}}$, where
\begin{footnotesize}
\begin{align*}
H_1&=\left|\mathbb{E}\left[\frac{1}{nh_n}\sum_{i=1}^{n}\sup_{\theta}|\partial_{{\theta}}g_n(\theta,\Delta X_i^n,X_{t_{i-1}^n})|\boldsymbol{1}_{D_{i,1}^n(D,\rho)}-\frac{1}{nh_n}\sum_{i=1}^{n}\sup_{\theta}|\partial_{{\theta}}g_n(\theta,\Delta X_{\tau_i^n},X_{t_{i-1}^n})|\boldsymbol{1}_{D_{i,1}^n(D,\rho)}|\right]\right|,\\
H_2&=\left|\mathbb{E}\left[\frac{1}{nh_n}\sum_{i=1}^{n}\sup_{\theta}|\partial_{{\theta}}g_n(\theta,\Delta X_{\tau_i^n},X_{t_{i-1}^n})|\boldsymbol{1}_{D_{i,1}^n(D,\rho)}-\frac{1}{nh_n}\sum_{i=1}^{n}\sup_{\theta}|\partial_{{\theta}}g_n(\theta,\Delta X_{\tau_i^n},X_{t_{i-1}^n})|\boldsymbol{1}_{\{J_i^n=1\}}|\right]\right|,\\
H_3&=\left|\mathbb{E}\left[\frac{1}{nh_n}\sum_{i=1}^{n}\sup_{\theta}|\partial_{{\theta}}g_n(\theta,\Delta X_{\tau_i^n},X_{t_{i-1}^n})|\boldsymbol{1}_{\{J_i^n=1\}}-\frac{1}{nh_n}\sum_{i=1}^{n}\int_{t_{i-1}^n}^{t_i^n}\int\sup_{\theta}|\partial_{{\theta}}g_n(\theta,c_{i-1}(z,\beta_0),X_{t_{i-1}^n})|p(ds,dz)\right]\right|,\\
H_4&=\left|\mathbb{E}\left[\frac{1}{nh_n}\sum_{i=1}^{n}\int_{t_{i-1}^n}^{t_i^n}\int\sup_{\theta}|\partial_{{\theta}}g_n(\theta,c_{i-1}(z,\beta_0),X_{t_{i-1}^n})|p(ds,dz)\right.\right.\\
&\qquad\qquad\qquad\qquad\qquad\qquad\qquad\qquad\left.\left.-\frac{1}{nh_n}\sum_{i=1}^{n}\int_{t_{i-1}^n}^{t_i^n}\int\sup_{\theta}|\partial_{{\theta}}g_n(\theta,c_{i-1}(z,\beta_0),X_{t_{i-1}^n})|q(ds,dz)\right]\right|,\\
H_5&=\left|\mathbb{E}\left[\frac{1}{nh_n}\sum_{i=1}^{n}\int_{t_{i-1}^n}^{t_i^n}\int\sup_{\theta}|\partial_{{\theta}}g_n(\theta,c_{i-1}(z,\beta_0),X_{t_{i-1}^n})|q(ds,dz)-\iint_B \sup_{\theta}|\partial_{{\theta}}g_n(\theta,y,x)|\Psi_{\beta_0}(y,x)dy\pi(dx)\right]\right|.
\end{align*}
\end{footnotesize}
By a similar argument to that for $I_l$ $(l=1,2,3)$, we have
\begin{align}
H_1&=O\left(\sqrt{h_n}\varepsilon_n^{-k-1}\right),\quad H_3=O\left(\sqrt{h_n}\varepsilon_n^{-k-1}\right)\label{p5},\\
H_2&=\begin{cases}
O\left(\frac{1}{h_nu_n^p}\right)+O\left(h_n^\rho u_n\right)& 
\text{(under {\bf[P1]}),}\label{p6}\\
O\left(h_n^\rho\varepsilon_n^{-k}\right)& 
\text{(under {\bf[P2]}).}
\end{cases}
\end{align}
Moreover, it is obvious that $H_4=0$ from martingale property, and that $H_5=0$ from changes of variables and stationarity.
Hence, $\sum_{l=1}^{5}{H_l}\to 0$, and we have
$\mathbb{E}\left[\sup_{\theta}|\partial_{{\theta}}s_n(\theta)|\right]<\infty$.
Finally, we evaluate the case of $k=3$ under {\bf [P2]}. In a similar way to the case of $k=1,2$ under {\bf [P2]}, we have the following modified evaluations:
\begin{align*}
\text{(\ref{p1})}&=O\left(h_n^{\frac{1-p}{1+p}}\varepsilon_n^{-(k-1)}\right)=O\left(\sqrt{h_n}\varepsilon_n^{-(k-1)}\cdot h_n^\frac{1-3p}{2+2p}\right)\to 0,\\
\text{(\ref{p2})}&=O\left(h_n^\rho\varepsilon_n^{-(k-1)}\right)\to 0,\\
\text{(\ref{p3})}&=O\left(\sqrt{h_n}\varepsilon_n^{-(k-1)}\right)\to 0,\\
\text{(\ref{p4})}&=O\left(\frac{1}{nh_n\varepsilon_n^{2(k-1)}}\right)\to 0,
\end{align*}
and the others are the same as the case of $k=1,2$ under {\bf [P2]}.
This completes the proof. 
\end{proof}


\subsubsection{Proof of Theorem \ref{thm1-2:jump}}\label{ada:consitencyproof}
\begin{proof}[{\bf Consistency for $\check{\alpha}_{n}$.} ]
We define the function $U_1(\alpha,\alpha_0)$ as follows:
\[U_1(\alpha,\alpha_0):=-\frac{1}{2}\int_{}^{}{\left\{\mathrm{tr}\left(S^{-1}(x,\alpha)S(x,\alpha_0)\right)+\log\det S(x,\alpha)\right\}}\pi(dx).\]
Since $\Delta X_i^n=\bar{X}_{i,n}(\beta_0)+h_n b_{i-1}(\beta_0)$, 
\begin{align*}
\frac{1}{n}l_n^{(1)}(\alpha)&=-\frac{1}{2nh_n}\sum_{i=1}^{n}\{\bar{X}_{i,n}(\beta_0)+h_n b_{i-1}(\beta_0)\}^\top S_{i-1}^{-1}(\alpha)\{\bar{X}_{i,n}(\beta_0)+h_n b_{i-1}(\beta_0)\}\boldsymbol{1}_{\{|\Delta X_i^n|\le D_1h_n^{\rho_1}\}}\\
&\quad -\frac{1}{2n}\sum_{i=1}^{n}\log\det S_{i-1}(\alpha)\boldsymbol{1}_{\{|\Delta X_i^n|\le D_1h_n^{\rho_1}\}}\\
&=-\frac{1}{2nh_n}\sum_{i=1}^{n}\bar{X}_{i,n}(\beta_0)^\top S_{i-1}^{-1}(\alpha)\bar{X}_{i,n}(\beta_0)\boldsymbol{1}_{\{|\Delta X_i^n|\le D_1h_n^{\rho_1}\}}\\
&\quad-h_n\cdot\frac{1}{nh_n}\sum_{i=1}^{n}b_{i-1}^\top(\beta_0)S_{i-1}^{-1}(\alpha)\bar{X}_{i,n}(\beta_0)\boldsymbol{1}_{\{|\Delta X_i^n|\le D_1h_n^{\rho_1}\}}\\
&\quad-h_n\cdot\frac{1}{2n}\sum_{i=1}^{n}b_{i-1}^\top(\beta_0)S_{i-1}^{-1}(\alpha)b_{i-1}(\beta_0)\boldsymbol{1}_{\{|\Delta X_i^n|\le D_1h_n^{\rho_1}\}}\\
&\quad -\frac{1}{2n}\sum_{i=1}^{n}\log\det S_{i-1}(\alpha)\boldsymbol{1}_{\{|\Delta X_i^n|\le D_1h_n^{\rho_1}\}}.
\end{align*}
Therefore, we see from Propositions \ref{prop4:jump}-(ii), \ref{prop5:jump} and \ref{prop6:jump} that
\begin{align}
&\sup_{\alpha\in\Theta_\alpha}\left|\frac{1}{n}l_n^{(1)}(\alpha)-U_1(\alpha,\alpha_0)\right|\notag\\
&\leq \sup_{\alpha\in\Theta_\alpha}\left|-\frac{1}{2nh_n}\sum_{i=1}^{n}\bar{X}_{i,n}(\beta_0)^\top S_{i-1}^{-1}(\alpha)\bar{X}_{i,n}(\beta_0)\boldsymbol{1}_{\{|\Delta X_i^n|\le D_1h_n^{\rho_1}\}}+\frac{1}{2}\int_{}^{}{\mathrm{tr}\left(S^{-1}(x,\alpha)S(x,\alpha_0)\right)}\pi(dx)\right|\notag\\
&\quad+\sup_{\alpha\in\Theta_\alpha}\left|-h_n\cdot\frac{1}{nh_n}\sum_{i=1}^{n}b_{i-1}^\top(\beta_0)S_{i-1}^{-1}(\alpha)\bar{X}_{i,n}(\beta_0)\boldsymbol{1}_{\{|\Delta X_i^n|\le D_1h_n^{\rho_1}\}}\right|\notag\\
&\quad+\sup_{\alpha\in\Theta_\alpha}\left|-h_n\cdot\frac{1}{2n}\sum_{i=1}^{n}b_{i-1}^\top(\beta_0)S_{i-1}^{-1}(\alpha)b_{i-1}(\beta_0)\boldsymbol{1}_{\{|\Delta X_i^n|\le D_1h_n^{\rho_1}\}}\right|\notag\\
&\quad+\sup_{\alpha\in\Theta_\alpha}\left|-\frac{1}{2n}\sum_{i=1}^{n}\log\det S_{i-1}(\alpha)\boldsymbol{1}_{\{|\Delta X_i^n|\le D_1h_n^{\rho_1}\}}+\frac{1}{2}\int_{}^{}{\log\det S(x,\alpha)}\pi(dx)\right|\notag\\
&\overset{P}{\to}0\label{a-cons:uni}.
\end{align}
By the assumption {\bf [A8]}, let $Z_x(y,\alpha)=\frac{1}{(\sqrt{2\pi})^d(\det S(x,\alpha))^\frac{1}{2}}\exp\left\{-\frac{1}{2}y^\top S^{-1}(x,\alpha)y\right\}, U(y)=\frac{Z_x(y,\alpha)}{Z_x(y,\alpha_0)}, Y\sim Z_x(y,\alpha_0)$. Then 
\begin{align*}
    \mathbb{E}[U(Y)]=\int U(y)Z_x(y,\alpha_0)dy=\int Z_x(y,\alpha)dy=1,
\end{align*}
and  it follows from Jensen's inequality that  
\begin{align*}
    -\log\mathbb{E}[U(Y)]\le \mathbb{E}[-\log U(Y)]
\end{align*}
with equality if and only if  the distribution of $U(Y)$ is degenerate,
that is, 
\begin{align*}
    U(Y)=1\ \mathrm{a.e.}\ &\Longleftrightarrow\ Z_x(y,\alpha)=Z_x(y,\alpha_0)\ \mathrm{a.e.}\\
    &\Longleftrightarrow\ S(x,\alpha)=S(x,\alpha_0).
\end{align*}
Therefore, we have
\begin{align*}
    0&\ge \mathbb{E}[\log U(Y)]\\
    &= \mathbb{E}[\log Z_x(Y,\alpha)]-\mathbb{E}[\log Z_x(Y,\alpha_0)]\\
    &=-\frac{1}{2}\log\det S(x,\alpha)-\frac{1}{2}\tr\left(S^{-1}(x,\alpha)S(x,\alpha_0)\right)\\
    &\quad+\frac{1}{2}\log\det S(x,\alpha_0)+\frac{1}{2}\tr\left(S^{-1}(x,\alpha_0)S(x,\alpha_0)\right).
\end{align*}
Hence,
\begin{align*}
    -\frac{1}{2}&\int{\left\{\mathrm{tr}\left(S^{-1}(x,\alpha)S(x,\alpha_0)\right)+\log\det S(x,\alpha)\right\}}\pi(dx)\\
    &\le
    -\frac{1}{2}\int{\left\{\mathrm{tr}\left(S^{-1}(x,\alpha_0)S(x,\alpha_0)\right)+\log\det S(x,\alpha_0)\right\}}\pi(dx)
\end{align*}
with equality if and only if
\begin{align*}
    S(x,\alpha)=S(x,\alpha_0)\ \mathrm{for\ a.s.\ all}\ x.
\end{align*}
Thus,  it follows from {\bf [A9]} that for all $\alpha\in\Theta_\alpha$,
\begin{align*}
U_1(\alpha,\alpha_0)\leq U_1(\alpha_0,\alpha_0)(=0)
\end{align*}
with equality if and only if $\alpha=\alpha_0$. Therefore,  it holds that for all $\varepsilon>0$,
\begin{align}\label{a-cons:iden}
\sup_{\alpha:|\alpha-\alpha_0|\geq\varepsilon}U_1(\alpha,\alpha_0)<U_1(\alpha_0,\alpha_0)(=0),
\end{align}
and  we see from the definition of $\check{\alpha}$ that for all $\varepsilon>0$,
\begin{align}\label{a-cons:compose}
P\left(\frac{1}{n}l_n^{(1)}(\check{\alpha})+\epsilon<\frac{1}{n}l_n^{(1)}(\alpha_0)\right)=0.
\end{align}
Hence, for all $\varepsilon>0$, there  exists $\delta>0$ such that
\[\sup_{\alpha:|\alpha-\alpha_0|\geq\varepsilon}U_1(\alpha,\alpha_0)<U_1(\alpha_0,\alpha_0)-\delta.\]
Thus,  it follows from (\ref{a-cons:uni}) and (\ref{a-cons:compose}) that
\begin{align*}
0\leq P\left(|\check{\alpha}_n-\alpha_0|\geq\varepsilon\right)&\leq P\left(U_1(\check{\alpha}_n,\alpha_0)<U_1(\alpha_0,\alpha_0)-\delta\right)\\
&\leq P\left(U_1(\check{\alpha}_n,\alpha_0)-\frac{1}{n}l_n^{(1)}(\check{\alpha})<-\frac{\delta}{3}\right)\\
&\quad+P\left(\frac{1}{n}l_n^{(1)}(\check{\alpha})-\frac{1}{n}l_n^{(1)}(\alpha_0)<-\frac{\delta}{3}\right)\\
&\quad+P\left(\frac{1}{n}l_n^{(1)}(\alpha_0)-U_1(\alpha_0,\alpha_0)<-\frac{\delta}{3}\right)\\
&\leq 2P\left(\sup_{\alpha\in\Theta_\alpha}\left|\frac{1}{n}l_n^{(1)}(\alpha)-U_1(\alpha,\alpha_0)\right|>\frac{\delta}{3}\right)\\
&\quad+P\left(\frac{1}{n}l_n^{(1)}(\check{\alpha})+\frac{\delta}{3}<\frac{1}{n}l_n^{(1)}(\alpha_0)\right)\\
&\to 0.
\end{align*}
This means that
\begin{align}\label{a-cons:result}
\check{\alpha}_n\overset{P}{\to}\alpha.
\end{align}

{\bf Consistency for $\check{\beta}_{n}$.} 
Let $\bar{U}_{\beta_0}^{(2)}(\alpha,\beta)$, $\tilde{U}_{\beta_0}^{(2)}(\beta)$ and $V_{\beta_0}(\alpha,\beta)$ be functions as follows:
\begin{align*}
&\bar{U}_{\beta_0}^{(2)}(\alpha,\beta):= -\frac{1}{2}\int{(b(x,\beta)-b(x,\beta_0))^\top S^{-1}(x,\alpha)(b(x,\beta)-b(x,\beta_0))}\pi(dx),\\
&\tilde{U}_{\beta_0}^{(2)}(\beta):=\iint_A{\left\{(\log\Psi_\beta (y,x))\Psi_{\beta_0} (y,x)-\Psi_\beta (y,x)\right\}}dy\pi(dx),\\
&V_{\beta_0}(\alpha,\beta):=\bar{U}_{\beta_0}^{(2)}(\alpha,\beta)+\tilde{U}_{\beta_0}^{(2)}(\beta)-\tilde{U}_{\beta_0}^{(2)}(\beta_0)
\end{align*}
It follows from Propositions \ref{prop4:jump}-(i) and \ref{prop7:jump} that 
\begin{align}\label{b-cons:uni_jump}
&\sup_{\beta\in\Theta_\beta}\left|\frac{1}{nh_n}\tilde{l}_n^{(2)}(\beta)-\tilde{U}_{\beta_0}^{(2)}(\beta)\right|\notag\\
&\le\sup_{\beta\in\Theta_\beta}\left|\frac{1}{nh_n}\sum_{i=1}^{n}\{\log \Psi_\beta(\Delta X_i^n,X_{t_{i-1}^n})\}\varphi_n(X_{t_{i-1}^n},\Delta X_i^n)\boldsymbol{1}_{\{|\Delta X_i^n|> D_2h_n^{\rho_2}\}}\right.\notag\\
&\qquad\qquad\qquad\qquad\qquad\qquad\qquad\qquad\qquad\qquad\quad\left.-\iint_A{(\log\Psi_\beta (y,x))\Psi_{\beta_0} (y,x)}dy\pi(dx)\right|\notag\\
&\quad+\sup_{\beta\in\Theta_\beta}\left|-\frac{1}{n}\sum_{i=1}^{n}{\int_{B}{\Psi_\beta(y,X_{t_{i-1}^n})\varphi_n(X_{t_{i-1}^n},y)dy}}+\iint_A{\Psi_\beta (y,x)}dy\pi(dx)\right|\notag\\
&\overset{P}{\to}0.
\end{align}
Moreover, we can calculate  
\begin{align*}
\frac{1}{nh_n}\bar{l}_n^{(2)}(\beta|\bar{\alpha})&=-\frac{1}{nh_n}\sum_{i=1}^{n}(b_{i-1}(\beta)-b_{i-1}(\beta_0))^\top S_{i-1}^{-1}(\bar{\alpha})\bar{X}_{i,n}(\beta_0)\boldsymbol{1}_{\{|\Delta X_i^n|\le D_3h_n^{\rho_3}\}}\\
&\quad+\frac{1}{2n}\sum_{i=1}^{n}(b_{i-1}(\beta)-b_{i-1}(\beta_0))^\top S_{i-1}^{-1}(\bar{\alpha})(b_{i-1}(\beta)-b_{i-1}(\beta_0))\boldsymbol{1}_{\{|\Delta X_i^n|\le D_3h_n^{\rho_3}\}}\\
&\quad-\frac{1}{2nh_n^2}\sum_{i=1}^n(\bar{X}_{i,n}(\beta_0))^\top S_{i-1}^{-1}(\bar{\alpha})\bar{X}_{i,n}(\beta_0)\boldsymbol{1}_{\{|\Delta X_i^n|\le D_3h_n^{\rho_3}\}}.
\end{align*}
Then, we see from Propositions \ref{prop4:jump}-(ii) and \ref{prop6:jump} that 
\begin{align}\label{b-cons:uni}
&\sup_{\theta\in\Theta}\left|\frac{1}{nh_n}\bar{l}_n^{(2)}(\beta|\alpha)-\frac{1}{nh_n}\bar{l}_n^{(2)}(\beta_0|\alpha)-\bar{U}_{\beta_0}^{(2)}(\alpha,\beta)\right|\notag\\
&\leq\sup_{\theta\in\Theta}\left|\frac{1}{nh_n}\sum_{i=1}^{n}(b_{i-1}(\beta)-b_{i-1}(\beta_0))^\top S_{i-1}^{-1}(\bar{\alpha})\bar{X}_{i,n}(\beta_0)\boldsymbol{1}_{\{|\Delta X_i^n|\le D_3h_n^{\rho_3}\}}\right|\notag\\
&\quad+\sup_{\theta\in\Theta}\left|-\frac{1}{2n}\sum_{i=1}^{n}(b_{i-1}(\beta)-b_{i-1}(\beta_0))^\top S_{i-1}^{-1}(\bar{\alpha})(b_{i-1}(\beta)-b_{i-1}(\beta_0))\boldsymbol{1}_{\{|\Delta X_i^n|\le D_3h_n^{\rho_3}\}}-\bar{U}_{\beta_0}^{(2)}(\alpha,\beta)\right|\notag\\
&\overset{P}{\to}0.
\end{align}
By the assumption {\bf [A8]}, it follows that
\begin{align}\label{b-cons:iden1}
\bar{U}_{\beta_0}^{(2)}(\alpha_0,\beta)= -\frac{1}{2}\int{(b(x,\beta)-b(x,\beta_0))^\top S^{-1}(x,\alpha_0)(b(x,\beta)-b(x,\beta_0))}\pi(dx)\leq 0
\end{align}
with equality if and only if $b(x,\beta)=b(x,\beta_0)\ x$-a.s.. On the other hand, 
 it holds that $\Psi_{\beta_0} (y,x)>0$ on the set $A$, and that for all $x>0$,
\[1+\log x-x\leq 0\]
with equality if and only if $x=1$. Therefore, we have
\begin{align}\label{b-cons:iden2}
&\tilde{U}_{\beta_0}^{(2)}(\beta)-\tilde{U}_{\beta_0}^{(2)}(\beta_0)\notag\\
&=\iint_A{\left\{(\log\Psi_\beta (y,x))\Psi_{\beta_0} (y,x)-\Psi_\beta (y,x)\right\}}dy\pi(dx)-\iint_A{\left\{(\log\Psi_{\beta_0} (y,x))\Psi_{\beta_0} (y,x)-\Psi_{\beta_0} (y,x)\right\}}dy\pi(dx)\notag\\
&=\iint_A{\Psi_{\beta_0} (y,x)\left\{1+\log\frac{\Psi_\beta (y,x)}{\Psi_{\beta_0} (y,x)}-\frac{\Psi_\beta (y,x)}{\Psi_{\beta_0} (y,x)}\right\}}dy\pi(dx)\notag\\
&\leq 0
\end{align}
with equality if and only if $\Psi_\beta (y,x)=\Psi_{\beta_0} (y,x)\ (x,y)$-a.s..
Hence, it follows from (\ref{b-cons:iden1}), (\ref{b-cons:iden2}) and {\bf [A9]} that for all $\beta\in\Theta_\beta$, 
\[V_{\beta_0}(\alpha_0,\beta)\leq V_{\beta_0}(\alpha_0,\beta_0)(=0)\]
with equality if and only if $\beta=\beta_0$ That is, for all $\varepsilon>0$,
\begin{align}\label{b-cons:idenV}
\sup_{\beta:|\beta-\beta_0|\geq \varepsilon}V_{\beta_0}(\alpha_0,\beta)<V_{\beta_0}(\alpha_0,\beta_0)(=0).
\end{align}
Moreover,  it follows from the definition of $\check{\beta}_n$ that for all $\varepsilon>0$,
\begin{align}\label{b-cons:compose}
P\left(\frac{1}{nh_n}l_n^{(2)}(\check{\beta}_n|\check{\alpha}_n)+\varepsilon<\frac{1}{nh_n}l_n^{(2)}(\beta_0|\check{\alpha}_n)\right)=0,
\end{align}
and  we see from the assumption {\bf [A4]} that
\begin{align*}
&|\bar{U}_{\beta_0}^{(2)}(\check{\alpha}_n,\check{\beta}_n)-\bar{U}_{\beta_0}^{(2)}(\alpha_0,\check{\beta}_n)|\\
&=\frac{1}{2}\left|\int{(b(x,\check{\beta}_n)-b(x,\beta_0))^\top \left(S^{-1}(x,\check{\alpha}_n)-S^{-1}(x,\alpha_0)\right)(b(x,\check{\beta}_n)-b(x,\beta_0))}\pi(dx)\right|\\
&\leq \frac{1}{2}\sum_{k_1,k_2=1}^d\int\left|b^{(k_1)}(x,\check{\beta}_n)-b^{(k_1)}(x,\beta_0)\right|\cdot\left|b^{(k_2)}(x,\check{\beta}_n)-b^{(k_2)}(x,\beta_0)\right|\cdot\left|S^{{-1}^{(k_1,k_2)}}(x,\check{\alpha}_n)-S^{{-1}^{(k_1,k_2)}}(x,\alpha_0)\right|\pi(dx)\\
&\leq C\sum_{k_1,k_2=1}^d\int(1+|x|)^c\left|S^{{-1}^{(k_1,k_2)}}(x,\check{\alpha}_n)-S^{{-1}^{(k_1,k_2)}}(x,\alpha_0)\right|\pi(dx).
\end{align*}
By (\ref{a-cons:result}) and the continuity of the right-hand side, we have  
\begin{align}\label{b-cons:tech}
|\bar{U}_{\beta_0}^{(2)}(\check{\alpha}_n,\check{\beta}_n)-\bar{U}_{\beta_0}^{(2)}(\alpha_0,\check{\beta}_n)|\overset{P}{\to}0.
\end{align}
By (\ref{b-cons:idenV}), for all $\varepsilon>0$, there  exists $\delta>0$ such that
\[\sup_{\beta:|\beta-\beta_0|\geq \varepsilon}V_{\beta_0}(\alpha_0,\beta)< -\delta.\]
Hence, it follows from (\ref{b-cons:uni_jump}), (\ref{b-cons:uni}), (\ref{b-cons:compose}) and (\ref{b-cons:tech}) that 
\begin{align*}
0\leq P\left(|\check{\beta}_n-\beta_0|\geq\varepsilon\right)&\leq P\left(V_{\beta_0}(\alpha_0,\check{\beta}_n)< -\delta\right)\\
&= P\left(\bar{U}_{\beta_0}^{(2)}(\alpha_0,\check{\beta}_n)+\tilde{U}_{\beta_0}^{(2)}(\check{\beta}_n)-\tilde{U}_{\beta_0}^{(2)}(\beta_0)<-\delta\right)\\
&\leq P\left(\frac{1}{nh_n}\left[l_n^{(2)}(\check{\beta}_n|\check{\alpha}_n)-l_n^{(2)}(\beta_0|\check{\alpha}_n)\right]-\bar{U}_{\beta_0}^{(2)}(\alpha_0,\check{\beta}_n)-\tilde{U}_{\beta_0}^{(2)}(\check{\beta}_n)+\tilde{U}_{\beta_0}^{(2)}(\beta_0)>\frac{\delta}{2}\right)\\
&\quad+P\left(-\frac{1}{nh_n}\left[l_n^{(2)}(\check{\beta}_n|\check{\alpha}_n)-l_n^{(2)}(\beta_0|\check{\alpha}_n)\right]>\frac{\delta}{2}\right)\\
&\leq P\left(\frac{1}{nh_n}\left[\bar{l}_n^{(2)}(\check{\beta}_n|\check{\alpha}_n)-\bar{l}_n^{(2)}(\beta_0|\check{\alpha}_n)\right]-\bar{U}_{\beta_0}^{(2)}(\alpha_0,\check{\beta}_n)>\frac{\delta}{6}\right)\\
&\quad+P\left(\frac{1}{nh_n}\tilde{l}_n^{(2)}(\check{\beta}_n)-\tilde{U}_{\beta_0}^{(2)}(\check{\beta}_n)>\frac{\delta}{6}\right)\\
&\quad+P\left(-\frac{1}{nh_n}\tilde{l}_n^{(2)}(\beta_0)+\tilde{U}_{\beta_0}^{(2)}(\beta_0)>\frac{\delta}{6}\right)\\
&\quad+P\left(\frac{1}{nh_n}l_n^{(2)}(\check{\beta}_n|\check{\alpha}_n)+\frac{\delta}{2}<\frac{1}{nh_n}l_n^{(2)}(\beta_0|\check{\alpha}_n)\right)\\
&\leq P\left(\frac{1}{nh_n}\left[\bar{l}_n^{(2)}(\check{\beta}_n|\check{\alpha}_n)-\bar{l}_n^{(2)}(\beta_0|\check{\alpha}_n)\right]-\bar{U}_{\beta_0}^{(2)}(\check{\alpha}_n,\check{\beta}_n)>\frac{\delta}{12}\right)\\
&\quad+P\left(\bar{U}_{\beta_0}^{(2)}(\check{\alpha}_n,\check{\beta}_n)-\bar{U}_{\beta_0}^{(2)}(\alpha_0,\check{\beta}_n)>\frac{\delta}{12}\right)\\
&\quad+2P\left(\sup_{\beta\in\Theta_\beta}\left|\frac{1}{nh_n}\tilde{l}_n^{(2)}(\beta)-\tilde{U}_{\beta_0}^{(2)}(\beta)\right|>\frac{\delta}{6}\right)\\
&\quad+P\left(\frac{1}{nh_n}l_n^{(2)}(\check{\beta}_n|\check{\alpha}_n)+\frac{\delta}{2}<\frac{1}{nh_n}l_n^{(2)}(\beta_0|\check{\alpha}_n)\right)\\
&\leq P\left(\sup_{\theta\in\Theta}\left|\frac{1}{nh_n}\left[\bar{l}_n^{(2)}(\beta|\alpha)-\bar{l}_n^{(2)}(\beta_0|\alpha)\right]-\bar{U}_{\beta_0}^{(2)}(\alpha,\beta)\right|>\frac{\delta}{12}\right)\\
&\quad+P\left(\left|\bar{U}_{\beta_0}^{(2)}(\check{\alpha}_n,\check{\beta}_n)-\bar{U}_{\beta_0}^{(2)}(\alpha_0,\check{\beta}_n)\right|>\frac{\delta}{12}\right)\\
&\quad+2P\left(\sup_{\beta\in\Theta_\beta}\left|\frac{1}{nh_n}\tilde{l}_n^{(2)}(\beta)-\tilde{U}_{\beta_0}^{(2)}(\beta)\right|>\frac{\delta}{6}\right)\\
&\quad+P\left(\frac{1}{nh_n}l_n^{(2)}(\check{\beta}_n|\check{\alpha}_n)+\frac{\delta}{2}<\frac{1}{nh_n}l_n^{(2)}(\beta_0|\check{\alpha}_n)\right)\\
&\to 0.
\end{align*}
This means that
\begin{align}\label{b-cons:result}
    \check{\beta}_n\overset{P}{\to}\beta_0.
\end{align}
\end{proof}


\subsubsection{Proof of Theorem \ref{thm2-2:jump}}
\begin{proof}[\bf Proof.]
Let us define some symbols. For $1\le m_1\le p$, 
\begin{align*}
&\partial_{{\alpha}_{m_1}}l_n^{(1)}(\alpha)=\sum_{i=1}^{n}{\xi_i^{m_1}(\alpha)},\\
&\xi_i^{m_1}(\alpha):=-\frac{1}{2}\left\{h_n^{-1}(\Delta X_i^n)^\top\partial_{{\alpha_{m_1}}}S_{i-1}^{-1}(\alpha)\Delta X_i^n+\partial_{{\alpha}_{m_1}}\log\det S_{i-1}(\alpha)\right\}\boldsymbol{1}_{\{|\Delta X_i^n|\le D_1h_n^{\rho_1}\}},
\end{align*}
and for $1\le m_2\le q$,
\begin{align*}
&\partial_{{\beta}_{m_2}}l_n^{(2)}(\beta|\bar{\alpha})=\partial_{{\beta}_{m_2}}\bar{l}_n^{(2)}(\beta|\bar{\alpha})+\partial_{{\beta}_{m_2}}\tilde{l}_n^{(2)}(\beta)=\sum_{i=1}^{n}{(\eta_{i,1}^{m_2}(\beta|\bar{\alpha})+\eta_{i,2}^{m_2}(\beta))},\\
&\eta_{i,1}^{m_1}(\beta|\bar{\alpha}):=(\partial_{{\beta_{m_2}}}b_{i-1}(\beta))^\top S_{i-1}^{-1}(\bar{\alpha})\bar{X}_{i,n}(\beta)\boldsymbol{1}_{\{|\Delta X_i^n|\le D_3h_n^{\rho_3}\}},\\
&\eta_{i,2}^{m_1}(\beta):=\left\{\partial_{{\beta_{m_2}}}\log\Psi_{\beta}(\Delta X_i^n,X_{t_{i-1}^n})\right\}\varphi_n(X_{t_{i-1}^n},\Delta X_i^n)\boldsymbol{1}_{\{|\Delta X_i^n|> D_2h_n^{\rho_2}\}}\\
&\quad-h_n\int_{B}{\partial_{{\beta_{m_2}}}\Psi_{\beta}(y,X_{t_{i-1}^n})\varphi_n(X_{t_{i-1}^n},y)}dy.
\end{align*}
Moreover, we can calculate that for $1\le m_1,m_1'\le p$,
\begin{align*}
\partial_{\alpha_{m_1}\alpha_{m_1'}}^2 l_n^{(1)}(\alpha)=-\frac{1}{2}\sum_{i=1}^{n}\left\{h_n^{-1}(\Delta X_i^n)^\top\partial_{{\alpha_{m_1}}\alpha_{m_1'}}^2S_{i-1}^{-1}(\alpha)\Delta X_i^n+\partial_{{\alpha_{m_1}}\alpha_{m_1'}}^2\log\det S_{i-1}(\alpha)\right\}\boldsymbol{1}_{\{|\Delta X_i^n|\le D_1h_n^{\rho_1}\}},
\end{align*}
that for $1\le m_2,m_2'\le q$,
\begin{align*}
&\partial_{\beta_{m_2}\beta_{m_2'}}^2 l_n^{(2)}(\beta|\bar{\alpha})\\
&=\sum_{i=1}^{n}\left\{(\partial_{\beta_{m_2}\beta_{m_2'}}^2 b_{i-1}(\beta))^\top S_{i-1}^{-1}(\bar{\alpha})\bar{X}_{i,n}(\beta)-h_n(\partial_{\beta_{m_2}} b_{i-1}(\beta))^\top S_{i-1}^{-1}(\bar{\alpha})\partial_{\beta_{m_2'}} b_{i-1}(\beta)\right\}\boldsymbol{1}_{\{|\Delta X_i^n|\le D_3h_n^{\rho_3}\}}\\
&\quad+\sum_{i=1}^{n}{}\left\{\partial_{{\beta_{m_2}}{\beta_{m_2'}}}^2\log\Psi_\beta(\Delta X_i^n,X_{t_{i-1}^n})\right\}\varphi_n(X_{t_{i-1}^n},\Delta X_i^n)\boldsymbol{1}_{\{|\Delta X_i^n|> D_2h_n^{\rho_2}\}}\\
&\quad-h_n\sum_{i=1}^{n}{}\int_{B}{\partial_{{\beta_{m_2}}{\beta_{m_2'}}}^2\Psi_\beta(y,X_{t_{i-1}^n})\varphi_n(X_{t_{i-1}^n},y)}dy
\end{align*}
and that for $1\le m_1\le p$, $1\le m_2\le q$,
\begin{align*}
\partial_{\alpha_{m_1}\beta_{m_2}}^2 l_n^{(2)}(\beta|\bar{\alpha})&=\sum_{i=1}^{n}{(\partial_{{\beta_{m_2}}}b_{i-1}(\beta))^\top\partial_{{\alpha_{m_1}}}S_{i-1}^{-1}(\bar{\alpha})\bar{X}_{i,n}(\beta)\boldsymbol{1}_{\{|\Delta X_i^n|\le D_3h_n^{\rho_3}\}}}.
\end{align*}
 Let $\varepsilon_0$ be a positive constant such that $\{\alpha\in\Theta_\alpha\ ;\ |\alpha-\alpha_0|<\varepsilon_0\}\subset\mathrm{Int}(\Theta_\alpha)$ and $\{\beta\in\Theta_\beta\ ;\ |\beta-\beta_0|<\varepsilon_0\}\subset\mathrm{Int}(\Theta_\beta)$. Then it follows from consistency of $\check{\alpha}_n$ and $\check{\beta}_n$ that there exists a real valued sequence $\varepsilon_n<\varepsilon_0$ such that $P(A_n\cap B_n)\to 1$, where $A_{n}:=\{\omega\in\Omega\ |\ |\check{\alpha}_n(\omega)-\alpha_0|<\varepsilon_n\}$ and $B_n:=\{\omega\in\Omega\ |\ |\check{\beta}_n(\omega)-\beta_0|<\varepsilon_n\}$. In particular, we have $\check{\alpha}_n\in\mathrm{Int}(\Theta_\alpha),\check{\beta}_n\in\mathrm{Int}(\Theta_\beta)$ on the set $A_n\cap B_n$. Therefore, since the functions $l_n^{(1)}(\alpha)$ and $l_n^{(2)}(\beta|\alpha)$ are maximized on the interior of $\Theta$, 
 one has $\partial_\alpha l_n^{(1)}(\check{\alpha}_n)=0,\partial_\beta l_n^{(2)}(\check{\beta}_n|\check{\alpha}_n)=0$.
Hence, By using Taylor's theorem, we have the following equations on the set $A_n\cap B_n$:
\begin{align*}
-\frac{1}{\sqrt{n}}\partial_{{\alpha}}l_n^{(1)}(\alpha_0)&=\left(\int_{0}^{1}{\frac{1}{n}\partial_{{\alpha}}^2 l_n^{(1)}(\alpha_0+u(\check{\alpha}_{n}-\alpha_0))du}\right)\sqrt{n}(\check{\alpha}_{n}-\alpha_0),\\
-\frac{1}{\sqrt{nh_n}}\partial_{{\beta}}l_n^{(2)}(\beta_0|\check{\alpha}_{n})&=\left(\int_{0}^{1}{\frac{1}{nh_n}\partial_{{\beta}}^2 l_n^{(2)}(\beta_0+u(\check{\beta}_{n}-\beta_0)|\check{\alpha}_{n})du}\right)\sqrt{nh_n}(\check{\beta}_{n}-\beta_0),\\
\frac{1}{\sqrt{nh_n}}\partial_{{\beta}}l_n^{(2)}(\beta_0|\check{\alpha}_{n})&=\frac{1}{\sqrt{nh_n}}\partial_{{\beta}}l_n^{(2)}(\beta_0|\alpha_0)+
\left(\int_{0}^{1}{\frac{1}{n\sqrt{h_n}}\partial_{{\alpha\beta}}^2 l_n^{(2)}(\beta_0|\alpha_0+u(\check{\alpha}_{n}-\alpha_0))du}\right)\sqrt{n}(\check{\alpha}_{n}-\alpha_0).
\end{align*}
Therefore, we obtain $L_n=C_nS_n\ (\omega\in A_n\cap B_n)$, where
\begin{align*}
S_n&:=\begin{pmatrix}
\sqrt{n}(\check{\alpha}_{n}-\alpha_0)\\
\sqrt{nh_n}(\check{\beta}_{n}-\beta_0)
\end{pmatrix},\quad
L_n:=\begin{pmatrix}
-\frac{1}{\sqrt{n}}\partial_{{\alpha}}l_n^{(1)}(\alpha_0)\\
-\frac{1}{\sqrt{nh_n}}\partial_{{\beta}} l_n^{(2)}(\beta_0|\alpha_0)
\end{pmatrix},\\
C_n&:=\begin{pmatrix}
\int_{0}^{1}{\frac{1}{n}\partial_{{\alpha}}^2 l_n^{(1)}(\alpha_0+u(\check{\alpha}_{n}-\alpha_0))du}&0\\
\int_{0}^{1}{\frac{1}{n\sqrt{h_n}}\partial_{{\alpha\beta}}^2 l_n^{(2)}(\beta_0|\alpha_0+u(\check{\alpha}_{n}-\alpha_0))du}&
\int_{0}^{1}{\frac{1}{nh_n}\partial_{{\beta}}^2 l_n^{(2)}(\beta_0+u(\check{\beta}_{n}-\beta_0)|\check{\alpha}_{n})du}
\end{pmatrix}.
\end{align*}
Thus, it is sufficient to show $S_n\overset{d}{\to}N_{p+q}(0,I(\theta_0)^{-1})$ that
\begin{align}
&\sup_{u\in[0,1]}\left|\frac{1}{n}\partial_{{\alpha}}^2 l_n^{(1)}(\alpha_0+u(\check{\alpha}_{n}-\alpha_0))+I_a(\alpha_0)\right|\overset{P}{\to}0\label{thm2:goal1-jump}%
,\\
&\sup_{u\in[0,1]}\left|\frac{1}{nh_n}\partial_{{\beta}}^2 l_n^{(2)}(\beta_0+u(\check{\beta}_{n}-\beta_0)|\check{\alpha}_{n})+I_{b,c}(\theta_0)\right|\overset{P}{\to}0\label{thm2:goal2-jump}%
,\\
&\sup_{u\in[0,1]}\left|\frac{1}{n\sqrt{h_n}}\partial_{{\alpha\beta}}^2 l_n^{(2)}(\beta_0|\alpha_0+u(\check{\alpha}_{n}-\alpha_0))\right|\overset{P}{\to}0\label{thm2:goal3-jump}%
,\\
&L_n\overset{d}{\to}N_{p+q}(0, I(\theta_0))\label{thm2:goal4-jump}%
.
\end{align}
{\bf Proof of \eqref{thm2:goal1-jump}.}
For $1\le m_1,m_1'\le p$, it follows from 
 Propositions \ref{prop4:jump}-(ii), \ref{prop5:jump} and \ref{prop6:jump} 
that 
\begin{align}
&\sup_{\alpha\in\Theta_\alpha}\left|\frac{1}{n}\partial_{\alpha_{m_1}\alpha_{m_1'}}^2 l_n^{(1)}(\alpha)+I_a^{(m_1,m_1')}(\alpha)\right|\notag\\
&=\sup_{\alpha\in\Theta_\alpha}\left|-\frac{1}{2n}\sum_{i=1}^{n}\left\{h_n^{-1}(\Delta X_i^n)^\top\partial_{{\alpha_{m_1}}\alpha_{m_1'}}^2S_{i-1}^{-1}(\alpha)\Delta X_i^n+\partial_{{\alpha_{m_1}}\alpha_{m_1'}}^2\log\det S_{i-1}(\alpha)\right\}\boldsymbol{1}_{\{|\Delta X_i^n|\le D_1h_n^{\rho_1}\}}+I_a^{(m_1,m_1')}(\alpha)\right|\notag\\
&\leq\frac{1}{2}\sup_{\alpha\in\Theta_\alpha}\left|\frac{1}{nh_n}\sum_{i=1}^{n}\bar{X}_{i,n}(\beta_0)^\top \partial_{{\alpha_{m_1}}\alpha_{m_1'}}^2S_{i-1}^{-1}(\alpha) \bar{X}_{i,n}(\beta_0)\boldsymbol{1}_{\{|\Delta X_i^n|\le D_1h_n^{\rho_1}\}}-\int{\mathrm{tr}\left\{\left( \partial_{{\alpha_{m_1}}\alpha_{m_1'}}^2S^{-1}(x,\alpha)\right)S(x,\alpha_0)\right\}\pi(dx)}\right|\notag\\
&\quad +h_n\cdot\sup_{\alpha\in\Theta_\alpha}\left|\frac{1}{nh_n}\sum_{i=1}^{n}b_{i-1}(\beta_0)^\top\partial_{{\alpha_{m_1}}\alpha_{m_1'}}^2S_{i-1}^{-1}(\alpha) \bar{X}_{i,n}(\beta_0)\boldsymbol{1}_{\{|\Delta X_i^n|\le D_1h_n^{\rho_1}\}}\right|\notag\\
&\quad +h_n\cdot\sup_{\alpha\in\Theta_\alpha}\left|\frac{1}{2n}\sum_{i=1}^{n}b_{i-1}(\beta_0)^\top\partial_{{\alpha_{m_1}}\alpha_{m_1'}}^2S_{i-1}^{-1}(\alpha)b_{i-1}(\beta_0)\boldsymbol{1}_{\{|\Delta X_i^n|\le D_1h_n^{\rho_1}\}}\right|\notag\\
&\quad+\frac{1}{2}\sup_{\alpha\in\Theta_\alpha}\left|\frac{1}{n}\sum_{i=1}^{n}\partial_{{\alpha_{m_1}}\alpha_{m_1'}}^2\log\det S_{i-1}(\alpha)\boldsymbol{1}_{\{|\Delta X_i^n|\le D_1h_n^{\rho_1}\}}-\int{\partial_{{\alpha_{m_1}}\alpha_{m_1'}}^2\log\det S(x,\alpha)\pi(dx)}\right|\notag\\
&\overset{P}{\to}0\notag.
\end{align}
Therefore,  one has
\begin{align}\label{a-2der}
\sup_{\alpha\in\Theta_\alpha}\left|\frac{1}{n}\partial_{\alpha}^2 l_n^{(1)}(\alpha)+I_a(\alpha)\right|\overset{P}{\to}0.
\end{align}
Note that for all $u\in[0,1]$, $\alpha_0+u(\check{\alpha}_{n}-\alpha_0)\in\{\alpha\in\Theta_\alpha\ |\ |\alpha-\alpha_0|<\varepsilon_n\}$ on the set $A_n$, then we have from \eqref{a-2der}, 
continuity of $I_a(\alpha)$ and consistency of $\check{\alpha}_n$ that for all $\varepsilon>0$, 
\begin{align*}
&P\left(\sup_{u\in[0,1]}\left|\frac{1}{n}\partial_{{\alpha}}^2 l_n^{(1)}(\alpha_0+u(\check{\alpha}_{n}-\alpha_0))+I_a(\alpha_0)\right|>\varepsilon\right)\\
\quad&\le  P\left(\sup_{\alpha\in\Theta_\alpha}\left|\frac{1}{n}\partial_{{\alpha}}^2 l_n^{(1)}(\alpha)+I_a(\alpha)\right|>\frac{\varepsilon}{2}\right)+
P\left(\sup_{u\in[0,1]}\left|I_a(\alpha_0+u(\check{\alpha}_{n}-\alpha_0))-I_a(\alpha_0)\right|>\frac{\varepsilon}{2}\right)\\
\quad&\le  P\left(\sup_{\alpha\in\Theta_\alpha}\left|\frac{1}{n}\partial_{{\alpha}}^2 l_n^{(1)}(\alpha)+I_a(\alpha)\right|>\frac{\varepsilon}{2}\right)+
P\left(\sup_{\alpha:|\alpha-\alpha_0|<\varepsilon_n}\left|I_a(\alpha)-I_a(\alpha_0)\right|>\frac{\varepsilon}{2}\right)+P\left(A_n^c\right)\\
&\to 0\quad(n\to\infty).
\end{align*}
This implies \eqref{thm2:goal1-jump}.

{\bf Proof of \eqref{thm2:goal2-jump}.}
Since $\bar{X}_{i,n}(\beta)=h_n(b_{i-1}(\beta_0)-b_{i-1}(\beta))+\bar{X}_{i,n}(\beta_0)$, 
 we see from Propositions \ref{prop4:jump}-(i), (ii), \ref{prop6:jump} and \ref{prop7:jump} that for $1\le m_2,m_2'\le q$,  
\begin{align*}
&\sup_{(\bar{\alpha},\beta)\in\Theta}\left|\frac{1}{nh_n}\partial_{\beta_{m_2}\beta_{m_2'}}^2 l_n^{(2)}(\beta|\bar{\alpha})+I_{b,c}^{(m_2,m_2')}(\bar{\alpha},\beta)\right|\\
&=\sup_{(\bar{\alpha},\beta)\in\Theta}\left|\frac{1}{nh_n}\sum_{i=1}^{n}(\partial_{\beta_{m_2}\beta_{m_2'}}^2 b_{i-1}(\beta))^\top S_{i-1}^{-1}(\bar{\alpha})\bar{X}_{i,n}(\beta)\boldsymbol{1}_{\{|\Delta X_i^n|\le D_3h_n^{\rho_3}\}}\right.\\
&\qquad\quad-\frac{1}{n}\sum_{i=1}^n(\partial_{\beta_{m_2}} b_{i-1}(\beta))^\top S_{i-1}^{-1}(\bar{\alpha})\partial_{\beta_{m_2'}} b_{i-1}(\beta)\boldsymbol{1}_{\{|\Delta X_i^n|\le D_3h_n^{\rho_3}\}}\\
&\qquad\quad+\frac{1}{nh_n}\sum_{i=1}^{n}{}\left\{\partial_{{\beta_{m_2}}{\beta_{m_2'}}}^2\log\Psi_\beta(\Delta X_i^n,X_{t_{i-1}^n})\right\}\varphi_n(X_{t_{i-1}^n},\Delta X_i^n)\boldsymbol{1}_{\{|\Delta X_i^n|> D_2h_n^{\rho_2}\}}\\
&\qquad\quad\left.-\frac{1}{n}\sum_{i=1}^{n}{}\int_{B}{\partial_{{\beta_{m_2}}{\beta_{m_2'}}}^2\Psi_\beta(y,X_{t_{i-1}^n})\varphi_n(X_{t_{i-1}^n},y)}dy+I_{b,c}^{(m_2,m_2')}(\bar{\alpha},\beta)\right|\\
&\leq \sup_{(\bar{\alpha},\beta)\in\Theta}\left|\frac{1}{n}\sum_{i=1}^{n}(\partial_{\beta_{m_2}\beta_{m_2'}}^2 b_{i-1}(\beta))^\top S_{i-1}^{-1}(\bar{\alpha})(b_{i-1}(\beta_0)-b_{i-1}(\beta))\boldsymbol{1}_{\{|\Delta X_i^n|\le D_3h_n^{\rho_3}\}}\right.\\
&\hspace{20mm}\left.+\int{(\partial^2_{\beta_{m_2}\beta_{m_2'}} b(x,\beta))^\top S^{-1}(x,\bar{\alpha})(b(x,\beta)-b(x,\beta_0))}\pi(dx)\right|\\
&\quad+\sup_{(\bar{\alpha},\beta)\in\Theta}\left|\frac{1}{nh_n}\sum_{i=1}^{n}(\partial_{\beta_{m_2}\beta_{m_2'}}^2 b_{i-1}(\beta))^\top S_{i-1}^{-1}(\bar{\alpha})\bar{X}_{i,n}(\beta_0)\boldsymbol{1}_{\{|\Delta X_i^n|\le D_3h_n^{\rho_3}\}}\right|\\
&\quad+\sup_{(\bar{\alpha},\beta)\in\Theta}\left|-\frac{1}{n}\sum_{i=1}^n(\partial_{\beta_{m_2}} b_{i-1}(\beta))^\top S_{i-1}^{-1}(\bar{\alpha})\partial_{\beta_{m_2'}} b_{i-1}(\beta)\boldsymbol{1}_{\{|\Delta X_i^n|\le D_3h_n^{\rho_3}\}}\right.\\
&\hspace{20mm}+\left.\int{(\partial_{\beta_{m_2}} b(x,\beta))^\top S^{-1}(x,\bar{\alpha})\partial_{\beta_{m_2'}} b(x,\beta)}\pi(dx)\right|\\
&\quad+\sup_{(\bar{\alpha},\beta)\in\Theta}\left|\frac{1}{nh_n}\sum_{i=1}^{n}{}\left\{\partial_{{\beta_{m_2}}{\beta_{m_2'}}}^2\log\Psi_\beta(\Delta X_i^n,X_{t_{i-1}^n})\right\}\varphi_n(X_{t_{i-1}^n},\Delta X_i^n)\boldsymbol{1}_{\{|\Delta X_i^n|> D_2h_n^{\rho_2}\}}\right.\\
&\hspace{20mm}\left.-\iint_A{\left\{\partial_{{\beta_{m_2}}{\beta_{m_2'}}}^2\log\Psi_\beta(y,x)\right\}\Psi_{\beta_0}(y,x)}dy\pi(dx)\right|\\
&\quad+\sup_{(\bar{\alpha},\beta)\in\Theta}\left|-\frac{1}{n}\sum_{i=1}^{n}{}\int_{B}{\partial_{{\beta_{m_2}}{\beta_{m_2'}}}^2\Psi_\beta(y,X_{t_{i-1}^n})\varphi_n(X_{t_{i-1}^n},y)}dy+\iint_B{\partial_{{\beta_{m_2}}{\beta_{m_2'}}}^2\Psi_\beta(y,x)}dy\pi(dx)\right|\\
&\overset{P}{\to}0.
\end{align*}
Therefore, 
\begin{align}\label{b-2der}
\sup_{(\bar{\alpha},\beta)\in\Theta}\left|\frac{1}{nh_n}\partial_{\beta}^2 l_n^{(2)}(\beta|\bar{\alpha})+I_{b,c}(\bar{\alpha},\beta)\right|\overset{P}{\to}0.
\end{align}
Since, on the set $A_n\cap B_n$, for all $u\in[0,1]$, $\beta_0+u(\check{\beta}_{n}-\beta_0)\in\{\beta\in\Theta_\beta\ |\ |\beta-\beta_0|<\varepsilon_n\}$ and $|\check{\alpha}_n-\alpha_0|<\varepsilon_n$, we have $(\check{\alpha}_n,\beta_0+u(\check{\beta}_{n}-\beta_0))\in\{\theta\in\Theta\ |\ |\theta-\theta_0|<2\varepsilon_n\}$.
Hence, by using \eqref{b-2der}, 
continuity of $I_{b,c}(\theta)$, consistency of $\check{\theta}_n$, for all $\varepsilon>0$, 
\begin{align*}
&P\left(\sup_{u\in[0,1]}\left|\frac{1}{nh_n}\partial_{{\beta}}^2 l_n^{(2)}(\beta_0+u(\check{\beta}_{n}-\beta_0)|\check{\alpha}_{n})+I_{b,c}(\theta_0)\right|>\varepsilon\right)\\
\quad&\le  P\left(\sup_{(\bar{\alpha},\beta)\in\Theta}\left|\frac{1}{nh_n}\partial_{{\beta}}^2 l_n^{(2)}(\beta|\bar{\alpha})+I_{b,c}(\bar{\alpha},\beta)\right|>\frac{\varepsilon}{2}\right)+
P\left(\sup_{u\in[0,1]}\left|I_{b,c}(\check{\alpha}_n,\beta_0+u(\check{\beta}_{n}-\beta_0))-I_{b,c}(\theta_0)\right|>\frac{\varepsilon}{2}\right)\\
&\leq P\left(\sup_{(\bar{\alpha},\beta)\in\Theta}\left|\frac{1}{nh_n}\partial_{{\beta}}^2 l_n^{(2)}(\beta|\bar{\alpha})+I_{b,c}(\bar{\alpha},\beta)\right|>\frac{\varepsilon}{2}\right)+
P\left(\sup_{\theta:|\theta-\theta_0|<2\varepsilon_n}\left|I_{b,c}(\theta)-I_{b,c}(\theta_0)\right|>\frac{\varepsilon}{2}\right)+P(A_n^c\cup B_n^c)\\
&\to 0\quad(n\to\infty).
\end{align*}
This implies \eqref{thm2:goal2-jump}.

{\bf Proof of \eqref{thm2:goal3-jump}.}
As $\bar{X}_{i,n}(\beta)=h_n(b_{i-1}(\beta_0)-b_{i-1}(\beta))+\bar{X}_{i,n}(\beta_0)$, 
it follows from  Propositions \ref{prop6:jump} and \ref{prop4:jump}-(ii) that
for $1\le m_1\le p$, $1\le m_2\le q$,
\begin{align*}
&\sup_{(\bar{\alpha},\beta)\in\Theta}\left|\frac{1}{n\sqrt{h_n}}\partial_{\alpha_{m_1}\beta_{m_2}}^2 l_n^{(2)}(\beta|\bar{\alpha})\right|\\
&\leq\sqrt{h_n}\cdot\sup_{(\bar{\alpha},\beta)\in\Theta}\left| \frac{1}{nh_n}\sum_{i=1}^{n}{(\partial_{{\beta_{m_2}}}b_{i-1}(\beta))^\top\partial_{{\alpha_{m_1}}}S_{i-1}^{-1}(\bar{\alpha})\bar{X}_{i,n}(\beta_0)\boldsymbol{1}_{\{|\Delta X_i^n|\le D_3h_n^{\rho_3}\}}}\right|\\
&\quad+\sqrt{h_n}\cdot\sup_{(\bar{\alpha},\beta)\in\Theta}\left| \frac{1}{n}\sum_{i=1}^{n}{(\partial_{{\beta_{m_2}}}b_{i-1}(\beta))^\top\partial_{{\alpha_{m_1}}}S_{i-1}^{-1}(\bar{\alpha})(b_{i-1}(\beta_0)-b_{i-1}(\beta))\boldsymbol{1}_{\{|\Delta X_i^n|\le D_3h_n^{\rho_3}\}}}\right|\\
&\overset{P}{\to}0.
\end{align*}
Hence, 
\begin{align}\label{ab-2der}
\sup_{(\bar{\alpha},\beta)\in\Theta}\left|\frac{1}{n\sqrt{h_n}}\partial_{\alpha\beta}^2 l_n^{(2)}(\beta|\bar{\alpha})\right|&\overset{P}{\to}0.
\end{align}
Thus  one has
\begin{align*}
&P\left(\sup_{u\in[0,1]}\left|\frac{1}{n\sqrt{h_n}}\partial_{{\alpha\beta}}^2 l_n^{(2)}(\beta_0|\alpha_0+u(\check{\alpha}_{n}-\alpha_0))\right|>\varepsilon\right)\le
P\left(\sup_{(\bar{\alpha},\beta)\in\Theta}\left|\frac{1}{n\sqrt{h_n}}\partial_{{\alpha\beta}}^2 l_n^{(2)}(\beta|\bar{\alpha})\right|>\varepsilon\right)\overset{P}{\to}0.
\end{align*}
This implies \eqref{thm2:goal3-jump}.

{\bf Proof of \eqref{thm2:goal4-jump}.}
From \citet{Hall-Heyde}, the following types of convergence are sufficient for \eqref{thm2:goal4-jump}: for $1\le m_1,m_1'\le p_1$, $1\le m_2,m_2'\le p_2$ and some $\nu_1,\nu_2>0$,
\begin{align}
&\sum_{i=1}^{n}{\mathbb{E}\left[\frac{1}{\sqrt{n}}\xi_i^{m_1}(\alpha_0)\ |\mathcal{F}_{i-1}^n\right]}\overset{P}{\to}0\label{MCLT1-jump},\\
&\sum_{i=1}^{n}{\mathbb{E}\left[\frac{1}{\sqrt{nh_n}}\left(\eta_{i,1}^{m_2}(\beta_0|\alpha_0)+\eta_{i,2}^{m_2}(\beta_0)\right)\ |\mathcal{F}_{i-1}^n\right]}\overset{P}{\to}0\label{MCLT2-jump},\\
&\sum_{i=1}^{n}{\mathbb{E}\left[\frac{1}{n}\xi_i^{m_1}(\alpha_0)\xi_i^{m_1'}(\alpha_0)\ |\mathcal{F}_{i-1}^n\right]}\overset{P}{\to}I_a^{m_1,m_1'}(\alpha_0)\label{MCLT3-jump},\\
&\sum_{i=1}^{n}{\mathbb{E}\left[\frac{1}{\sqrt{n}}\xi_i^{m_1}(\alpha_0)\ |\mathcal{F}_{i-1}^n\right]\mathbb{E}\left[\frac{1}{\sqrt{n}}\xi_i^{m_1'}(\alpha_0)\ |\mathcal{F}_{i-1}^n\right]}\overset{P}{\to}0\label{MCLT4-jump},\\
&\sum_{i=1}^{n}{\mathbb{E}\left[\frac{1}{nh_n}\left(\eta_{i,1}^{m_2}(\beta_0|\alpha_0)+\eta_{i,2}^{m_2}(\beta_0)\right)\left(\eta_{i,1}^{m_2'}(\beta_0|\alpha_0)+\eta_{i,2}^{m_2'}(\beta_0)\right)\ |\mathcal{F}_{i-1}^n\right]}\overset{P}{\to}I_{b,c}^{m_2,m_2'}(\theta_0)\label{MCLT5-jump},\\
&\sum_{i=1}^{n}{\mathbb{E}\left[\frac{1}{\sqrt{nh_n}}\left(\eta_{i,1}^{m_2}(\beta_0|\alpha_0)+\eta_{i,2}^{m_2}(\beta_0)\right)\ |\mathcal{F}_{i-1}^n\right]\mathbb{E}\left[\frac{1}{\sqrt{nh_n}}\left(\eta_{i,1}^{m_2'}(\beta_0|\alpha_0)+\eta_{i,2}^{m_2'}(\beta_0)\right)\ |\mathcal{F}_{i-1}^n\right]}\overset{P}{\to}0\label{MCLT6-jump},\\
&\sum_{i=1}^{n}{\mathbb{E}\left[\frac{1}{n\sqrt{h_n}}\xi_i^{m_1}(\alpha_0)\left(\eta_{i,1}^{m_2}(\beta_0|\alpha_0)+\eta_{i,2}^{m_2}(\beta_0)\right)\ |\mathcal{F}_{i-1}^n\right]}\overset{P}{\to}0\label{MCLT7-jump},\\
&\sum_{i=1}^{n}{\mathbb{E}\left[\frac{1}{\sqrt{n}}\xi_i^{m_1}(\alpha_0)\ |\mathcal{F}_{i-1}^n\right]\mathbb{E}\left[\frac{1}{\sqrt{nh_n}}\left(\eta_{i,1}^{m_2}(\beta_0|\alpha_0)+\eta_{i,2}^{m_2}(\beta_0)\right)\ |\mathcal{F}_{i-1}^n\right]}\overset{P}{\to}0\label{MCLT8-jump},\\
&\sum_{i=1}^{n}{\mathbb{E}\left[\left|\frac{1}{\sqrt{n}}\xi_i^{m_1}(\alpha_0)\right|^{2+\nu_1}|\mathcal{F}_{i-1}^n\right]}\overset{P}{\to}0\label{MCLT9-jump},\\
&\sum_{i=1}^{n}{\mathbb{E}\left[\left|\frac{1}{\sqrt{nh_n}}\left(\eta_{i,1}^{m_2}(\beta_0|\alpha_0)+\eta_{i,2}^{m_2}(\beta_0)\right)\right|^{2+\nu_2}|\mathcal{F}_{i-1}^n\right]}\overset{P}{\to}0\label{MCLT10-jump}.
\end{align}

{\bf Proof of \eqref{MCLT1-jump}.}
Since for $k=1,2$ and $j=0,1,2$,
\begin{align}\label{eq:barX}
|\bar{X}_{i,n}(\beta_0)|^k\boldsymbol{1}_{C_{i,j}^n(D,\rho)}\le C\left(|\Delta X_i^n|^k+h_n^k|b_{i-1}(\beta_0)|^k\right)\boldsymbol{1}_{C_{i,j}^n(D,\rho)}=R(\theta,h_n^{k\rho},X_{t_{i-1}^n})\boldsymbol{1}_{C_{i,j}^n(D,\rho)},
\end{align} 
it follows from 
$\partial_{\alpha_{m_1}}\log\det S_{i-1}(\alpha_0)=-\tr\left(\partial_{\alpha_{m_1}}S_{i-1}^{-1}(\alpha_0)S_{i-1}(\alpha_0)\right)=-\sum_{k,l=1}^d\partial_{\alpha_{m_1}}S_{i-1}^{-1^{(k,l)}}(\alpha_0)S_{i-1}^{(l,k)}(\alpha_0)$, $\rho_1\geq\frac{1+\delta}{6}$ under either conditions [$\textbf{C}_2${\bf 1}] or [$\textbf{C}_2${\bf 2}], 
 Propositions \ref{prop2:jump} and \ref{prop3:jump}, Lemma \ref{lem1:jump} that  
\begin{align*}
&\left|\sum_{i=1}^{n}{\mathbb{E}\left[\frac{1}{\sqrt{n}}\xi_i^{m_1}(\alpha_0)\ |\mathcal{F}_{i-1}^n\right]}\right|\\
&\le\frac{1}{2\sqrt{n}}\sum_{i=1}^{n}{\left|\mathbb{E}\left[\left\{h_n^{-1}(\Delta X_i^n)^\top\partial_{{\alpha_{m_1}}}S_{i-1}^{-1}(\alpha_0)\Delta X_i^n+\partial_{{\alpha}_{m_1}}\log\det S_{i-1}(\alpha_0)\right\}\boldsymbol{1}_{\{|\Delta X_i^n|\le D_1h_n^{\rho_1}\}}\  |\mathcal{F}_{i-1}^n \right]\right|}\\
&\le\frac{1}{2\sqrt{n}}\sum_{i=1}^{n}{\left|\mathbb{E}\left[\left\{h_n^{-1}\bar{X}_{i,n}(\beta_0)^\top\partial_{{\alpha_{m_1}}}S_{i-1}^{-1}(\alpha_0)\bar{X}_{i,n}(\beta_0)+\partial_{{\alpha}_{m_1}}\log\det S_{i-1}(\alpha_0)\right\}\boldsymbol{1}_{\{|\Delta X_i^n|\le D_1h_n^{\rho_1}\}}\  |\mathcal{F}_{i-1}^n \right]\right|}\\
&\quad+\frac{1}{\sqrt{n}}\sum_{i=1}^{n}{\left|\mathbb{E}\left[\left\{b_{i-1}(\beta_0)^\top\partial_{{\alpha_{m_1}}}S_{i-1}^{-1}(\alpha_0)\bar{X}_{i,n}(\beta_0)\right\}\boldsymbol{1}_{\{|\Delta X_i^n|\le D_1h_n^{\rho_1}\}}\  |\mathcal{F}_{i-1}^n \right]\right|}\\
&\quad+\frac{h_n}{2\sqrt{n}}\sum_{i=1}^{n}{\left|b_{i-1}(\beta_0)^\top\partial_{{\alpha_{m_1}}}S_{i-1}^{-1}(\alpha_0)b_{i-1}(\beta_0)\right| P\left(|\Delta X_i^n|\le D_1h_n^{\rho_1}\ |\mathcal{F}_{i-1}^n\right)}\\
&\le\frac{1}{2\sqrt{n}}\sum_{i=1}^{n}{\left|\sum_{k_1,k_2=1}^{d}h_n^{-1}\partial_{{\alpha_{m_1}}}S_{i-1}^{{-1}^{(k_1,k_2)}}(\alpha_0)\mathbb{E}\left[\bar{X}_{i,n}^{(k_1)}(\beta_0)\bar{X}_{i,n}^{(k_2)}(\beta_0)\boldsymbol{1}_{C_{i,0}^n(D_1,\rho_1)} \  |\mathcal{F}_{i-1}^n\right]\right.}\\
&\qquad{\left. +\partial_{{\alpha}_{m_1}}\log\det S_{i-1}(\alpha_0)P\left(|\Delta X_i^n|\le D_1h_n^{\rho_1}\ |\mathcal{F}_{i-1}^n\right)\right|}\\
&\quad+\frac{1}{2\sqrt{n}}\sum_{i=1}^{n}{\left|\sum_{k_1,k_2=1}^{d}\sum_{j=1}^{2}{}h_n^{-1}\partial_{{\alpha_{m_1}}}S_{i-1}^{{-1}^{(k_1,k_2)}}(\alpha_0)\mathbb{E}\left[\bar{X}_{i,n}^{(k_1)}(\beta_0)\bar{X}_{i,n}^{(k_2)}(\beta_0)\boldsymbol{1}_{C_{i,j}^n(D_1,\rho_1)} \  |\mathcal{F}_{i-1}^n\right]\right|}\\
&\quad+\frac{1}{\sqrt{n}}\sum_{i=1}^{n}{\left|\sum_{k_1,k_2=1}^{d}b_{i-1}^{(k_1)}(\beta_0)\partial_{{\alpha_{m_1}}}S_{i-1}^{{-1}^{(k_1,k_2)}}(\alpha_0)\mathbb{E}\left[\bar{X}_{i,n}^{(k_2)}(\beta_0)\boldsymbol{1}_{C_{i,0}^n(D_1,\rho_1)}\  |\mathcal{F}_{i-1}^n \right]\right|}\\
&\quad+\frac{1}{\sqrt{n}}\sum_{i=1}^{n}{\left|\sum_{k_1,k_2=1}^{d}\sum_{j=1}^{2}b_{i-1}^{(k_1)}(\beta_0)\partial_{{\alpha_{m_1}}}S_{i-1}^{{-1}^{(k_1,k_2)}}(\alpha_0)\mathbb{E}\left[\bar{X}_{i,n}^{(k_2)}(\beta_0)\boldsymbol{1}_{C_{i,j}^n(D_1,\rho_1)}\  |\mathcal{F}_{i-1}^n \right]\right|}\\
&\quad+\frac{h_n}{2\sqrt{n}}\sum_{i=1}^{n}{\left|b_{i-1}(\beta_0)^\top\partial_{{\alpha_{m_1}}}S_{i-1}^{-1}(\alpha_0)b_{i-1}(\beta_0)\right| P\left(|\Delta X_i^n|\le D_1h_n^{\rho_1}\ |\mathcal{F}_{i-1}^n\right)}\\
&=\frac{1}{2\sqrt{n}}\sum_{i=1}^{n}{R(\theta,h_n,X_{t_{i-1}^n})}+\frac{1}{2\sqrt{n}}\sum_{i=1}^{n}{R(\theta,h_n^{3\rho_1},X_{t_{i-1}^n})}+\frac{1}{\sqrt{n}}\sum_{i=1}^{n}{R(\theta,h_n^{2},X_{t_{i-1}^n})}\\
&\quad+\frac{1}{\sqrt{n}}\sum_{i=1}^{n}{R(\theta,h_n^{1+2\rho_1},X_{t_{i-1}^n})}+\frac{h_n}{2\sqrt{n}}\sum_{i=1}^{n}{R(\theta,1,X_{t_{i-1}^n})}\\
&=O_P\left(\sqrt{nh_n^2}\right)+O_P\left(\sqrt{nh_n^{6\rho_1}}\right)\\
&\le O_P\left(\sqrt{nh_n^{1+\delta}}\right)\\
&\overset{P}{\to} 0.
\end{align*}
This implies \eqref{MCLT1-jump}.

{\bf Proof of \eqref{MCLT2-jump}.}
 We see from \eqref{eq:barX}, Proposition \ref{prop2:jump}, Proposition \ref{prop3:jump}, H\"{o}lder's inequality that 
\begin{align*}
&\left|\sum_{i=1}^{n}{\mathbb{E}\left[\frac{1}{\sqrt{nh_n}}\left(\eta_{i,1}^{m_2}(\beta_0|\alpha_0)+\eta_{i,2}^{m_2}(\beta_0)\right)\ |\mathcal{F}_{i-1}^n\right]}\right|\\
&\le\frac{1}{\sqrt{nh_n}}\sum_{i=1}^{n}{\left|\sum_{k_1,k_2=1}^{d}{\partial_{{\beta_{m_2}}}b_{i-1}^{(k_1)}(\beta_0)S_{i-1}^{{-1}^{(k_1,k_2)}}(\alpha_0)\mathbb{E}\left[\bar{X}_{i,n}^{(k_2)}(\beta_0)\boldsymbol{1}_{C_{i,0}^n(D_3,\rho_3)}\ |\mathcal{F}_{i-1}^n\right]}\right|}\\
&\quad+\frac{1}{\sqrt{nh_n}}\sum_{i=1}^{n}{\sum_{k_1,k_2=1}^{d}{\left|\partial_{{\beta_{m_2}}}b_{i-1}^{(k_1)}(\beta_0)\right|\cdot\left|S_{i-1}^{{-1}^{(k_1,k_2)}}(\alpha_0)\right|\sum_{j=1}^{2}{}\mathbb{E}\left[|\bar{X}_{i,n}^{(k_2)}(\beta_0)|\boldsymbol{1}_{C_{i,j}^n(D_3,\rho_3)}\ |\mathcal{F}_{i-1}^n\right]}}\\
&\quad+\left|\frac{1}{\sqrt{nh_n}}\sum_{i=1}^{n}\mathbb{E}\left[\left(\partial_{\beta_{m_2}}\log \Psi_{\beta_0}(\Delta X_i^n,X_{t_{i-1}^n})\right)\varphi_n(X_{t_{i-1}^n},\Delta X_i^n)\boldsymbol{1}_{D_{i,1}^n(D_2,\rho_2)}\right.\right.\\
&\left. \left.\qquad\qquad\qquad\quad-h_n\int_{B}\partial_{{\beta_{m_2}}}\Psi_{\beta_0}(y,X_{t_{i-1}^n})\varphi_n(X_{t_{i-1}^n},y)dy\ |\mathcal{F}_{i-1}^n\right] \right|\\
&\quad+\frac{1}{\sqrt{nh_n}}\sum_{i=1}^{n}\sum_{j=0,2}\mathbb{E}\left[\left|\partial_{\beta_{m_2}}\log \Psi_{\beta_0}(\Delta X_i^n,X_{t_{i-1}^n})\right|\varphi_n(X_{t_{i-1}^n},\Delta X_i^n)\boldsymbol{1}_{D_{i,j}^n(D_2,\rho_2)}\ |\mathcal{F}_{i-1}^n\right]\\
&\leq \frac{1}{n}\sum_{i=1}^{n}{R(\theta,\sqrt{nh_n^3},X_{t_{i-1}^n})}+\frac{1}{n}\sum_{i=1}^{n}{R\left(\theta,\sqrt{nh_n^{1+4\rho_3}},X_{t_{i-1}^n}\right)}\\
&\quad+\left|\frac{1}{\sqrt{nh_n}}\sum_{i=1}^{n}\mathbb{E}\left[\left(\partial_{\beta_{m_2}}\log \Psi_{\beta_0}(\Delta X_i^n,X_{t_{i-1}^n})\right)\varphi_n(X_{t_{i-1}^n},\Delta X_i^n)\boldsymbol{1}_{D_{i,1}^n(D_2,\rho_2)}\right.\right.\\
&\left. \left.\qquad\qquad\qquad\quad-h_n\int_{B}\partial_{{\beta_{m_2}}}\Psi_{\beta_0}(y,X_{t_{i-1}^n})\varphi_n(X_{t_{i-1}^n},y)dy\ |\mathcal{F}_{i-1}^n\right] \right|\\
&\quad+\frac{1}{\sqrt{nh_n}}\sum_{i=1}^{n}\mathbb{E}\left[\left|\partial_{\beta_{m_2}}\log \Psi_{\beta_0}(\Delta X_i^n,X_{t_{i-1}^n})\right|^4\varphi_n(X_{t_{i-1}^n},\Delta X_i^n)^4\ |\mathcal{F}_{i-1}^n\right]^{\frac{1}{4}}\sum_{j=0,2}P\left(D_{i,j}^n(D_2,\rho_2)|\mathcal{F}_{i-1}\right)^{\frac{3}{4}}.
\end{align*}
We evaluate the third term on the right-hand side in an analogous manner to the proof of the uniform convergence in Proposition \ref{prop7:jump}. 
 Let $g_n(\beta,y,x)=\log\Psi_{\beta}(y,x)\varphi_n(x,y)$. We can calculate 
\begin{footnotesize}
\begin{align*}
&\left|\frac{1}{\sqrt{nh_n}}\sum_{i=1}^{n}\mathbb{E}\left[\left(\partial_{\beta_{m_2}}\log \Psi_{\beta_0}(\Delta X_i^n,X_{t_{i-1}^n})\right)\varphi_n(X_{t_{i-1}^n},\Delta X_i^n)\boldsymbol{1}_{D_{i,1}^n(D_2,\rho_2)}\right.\right.\\
&\left. \left.\qquad\qquad\qquad\quad-h_n\int_{B}\partial_{{\beta_{m_2}}}\Psi_{\beta_0}(y,X_{t_{i-1}^n})\varphi_n(X_{t_{i-1}^n},y)dy\ |\mathcal{F}_{i-1}^n\right] \right|\\
&\leq\sqrt{nh_n}\left\{\left|\mathbb{E}\left[\frac{1}{nh_n}\sum_{i=1}^{n}\partial_{{\beta_{m_2}}}g_n(\beta_0,\Delta X_i^n,X_{t_{i-1}^n})\boldsymbol{1}_{D_{i,1}^n(D_2,\rho_2)}-\frac{1}{nh_n}\sum_{i=1}^{n}\partial_{{\beta_{m_2}}}g_n(\beta_0,\Delta X_{\tau_i^n},X_{t_{i-1}^n})\boldsymbol{1}_{D_{i,1}^n(D_2,\rho_2)}|\mathcal{F}_{i-1}^n\right]\right|\right.\\
&\qquad\qquad+\left|\mathbb{E}\left[\frac{1}{nh_n}\sum_{i=1}^{n}\partial_{{\beta_{m_2}}}g_n(\beta_0,\Delta X_{\tau_i^n},X_{t_{i-1}^n})\boldsymbol{1}_{D_{i,1}^n(D_2,\rho_2)}-\frac{1}{nh_n}\sum_{i=1}^{n}\partial_{{\beta_{m_2}}}g_n(\beta_0,\Delta X_{\tau_i^n},X_{t_{i-1}^n})\boldsymbol{1}_{\{J_i^n=1\}}|\mathcal{F}_{i-1}^n\right]\right|\\
&\qquad\qquad+\left|\mathbb{E}\left[\frac{1}{nh_n}\sum_{i=1}^{n}\partial_{{\beta_{m_2}}}g_n(\beta_0,\Delta X_{\tau_i^n},X_{t_{i-1}^n})\boldsymbol{1}_{\{J_i^n=1\}}-\frac{1}{nh_n}\sum_{i=1}^{n}\int_{t_{i-1}^n}^{t_i^n}\int_E\partial_{{\beta_{m_2}}}g_n(\beta_0,c_{i-1}(z,\beta_0),X_{t_{i-1}^n})p(ds,dz)|\mathcal{F}_{i-1}^n\right]\right|\\
&\qquad\qquad+\left|\mathbb{E}\left[\frac{1}{nh_n}\sum_{i=1}^{n}\int_{t_{i-1}^n}^{t_i^n}\int_E\partial_{{\beta_{m_2}}}g_n(\beta_0,c_{i-1}(z,\beta_0),X_{t_{i-1}^n})p(ds,dz)\right.\right.\\
&\qquad\qquad\qquad\qquad\left.\left.-\frac{1}{nh_n}\sum_{i=1}^{n}\int_{t_{i-1}^n}^{t_i^n}\int_E\partial_{{\beta_{m_2}}}g_n(\beta_0,c_{i-1}(z,\beta_0),X_{t_{i-1}^n})q^{\beta_0}(ds,dz)|\mathcal{F}_{i-1}^n\right]\right|\\
&\qquad\qquad+\left.\left|\mathbb{E}\left[\frac{1}{nh_n}\sum_{i=1}^{n}\int_{t_{i-1}^n}^{t_i^n}\int_E\partial_{{\beta_{m_2}}}g_n(\beta_0,c_{i-1}(z,\beta_0),X_{t_{i-1}^n})q^{\beta_0}(ds,dz)-\frac{1}{n}\sum_{i=1}^n\int_{B}\partial_{{\beta_{m_2}}}\Psi_{\beta_0}(y,X_{t_{i-1}^n})\varphi_n(X_{t_{i-1}^n},y)dy\ |\mathcal{F}_{i-1}^n\right]\right|\right\}\\
&\leq\frac{1}{\sqrt{nh_n}}\sum_{i=1}^{n}\mathbb{E}\left[|\partial_{{\beta_{m_2}}}g_n(\beta_0,\Delta X_i^n,X_{t_{i-1}^n})-\partial_{{\beta_{m_2}}}g_n(\beta_0,\Delta X_{\tau_i^n},X_{t_{i-1}^n})|\boldsymbol{1}_{D_{i,1}^n(D_2,\rho_2)}|\mathcal{F}_{i-1}^n\right]\\
&\quad+\frac{1}{\sqrt{nh_n}}\sum_{i=1}^{n}\mathbb{E}\left[|\partial_{{\beta_{m_2}}}g_n(\beta_0,\Delta X_{\tau_i^n},X_{t_{i-1}^n})|\boldsymbol{1}_{C_{i,1}^n(D_2,\rho_2)}|\mathcal{F}_{i-1}^n\right]\\
&\quad+\frac{1}{\sqrt{nh_n}}\sum_{i=1}^{n}\left|\mathbb{E}\left[\partial_{{\beta_{m_2}}}g_n(\beta_0,\Delta X_{\tau_i^n},X_{t_{i-1}^n})\boldsymbol{1}_{\{J_i^n=1\}}-\int_{t_{i-1}^n}^{t_i^n}\int_E\partial_{{\beta_{m_2}}}g_n(\beta_0,c_{i-1}(z,\beta_0),X_{t_{i-1}^n})p(ds,dz)|\mathcal{F}_{i-1}^n\right]\right|\\
&\quad+\frac{1}{\sqrt{nh_n}}\sum_{i=1}^{n}\left|\mathbb{E}\left[\int_{t_{i-1}^n}^{t_i^n}\int_E\partial_{{\beta_{m_2}}}g_n(\beta_0,c_{i-1}(z,\beta_0),X_{t_{i-1}^n})(p-q^{\beta_0})(ds,dz)|\mathcal{F}_{i-1}^n\right]\right|\\
&\quad+\frac{1}{\sqrt{nh_n}}\sum_{i=1}^{n}\left|\mathbb{E}\left[\int_{t_{i-1}^n}^{t_i^n}\int_E\partial_{{\beta_{m_2}}}g_n(\beta_0,c_{i-1}(z,\beta_0),X_{t_{i-1}^n})q^{\beta_0}(ds,dz)-h_n\int_{B}\partial_{{\beta_{m_2}}}\Psi_{\beta_0}(y,X_{t_{i-1}^n})\varphi_n(X_{t_{i-1}^n},y)dy\ |\mathcal{F}_{i-1}^n\right]\right|\\
&=:\sum_{i=1}^5H_n^i.
\end{align*}
\end{footnotesize}
It is obvious from martingale property that $H_n^4=0$. Since it follows from change of variables that
\begin{align*}
&\int_{t_{i-1}^n}^{t_i^n}\int_E\partial_{{\beta_{m_2}}}g_n(\beta_0,c_{i-1}(z,\beta_0),X_{t_{i-1}^n})q^{\beta_0}(ds,dz)\\
&=\int_{t_{i-1}^n}^{t_i^n}\int_E\partial_{{\beta_{m_2}}}\log\Psi_{\beta_0}(c_{i-1}(z,\beta_0),X_{t_{i-1}^n})\varphi_n(X_{t_{i-1}^n},c_{i-1}(z,\beta_0))f_{\beta_0}(z)dzds\\
&=\int_{t_{i-1}^n}^{t_i^n}ds\int_B\partial_{{\beta_{m_2}}}\log\Psi_{\beta_0}(y,X_{t_{i-1}^n})\varphi_n(X_{t_{i-1}^n},y)\Psi_{\beta_0}(y,X_{t_{i-1}^n})dy\\
&=h_n\int_B\frac{\partial_{{\beta_{m_2}}}\Psi_{\beta_0}(y,X_{t_{i-1}^n})}{\Psi_{\beta_0}(y,X_{t_{i-1}^n})}\varphi_n(X_{t_{i-1}^n},y)\Psi_{\beta_0}(y,X_{t_{i-1}^n})dy\\
&=h_n\int_B\partial_{\beta_{m_2}}\Psi_{\beta_0}(y,X_{t_{i-1}^n})\varphi_n(X_{t_{i-1}^n},y)dy.
\end{align*}
Hence, we have $H_n^5=0$.
For $H_n^1,H_n^2,H_n^3$, in a similar manner to the evaluations of $I_1,I_2$ and $I_3$ in the proof of Proposition \ref{prop7:jump} with $k=1$,  one has 
\begin{align*}
H_n^1&=O_p\left(\sqrt{nh_n\varepsilon_n^{-4}}\right),\ H_n^3=O_p\left(nh_n^2\varepsilon_n^{-4}\right)+O_p\left(\sqrt{nh_n^2\varepsilon_n^{-2}}\right), \\
H_n^2&=\begin{cases}
O_p\left(\sqrt{nh_n}/(h_nu_n^p)\right)+O_p\left(\sqrt{nh_n^{1+2\rho_2}u_n^2}\right)&(\text{under}\ 
[\textbf{C}_2\textbf{1}]),\\
O_p\left(\sqrt{nh_n^{1+2\rho_2}\varepsilon_n^{-2}}\right)&(\text{under}\ [\textbf{C}_2\textbf{2}]),
\end{cases}
\end{align*}
 where a real valued sequence $u_n$ and an integer $p\ge 2$ satisfy
$nh_n^{1+2\rho_2}u_n^2\to 0$ and $h_nu_n^p/\sqrt{nh_n}\to\infty$.
For example, we can take $u_n=h_n^{(\delta-2\rho_2)/2}$ and choose sufficient large $p$ such that $h_nu_n^p/\sqrt{nh_n}\to\infty$.
Hence, since $\rho_2\in B_2$ under the condition [$\textbf{C}_2${\bf 2}], it follows from $\sum_{i=1}^5H_n^i\overset{P}{\to}0$ that the third term is bounded by $o_p(1)$. For the fourth term, since by using Proposition \ref{prop2:jump},
\begin{align*}
\sum_{j=0,2}P\left(D_{i,j}^n(D_2,\rho_2)|\mathcal{F}_{i-1}\right)^{\frac{3}{4}}=R(\theta,h_n^\frac{3}{2},X_{t_{i-1}^n}),
\end{align*}
and  
\begin{align*}
\frac{1}{\sqrt{nh_n}}\sum_{i=1}^{n}&\mathbb{E}\left[\left|\partial_{\beta_{m_2}}\log \Psi_{\beta_0}(\Delta X_i^n,X_{t_{i-1}^n})\right|^4\varphi_n(X_{t_{i-1}^n},\Delta X_i^n)^4\ |\mathcal{F}_{i-1}^n\right]^{\frac{1}{4}}\\
&\leq\begin{cases}\frac{1}
{\sqrt{nh_n}}\sum_{i=1}^{n}R(\theta,1,X_{t_{i-1}^n})&(\text{under} [\textbf{C}_2\textbf{1}]), \\
\frac{1}{\sqrt{nh_n}}\sum_{i=1}^{n}R(\theta,\varepsilon_n^{-1},X_{t_{i-1}^n})&(\text{under}\ [\textbf{C}_2\textbf{2}]),
\end{cases}
\end{align*}
then we have
\begin{align*}
\frac{1}{\sqrt{nh_n}}\sum_{i=1}^{n}&\mathbb{E}\left[\left|\partial_{\beta_{m_2}}\log \Psi_{\beta_0}(\Delta X_i^n,X_{t_{i-1}^n})\right|^4\varphi_n(X_{t_{i-1}^n},\Delta X_i^n)^4\ |\mathcal{F}_{i-1}^n\right]^{\frac{1}{4}}\sum_{j=0,2}P\left(D_{i,j}^n(D_2,\rho_2)|\mathcal{F}_{i-1}\right)^{\frac{3}{4}}\\
&\leq\begin{cases}
\frac{1}{n}\sum_{i=1}^{n}{R\left(\theta,\sqrt{nh_n^2},X_{t_{i-1}^n}\right)}&(\text{under}\ [\textbf{C}_2\textbf{1}])\\
\frac{1}{n}\sum_{i=1}^{n}R\left(\theta,\sqrt{nh_n^2\varepsilon_n^{-2}},X_{t_{i-1}^n}\right)&(\text{under} [\textbf{C}_2\textbf{2}])
\end{cases}\\
&=O_p\left(\sqrt{nh_n^2\varepsilon_n^{-2}}\right).
\end{align*}
Therefore, it holds from $\rho_3\geq\frac{\delta}{4}$ that 
\begin{align*}
&\left|\sum_{i=1}^{n}{\mathbb{E}\left[\frac{1}{\sqrt{nh_n}}\left(\eta_{i,1}^{m_2}(\beta_0|\alpha_0)+\eta_{i,2}^{m_2}(\beta_0)\right)\ |\mathcal{F}_{i-1}^n\right]}\right|\\
&\leq \frac{1}{n}\sum_{i=1}^{n}{R(\theta,\sqrt{nh_n^3},X_{t_{i-1}^n})}+\frac{1}{n}\sum_{i=1}^{n}{R\left(\theta,\sqrt{nh_n^{1+4\rho_3}},X_{t_{i-1}^n}\right)}\\
&\quad+\left|\frac{1}{\sqrt{nh_n}}\sum_{i=1}^{n}\mathbb{E}\left[\left(\partial_{\beta_{m_2}}\log \Psi_{\beta_0}(\Delta X_i^n,X_{t_{i-1}^n})\right)\varphi_n(X_{t_{i-1}^n},\Delta X_i^n)\boldsymbol{1}_{D_{i,1}^n(D_2,\rho_2)}\right.\right.\\
&\left. \left.\qquad\qquad\qquad\quad-h_n\int_{B}\partial_{{\beta_{m_2}}}\Psi_{\beta_0}(y,X_{t_{i-1}^n})\varphi_n(X_{t_{i-1}^n},y)dy\ |\mathcal{F}_{i-1}^n\right] \right|\\
&\quad+\frac{1}{\sqrt{nh_n}}\sum_{i=1}^{n}\mathbb{E}\left[\left|\partial_{\beta_{m_2}}\log \Psi_{\beta_0}(\Delta X_i^n,X_{t_{i-1}^n})\right|^4\varphi_n(X_{t_{i-1}^n},\Delta X_i^n)^4\ |\mathcal{F}_{i-1}^n\right]^{\frac{1}{4}}\sum_{j=0,2}P\left(D_{i,j}^n(D_2,\rho_2)|\mathcal{F}_{i-1}\right)^{\frac{3}{4}}\\
&\leq O_p\left(\sqrt{nh_n^3}\right)+O_p\left(\sqrt{nh_n^{1+\delta}}\right)+o_p(1)+O_p\left(\sqrt{nh_n^2\varepsilon_n^{-2}}\right)\\
\overset{P}{\to}0
\end{align*}
under either condition [$\textbf{C}_2${\bf 1}] or [$\textbf{C}_2${\bf 2}].
This implies \eqref{MCLT2-jump}.

{\bf Proof of \eqref{MCLT3-jump}.}
Since $\Delta X_{i}^n=\bar{X}_{i,n}(\beta_0)+h_nb_{i-1}(\beta_0)$, we can calculate  
\begin{align*}
&\sum_{i=1}^{n}{\mathbb{E}\left[\frac{1}{n}\xi_i^{m_1}(\alpha_0)\xi_i^{m_1'}(\alpha_0)\ |\mathcal{F}_{i-1}^n\right]}\\
&=\frac{1}{4n}\sum_{i=1}^{n}\mathbb{E}\left[\left\{h_n^{-1}(\Delta X_i^n)^\top\partial_{{\alpha_{m_1}}}S_{i-1}^{-1}(\alpha_0)\Delta X_i^n+\partial_{{\alpha}_{m_1}}\log\det S_{i-1}(\alpha_0)\right\}\right.\\
&\qquad\qquad\qquad \times\left.\left\{h_n^{-1}(\Delta X_i^n)^\top\partial_{{\alpha_{m_1'}}}S_{i-1}^{-1}(\alpha_0)\Delta X_i^n+\partial_{{\alpha}_{m_1'}}\log\det S_{i-1}(\alpha_0)\right\}\boldsymbol{1}_{\{|\Delta X_i^n|\le D_1h_n^{\rho_1}\}}\ |\mathcal{F}_{i-1}^n\right]\\
&=\frac{1}{4n}\sum_{i=1}^{n}\mathbb{E}\left[\left\{h_n^{-1}(\bar{X}_{i,n}(\beta_0))^\top\partial_{{\alpha_{m_1}}}S_{i-1}^{-1}(\alpha_0)\bar{X}_{i,n}(\beta_0)+
2b_{i-1}^\top(\beta_0)\partial_{{\alpha_{m_1}}}S_{i-1}^{-1}(\alpha_0)\bar{X}_{i,n}(\beta_0)\right.\right.\\
&\qquad\qquad\qquad\quad +h_nb_{i-1}^\top(\beta_0)\partial_{{\alpha_{m_1}}}S_{i-1}^{-1}(\alpha_0)b_{i-1}^\top(\beta_0)+\partial_{{\alpha}_{m_1}}\log\det S_{i-1}(\alpha_0)\left.\right\}\\
&\qquad\qquad\qquad \times\left.\left\{h_n^{-1}(\bar{X}_{i,n}(\beta_0))^\top\partial_{{\alpha_{m_1'}}}S_{i-1}^{-1}(\alpha_0)\bar{X}_{i,n}(\beta_0)+
2b_{i-1}^\top(\beta_0)\partial_{{\alpha_{m_1'}}}S_{i-1}^{-1}(\alpha_0)\bar{X}_{i,n}(\beta_0)\right.\right.\\
&\qquad\qquad\qquad\quad\left.\left. +h_nb_{i-1}^\top(\beta_0)\partial_{{\alpha_{m_1'}}}S_{i-1}^{-1}(\alpha_0)b_{i-1}(\beta_0)+\partial_{{\alpha}_{m_1'}}\log\det S_{i-1}(\alpha_0)\right\}\boldsymbol{1}_{\{|\Delta X_i^n|\le D_1h_n^{\rho_1}\}}\ |\mathcal{F}_{i-1}^n\right]\\
&=\frac{1}{4nh_n^2}\sum_{i=1}^{n}\sum_{k_1,k_2=1}^{d}\sum_{k_3,k_4=1}^{d}\partial_{{\alpha_{m_1}}}S_{i-1}^{{-1}^ {(k_1,k_2)}}\partial_{{\alpha_{m_1'}}}S_{i-1}^{{-1}^ {(k_3,k_4)}}\mathbb{E}\left[\bar{X}_{i,n}^{(k_1)}\bar{X}_{i,n}^{(k_2)}\bar{X}_{i,n}^{(k_3)}\bar{X}_{i,n}^{(k_4)}\boldsymbol{1}_{\{|\Delta X_i^n|\le D_1h_n^{\rho_1}\}} \ |\mathcal{F}_{i-1}^n\right]\\
&\quad+\frac{1}{nh_n}\sum_{i=1}^{n}\sum_{k_1,k_2=1}^{d}\sum_{k_3,k_4=1}^{d}\partial_{{\alpha_{m_1}}}S_{i-1}^{{-1}^ {(k_1,k_2)}}\partial_{{\alpha_{m_1'}}}S_{i-1}^{{-1}^ {(k_3,k_4)}}b_{i-1}^{(k_3)}\mathbb{E}\left[\bar{X}_{i,n}^{(k_1)}\bar{X}_{i,n}^{(k_2)}\bar{X}_{i,n}^{(k_3)}\boldsymbol{1}_{\{|\Delta X_i^n|\le D_1h_n^{\rho_1}\}} \ |\mathcal{F}_{i-1}^n\right]\\
&\quad+\frac{1}{4nh_n}\sum_{i=1}^{n}\partial_{{\alpha}_{m_1'}}\log\det S_{i-1}\sum_{k_1,k_2=1}^{d}\partial_{{\alpha_{m_1}}}S_{i-1}^{{-1}^ {(k_1,k_2)}}\mathbb{E}\left[\bar{X}_{i,n}^{(k_1)}\bar{X}_{i,n}^{(k_2)}\boldsymbol{1}_{\{|\Delta X_i^n|\le D_1h_n^{\rho_1}\}} \ |\mathcal{F}_{i-1}^n\right]\\
&\quad+\frac{1}{4nh_n}\sum_{i=1}^{n}\partial_{{\alpha}_{m_1}}\log\det S_{i-1}\sum_{k_1,k_2=1}^{d}\partial_{{\alpha_{m_1'}}}S_{i-1}^{{-1}^ {(k_1,k_2)}}\mathbb{E}\left[\bar{X}_{i,n}^{(k_1)}\bar{X}_{i,n}^{(k_2)}\boldsymbol{1}_{\{|\Delta X_i^n|\le D_1h_n^{\rho_1}\}} \ |\mathcal{F}_{i-1}^n\right]\\
&\quad+\frac{1}{n}\sum_{i=1}^nR(\theta,1,X_{t_{i-1}^n})\sum_{k_1,k_2=1}^d\mathbb{E}\left[\bar{X}_{i,n}^{(k_1)}\bar{X}_{i,n}^{(k_2)}\boldsymbol{1}_{\{|\Delta X_i^n|\le D_1h_n^{\rho_1}\}} \ |\mathcal{F}_{i-1}^n\right]\\
&\quad+(1+h_n)\frac{1}{n}\sum_{i=1}^{n}\sum_{k_1=1}^{d}R(\theta,1,X_{t_{i-1}^n})\mathbb{E}\left[\bar{X}_{i,n}^{(k_1)}\boldsymbol{1}_{\{|\Delta X_i^n|\le D_1h_n^{\rho_1}\}} \ |\mathcal{F}_{i-1}^n\right]\\
&\quad+h_n(1+h_n)\frac{1}{n}\sum_{i=1}^{n}R(\theta,1,X_{t_{i-1}^n})P(|\Delta X_i^n|\le D_1h_n^{\rho_1}|\mathcal{F}_{i-1}^n)\\
&\quad+\frac{1}{4n}\sum_{i=1}^{n}\left(\partial_{{\alpha}_{m_1}}\log\det S_{i-1}\right)\left(\partial_{{\alpha}_{m_1'}}\log\det S_{i-1}\right)P(|\Delta X_i^n|\le D_1h_n^{\rho_1}\ |\mathcal{F}_{i-1}^n).
\end{align*}
Since it holds from Proposition \ref{prop2:jump} that for $p=1,2,3,4$ and $k_l=1,\ldots,d$,\quad$(l=1,\ldots,p)$,
\begin{align}
&\mathbb{E}\left[\prod_{l=1}^p \bar{X}_{i,n}^{(k_l)}\boldsymbol{1}_{\{|\Delta X_i^n|\le D_1h_n^{\rho_1}\}}\ |\mathcal{F}_{i-1}^n\right]\notag\\
&=\mathbb{E}\left[\prod_{l=1}^p \bar{X}_{i,n}^{(k_l)}\boldsymbol{1}_{C_{i,0}^n(D_1,\rho_1)}\ |\mathcal{F}_{i-1}^n\right]+\sum_{j=1}^{2}{}\mathbb{E}\left[\prod_{l=1}^p \bar{X}_{i,n}^{(k_l)}\boldsymbol{1}_{C_{i,j}^n(D_1,\rho_1)}\ |\mathcal{F}_{i-1}^n\right]\notag\\
&=\mathbb{E}\left[\prod_{l=1}^p \bar{X}_{i,n}^{(k_l)}\boldsymbol{1}_{C_{i,0}^n(D_1,\rho_1)}\ |\mathcal{F}_{i-1}^n\right]+R\left(\theta,h_n^{1+(p+1)\rho_1},X_{t_{i-1}^n}\right),\label{devide:jump}
\end{align}
 it follows from Proposition \ref{prop3:jump} that for $k_l=1,\ldots,d$,\quad$(l=1,\ldots,p)$,
\begin{align}
&\mathbb{E}\left[\bar{X}_{i,n}^{(k_1)}\bar{X}_{i,n}^{(k_2)}\bar{X}_{i,n}^{(k_3)}\bar{X}_{i,n}^{(k_4)}\boldsymbol{1}_{\{|\Delta X_i^n|\le D_1h_n^{\rho_1}\}} \ |\mathcal{F}_{i-1}^n\right]\nonumber\\
&\qquad=h_n^2\left(S_{i-1}^{(k_1,k_2)}S_{i-1}^{(k_3,k_4)}+S_{i-1}^{(k_1,k_3)}S_{i-1}^{(k_2,k_4)}+S_{i-1}^{(k_1,k_4)}S_{i-1}^{(k_2,k_3)}\right)(\alpha_0)+R(\theta,h_n^{1+5\rho_1},X_{t_{i-1}^n})+R(\theta,h_n^3,X_{t_{i-1}^n}), \label{under:rep4}\\
&\mathbb{E}\left[\bar{X}_{i,n}^{(k_1)}\bar{X}_{i,n}^{(k_2)}\bar{X}_{i,n}^{(k_3)}\boldsymbol{1}_{\{|\Delta X_i^n|\le D_1h_n^{\rho_1}\}} \ |\mathcal{F}_{i-1}^n\right]= R(\theta,h_n^{2},X_{t_{i-1}^n})+R(\theta,h_n^{1+4\rho_1},X_{t_{i-1}^n}), \label{under:rep3}\\
&\mathbb{E}\left[\bar{X}_{i,n}^{(k_1)}\bar{X}_{i,n}^{(k_2)}\boldsymbol{1}_{\{|\Delta X_i^n|\le D_1h_n^{\rho_1}\}} \ |\mathcal{F}_{i-1}^n\right]= h_nS_{i-1}^{(k_1,k_2)}(\alpha_0)+R(\theta,h_n^{2},X_{t_{i-1}^n})+R(\theta,h_n^{1+3\rho_1},X_{t_{i-1}^n}), \label{under:rep2}\\
&\mathbb{E}\left[\bar{X}_{i,n}^{(k_1)}\boldsymbol{1}_{\{|\Delta X_i^n|\le D_1h_n^{\rho_1}\}} \ |\mathcal{F}_{i-1}^n\right]= R(\theta,h_n^{2},X_{t_{i-1}^n})+R(\theta,h_n^{1+2\rho_1},X_{t_{i-1}^n}).\label{under:rep1}
\end{align}
In particular, the right-hand side of \eqref{under:rep3} and \eqref{under:rep1}, and the second and third term on the right-hand side of \eqref{under:rep2} can be expressed by $R(\theta,h_n^{1+2\rho_1},X_{t_{i-1}^n})$ since $0<\rho_1<\frac{1}{2}$.
Therefore, it follows from Proposition \ref{prop4:jump}-(i), \eqref{under:rep4}, \eqref{under:rep3}, \eqref{under:rep2}, \eqref{under:rep1} and $\frac{1}{5}<\rho_1$ that 
\begin{align*}
&\sum_{i=1}^{n}{\mathbb{E}\left[\frac{1}{n}\xi_i^{m_1}(\alpha_0)\xi_i^{m_1'}(\alpha_0)\ |\mathcal{F}_{i-1}^n\right]}\\
&=\frac{1}{4n}\sum_{i=1}^{n}\sum_{k_1,k_2=1}^{d}\sum_{k_3,k_4=1}^{d}\partial_{{\alpha_{m_1}}}S_{i-1}^{{-1}^ {(k_1,k_2)}}\partial_{{\alpha_{m_1'}}}S_{i-1}^{{-1}^ {(k_3,k_4)}}(S_{i-1}^{(k_1,k_2)}S_{i-1}^{(k_3,k_4)}+S_{i-1}^{(k_1,k_3)}S_{i-1}^{(k_2,k_4)}+S_{i-1}^{(k_1,k_4)}S_{i-1}^{(k_2,k_3)})\\
&\quad+\frac{1}{4n}\sum_{i=1}^{n}\partial_{{\alpha}_{m_1'}}\log\det S_{i-1}\sum_{k_1,k_2=1}^{d}\left(\partial_{{\alpha_{m_1}}}S_{i-1}^{{-1}^ {(k_1,k_2)}}\right)S_{i-1}^{(k_1,k_2)}\\
&\quad+\frac{1}{4n}\sum_{i=1}^{n}\partial_{{\alpha}_{m_1}}\log\det S_{i-1}\sum_{k_1,k_2=1}^{d}\left(\partial_{{\alpha_{m_1'}}}S_{i-1}^{{-1}^ {(k_1,k_2)}}\right)S_{i-1}^{(k_1,k_2)}\\
&\quad+\frac{1}{4n}\sum_{i=1}^{n}\left(\partial_{{\alpha}_{m_1}}\log\det S_{i-1}\right)\left(\partial_{{\alpha}_{m_1'}}\log\det S_{i-1}\right)\\
&\quad+O_P(h_n)+O_P\left(h_n^{5\rho_1-1}\right)+O_P\left(h_n^{2\rho_1}\right)+O_P\left(h_n^{1+2\rho_1}\right)\\
&\overset{P}{\longrightarrow}\frac{1}{4}\sum_{k_1,k_2=1}^{d}\sum_{k_3,k_4=1}^{d}\int \partial_{{\alpha_{m_1}}}S^{{-1}^ {(k_1,k_2)}}(x,\alpha_0)\partial_{{\alpha_{m_1'}}}S^{{-1}^ {(k_3,k_4)}}(x,\alpha_0)\\
&\qquad\qquad\quad\times(S^{(k_1,k_2)}S^{(k_3,k_4)}+S^{(k_1,k_3)}S^{(k_2,k_4)}+S^{(k_1,k_4)}S^{(k_2,k_3)})(x,\alpha_0)\pi(dx)\\
&\quad+\frac{1}{4}\int \partial_{{\alpha}_{m_1'}}\log\det S(x,\alpha_0)\sum_{k_1,k_2=1}^{d}\left(\partial_{{\alpha_{m_1}}}S^{{-1}^ {(k_1,k_2)}}(x,\alpha_0)\right)S^{(k_1,k_2)}(x,\alpha_0)\pi(dx)\\
&\quad+\frac{1}{4}\int \partial_{{\alpha}_{m_1}}\log\det S(x,\alpha_0)\sum_{k_1,k_2=1}^{d}\left(\partial_{{\alpha_{m_1'}}}S^{{-1}^ {(k_1,k_2)}}(x,\alpha_0)\right)S^{(k_1,k_2)}(x,\alpha_0)\pi(dx)\\
&\quad+\frac{1}{4}\int \left(\partial_{{\alpha}_{m_1'}}\log\det S(x,\alpha_0)\right)\left(\partial_{{\alpha}_{m_1}}\log\det S(x,\alpha_0)\right)\pi(dx)\\
&=\frac{1}{2}\int \mathrm{tr}\left\{\left(\partial_{{\alpha}_{m_1}}S^{-1}\right)S\left(\partial_{{\alpha}_{m_1'}}S^{-1}\right)S\right\}(x,\alpha_0)\pi(dx)\\
&=I_{a}^{(m_1,m_1')}(\alpha_0).
\end{align*}
This implies \eqref{MCLT3-jump}.

{\bf Proof of \eqref{MCLT4-jump}.}
From the proof of \eqref{MCLT1-jump}, it is easy to show that 
\begin{align*}
&\left|\sum_{i=1}^{n}{\mathbb{E}\left[\frac{1}{\sqrt{n}}\xi_i^{m_1}(\alpha_0)\ |\mathcal{F}_{i-1}^n\right]\mathbb{E}\left[\frac{1}{\sqrt{n}}\xi_i^{m_1'}(\alpha_0)\ |\mathcal{F}_{i-1}^n\right]}\right|\\
\le&\ \left|\frac{1}{n}\sum_{i=1}^{n}R\left(\theta,\sqrt{nh_n^{1+\delta}},X_{t_{i-1}^n}\right)R\left(\theta,\sqrt{nh_n^{1+\delta}},X_{t_{i-1}^n}\right)\right|\\
=&\ O_p\left(nh_n^{1+\delta}\right)\\
\overset{P}{\to}&\ 0.
\end{align*}
{\bf Proof of \eqref{MCLT5-jump}.}
By using \eqref{under:rep2}, \eqref{under:rep1}, we can calculate
\begin{align*}
&\sum_{i=1}^{n}{\mathbb{E}\left[\frac{1}{nh_n}\left(\eta_{i,1}^{m_2}(\beta_0|\alpha_0)+\eta_{i,2}^{m_2}(\beta_0)\right)\left(\eta_{i,1}^{m_2'}(\beta_0|\alpha_0)+\eta_{i,2}^{m_2'}(\beta_0)\right)\ |\mathcal{F}_{i-1}^n\right]}\\
&=\frac{1}{nh_n}\sum_{i=1}^{n}\mathbb{E}\left[\left\{(\partial_{{\beta_{m_2}}}b_{i-1}(\beta_0))^\top S_{i-1}^{-1}(\alpha_0)\bar{X}_{i,n}(\beta_0)\boldsymbol{1}_{\{|\Delta X_i^n|\le D_3h_n^{\rho_3}\}}\right.\right.\\
&\quad\left.\qquad\qquad\qquad+\left(\partial_{{\beta_{m_2}}}\log\Psi_{\beta_0}(\Delta X_i^n,X_{t_{i-1}^n})\right)\varphi_n(X_{t_{i-1}^n},\Delta X_i^n)\boldsymbol{1}_{\{|\Delta X_i^n|> D_2h_n^{\rho_2}\}}\right.\\
&\quad\left.\left.\qquad\qquad\qquad-h_n\int_{B}{\partial_{{\beta_{m_2}}}\Psi_{\beta_0}(y,X_{t_{i-1}^n})\varphi_n(X_{t_{i-1}^n},y)}dy\right\}\right.\\
&\quad\left.\qquad\qquad\quad\times\left\{(\partial_{{\beta_{m_2'}}}b_{i-1}(\beta_0))^\top S_{i-1}^{-1}(\alpha_0)\bar{X}_{i,n}(\beta_0)\boldsymbol{1}_{\{|\Delta X_i^n|\le D_3h_n^{\rho_3}\}}\right.\right.\\
&\quad\left.\qquad\qquad\qquad+\left(\partial_{{\beta_{m_2'}}}\log\Psi_{\beta_0}(\Delta X_i^n,X_{t_{i-1}^n})\right)\varphi_n(X_{t_{i-1}^n},\Delta X_i^n)\boldsymbol{1}_{\{|\Delta X_i^n|> D_2h_n^{\rho_2}\}}\right.\\
&\quad\left.\left.\qquad\qquad\qquad-h_n\int_{B}{\partial_{{\beta_{m_2'}}}\Psi_{\beta_0}(y,X_{t_{i-1}^n})\varphi_n(X_{t_{i-1}^n},y)}dy\right\}\ |\mathcal{F}_{i-1}^n\right]\\
&=\frac{1}{nh_n}\sum_{i=1}^{n}\sum_{k_1,k_2=1}^{d}\sum_{k_3,k_4=1}^{d}\partial_{{\beta_{m_2}}}b_{i-1}^{(k_1)} S_{i-1}^{{-1}^{(k_1,k_2)}}\partial_{{\beta_{m_2'}}}b_{i-1}^{(k_3)} S_{i-1}^{{-1}^{(k_3,k_4)}}\mathbb{E}\left[\bar{X}_{i,n}^{(k_2)}\bar{X}_{i,n}^{(k_4)}\boldsymbol{1}_{\{|\Delta X_i^n|\le D_3h_n^{\rho_3}\}} \ |\mathcal{F}_{i-1}^n\right]\\
&\quad+\frac{1}{nh_n}\sum_{i=1}^{n}\mathbb{E}\left[\partial_{{\beta_{m_2}}}\log\Psi_{\beta_0}(\Delta X_i^n,X_{t_{i-1}^n})\partial_{{\beta_{m_2'}}}\log\Psi_{\beta_0}(\Delta X_i^n,X_{t_{i-1}^n})\varphi_n^2(X_{t_{i-1}^n},\Delta X_i^n)\boldsymbol{1}_{\{|\Delta X_i^n|> D_2h_n^{\rho_2}\}}\ |\mathcal{F}_{i-1}^n\right]\\
&\quad+\frac{h_n}{n}\sum_{i=1}^{n}\left(\int_{B}{\partial_{{\beta_{m_2}}}\Psi_{\beta_0}(y,X_{t_{i-1}^n})\varphi_n(X_{t_{i-1}^n},y)}dy\right)\left(\int_{B}{\partial_{{\beta_{m_2'}}}\Psi_{\beta_0}(y,X_{t_{i-1}^n})\varphi_n(X_{t_{i-1}^n},y)}dy\right)\\
&\quad+\frac{1}{nh_n}\sum_{i=1}^{n}\sum_{k_1,k_2=1}^{d}\partial_{{\beta_{m_2}}}b_{i-1}^{(k_1)} S_{i-1}^{{-1}^{(k_1,k_2)}}\\
&\qquad\times\mathbb{E}\left[\bar{X}_{i,n}^{(k_2)}\partial_{{\beta_{m_2'}}}\log\Psi_{\beta_0}(\Delta X_i^n,X_{t_{i-1}^n})\varphi_n(X_{t_{i-1}^n},\Delta X_i^n)\boldsymbol{1}_{\{|\Delta X_i^n|\le D_3h_n^{\rho_3}\}}\boldsymbol{1}_{\{|\Delta X_i^n|> D_2h_n^{\rho_2}\}} \ |\mathcal{F}_{i-1}^n\right]\\
&\quad+\frac{1}{nh_n}\sum_{i=1}^{n}\sum_{k_1,k_2=1}^{d}\partial_{{\beta_{m_2'}}}b_{i-1}^{(k_1)} S_{i-1}^{{-1}^{(k_1,k_2)}}\\
&\qquad\times\mathbb{E}\left[\bar{X}_{i,n}^{(k_2)}\partial_{{\beta_{m_2}}}\log\Psi_{\beta_0}(\Delta X_i^n,X_{t_{i-1}^n})\varphi_n(X_{t_{i-1}^n},\Delta X_i^n)\boldsymbol{1}_{\{|\Delta X_i^n|\le D_3h_n^{\rho_3}\}}\boldsymbol{1}_{\{|\Delta X_i^n|> D_2h_n^{\rho_2}\}} \ |\mathcal{F}_{i-1}^n\right]\\
&\quad+\frac{1}{n}\sum_{i=1}^{n}\sum_{k_1=1}^{d}R(\theta,1,X_{t_{i-1}^n})\mathbb{E}\left[\bar{X}_{i,n}^{(k_1)}\boldsymbol{1}_{\{|\Delta X_i^n|\le D_3h_n^{\rho_3}\}} |\mathcal{F}_{i-1}^n\right]\\
&\quad+\frac{1}{n}\sum_{i=1}^{n}R(\theta,1,X_{t_{i-1}^n})\mathbb{E}\left[\partial_{{\beta_{m_2}}}\log\Psi_{\beta_0}(\Delta X_i^n,X_{t_{i-1}^n})\varphi_n(X_{t_{i-1}^n},\Delta X_i^n)\boldsymbol{1}_{\{|\Delta X_i^n|> D_2h_n^{\rho_2}\}} |\mathcal{F}_{i-1}^n\right]\\
&\quad+\frac{1}{n}\sum_{i=1}^{n}R(\theta,1,X_{t_{i-1}^n})\mathbb{E}\left[\partial_{{\beta_{m_2'}}}\log\Psi_{\beta_0}(\Delta X_i^n,X_{t_{i-1}^n})\varphi_n(X_{t_{i-1}^n},\Delta X_i^n)\boldsymbol{1}_{\{|\Delta X_i^n|> D_2h_n^{\rho_2}\}} |\mathcal{F}_{i-1}^n\right]\\
&=\frac{1}{n}\sum_{i=1}^{n}\sum_{k_1,k_2=1}^{d}\sum_{k_3,k_4=1}^{d}\partial_{{\beta_{m_2}}}b_{i-1}^{(k_1)} \partial_{{\beta_{m_2'}}}b_{i-1}^{(k_3)} S_{i-1}^{{-1}^{(k_1,k_2)}} S_{i-1}^{{-1}^{(k_3,k_4)}} S_{i-1}^{{-1}^{(k_2,k_4)}}\\
&\quad+\frac{1}{nh_n}\sum_{i=1}^{n}\mathbb{E}\left[\partial_{{\beta_{m_2}}}\log\Psi_{\beta_0}(\Delta X_i^n,X_{t_{i-1}^n})\partial_{{\beta_{m_2'}}}\log\Psi_{\beta_0}(\Delta X_i^n,X_{t_{i-1}^n})\varphi_n^2(X_{t_{i-1}^n},\Delta X_i^n)\boldsymbol{1}_{D_{i,1}^n(D_2,\rho_2)}\ |\mathcal{F}_{i-1}^n\right]\\
&\quad-\frac{1}{n}\sum_{i=1}^n\int_A\frac{\partial_{\beta_{m_2}}\Psi_{\beta_0}\partial_{\beta_{m_2'}}\Psi_{\beta_0}}{\Psi_{\beta_0}}(y,X_{t_{i-1}^n})\varphi_n^2(X_{t_{i-1}^n},y)dy\\
&\quad+\frac{1}{n}\sum_{i=1}^n\int_A\frac{\partial_{\beta_{m_2}}\Psi_{\beta_0}\partial_{\beta_{m_2'}}\Psi_{\beta_0}}{\Psi_{\beta_0}}(y,X_{t_{i-1}^n})\varphi_n^2(X_{t_{i-1}^n},y)dy\\
&\quad+\frac{1}{nh_n}\sum_{i=1}^{n}\sum_{j=0,2}\mathbb{E}\left[\partial_{{\beta_{m_2}}}\log\Psi_{\beta_0}(\Delta X_i^n,X_{t_{i-1}^n})\partial_{{\beta_{m_2'}}}\log\Psi_{\beta_0}(\Delta X_i^n,X_{t_{i-1}^n})\varphi_n^2(X_{t_{i-1}^n},\Delta X_i^n)\boldsymbol{1}_{D_{i,j}^n(D_2,\rho_2)}\ |\mathcal{F}_{i-1}^n\right]\\
&\quad+\frac{1}{nh_n}\sum_{i=1}^{n}\sum_{k_1,k_2=1}^{d}\partial_{{\beta_{m_2}}}b_{i-1}^{(k_1)} S_{i-1}^{{-1}^{(k_1,k_2)}}\\
&\qquad\times\mathbb{E}\left[\bar{X}_{i,n}^{(k_2)}\partial_{{\beta_{m_2'}}}\log\Psi_{\beta_0}(\Delta X_i^n,X_{t_{i-1}^n})\varphi_n(X_{t_{i-1}^n},\Delta X_i^n)\boldsymbol{1}_{\{|\Delta X_i^n|\le D_3h_n^{\rho_3}\}}\boldsymbol{1}_{\{|\Delta X_i^n|> D_2h_n^{\rho_2}\}} \ |\mathcal{F}_{i-1}^n\right]\\
&\quad+\frac{1}{nh_n}\sum_{i=1}^{n}\sum_{k_1,k_2=1}^{d}\partial_{{\beta_{m_2'}}}b_{i-1}^{(k_1)} S_{i-1}^{{-1}^{(k_1,k_2)}}\\
&\qquad\times\mathbb{E}\left[\bar{X}_{i,n}^{(k_2)}\partial_{{\beta_{m_2}}}\log\Psi_{\beta_0}(\Delta X_i^n,X_{t_{i-1}^n})\varphi_n(X_{t_{i-1}^n},\Delta X_i^n)\boldsymbol{1}_{\{|\Delta X_i^n|\le D_3h_n^{\rho_3}\}}\boldsymbol{1}_{\{|\Delta X_i^n|> D_2h_n^{\rho_2}\}} \ |\mathcal{F}_{i-1}^n\right]\\
&\quad+\frac{1}{n}\sum_{i=1}^{n}R(\theta,1,X_{t_{i-1}^n})\mathbb{E}\left[\partial_{{\beta_{m_2}}}\log\Psi_{\beta_0}(\Delta X_i^n,X_{t_{i-1}^n})\varphi_n(X_{t_{i-1}^n},\Delta X_i^n)\boldsymbol{1}_{\{|\Delta X_i^n|> D_2h_n^{\rho_2}\}} |\mathcal{F}_{i-1}^n\right]\\
&\quad+\frac{1}{n}\sum_{i=1}^{n}R(\theta,1,X_{t_{i-1}^n})\mathbb{E}\left[\partial_{{\beta_{m_2'}}}\log\Psi_{\beta_0}(\Delta X_i^n,X_{t_{i-1}^n})\varphi_n(X_{t_{i-1}^n},\Delta X_i^n)\boldsymbol{1}_{\{|\Delta X_i^n|> D_2h_n^{\rho_2}\}} |\mathcal{F}_{i-1}^n\right]\\
&\quad+O_p\left(h_n^{1+2\rho_3}\right)+O_p\left(h_n^{2\rho_3}\right)+O_p\left(h_n\right),
\end{align*}
where the third and fourth terms on the right-hand side are adjusted by adding and subtracting the same term. By using Proposition \ref{prop4:jump}-(i), the first and fourth  terms on the right-hand side converge to $I^{(m_2,m_2')}_b(\beta_0)$ and $I^{(m_2,m_2')}_c(\beta_0)$ in probability, respectively. We evaluate the second and third  terms in an analogous manner to Proposition \ref{prop7:jump}. 
 Let $g_n(\beta,y,x)=\partial_{{\beta_{m_2}}}\log\Psi_{\beta}(y,x)\partial_{{\beta_{m_2'}}}\log\Psi_{\beta}(y,x)\varphi_n^2(x,y)$.
We can calculate
\begin{footnotesize}    
\begin{align*}
&\left|\frac{1}{{nh_n}}\sum_{i=1}^{n}\mathbb{E}\left[\partial_{{\beta_{m_2}}}\log\Psi_{\beta_0}(\Delta X_i^n,X_{t_{i-1}^n})\partial_{{\beta_{m_2'}}}\log\Psi_{\beta_0}(\Delta X_i^n,X_{t_{i-1}^n})\varphi_n^2(X_{t_{i-1}^n},\Delta X_i^n)\boldsymbol{1}_{D_{i,1}^n(D_2,\rho_2)}\right.\right.\\
&\left. \left.\qquad\qquad\qquad\quad-h_n\int_{A}\frac{\partial_{{\beta_{m_2}}}\Psi_{\beta_0}\partial_{{\beta_{m_2'}}}\Psi_{\beta_0}}{\Psi_{\beta_0}}(y,X_{t_{i-1}^n})\varphi_n^2(X_{t_{i-1}^n},y)dy\ |\mathcal{F}_{i-1}^n\right] \right|\\
&\leq\left|\mathbb{E}\left[\frac{1}{nh_n}\sum_{i=1}^{n}g_n(\beta_0,\Delta X_i^n,X_{t_{i-1}^n})\boldsymbol{1}_{D_{i,1}^n(D_2,\rho_2)}-\frac{1}{nh_n}\sum_{i=1}^{n}g_n(\beta_0,\Delta X_{\tau_i^n},X_{t_{i-1}^n})\boldsymbol{1}_{D_{i,1}^n(D_2,\rho_2)}|\mathcal{F}_{i-1}^n\right]\right|\\
&\qquad\qquad+\left|\mathbb{E}\left[\frac{1}{nh_n}\sum_{i=1}^{n}g_n(\beta_0,\Delta X_{\tau_i^n},X_{t_{i-1}^n})\boldsymbol{1}_{D_{i,1}^n(D_2,\rho_2)}-\frac{1}{nh_n}\sum_{i=1}^{n}g_n(\beta_0,\Delta X_{\tau_i^n},X_{t_{i-1}^n})\boldsymbol{1}_{\{J_i^n=1\}}|\mathcal{F}_{i-1}^n\right]\right|\\
&\qquad\qquad+\left|\mathbb{E}\left[\frac{1}{nh_n}\sum_{i=1}^{n}g_n(\beta_0,\Delta X_{\tau_i^n},X_{t_{i-1}^n})\boldsymbol{1}_{\{J_i^n=1\}}-\frac{1}{nh_n}\sum_{i=1}^{n}\int_{t_{i-1}^n}^{t_i^n}\int_E g_n(\beta_0,c_{i-1}(z,\beta_0),X_{t_{i-1}^n})p(ds,dz)|\mathcal{F}_{i-1}^n\right]\right|\\
&\qquad\qquad+\left|\mathbb{E}\left[\frac{1}{nh_n}\sum_{i=1}^{n}\int_{t_{i-1}^n}^{t_i^n}\int_E g_n(\beta_0,c_{i-1}(z,\beta_0),X_{t_{i-1}^n})p(ds,dz)\right.\right.\\
&\qquad\qquad\qquad\qquad\left.\left.-\frac{1}{nh_n}\sum_{i=1}^{n}\int_{t_{i-1}^n}^{t_i^n}\int_E g_n(\beta_0,c_{i-1}(z,\beta_0),X_{t_{i-1}^n})q^{\beta_0}(ds,dz)|\mathcal{F}_{i-1}^n\right]\right|\\
&\qquad\qquad+\left|\mathbb{E}\left[\frac{1}{nh_n}\sum_{i=1}^{n}\int_{t_{i-1}^n}^{t_i^n}\int_E g_n(\beta_0,c_{i-1}(z,\beta_0),X_{t_{i-1}^n})q^{\beta_0}(ds,dz)-\frac{1}{n}\sum_{i=1}^n\int_{A}\frac{\partial_{{\beta_{m_2}}}\Psi_{\beta_0}\partial_{{\beta_{m_2'}}}\Psi_{\beta_0}}{\Psi_{\beta_0}}(y,X_{t_{i-1}^n})\varphi_n^2(X_{t_{i-1}^n},y)dy\ |\mathcal{F}_{i-1}^n\right]\right|\\
&\leq\frac{1}{nh_n}\sum_{i=1}^{n}\mathbb{E}\left[|g_n(\beta_0,\Delta X_i^n,X_{t_{i-1}^n})-g_n(\beta_0,\Delta X_{\tau_i^n},X_{t_{i-1}^n})|\boldsymbol{1}_{D_{i,1}^n(D_2,\rho_2)}|\mathcal{F}_{i-1}^n\right]\\
&\quad+\frac{1}{nh_n}\sum_{i=1}^{n}\mathbb{E}\left[|g_n(\beta_0,\Delta X_{\tau_i^n},X_{t_{i-1}^n})|\boldsymbol{1}_{C_{i,1}^n(D_2,\rho_2)}|\mathcal{F}_{i-1}^n\right]\\
&\quad+\frac{1}{nh_n}\sum_{i=1}^{n}\left|\mathbb{E}\left[g_n(\beta_0,\Delta X_{\tau_i^n},X_{t_{i-1}^n})\boldsymbol{1}_{\{J_i^n=1\}}-\int_{t_{i-1}^n}^{t_i^n}\int_E g_n(\beta_0,c_{i-1}(z,\beta_0),X_{t_{i-1}^n})p(ds,dz)|\mathcal{F}_{i-1}^n\right]\right|\\
&\quad+\frac{1}{nh_n}\sum_{i=1}^{n}\left|\mathbb{E}\left[\int_{t_{i-1}^n}^{t_i^n}\int_E g_n(\beta_0,c_{i-1}(z,\beta_0),X_{t_{i-1}^n})(p-q^{\beta_0})(ds,dz)|\mathcal{F}_{i-1}^n\right]\right|\\
&\quad+\frac{1}{nh_n}\sum_{i=1}^{n}\left|\mathbb{E}\left[\int_{t_{i-1}^n}^{t_i^n}\int_E g_n(\beta_0,c_{i-1}(z,\beta_0),X_{t_{i-1}^n})q^{\beta_0}(ds,dz)-h_n\int_{A}\frac{\partial_{{\beta_{m_2}}}\Psi_{\beta_0}\partial_{{\beta_{m_2'}}}\Psi_{\beta_0}}{\Psi_{\beta_0}}(y,X_{t_{i-1}^n})\varphi_n^2(X_{t_{i-1}^n},y)dy\ |\mathcal{F}_{i-1}^n\right]\right|\\
&=:\sum_{i=1}^5H_n^i.
\end{align*}
\end{footnotesize}
It is obvious from martingale property that $H_n^4=0$.  Using change of variables, we have
\begin{align*}
&\int_{t_{i-1}^n}^{t_i^n}\int_E g_n(\beta_0,c_{i-1}(z,\beta_0),X_{t_{i-1}^n})q^{\beta_0}(ds,dz)\\
&=\int_{t_{i-1}^n}^{t_i^n}\int_E\partial_{{\beta_{m_2}}}\log\Psi_{\beta_0}(c_{i-1}(z,\beta_0),X_{t_{i-1}^n})\partial_{{\beta_{m_2'}}}\log\Psi_{\beta_0}(c_{i-1}(z,\beta_0),X_{t_{i-1}^n})\varphi_n^2(X_{t_{i-1}^n},c_{i-1}(z,\beta_0))f_{\beta_0}(z)dzds\\
&=\int_{t_{i-1}^n}^{t_i^n}ds\int_A\frac{\partial_{{\beta_{m_2}}}\Psi_{\beta_0}\partial_{{\beta_{m_2'}}}\Psi_{\beta_0}}{\Psi_{\beta_0}^2}(y,X_{t_{i-1}^n})\varphi_n^2(X_{t_{i-1}^n},y)\Psi_{\beta_0}(y,X_{t_{i-1}^n})dy\\
&=h_n\int_A\frac{\partial_{{\beta_{m_2}}}\Psi_{\beta_0}\partial_{{\beta_{m_2'}}}\Psi_{\beta_0}}{\Psi_{\beta_0}}(y,X_{t_{i-1}^n})\varphi_n^2(X_{t_{i-1}^n},y)dy.
\end{align*}
This implies $H_n^5=0$. 
For $H_n^1,H_n^2,H_n^3$, in a similar way to the evaluation $I_1,I_2$ and $I_3$ in the proof of Proposition \ref{prop7:jump} with $k=2$, 
 one has 
\begin{align*}
H_n^1&=O_p\left(\sqrt{h_n\varepsilon_n^{-6}}\right),\ H_n^3=O_p\left(\sqrt{h_n\varepsilon_n^{-6}}\right), \\
H_n^2&=\begin{cases}
O_p\left(\frac{1}{h_nu_n^p}\right)+O_p\left(h_n^{\rho_2}u_n\right)&(\text{under}\ [\textbf{C}_2\textbf{1}]), \\
O_p\left(h_n^{\rho_2}\varepsilon_n^{-2}\right)&(\text{under}\ [\textbf{C}_2\textbf{2}]), 
\end{cases}
\end{align*}
where a real valued sequence $u_n$ and an integer $p\ge 2$ satisfy $h_n^{\rho_2}u_n\to 0$ and $h_nu_n^p\to\infty$, respectively.
Therefore, since $\rho_2\in B_2$ under the condition [$\textbf{C}_2${\bf 2}], one has from $\sum_{i=1}^5H_n^i\overset{P}{\to}0$ that the sum of the second and third term is bounded by
$o_p(1)$. For the fifth term, it follows from H\"{o}lder's inequality and Proposition \ref{prop2:jump} that 
\begin{align*}
&\left|\frac{1}{nh_n}\sum_{i=1}^{n}\sum_{j=0,2}\mathbb{E}\left[\partial_{{\beta_{m_2}}}\log\Psi_{\beta_0}(\Delta X_i^n,X_{t_{i-1}^n})\partial_{{\beta_{m_2'}}}\log\Psi_{\beta_0}(\Delta X_i^n,X_{t_{i-1}^n})\varphi_n^2(X_{t_{i-1}^n},\Delta X_i^n)\boldsymbol{1}_{D_{i,j}^n(D_2,\rho_2)}\ |\mathcal{F}_{i-1}^n\right]\right|\\
\le\ &\frac{1}{nh_n}\sum_{i=1}^{n}\left(\mathbb{E}\left[\left(\partial_{{\beta_{m_2}}}\log\Psi_{\beta_0}(\Delta X_i^n,X_{t_{i-1}^n})\right)^4\left(\partial_{{\beta_{m_2'}}}\log\Psi_{\beta_0}(\Delta X_i^n,X_{t_{i-1}^n})\right)^4\varphi_n^8(X_{t_{i-1}^n},\Delta X_i^n)\ |\mathcal{F}_{i-1}^n\right]\right)^{\frac{1}{4}}\\
&\qquad\quad\times\left(\sum_{j=0,2}P\left({D_{i,j}^n(D_2,\rho_2)}\ |\mathcal{F}_{i-1}^n\right)\right)^{\frac{3}{4}}\\
=\ &\frac{1}{nh_n}\sum_{i=1}^{n}R(\theta,\varepsilon_n^{-2},X_{t_{i-1}^n})\times R(\theta,h_n^{\frac{3}{2}},X_{t_{i-1}^n})\\
=\ &O_p\left(\sqrt{h_n\varepsilon_n^{-4}}\right).
\end{align*}
Therefore, the fifth term converges to 0 in probability.
Since $|\bar{X}_{i,n}(\beta_0)|\boldsymbol{1}_{\left\{|\Delta X_i^n|\le D_3h_n^{\rho_3}\right\}}=R(\theta,h_n^{\rho_3},X_{t_{i-1}^n})\boldsymbol{1}_{\left\{|\Delta X_i^n|\le D_3h_n^{\rho_3}\right\}}$, for the sixth term, 
 we see from 
Propositions \ref{prop1:jump} and \ref{prop2:jump} that
\begin{align*}
&\left|\frac{1}{nh_n}\sum_{i=1}^{n}\sum_{k_1,k_2=1}^{d}\partial_{{\beta_{m_2}}}b_{i-1}^{(k_1)} S_{i-1}^{{-1}^{(k_1,k_2)}}\right.\\
&\quad\left.\times\ \mathbb{E}\left[\bar{X}_{i,n}^{(k_2)}\partial_{{\beta_{m_2'}}}\log\Psi_{\beta_0}(\Delta X_i^n,X_{t_{i-1}^n})\varphi_n(X_{t_{i-1}^n},\Delta X_i^n)\boldsymbol{1}_{\{|\Delta X_i^n|\le D_3h_n^{\rho_3}\}}\boldsymbol{1}_{\{|\Delta X_i^n|> D_2h_n^{\rho_2}\}} \ |\mathcal{F}_{i-1}^n\right]\right|\\
&\le\frac{1}{nh_n}\sum_{i=1}^{n}\sum_{k_1,k_2=1}^{d}R(\theta,h_n^{\rho_3},X_{t_{i-1}^n})\\
&\quad\times\ \mathbb{E}\left[|\partial_{{\beta_{m_2'}}}\log\Psi_{\beta_0}(\Delta X_i^n,X_{t_{i-1}^n})|\varphi_n(X_{t_{i-1}^n},\Delta X_i^n)\boldsymbol{1}_{\{|\Delta X_i^n|\le D_3h_n^{\rho_3}\}}\boldsymbol{1}_{\{|\Delta X_i^n|> D_2h_n^{\rho_2}\}} \ |\mathcal{F}_{i-1}^n\right]\\
&\le\begin{cases}
\displaystyle
\frac{1}{nh_n}\sum_{i=1}^{n}\sum_{k_1,k_2=1}^{d}R(\theta,h_n^{\rho_3},X_{t_{i-1}^n}) \mathbb{E}\left[(1+|\Delta X_i^n|)^C\boldsymbol{1}_{\{|\Delta X_i^n|\le D_3h_n^{\rho_3}\}}\boldsymbol{1}_{\{|\Delta X_i^n|> D_2h_n^{\rho_2}\}} \ |\mathcal{F}_{i-1}^n\right]   &(\text{under}\ [\textbf{C}_2\textbf{1}])\\
\displaystyle
\frac{1}{nh_n}\sum_{i=1}^{n}\sum_{k_1,k_2=1}^{d}R(\theta,h_n^{\rho_3}\varepsilon_n^{-1},X_{t_{i-1}^n}) P\left(\{|\Delta X_i^n|\le D_3h_n^{\rho_3}\}\cap\{|\Delta X_i^n|> D_2h_n^{\rho_2}\}\ |\mathcal{F}_{i-1}^n\right)    &(\text{under}\ [\textbf{C}_2\textbf{2}])
\end{cases}\\
&\le \begin{cases}
\displaystyle
\frac{1}{nh_n}\sum_{i=1}^{n}\sum_{k_1,k_2=1}^{d}R(\theta,h_n^{\rho_3},X_{t_{i-1}^n})\\
\displaystyle
\quad\times\left\{P\left(\{|\Delta X_i^n|\le D_3h_n^{\rho_3}\}\cap\{|\Delta X_i^n|> D_2h_n^{\rho_2}\}\ |\mathcal{F}_{i-1}^n\right)+\mathbb{E}\left[|\Delta X_i^n|^C\ |\mathcal{F}_{i-1}^n\right]\right\}   &(\text{under}\ [\textbf{C}_2\textbf{1}])\\
\displaystyle
\frac{1}{nh_n}\sum_{i=1}^{n}\sum_{k_1,k_2=1}^{d}R(\theta,h_n^{\rho_3}\varepsilon_n^{-1},X_{t_{i-1}^n}) P\left(\{|\Delta X_i^n|\le D_3h_n^{\rho_3}\}\cap\{|\Delta X_i^n|> D_2h_n^{\rho_2}\}\ |\mathcal{F}_{i-1}^n\right)    &(\text{under}\ [\textbf{C}_2\textbf{2}])
\end{cases}\\
&\le\begin{cases}
\displaystyle
\frac{1}{nh_n}\sum_{i=1}^{n}\sum_{k_1,k_2=1}^{d}R(\theta,h_n^{1+\rho_3},X_{t_{i-1}^n})   &(\text{under}\ [\textbf{C}_2\textbf{1}])\\
\displaystyle
\frac{1}{nh_n}\sum_{i=1}^{n}\sum_{k_1,k_2=1}^{d}R(\theta,h_n^{1+2\rho_3}\varepsilon_n^{-1},X_{t_{i-1}^n})&(\text{under}\ [\textbf{C}_2\textbf{2}])
\end{cases}\\
&=\begin{cases}
\displaystyle
O_p\left(h_n^{\rho_3}\right)  &(\text{under}\ [\textbf{C}_2\textbf{1}]),\\
\displaystyle
O_p\left(h_n^{2\rho_3}\varepsilon_n^{-1}\right)&(\text{under}\ 
[\textbf{C}_2\textbf{2}]).
\end{cases}
\end{align*}
 Since $\rho_3\ge\frac{1}{16}$ under the condition $[\textbf{C}_2\textbf{2}]$, the sixth term converges to 0 in probability. The evaluation of the seventh term is the same.
For  the eighth term, it follows from  Propositions \ref{prop1:jump} and \ref{prop2:jump} that 
\begin{align*}
&\left|\frac{1}{n}\sum_{i=1}^{n}R(\theta,1,X_{t_{i-1}^n})\mathbb{E}\left[\partial_{{\beta_{m_2'}}}\log\Psi_{\beta_0}(\Delta X_i^n,X_{t_{i-1}^n})\varphi_n(X_{t_{i-1}^n},\Delta X_i^n)\boldsymbol{1}_{\{|\Delta X_i^n|> D_2h_n^{\rho_2}\}} |\mathcal{F}_{i-1}^n\right]\right|\\
&\le\begin{cases}
\displaystyle
\frac{1}{n}\sum_{i=1}^{n}R(\theta,1,X_{t_{i-1}^n}) \mathbb{E}\left[(1+|\Delta X_i^n|)^C\boldsymbol{1}_{\{|\Delta X_i^n|> D_2h_n^{\rho_2}\}} \ |\mathcal{F}_{i-1}^n\right]   &(\text{under}\ [\textbf{C}_2\textbf{1}])\\
\displaystyle
\frac{1}{n}\sum_{i=1}^{n}R(\theta,\varepsilon_n^{-1},X_{t_{i-1}^n}) P\left(|\Delta X_i^n|> D_2h_n^{\rho_2}\ |\mathcal{F}_{i-1}^n\right)    &(\text{under}\ [\textbf{C}_2\textbf{2}])
\end{cases}\\
&\le\begin{cases}
\displaystyle
\frac{1}{n}\sum_{i=1}^{n}R(\theta,1,X_{t_{i-1}^n})\left\{P\left(|\Delta X_i^n|> D_2h_n^{\rho_2}\ |\mathcal{F}_{i-1}^n\right)+ \mathbb{E}\left[|\Delta X_i^n|^C\ |\mathcal{F}_{i-1}^n\right]\right\}   &(\text{under}\ [\textbf{C}_2\textbf{1}])\\
\displaystyle
\frac{1}{n}\sum_{i=1}^{n}R(\theta,h_n\varepsilon_n^{-1},X_{t_{i-1}^n})     &(\text{under}\ [\textbf{C}_2\textbf{2}])
\end{cases}\\
&\le\begin{cases}
\displaystyle
O_p(h_n)&(\text{under}\ [\textbf{C}_2\textbf{1}]),\\
\displaystyle
O_p\left(h_n\varepsilon_n^{-1}\right)     &(\text{under}\ [\textbf{C}_2\textbf{2}]).
\end{cases}
\end{align*}
Hence, the eighth term converges to 0 in probability.
The evaluation of the ninth term is the same. 
Thus, we have
\begin{align*}
\sum_{i=1}^{n}{\mathbb{E}\left[\frac{1}{nh_n}\left(\eta_{i,1}^{m_2}(\beta_0|\alpha_0)+\eta_{i,2}^{m_2}(\beta_0)\right)\left(\eta_{i,1}^{m_2'}(\beta_0|\alpha_0)+\eta_{i,2}^{m_2'}(\beta_0)\right)\ |\mathcal{F}_{i-1}^n\right]} \overset{P}{\to}I^{(m_2,m_2')}_b(\beta_0)+I^{(m_2,m_2')}_c(\beta_0)=I^{(m_2,m_2')}_{b,c}(\beta_0).
\end{align*}
This implies \eqref{MCLT5-jump}.

{\bf Proof of \eqref{MCLT6-jump}.}
From the proof of \eqref{MCLT2-jump},  one has 
\begin{align*}
&\left|\sum_{i=1}^{n}{\mathbb{E}\left[\frac{1}{\sqrt{nh_n}}\left(\eta_{i,1}^{m_2}+\eta_{i,2}^{m_2}\right)\ |\mathcal{F}_{i-1}^n\right]\mathbb{E}\left[\frac{1}{\sqrt{nh_n}}\left(\eta_{i,1}^{m_2'}+\eta_{i,2}^{m_2'}\right)\ |\mathcal{F}_{i-1}^n\right]}\right|\\
&=\sum_{i=1}^{n}\left\{R\left(\theta,\frac{h_n\varepsilon_n^{-2}}{\sqrt{n}},X_{t_{i-1}^n}\right)+R\left(\theta,\sqrt{\frac{h_n^{1+4\rho_1}}{n}},X_{t_{i-1}^n}\right)\right\}^2+o_P(1)\\
&=\frac{1}{n}\sum_{i=1}^{n}\left\{R\left(\theta,h_n^2\varepsilon_n^{-4},X_{t_{i-1}^n}\right)+R\left(\theta,{h_n^{1+4\rho_1}},X_{t_{i-1}^n}\right)\right\}+o_P(1)\\
&\overset{P}{\to}0.
\end{align*}

{\bf Proof of \eqref{MCLT7-jump}.}
 Since $\Delta X_i^n=\bar{X}_{i,n}(\beta_0)+h_nb_{i-1}(\beta_0)$ 
and $\boldsymbol{1}_{\{|\Delta X_i^n|\le D_3h_n^{\rho_3}\}}\le 1$, 
it follows from \eqref{under:rep3}, \eqref{under:rep2} and \eqref{under:rep1} that we can calculate
\begin{align*}
&\left|\sum_{i=1}^{n}{\mathbb{E}\left[\frac{1}{n\sqrt{h_n}}\xi_i^{m_1}(\alpha_0)\left(\eta_{i,1}^{m_2}(\beta_0|\alpha_0)+\eta_{i,2}^{m_2}(\beta_0)\right)\ |\mathcal{F}_{i-1}^n\right]}\right|\\
&\le\frac{1}{2nh_n\sqrt{h_n}}\sum_{i=1}^{n}\left|\mathbb{E}\left[(\Delta X_i^n)^\top\left(\partial_{{\alpha_{m_1}}}S_{i-1}^{-1}\right)\Delta X_i^n(\partial_{{\beta_{m_2}}}b_{i-1})^\top S_{i-1}^{-1}\bar{X}_{i,n}\boldsymbol{1}_{\{|\Delta X_i^n|\le D_1h_n^{\rho_1}\}}\boldsymbol{1}_{\{|\Delta X_i^n|\le D_3h_n^{\rho_3}\}}\ |\mathcal{F}_{i-1}^n\right]\right|\\
&\quad+\frac{1}{2n\sqrt{h_n}}\sum_{i=1}^{n}\left|\partial_{{\alpha}_{m_1}}\log\det S_{i-1}\right|\left|\mathbb{E}\left[(\partial_{{\beta_{m_2}}}b_{i-1})^\top S_{i-1}^{-1}\bar{X}_{i,n}\boldsymbol{1}_{\{|\Delta X_i^n|\le D_1h_n^{\rho_1}\}}\boldsymbol{1}_{\{|\Delta X_i^n|\le D_3h_n^{\rho_3}\}}\ |\mathcal{F}_{i-1}^n\right]\right|\\
&\quad+\frac{1}{2nh_n\sqrt{h_n}}\sum_{i=1}^{n}\left|\mathbb{E}\left[(\Delta X_i^n)^\top\left(\partial_{{\alpha_{m_1}}}S_{i-1}^{-1}\right)\Delta X_i^n\left(\partial_{{\beta_{m_2}}}\log\Psi_{\beta_0}(\Delta X_i^n,X_{t_{i-1}^n})\right)\varphi_n(X_{t_{i-1}^n},\Delta X_i^n)\right.\right.\\
&\qquad\qquad\qquad\qquad\qquad\left.\left.\times\ \boldsymbol{1}_{\{|\Delta X_i^n|\le D_1h_n^{\rho_1}\}}\boldsymbol{1}_{\{|\Delta X_i^n|> D_2h_n^{\rho_2}\}}\ |\mathcal{F}_{i-1}^n\right]\right|\\
&\quad+\frac{1}{2n\sqrt{h_n}}\sum_{i=1}^{n}\left|\partial_{{\alpha}_{m_1}}\log\det S_{i-1}\right|\mathbb{E}\left[\left|\left(\partial_{{\beta_{m_2}}}\log\Psi_{\beta_0}(\Delta X_i^n,X_{t_{i-1}^n})\right)\varphi_n(X_{t_{i-1}^n},\Delta X_i^n)\right|\right.\\
&\qquad\qquad\qquad\qquad\qquad\left.\times\ \boldsymbol{1}_{\{|\Delta X_i^n|\le D_1h_n^{\rho_1}\}}\boldsymbol{1}_{\{|\Delta X_i^n|> D_2h_n^{\rho_2}\}}\ |\mathcal{F}_{i-1}^n\right]\\
&\quad+\frac{1}{2n\sqrt{h_n}}\sum_{i=1}^{n}\left|\int_{B}{\partial_{{\beta_{m_2}}}\Psi_{\beta_0}(y,X_{t_{i-1}^n})\varphi_n(X_{t_{i-1}^n},y)}dy\right|\\
&\qquad\qquad\qquad\qquad\times \left|\mathbb{E}\left[(\Delta X_i^n)^\top\partial_{{\alpha_{m_1}}}S_{i-1}^{-1}(\alpha_0)\Delta X_i^n\boldsymbol{1}_{\{|\Delta X_i^n|\le D_1h_n^{\rho_1}\}}\ |\mathcal{F}_{i-1}^n\right]\right|\\
&\quad+\frac{\sqrt{h_n}}{2n}\sum_{i=1}^{n}\left|\partial_{{\alpha}_{m_1}}\log\det S_{i-1}\right|\left|\int_{B}{\partial_{{\beta_{m_2}}}\Psi_{\beta_0}(y,X_{t_{i-1}^n})\varphi_n(X_{t_{i-1}^n},y)}dy\right|P(|\Delta X_i^n|\le D_1h_n^{\rho_1}\ |\mathcal{F}_{i-1}^n)\\
&\le \frac{1}{2nh_n\sqrt{h_n}}\sum_{i=1}^{n}\sum_{k_1,k_2=1}^{d}\sum_{k_3,k_4=1}^{d} \left|{\partial_{\alpha_{m_1}}S_{i-1}^{-1}}^{(k_1,k_2)}\right|\left|\partial_{\beta_{m_2}}b_{i-1}^{(k_3)}\right|\left|S_{i-1}^{(k_3,k_4)}\right|\left|\mathbb{E}\left[\bar{X}_{i,n}^{(k_1)}\bar{X}_{i,n}^{(k_2)}\bar{X}_{i,n}^{(k_4)}\boldsymbol{1}_{\{|\Delta X_i^n|\le D_1h_n^{\rho_1}\}}\ |\mathcal{F}_{i-1}^n\right]\right|\\
&\quad+ \frac{1}{n\sqrt{h_n}}\sum_{i=1}^{n}\sum_{k_1,k_2=1}^{d}\sum_{k_3,k_4=1}^{d} \left|b_{i-1}^{(k_1)}\right|\left|{\partial_{\alpha_{m_1}}S_{i-1}^{-1}}^{(k_1,k_2)}\right|\left|\partial_{\beta_{m_2}}b_{i-1}^{(k_3)}\right|\left|S_{i-1}^{(k_3,k_4)}\right|\left|\mathbb{E}\left[\bar{X}_{i,n}^{(k_2)}\bar{X}_{i,n}^{(k_4)}\boldsymbol{1}_{\{|\Delta X_i^n|\le D_1h_n^{\rho_1}\}}\ |\mathcal{F}_{i-1}^n\right]\right|\\
&\quad+ \frac{\sqrt{h_n}}{2n}\sum_{i=1}^{n}\sum_{k_1,k_2=1}^{d}\sum_{k_3,k_4=1}^{d} \left|b_{i-1}^{(k_1)}\right|\left|{\partial_{\alpha_{m_1}}S_{i-1}^{-1}}^{(k_1,k_2)}\right|\left|b_{i-1}^{(k_2)}\right|\left|\partial_{\beta_{m_2}}b_{i-1}^{(k_3)}\right|\left|S_{i-1}^{(k_3,k_4)}\right|\left|\mathbb{E}\left[\bar{X}_{i,n}^{(k_4)}\boldsymbol{1}_{\{|\Delta X_i^n|\le D_1h_n^{\rho_1}\}}\ |\mathcal{F}_{i-1}^n\right]\right|\\
&\quad+ \frac{1}{2n\sqrt{h_n}}\sum_{i=1}^{n}\left|\partial_{\alpha_{m_1}}\log\det S_{i-1}\right|\sum_{k_1,k_2=1}^{d}\left|\partial_{\beta_{m_2}}b_{i-1}^{(k_1)}\right|\left|S_{i-1}^{(k_1,k_2)}\right|\left|\mathbb{E}\left[\bar{X}_{i,n}^{(k_2)}\boldsymbol{1}_{\{|\Delta X_i^n|\le D_1h_n^{\rho_1}\}}\ |\mathcal{F}_{i-1}^n\right]\right|\\
&\quad+ \frac{1}{2nh_n\sqrt{h_n}}\sum_{i=1}^{n}\sum_{k_1,k_2=1}^{d}\left|{\partial_{\alpha_{m_1}}S_{i-1}^{-1}}^{(k_1,k_2)}\right|  \\
&\qquad\quad\times\mathbb{E}\left[\left|{\Delta X_{i}^n}^{k_1}\right|\left|{\Delta X_{i}^n}^{k_2}\right|\left|\partial_{{\beta_{m_2}}}\log\Psi_{\beta_0}(\Delta X_i^n,X_{t_{i-1}^n})\right|\varphi_n(X_{t_{i-1}^n},\Delta X_i^n)\boldsymbol{1}_{\{|\Delta X_i^n|\le D_1h_n^{\rho_1}\}}\boldsymbol{1}_{\{|\Delta X_i^n|> D_2h_n^{\rho_2}\}}\ |\mathcal{F}_{i-1}^n\right]\\
&\quad+ \frac{1}{2n\sqrt{h_n}}\sum_{i=1}^{n}\left|\partial_{\alpha_{m_1}}\log\det S_{i-1}\right|\\
&\qquad\quad\times\mathbb{E}\left[\left|\partial_{{\beta_{m_2}}}\log\Psi_{\beta_0}(\Delta X_i^n,X_{t_{i-1}^n})\right|\varphi_n(X_{t_{i-1}^n},\Delta X_i^n)\boldsymbol{1}_{\{|\Delta X_i^n|\le D_1h_n^{\rho_1}\}}\boldsymbol{1}_{\{|\Delta X_i^n|> D_2h_n^{\rho_2}\}}\ |\mathcal{F}_{i-1}^n\right]\\
&\quad+ \frac{1}{2n\sqrt{h_n}}\sum_{i=1}^{n}\left|\int_{B}{\partial_{{\beta_{m_2}}}\Psi_{\beta_0}(y,X_{t_{i-1}^n})\varphi_n(X_{t_{i-1}^n},y)}dy\right|\\
&\qquad\quad\times\sum_{k_1,k_2=1}^{d} \left|{\partial_{\alpha_{m_1}}S_{i-1}^{-1}}^{(k_1,k_2)}\right|\left|\mathbb{E}\left[\bar{X}_{i,n}^{(k_1)}\bar{X}_{i,n}^{(k_2)}\boldsymbol{1}_{\{|\Delta X_i^n|\le D_1h_n^{\rho_1}\}}\ |\mathcal{F}_{i-1}^n\right]\right|\\
&\quad+ \frac{\sqrt{h_n}}{n}\sum_{i=1}^{n}\left|\int_{B}{\partial_{{\beta_{m_2}}}\Psi_{\beta_0}(y,X_{t_{i-1}^n})\varphi_n(X_{t_{i-1}^n},y)}dy\right|\\
&\qquad\quad\times\sum_{k_1,k_2=1}^{d} \left|b_{i-1}^{(k_1)}\right|\left|{\partial_{\alpha_{m_1}}S_{i-1}^{-1}}^{(k_1,k_2)}\right|\left|\mathbb{E}\left[\bar{X}_{i,n}^{(k_2)}\boldsymbol{1}_{\{|\Delta X_i^n|\le D_1h_n^{\rho_1}\}}\ |\mathcal{F}_{i-1}^n\right]\right|\\
&\quad+ \frac{h_n\sqrt{h_n}}{2n}\sum_{i=1}^{n}\left|\int_{B}{\partial_{{\beta_{m_2}}}\Psi_{\beta_0}(y,X_{t_{i-1}^n})\varphi_n(X_{t_{i-1}^n},y)}dy\right|\\
&\qquad\quad\times\sum_{k_1,k_2=1}^{d} \left|b_{i-1}^{(k_1)}\right|\left|{\partial_{\alpha_{m_1}}S_{i-1}^{-1}}^{(k_1,k_2)}\right|\left|b_{i-1}^{(k_2)}\right|P\left(|\Delta X_i^n|\le D_1h_n^{\rho_1}\ |\mathcal{F}_{i-1}^n\right)\\
&\quad+ \frac{\sqrt{h_n}}{2n}\sum_{i=1}^{n}\left|\partial_{\alpha_{m_1}}\log\det S_{i-1}\right|\left|\int_{B}{\partial_{{\beta_{m_2}}}\Psi_{\beta_0}(y,X_{t_{i-1}^n})\varphi_n(X_{t_{i-1}^n},y)}dy\right|P\left(|\Delta X_i^n|\le D_1h_n^{\rho_1}\ |\mathcal{F}_{i-1}^n\right)\\
&\le \frac{1}{n}\sum_{i=1}^n R(\theta,\sqrt{h_n},X_{t_{i-1}^n})+ \frac{1}{n}\sum_{i=1}^n R(\theta,h_n^{\frac{3}{2}+2\rho_1},X_{t_{i-1}^n})+ \frac{1}{n}\sum_{i=1}^n R(\theta,h_n^{\frac{3}{2}},X_{t_{i-1}^n})+ \frac{1}{n}\sum_{i=1}^n R(\theta,h_n^{\frac{1}{2}+2\rho_1},X_{t_{i-1}^n})\\
&\quad+  \frac{1}{n}\sum_{i=1}^n R(\theta,h_n^{2\rho_1-\frac{3}{2}},X_{t_{i-1}^n})\mathbb{E}\left[\left|\partial_{{\beta_{m_2}}}\log\Psi_{\beta_0}(\Delta X_i^n,X_{t_{i-1}^n})\right|\varphi_n(X_{t_{i-1}^n},\Delta X_i^n)\boldsymbol{1}_{\{|\Delta X_i^n|\le D_1h_n^{\rho_1}\}}\boldsymbol{1}_{\{|\Delta X_i^n|> D_2h_n^{\rho_2}\}}\ |\mathcal{F}_{i-1}^n\right]\\
&\quad+  \frac{1}{n}\sum_{i=1}^n R(\theta,h_n^{-\frac{1}{2}},X_{t_{i-1}^n})\mathbb{E}\left[\left|\partial_{{\beta_{m_2}}}\log\Psi_{\beta_0}(\Delta X_i^n,X_{t_{i-1}^n})\right|\varphi_n(X_{t_{i-1}^n},\Delta X_i^n)\boldsymbol{1}_{\{|\Delta X_i^n|\le D_1h_n^{\rho_1}\}}\boldsymbol{1}_{\{|\Delta X_i^n|> D_2h_n^{\rho_2}\}}\ |\mathcal{F}_{i-1}^n\right].
\end{align*}
It is obvious that the first through fourth terms on the right-hand side converge to 0 in probability. 
We evaluate the fifth and sixth  terms. If we take $p=11+11\rho_1$ and $q=1+\frac{1}{10+11\rho_1}$, 
it follows from  Propositions \ref{prop1:jump} and \ref{prop2:jump} and H\"{o}lder's inequality that
\begin{align*}
&\mathbb{E}\left[\left|\partial_{{\beta_{m_2}}}\log\Psi_{\beta_0}(\Delta X_i^n,X_{t_{i-1}^n})\right|\varphi_n(X_{t_{i-1}^n},\Delta X_i^n)\boldsymbol{1}_{\{|\Delta X_i^n|\le D_1h_n^{\rho_1}\}}\boldsymbol{1}_{\{|\Delta X_i^n|> D_2h_n^{\rho_2}\}}\ |\mathcal{F}_{i-1}^n\right]\\
&\le\begin{cases}
\displaystyle
C(1+|X_{t_{i-1}^n}|)^C \left(\mathbb{E}\left[(1+|\Delta X_i^n|)^C\right]\right)^{\frac{1}{p}}P\left(\{|\Delta X_i^n|\le D_1h_n^{\rho_1}\}\cap\{|\Delta X_i^n|> D_2h_n^{\rho_2}\}\ |\mathcal{F}_{i-1}^n\right)^\frac{1}{q}   &(\text{under}\ [\textbf{C}_2\textbf{1}])\\
\displaystyle
C\varepsilon_n^{-1}(1+|X_{t_{i-1}^n}|)^C P\left(\{|\Delta X_i^n|\le D_1h_n^{\rho_1}\}\cap\{|\Delta X_i^n|> D_2h_n^{\rho_2}\}\ |\mathcal{F}_{i-1}^n\right)    &(\text{under}\ [\textbf{C}_2\textbf{2}])
\end{cases}\\
&\le\begin{cases}
Ch_n^{1+\rho_1-\frac{1}{11}}(1+|X_{t_{i-1}^n}|)^C   &(\text{under}\ [\textbf{C}_2\textbf{1}]),\\
\displaystyle
Ch_n^{1+\rho_1}\varepsilon_n^{-1}(1+|X_{t_{i-1}^n}|)^C     &(\text{under}\ [\textbf{C}_2\textbf{2}]).
\end{cases}
\end{align*}
Therefore, 
\begin{align*}
&\mathbb{E}\left[\left|\partial_{{\beta_{m_2}}}\log\Psi_{\beta_0}(\Delta X_i^n,X_{t_{i-1}^n})\right|\varphi_n(X_{t_{i-1}^n},\Delta X_i^n)\boldsymbol{1}_{\{|\Delta X_i^n|\le D_1h_n^{\rho_1}\}}\boldsymbol{1}_{\{|\Delta X_i^n|> D_2h_n^{\rho_2}\}}\ |\mathcal{F}_{i-1}^n\right]\\
&=\begin{cases}
R(\theta,h_n^{1+\rho_1-\frac{1}{11}},X_{t_{i-1}^n})  &(\text{under}\ [\textbf{C}_2\textbf{1}]),\\
\displaystyle
R(\theta,h_n^{1+\rho_1}\varepsilon_n^{-1},X_{t_{i-1}^n})     &(\text{under}\ [\textbf{C}_2\textbf{2}]).
\end{cases}
\end{align*}
Hence, for the fifth term, we have
\begin{align*}
&\frac{1}{n}\sum_{i=1}^n R(\theta,h_n^{2\rho_1-\frac{3}{2}},X_{t_{i-1}^n})\mathbb{E}\left[\left|\partial_{{\beta_{m_2}}}\log\Psi_{\beta_0}(\Delta X_i^n,X_{t_{i-1}^n})\right|\varphi_n(X_{t_{i-1}^n},\Delta X_i^n)\boldsymbol{1}_{\{|\Delta X_i^n|\le D_1h_n^{\rho_1}\}}\boldsymbol{1}_{\{|\Delta X_i^n|> D_2h_n^{\rho_2}\}}\ |\mathcal{F}_{i-1}^n\right]\\
&=\begin{cases}
\displaystyle
\frac{1}{n}\sum_{i=1}^nR(\theta,h_n^{3\rho_1-\frac{1}{2}-\frac{1}{11}},X_{t_{i-1}^n})  &(\text{under}\ [\textbf{C}_2\textbf{1}]),\\
\displaystyle
\frac{1}{n}\sum_{i=1}^nR(\theta,h_n^{3\rho_1-\frac{1}{2}}\varepsilon_n^{-1},X_{t_{i-1}^n})     &(\text{under}\ [\textbf{C}_2\textbf{2}]).
\end{cases}
\end{align*}
Since $\rho_1>\frac{1}{5}$ under the condition $[\textbf{C}_2\textbf{1}]$, 
 one has $h_n^{3\rho_1-\frac{1}{2}-\frac{1}{11}}<h_n^{\frac{3}{5}-\frac{1}{2}-\frac{1}{11}}=h_n^{\frac{1}{10}-\frac{1}{11}}\to 0$. Moreover, since $h_n^{3\rho_1-\frac{1}{2}}\varepsilon_n^{-1}\to 0$ under the condition $[\textbf{C}_2\textbf{2}]$, the fifth term converges to 0 in probability. In a similar way, the sixth term converges to 0 in probability. This implies \eqref{MCLT7-jump}.

{\bf Proof of \eqref{MCLT8-jump}.}
From the proof of \eqref{MCLT1-jump}, \eqref{MCLT2-jump}, it is easy to show that 
\begin{align*}
&\left|\sum_{i=1}^{n}{\mathbb{E}\left[\frac{1}{\sqrt{n}}\xi_i^{m_1}(\alpha_0)\ |\mathcal{F}_{i-1}^n\right]\mathbb{E}\left[\frac{1}{\sqrt{nh_n}}\left(\eta_{i,1}^{m_2}(\beta_0|\alpha_0)+\eta_{i,2}^{m_2}(\beta_0)\right)\ |\mathcal{F}_{i-1}^n\right]}\right|\\
&=\sum_{i=1}^{n}\left\{R\left(\theta,\frac{h_n}{\sqrt{n}},X_{t_{i-1}^n}\right)+R\left(\theta,\frac{h_n^{3\rho_1}}{\sqrt{n}},X_{t_{i-1}^n}\right)\right\}\\
&\qquad\qquad\times\left\{R\left(\theta,\frac{h_n\varepsilon_n^{-2}}{\sqrt{n}},X_{t_{i-1}^n}\right)+R\left(\theta,\sqrt{\frac{h_n^{1+4\rho_1}}{n}},X_{t_{i-1}^n}\right)\right\}+o_P(1)\\
&\overset{P}{\to}0.
\end{align*}

{\bf Proof of \eqref{MCLT9-jump}.}
Note that  since $\frac{1}{5}<\rho_1<\frac{1}{2}$, we obtain $\frac{5\rho_1-1}{1-2\rho_1}>0$. We show the proof for $\frac{5\rho_1-1}{1-2\rho_1}\ge\nu_1>0$.
First, let us remark that following moment estimate holds: For $p\ge1$, $t_{i-1}^n\le t\le t_i^n$,
\begin{align}\label{moment-eval:jump}
\mathbb{E}\left[|X_t-X_{t_{i-1}}^n|^p\boldsymbol{1}_{C_{i,0}^n(D,\rho)}\ |\mathcal{F}_{i-1}^n\right]= R\left(\theta,h_n^{\frac{p}{2}},X_{t_{i-1}^n}\right).
\end{align}
We can prove this estimate by the  same argument as Lemma 6 of \citet{Kessler}.
Since $C_{i,j}^n(D_1,\rho_1)\subset \{|\Delta X_i^n|\le D_1h_n^{\rho_1}\},\ (j=1,2)$, it holds from \eqref{moment-eval:jump} and Proposition \ref{prop2:jump} that 
\begin{align*}
&\mathbb{E}\left[\left|{\Delta X_i^n}\right|^{2(2+\nu_1)}\boldsymbol{1}_{\{|\Delta X_i^n|\le D_1h_n^{\rho_1}\}}\ |\mathcal{F}_{i-1}^n\right]\nonumber\\
&\quad \le \mathbb{E}\left[\left|{\Delta X_i^n}\right|^{2(2+\nu_1)}\boldsymbol{1}_{C_{i,0}^n(D_1,\rho_1)}\ |\mathcal{F}_{i-1}^n\right]+\sum_{j=1,2}\mathbb{E}\left[\left|{\Delta X_i^n}\right|^{2(2+\nu_1)}\boldsymbol{1}_{C_{i,j}^n(D_1,\rho_1)}\ |\mathcal{F}_{i-1}^n\right]\nonumber\\
&\quad \le R(\theta,h_n^{2+\nu_1},X_{t_{i-1}^n})+Ch_n^{2(2+\nu_1)\rho_1}\sum_{j=1,2}P(C_{i,j}^n(D_1,\rho_1)\|\mathcal{F}_{i-1}^n)\nonumber\\
&\quad = R(\theta,h_n^{2+\nu_1},X_{t_{i-1}^n})+ R(\theta,h_n^{1+\rho_1+2(2+\nu_1)\rho_1},X_{t_{i-1}^n}).
\end{align*}
Therefore,
\begin{align*}
&\sum_{i=1}^{n}{\mathbb{E}\left[\left|\frac{1}{\sqrt{n}}\xi_i^{m_1}(\alpha_0)\right|^{2+\nu_1}|\mathcal{F}_{i-1}^n\right]}\\
&\le\frac{C_{\nu_1}}{n^{1+\frac{\nu_1}{2}}h_n^{2+\nu_1}}\sum_{i=1}^{n}\mathbb{E}\left[\left|(\Delta X_i^n)^\top\left(\partial_{{\alpha_{m_1}}}S_{i-1}^{-1}\right)\Delta X_i^n\right|^{2+\nu_1}\boldsymbol{1}_{\{|\Delta X_i^n|\le D_1h_n^{\rho_1}\}}\ |\mathcal{F}_{i-1}^n\right]\\
&\quad+\frac{C_{\nu_1}}{n^{1+\frac{\nu_1}{2}}}\sum_{i=1}^{n}\left|\partial_{{\alpha}_{m_1}}\log\det S_{i-1}\right|^{2+\nu_1}P(|\Delta X_i^n|\le D_1h_n^{\rho_1}\ |\mathcal{F}_{i-1}^n)\\
&\le\frac{C_{\nu_1}}{n^{1+\frac{\nu_1}{2}}h_n^{2+\nu_1}}\sum_{i=1}^{n}\sum_{k_1,k_2=1}^{d}R(\theta,1,X_{t_{i-1}^n})\mathbb{E}\left[\left|{\Delta X_i^n}^{(k_1)}\right|^{2+\nu_1}\left|{\Delta X_i^n}^{(k_2)}\right|^{2+\nu_1}\boldsymbol{1}_{\{|\Delta X_i^n|\le D_1h_n^{\rho_1}\}}\ |\mathcal{F}_{i-1}^n\right]\\
&\quad+\frac{C_{\nu_1}}{n^{1+\frac{\nu_1}{2}}}\sum_{i=1}^{n}\sum_{k_1=1}^{d}R(\theta,1,X_{t_{i-1}^n})\\
&\le\frac{C_{\nu_1}}{n^{1+\frac{\nu_1}{2}}h_n^{2+\nu_1}}\sum_{i=1}^{n}\sum_{k_1,k_2=1}^{d}R(\theta,1,X_{t_{i-1}^n})\mathbb{E}\left[\left|{\Delta X_i^n}\right|^{2(2+\nu_1)}\boldsymbol{1}_{\{|\Delta X_i^n|\le D_1h_n^{\rho_1}\}}\ |\mathcal{F}_{i-1}^n\right]+\frac{C_{\nu_1}}{n^{1+\frac{\nu_1}{2}}}\sum_{i=1}^{n}R(\theta,1,X_{t_{i-1}^n})\\
&\le\frac{C_{\nu_1}}{n^{1+\frac{\nu_1}{2}}h_n^{2+\nu_1}}\sum_{i=1}^{n}\left\{R(\theta,h_n^{2+\nu_1},X_{t_{i-1}^n})+ R(\theta,h_n^{1+\rho_1+2(2+\nu_1)\rho_1},X_{t_{i-1}^n})\right\}+O_P\left({n^{-\frac{\nu_1}{2}}}\right)\\
&=O_P\left(n^{-\frac{\nu_1}{2}}{h_n^\mu}\right)+O_P\left({n^{-\frac{\nu_1}{2}}}\right),
\end{align*}
 where $\mu:=(5+2\nu_1)\rho_1-1-\nu_1=(5+2\nu_1)\left(\rho_1-\frac{1+\nu_1}{5+2\nu_1}\right)$. Hence, if we take $nu_1$ which satisfies $\frac{5\rho_1-1}{1-2\rho_1}\ge\nu_1>0$, then $\mu\ge 0$, and \eqref{MCLT9-jump} holds.

{\bf Proof of \eqref{MCLT10-jump}.}
We show the proof for $\nu_2> \left(\frac{4-2\delta}{\delta}\vee 2\right)>0$. it follows from
 Propositions \ref{prop1:jump} and \ref{prop2:jump} that
\begin{align*}
&\mathbb{E}\left[\left|\partial_{{\beta_{m_2'}}}\log\Psi_{\beta_0}(\Delta X_i^n,X_{t_{i-1}^n})\right|^{2+\nu_2}\varphi_n(X_{t_{i-1}^n},\Delta X_i^n)^{2+\nu_2}\boldsymbol{1}_{\{|\Delta X_i^n|> D_2h_n^{\rho_2}\}} |\mathcal{F}_{i-1}^n\right]\\
&\le\begin{cases}
\displaystyle
C(1+|X_{t_{i-1}^n}|)^C \mathbb{E}\left[(1+|\Delta X_i^n|)^C\boldsymbol{1}_{\{|\Delta X_i^n|> D_2h_n^{\rho_2}\}} \ |\mathcal{F}_{i-1}^n\right]   &(\text{under}\ [\textbf{C}_2\textbf{1}])\\
\displaystyle
C\varepsilon_n^{-(2+\nu_2)}(1+|X_{t_{i-1}^n}|)^C P\left(|\Delta X_i^n|> D_2h_n^{\rho_2}\ |\mathcal{F}_{i-1}^n\right)    &(\text{under}\ [\textbf{C}_2\textbf{2}])
\end{cases}\\
&\le\begin{cases}
\displaystyle
C(1+|X_{t_{i-1}^n}|)^C \left\{\mathbb{E}\left[|\Delta X_i^n|^C \ |\mathcal{F}_{i-1}^n\right]+P\left(|\Delta X_i^n|> D_2h_n^{\rho_2}\ |\mathcal{F}_{i-1}^n\right)\right\}   &(\text{under}\ [\textbf{C}_2\textbf{1}])\\
\displaystyle
Ch_n\varepsilon_n^{-(2+\nu_2)}(1+|X_{t_{i-1}^n}|)^C     &(\text{under}\ [\textbf{C}_2\textbf{2}])
\end{cases}\\
&\le\begin{cases}
\displaystyle
Ch_n(1+|X_{t_{i-1}^n}|)^C   &(\text{under}\ [\textbf{C}_2\textbf{1}])\\
\displaystyle
Ch_n\varepsilon_n^{-(2+\nu_2)}(1+|X_{t_{i-1}^n}|)^C     &(\text{under}\ [\textbf{C}_2\textbf{2}])
\end{cases}\\
&\le Ch_n\varepsilon_n^{-(2+\nu_2)}(1+|X_{t_{i-1}^n}|)^C.
\end{align*}
Therefore, we have 
\begin{align}\label{eq:lindberg-eval2-1}
&\mathbb{E}\left[\left|\partial_{{\beta_{m_2'}}}\log\Psi_{\beta_0}(\Delta X_i^n,X_{t_{i-1}^n})\right|^{2+\nu_2}\varphi_n(X_{t_{i-1}^n},\Delta X_i^n)^{2+\nu_2}\boldsymbol{1}_{\{|\Delta X_i^n|> D_2h_n^{\rho_2}\}} |\mathcal{F}_{i-1}^n\right]\nonumber\\
&\quad = R(\theta,h_n\varepsilon_n^{-(2+\nu_2)},X_{t_{i-1}^n}).
\end{align} 
Moreover,  one has
\begin{align}
\mathbb{E}\left[|\bar{X}_{i,n}(\beta_0)|^{2+\nu_2}\boldsymbol{1}_{\{|\Delta X_i^n|\le D_3h_n^{\rho_3}\}}\ |\mathcal{F}_{i-1}^n\right]&\le C\mathbb{E}\left[\left(|\Delta X_i^n|^{2+\nu_2}+h_n^{2+\nu_2}|b_{i-1}(\beta_0)|^{2+\nu_2}\right)\boldsymbol{1}_{\{|\Delta X_i^n|\le D_3h_n^{\rho_3}\}}\ |\mathcal{F}_{i-1}^n\right]\nonumber\\
&\le R(\theta,h_n^{(2+\nu_2)\rho_3},X_{t_{i-1}^n})P(|\Delta X_i^n|\le D_3h_n^{\rho_3}|\mathcal{F}_{i-1}^n)\nonumber\\
&\le R(\theta,h_n^{(2+\nu_2)\rho_3},X_{t_{i-1}^n})\label{eq:lindberg-eval2-2}
\end{align}
By using \eqref{eq:lindberg-eval2-1}, \eqref{eq:lindberg-eval2-2}, we obtain 
\begin{align*}
&\sum_{i=1}^{n}{\mathbb{E}\left[\left|\frac{1}{\sqrt{nh_n}}\left(\eta_{i,1}^{m_2}(\beta_0|\alpha_0)+\eta_{i,2}^{m_2}(\beta_0)\right)\right|^{2+\nu_2}|\mathcal{F}_{i-1}^n\right]}\\
&\le \frac{1}{{(nh_n)}^{\frac{\nu_2}{2}}}\left\{\frac{1}{nh_n}\sum_{i=1}^n\sum_{k_1,k_2=1}^d\left|\partial_{\beta_{m_2}}b_{i-1}^{(k_1)}(\beta_0)\right|^{2+\nu_2}|{S_{i-1}^{-1}(\alpha_0)}^{(k_1,k_2)}|^{2+\nu_2}\mathbb{E}\left[|\bar{X}_{i,n}^{(k_2)}(\beta_0)|^{2+\nu_2}\boldsymbol{1}_{\{|\Delta X_i^n|\le D_3h_n^{\rho_3}\}}\ |\mathcal{F}_{i-1}^n\right|\right.\\
&\qquad\qquad + \frac{1}{nh_n}\sum_{i=1}^n \mathbb{E}\left[\left|\partial_{{\beta_{m_2'}}}\log\Psi_{\beta_0}(\Delta X_i^n,X_{t_{i-1}^n})\right|^{2+\nu_2}\varphi_n(X_{t_{i-1}^n},\Delta X_i^n)^{2+\nu_2}\boldsymbol{1}_{\{|\Delta X_i^n|> D_2h_n^{\rho_2}\}} |\mathcal{F}_{i-1}^n\right]\\
&\qquad\qquad \left.+ h_n^{1+\nu_2}\frac{1}{n}\sum_{i=1}^n\left|\int_B\partial_{\beta_{m_2}}\Psi_{\beta_0}(y,X_{t_{i-1}}^n)\varphi_n(X_{t_{i-1}}^n,y)dy\right|\right\}\\
&\le \frac{1}{{(nh_n)}^{\frac{\nu_2}{2}}}\left\{\frac{1}{nh_n}\sum_{i=1}^nR(\theta,h_n^{(2+\nu_2)\rho_3},X_{t_{i-1}^n})+\frac{1}{n}\sum_{i=1}^nR(\theta,\varepsilon_n^{-(2+\nu_2)},X_{t_{i-1}^n})+\frac{1}{n}\sum_{i=1}^nR(\theta,h_n^{1+\nu_2},X_{t_{i-1}^n})\right\}\\
&\le O_p\left((nh_n)^{-\frac{\nu_2}{2}}h_n^{(2+\nu_2)\rho_3-1}\right)+O_p\left((nh_n)^{-\frac{\nu_2}{2}}\varepsilon_n^{-(2+\nu_2)}\right)+O_p\left((nh_n)^{-\frac{\nu_2}{2}}h_n^{1+\nu_2}\right).
\end{align*}
It is obvious from $\nu_2> \left(\frac{4-2\delta}{\delta}\vee 2\right)$ that the third term on the right-hand side converges to 0 in probability. Since $\rho_3\ge\frac{\delta}{4}$, we have $(2+\nu_2)\rho_3-1>\left(2+\frac{4-2\delta}{\delta}\right)\frac{\delta}{4}-1=0$. Hence, the first term converges to 0 in probability. As $nh_n\varepsilon_n^4\to\infty$ under the condition [\textbf{B}\textbf{2}], it follows that 
\begin{align*}
(nh_n)^{-\frac{\nu_2}{2}}\varepsilon_n^{-(2+\nu_2)}&=\left(\frac{1}{nh_n\varepsilon_n^4}\right)^\frac{\nu_2}{2}\cdot \varepsilon_n^{\nu_2-2}.
\end{align*}
Therefore,  we see from $\nu_2> 2$ that the second term converges to 0 in probability. This implies \eqref{MCLT10-jump}, and  completes the proof of \eqref{thm2:goal4-jump}.

Next, let $J_n=\{\omega\in\Omega\ |\ C_n(\omega)\ \text{is non-singular}\}$ and 
\begin{align*}
    \tilde{C}_n(\omega):=\begin{cases}
        C_n(\omega)&(\omega\in A_n\cap B_n\cap J_n), \\
        -I(\theta_0)&(\omega\notin A_n\cap B_n\cap J_n).
    \end{cases}
\end{align*}
 Note that $\tilde{C}_n$ is non-singular. Moreover,  it follows from $L_n=C_nS_n\ (\omega\in A_n\cap B_n)$ that 
\begin{align}\label{thm2:asyN:eq}
S_n=\tilde{C}_n^{-1}L_n\ (\omega\in A_n\cap B_n\cap J_n).
\end{align}
By using \eqref{thm2:goal1-jump},\eqref{thm2:goal2-jump} and \eqref{thm2:goal3-jump}, we obtain $C_n\overset{P}{\to}-I(\theta_0)$, and since non-singularity of $I(\theta_0)$, we have $P(J_n)\to 1$. Hence, for $\varepsilon>0$, it follows that
\begin{align*}
P\left(|\tilde{C}_n+I(\theta_0)|>\varepsilon\right)&\leq P\left(|C_n+I(\theta_0)|>\varepsilon\right)+P(A_n^c\cup B_n^c\cup J_n^c)\\
&\leq P\left(|C_n+I(\theta_0)|>\varepsilon\right)+P(A_n^c)+ P(B_n^c)+P(J_n^c)\\
&\to 0.
\end{align*}
This implies $\tilde{C}_n$ converges to $-I(\theta_0)$ in probability. 
Therefore, 
 by the continuous mapping theorem, one has
\begin{align}\label{thm2:asyN:inv}
\tilde{C}_n^{-1}\overset{P}{\to}-I(\theta_0)^{-1}.
\end{align}
Finally, it follows from \eqref{thm2:goal4-jump}, \eqref{thm2:asyN:eq}, \eqref{thm2:asyN:inv}, Slutsky's theorem and Portmanteau's lemma that for any closed set $F\subset\mathbb{R}^{p+q}$,
\begin{align*}
\limsup_{n\to\infty}P\left(S_n\in F\right)&\leq \limsup_{n\to\infty}P\left(\{\omega\in\Omega\ |\ S_n(\omega)\in F\}\cap A_n\cap B_n\cap J_n\right)+\limsup_{n\to\infty}P(A_n^c\cup B_n^c\cup J_n^c)\\
&=\limsup_{n\to\infty}P\left(\{\omega\in\Omega\ |\ \tilde{C}_n^{-1}L_n(\omega)\in F\}\cap A_n\cap B_n\cap J_n\right)\\
&\leq\limsup_{n\to\infty}P\left(\tilde{C}_n^{-1}L_n\in F\right)\\
&\leq P\left(N(0,I(\theta_0)^{-1})\in F\right).
\end{align*}
This implies $S_n\overset{d}{\to}N_{p+q}(0,I(\theta_0)^{-1})$ and completes the proof of Theorem \ref{thm2-2:jump}. 
\end{proof}

\subsubsection{Proof of Corollary \ref{thm1-1:jump}}
\begin{proof}
We prove this by applying the proof of Theorem \ref{thm1-2:jump}.

{\bf Consistency of $\hat{\alpha}_{n}$.}

 We take $D_1$ and $\rho_1$ in $l_n^{(1)}(\alpha)$ in the same way as in $\bar{l}_n(\theta)$. 
Since $\bar{X}_{i,n}(\beta)=\Delta X_{i}^n-h_nb_{i-1}(\beta)$, we have 
\begin{align}
&\sup_{\theta\in\Theta}\left|\bar{l}_n(\theta)-l_n^{(1)}(\alpha)\right|\nonumber\\
&= \sup_{\theta\in\Theta}\left|\frac{-1}{2h_n}\sum_{i=1}^n\left\{\left(\bar{X}_{i,n}(\beta)\right)^\top S_{i-1}^{-1}(\alpha)\bar{X}_{i,n}(\beta)-\left(\Delta X_i^n\right)^\top S_{i-1}^{-1}(\alpha)\Delta X_i^n\right\}\boldsymbol{1}_{\{|\Delta X_i^n|\le D_1h_n^{\rho_1}\}}\right|\nonumber\\
&\le \sup_{\theta\in\Theta}\left|\sum_{i=1}^n\sum_{k_1,k_2=1}^d b_{i-1}^{(k_1)}(\beta){S_{i-1}^{-1}}^{(k_1,k_2)}(\alpha){\Delta X_i^n}^{(k_2)}\boldsymbol{1}_{\{|\Delta X_i^n|\le D_1h_n^{\rho_1}\}}\right.\nonumber\\
&\qquad\qquad\qquad\qquad\qquad\qquad \left.-\frac{h_n}{2}\sum_{i=1}^n\left(b_{i-1}(\beta)\right)^\top S_{i-1}^{-1}(\alpha)b_{i-1}(\beta)\boldsymbol{1}_{\{|\Delta X_i^n|\le D_1h_n^{\rho_1}\}}\right|\nonumber\\
&\le \frac{1}{n}\sum_{i=1}^n\sum_{k_1,k_2=1}^d R(\theta,1,X_{t_{i-1}^n})|\Delta X_i^n|\boldsymbol{1}_{\{|\Delta X_i^n|\le D_1h_n^{\rho_1}\}}+\frac{h_n}{2n}\sum_{i=1}^n R(\theta,1,X_{t_{i-1}^n})\boldsymbol{1}_{\{|\Delta X_i^n|\le D_1h_n^{\rho_1}\}}\nonumber\\
&\le \frac{1}{n}\sum_{i=1}^n R(\theta,h_n^{\rho_1},X_{t_{i-1}^n})+\frac{1}{n}\sum_{i=1}^n R(\theta,h_n,X_{t_{i-1}^n})\nonumber\\
&\overset{P}{\to} 0.\label{a-cons:col1:diff}
\end{align}
By using \eqref{a-cons:col1:diff} and \eqref{a-cons:uni}, 
\begin{align}\label{a-cons:col1:uni}
\sup_{\theta\in\Theta}\left|\frac{1}{n}\bar{l}_n(\theta)-U_1(\alpha,\alpha_0)\right|&\le \sup_{\theta\in\Theta}\left|\frac{1}{n}\bar{l}_n(\theta)-\frac{1}{n}l_n^{(1)}(\alpha)\right|+\sup_{\theta\in\Theta}\left|\frac{1}{n}l_n^{(1)}(\theta)-U_1(\alpha,\alpha_0)\right|\nonumber\\
&\overset{P}{\to}0.
\end{align}
Moreover, by the definition of $\hat{\theta}_n$, it follows that for all $\varepsilon>0$, 
\begin{align}\label{a-cons:col1:compose}
P\left(\frac{1}{n}\bar{l}_n(\hat{\theta}_n)+\varepsilon<\frac{1}{n}\bar{l}_n(\alpha_0,\hat{\beta}_n)\right)= 0.
\end{align}
Hence, 
 in an analogous manner to the proof of consistency of $\check{\alpha}_n$,
we see from \eqref{a-cons:iden}, \eqref{a-cons:col1:uni} and \eqref{a-cons:col1:compose} that
\begin{align}\label{a-cons:col1:result}
\hat{\alpha}_n\overset{P}{\to}\alpha_0.
\end{align}

{\bf Consistency of $\hat{\beta}_{n}$.} 

 For $D_3,\rho_3,D_2$ and $\rho_2$ in $l_n^{(2)}(\beta|\bar{\alpha})$ 
and $D_1,\rho_1,D_2$ and $\rho_2$ in $l_n(\theta)$,  
we set $D_3=D_1,\ \rho_3=\rho_1,\ D_2=D_2,\ \rho_2=\rho_2$. Then, it holds that
\begin{align}\label{b-cons:col1:eq1}
\tilde{l}_n(\beta)&= \tilde{l}_n^{(2)}(\beta)
\end{align}
and that
\begin{align}\label{b-cons:col1:eq2}
\frac{1}{nh_n}\left(\bar{l}_n(\alpha,\beta)-\bar{l}_n(\alpha,\beta_0)\right)&= 
\frac{1}{nh_n}\left(\bar{l}_n^{(2)}(\beta|\alpha)-\bar{l}_n^{(2)}(\beta_0|\alpha)\right).
\end{align}
By using \eqref{b-cons:uni_jump}, \eqref{b-cons:uni}, \eqref{b-cons:col1:eq1} and \eqref{b-cons:col1:eq2}, we have 
\begin{align}\label{b-cons:col1:uni_jump}
\sup_{\beta\in\Theta_{\beta}}\left|\frac{1}{nh_n}\tilde{l}_n(\beta)-\tilde{U}_{\beta_0}^{(2)}(\beta)\right|&\overset{P}{\to} 0
\end{align}
and
\begin{align}\label{b-cons:col1:uni}
\sup_{\theta\in\Theta}\left|\frac{1}{nh_n}\bar{l}_n(\alpha,\beta)-\frac{1}{nh_n}\bar{l}_n(\alpha,\beta_0)-\bar{U}_{\beta_0}^{(2)}(\alpha,\beta)\right|&\overset{P}{\to} 0.
\end{align}
Moreover, by the definition of $\hat{\beta}_n$, it follows that for all $\varepsilon>0$, 
\begin{align}\label{b-cons:col1:compose}
P\left(\frac{1}{nh_n}l_n^{(2)}(\hat{\beta}_n|\hat{\alpha}_n)+\varepsilon<\frac{1}{nh_n}l_n^{(2)}(\beta_0|\hat{\alpha}_n)\right)=0.
\end{align}
Furthermore, by using \eqref{a-cons:col1:result}, in an analogous manner to \eqref{b-cons:tech},  one has
\begin{align}\label{b-cons:col1:tech}
|\bar{U}_{\beta_0}^{(2)}(\hat{\alpha}_n,\hat{\beta}_n)-\bar{U}_{\beta_0}^{(2)}(\alpha_0,\hat{\beta}_n)|\overset{P}{\to} 0.
\end{align}
Therefore,  in a similar way to the proof of consistency of $\check{\beta}_n$, 
it holds from  \eqref{b-cons:idenV}, \eqref{b-cons:col1:uni_jump}, \eqref{b-cons:col1:uni}, \eqref{b-cons:col1:compose} and \eqref{b-cons:col1:tech} that
\begin{align}\label{b-cons:col1:result}
\hat{\beta}_n\overset{P}{\to}\beta_0.
\end{align}
\end{proof}
\subsubsection{Proof of corollary \ref{thm2-1:jump}}
\begin{proof}    
We prove this by applying the proof of Theorem \ref{thm2-2:jump}. 
Let us define the function $l_n(\theta)$ as follows:
\begin{align*}
&l_n(\theta):=\bar{l}_n(\theta)+\tilde{l}_n(\beta),\\
&\bar{l}_n(\theta):=-\frac{1}{2}\sum_{i=1}^{n}\left\{h_n^{-1}(\bar{X}_{i,n}(\beta))^\top S_{i-1}^{-1}(\alpha)\bar{X}_{i,n}(\beta)+\log\det S_{i-1}(\alpha)\right\}\boldsymbol{1}_{\{|\Delta X_i^n|\le D_1h_n^{\rho_1}\}},\\
&\tilde{l}_n(\beta):=\sum_{i=1}^{n}\left(\log\Psi_\beta(\Delta X_i^n,X_{t_{i-1}^n})\right)\varphi_n(X_{t_{i-1}^n},\Delta X_i^n)\boldsymbol{1}_{\{|\Delta X_i^n|> D_2h_n^{\rho_2}\}}\notag\\
&\qquad\qquad\qquad-h_n\sum_{i=1}^{n}\int_{B}\Psi_\beta(y,X_{t_{i-1}^n})\varphi_n(X_{t_{i-1}^n},y)dy.
\end{align*}
By applying the thresholds $D_1,\rho_1,D_2$ and $\rho_2$ contained in $l_n(\theta)$, we further define $l_n^{(1)}(\alpha)$ and $l_n^{(2)}(\beta|\bar{\alpha})$ as follows:
\begin{align*}
&l_n^{(1)}(\alpha):=-\frac{1}{2}\sum_{i=1}^{n}\left\{h_n^{-1}(\Delta X_i^n)^\top S_{i-1}^{-1}(\alpha)\Delta X_i^n+\log\det S_{i-1}(\alpha)\right\}\boldsymbol{1}_{\{|\Delta X_i^n|\le D_1h_n^{\rho_1}\}},\\
&l_n^{(2)}(\beta|\bar{\alpha}):=\bar{l}_n^{(2)}(\beta|\bar{\alpha})+\tilde{l}_n^{(2)}(\beta),\notag\\
&\bar{l}_n^{(2)}(\beta|\bar{\alpha}):=-\frac{1}{2h_n}\sum_{i=1}^{n}(\bar{X}_{i,n}(\beta))^\top S_{i-1}^{-1}(\bar{\alpha})\bar{X}_{i,n}(\beta)\boldsymbol{1}_{\{|\Delta X_i^n|\le D_1h_n^{\rho_1}\}},\\
&\tilde{l}_n^{(2)}(\beta):=\sum_{i=1}^{n}\left(\log\Psi_\beta(\Delta X_i^n,X_{t_{i-1}^n})\right)\varphi_n(X_{t_{i-1}^n},\Delta X_i^n)\boldsymbol{1}_{\{|\Delta X_i^n|> D_2h_n^{\rho_2}\}}\notag\\
&\qquad\qquad\qquad-h_n\sum_{i=1}^{n}\int_{B}\Psi_\beta(y,X_{t_{i-1}^n})\varphi_n(X_{t_{i-1}^n},y)dy.
\end{align*}
Then we can calculate for $1\le m_1, m_1'\le p$,
\begin{align*}
\partial_{\alpha_{m_1}}l_n(\theta)&= \partial_{\alpha_{m_1}}l_n^{(1)}(\alpha)\\
&\quad -\frac{1}{2h_n}\sum_{i=1}^n\left\{\left(\bar{X}_{i,n}(\beta)\right)^\top\partial_{\alpha_{m_1}}S_{i-1}^{-1}(\alpha)\bar{X}_{i,n}(\beta)-\left(\Delta X_i^n\right)^\top\partial_{\alpha_{m_1}}S_{i-1}^{-1}(\alpha)\Delta X_i^n\right\}\boldsymbol{1}_{\{|\Delta X_i^n|\le D_1h_n^{\rho_1}\}},\\
\partial_{\alpha_{m_1}\alpha_{m_1'}}^2 l_n(\theta)&= \partial_{\alpha_{m_1}\alpha_{m_1'}}^2 l_n^{(1)}(\alpha)\\
&\quad -\frac{1}{2h_n}\sum_{i=1}^n\left\{\left(\bar{X}_{i,n}(\beta)\right)^\top\partial_{\alpha_{m_1}\alpha_{m_1'}}^2 S_{i-1}^{-1}(\alpha)\bar{X}_{i,n}(\beta)-\left(\Delta X_i^n\right)^\top\partial_{\alpha_{m_1}\alpha_{m_1'}}^2 S_{i-1}^{-1}(\alpha)\Delta X_i^n\right\}\boldsymbol{1}_{\{|\Delta X_i^n|\le D_1h_n^{\rho_1}\}},
\end{align*}
for $1\le m_2,m_2'$,
\begin{align*}
\partial_{\beta_{m_2}}l_n(\theta)&= \partial_{\beta_{m_2}}l_n^{(2)}(\beta|\alpha),\\
\partial_{\beta_{m_2}\beta_{m_2'}}^2 l_n(\theta)&= \partial_{\beta_{m_2}\beta_{m_2'}}^2 l_n^{(2)}(\beta|\alpha),
\end{align*}
and for $1\le m_1\le p, 1\le m_2\le q$,
\begin{align*}
\partial_{\alpha_{m_1}\beta_{m_2}}^2 l_n(\theta)&= \partial_{\alpha_{m_1}\beta_{m_2}}^2 l_n^{(2)}(\beta|\alpha).
\end{align*}

Let  $\varepsilon_0$ be a positive constant such that $\{\theta\in\Theta\ ;\ |\theta-\theta_0|<\varepsilon_0\}\subset\mathrm{Int}(\Theta)$. Then it follows from consistency of $\hat{\theta}_n$ that there exists an real valued sequence $\varepsilon_n<\varepsilon_0$ such that $P(\hat{A}_n)\to 1$, where $\hat{A}_{n}:=\{\omega\in\Omega\ |\ |\hat{\theta}_n(\omega)-\theta_0|<\varepsilon_n\}$. In particular, we have $\hat{\theta}_n\in\mathrm{Int}(\Theta)$ on the set $\hat{A}_n$.
Hence, By using Taylor's theorem, we have the following equations on the set $\hat{A}_n$:
\begin{align*}
-\begin{pmatrix}
\partial_{{\alpha}} l_n(\theta_0)\\
\partial_{{\beta}} l_n(\theta_0)
\end{pmatrix}&=\left(\int_{0}^{1}{\partial_{{\theta}}^2 l_n(\theta_0+u(\hat{\theta}_{n}-\theta_0))du}\right)\begin{pmatrix}
\hat{\alpha}_n-\alpha_0\\
\hat{\beta}_n-\beta_0
\end{pmatrix},
\end{align*}
where 
\begin{align*}
\partial_{{\theta}}^2 l_n(\theta)&= \begin{pmatrix}
\partial_{{\alpha}}^2 l_n(\theta)&\partial_{{\alpha\beta}}^2 l_n(\theta)\\
\partial_{{\beta\alpha}}^2 l_n(\theta)&\partial_{{\beta}}^2 l_n(\theta)
\end{pmatrix}.
\end{align*}
Therefore, we obtain $\hat{L}_n=\hat{C}_n\hat{S}_n\ (\omega\in \hat{A}_n)$, where
\begin{align*}
\hat{S}_n&:=\begin{pmatrix}
\sqrt{n}(\hat{\alpha}_{n}-\alpha_0)\\
\sqrt{nh_n}(\hat{\beta}_{n}-\beta_0)
\end{pmatrix},\quad
\hat{L}_n:=\begin{pmatrix}
-\frac{1}{\sqrt{n}}\partial_{{\alpha}}l_n(\theta_0)\\
-\frac{1}{\sqrt{nh_n}}\partial_{{\beta}} l_n(\theta_0)
\end{pmatrix},\\
\hat{C}_n&:=\begin{pmatrix}
\int_{0}^{1}{\frac{1}{n}\partial_{{\alpha}}^2 l_n(\theta_0+u(\hat{\theta}_{n}-\theta_0))du}&\int_{0}^{1}{\frac{1}{n\sqrt{h_n}}\partial_{{\alpha\beta}}^2 l_n(\theta_0+u(\hat{\theta}_{n}-\theta_0))du}\\
\int_{0}^{1}{\frac{1}{n\sqrt{h_n}}\partial_{{\beta\alpha}}^2 l_n(\theta_0+u(\hat{\theta}_{n}-\theta_0))du}&
\int_{0}^{1}{\frac{1}{nh_n}\partial_{{\beta}}^2 l_n(\theta_0+u(\hat{\theta}_{n}-\theta_0))du}
\end{pmatrix}.
\end{align*}
Thus, it is sufficient to show $\hat{S}_n\overset{d}{\to}N_{p+q}(0,I(\theta_0)^{-1})$ that 
\begin{align}
&\sup_{u\in[0,1]}\left|\frac{1}{n}\partial_{{\alpha}}^2 l_n(\theta_0+u(\hat{\theta}_{n}-\theta_0))+I_a(\alpha_0)\right|\overset{P}{\to}0\label{col2:goal1-jump}%
,\\
&\sup_{u\in[0,1]}\left|\frac{1}{nh_n}\partial_{{\beta}}^2 l_n(\theta_0+u(\hat{\theta}_{n}-\theta_0))+I_{b,c}(\theta_0)\right|\overset{P}{\to}0\label{col2:goal2-jump}%
,\\
&\sup_{u\in[0,1]}\left|\frac{1}{n\sqrt{h_n}}\partial_{{\alpha\beta}}^2 l_n(\theta_0+u(\hat{\theta}_{n}-\theta_0))\right|\overset{P}{\to}0\label{col2:goal3-jump}%
,\\
&\hat{L}_n\overset{d}{\to}N_{p+q}(0, I(\theta_0))\label{col2:goal4-jump}%
.
\end{align}

{\bf Proof of \eqref{col2:goal1-jump}.}
For $1\le m_1,m_1'\le p$, we can calculate
\begin{align*}
\partial_{\alpha_{m_1}\alpha_{m_1'}}^2 l_n(\theta)&= \partial_{\alpha_{m_1}\alpha_{m_1'}}^2 l_n^{(1)}(\alpha)\\
&\quad -\frac{1}{2h_n}\sum_{i=1}^n\left\{\left(\bar{X}_{i,n}(\beta)\right)^\top\partial_{\alpha_{m_1}\alpha_{m_1'}}^2 S_{i-1}^{-1}(\alpha)\bar{X}_{i,n}(\beta)\right.\\
&\qquad\qquad\qquad\left.-\left(\Delta X_i^n\right)^\top\partial_{\alpha_{m_1}\alpha_{m_1'}}^2 S_{i-1}^{-1}(\alpha)\Delta X_i^n\right\}\boldsymbol{1}_{\{|\Delta X_i^n|\le D_1h_n^{\rho_1}\}}.
\end{align*} 
Since $\bar{X}_{i,n}(\beta)=\Delta X_{i,n}-h_nb_{i-1}(\beta)$, by using \eqref{a-2der}, we have  
\begin{align}
&\sup_{\theta\in\Theta}\left|\frac{1}{n}\partial_{\alpha_{m_1}\alpha_{m_1'}}^2 l_n(\alpha)+I_a^{(m_1,m_1')}(\alpha)\right|\notag\\
&\le \sup_{\alpha\in\Theta_\alpha}\left|\frac{1}{n}\partial_{\alpha_{m_1}\alpha_{m_1'}}^2 l_n^{(1)}(\alpha)+I_a^{(m_1,m_1')}(\alpha)\right|\notag\\
&\quad +\frac{1}{n}\sum_{i=1}^n\sum_{k_1,k_2=1}^d\sup_{\theta\in\Theta}|b_{i-1}^{(k_1)}|\sup_{\theta\in\Theta}|\partial_{\alpha_{m_1}\alpha_{m_1'}}^2{S_{i-1}^{-1}}^{(k_1,k_2)}||{\Delta X_{i}^n}^{(k_2)}|\boldsymbol{1}_{\{|\Delta X_i^n|\le D_1h_n^{\rho_1}\}}\notag\\
&\quad +\frac{h_n}{2n}\sum_{i=1}^n\sum_{k_1,k_2=1}^d\sup_{\theta\in\Theta}|b_{i-1}^{(k_1)}|\sup_{\theta\in\Theta}|\partial_{\alpha_{m_1}\alpha_{m_1'}}^2{S_{i-1}^{-1}}^{(k_1,k_2)}|\sup_{\theta\in\Theta}|b_{i-1}^{(k_2)}|\boldsymbol{1}_{\{|\Delta X_i^n|\le D_1h_n^{\rho_1}\}}\notag\\
&\le \sup_{\alpha\in\Theta_\alpha}\left|\frac{1}{n}\partial_{\alpha_{m_1}\alpha_{m_1'}}^2 l_n^{(1)}(\alpha)+I_a^{(m_1,m_1')}(\alpha)\right|+\frac{1}{n}\sum_{i=1}^nR(\theta,h_n^{\rho_1},X_{t_{i-1}^n})+\frac{1}{n}\sum_{i=1}^nR(\theta,h_n,X_{t_{i-1}^n})\notag\\
&\overset{P}{\to} 0.\notag
\end{align}
Therefore, 
\begin{align}\label{a-2der:col}
\sup_{\theta\in\Theta}\left|\frac{1}{n}\partial_{\alpha}^2 l_n(\theta)+I_a(\alpha)\right|\overset{P}{\to}0.
\end{align}
Since for all $u\in[0,1]$, $\alpha_0+u(\hat{\alpha}_{n}-\alpha_0)\in\{\alpha\in\Theta_\alpha\ |\ |\alpha-\alpha_0|<\varepsilon_n\}$ on the set $\hat{A}_n$, it holds from \eqref{a-2der:col}, continuity of $I_a(\alpha)$ and consistency of $\hat{\alpha}_n$ that for all $\varepsilon>0$, 
\begin{align*}
&P\left(\sup_{u\in[0,1]}\left|\frac{1}{n}\partial_{{\alpha}}^2 l_n(\theta_0+u(\hat{\theta}_{n}-\theta_0))+I_a(\alpha_0)\right|>\varepsilon\right)\\
\quad&\le  P\left(\sup_{\theta\in\Theta}\left|\frac{1}{n}\partial_{{\alpha}}^2 l_n(\theta)+I_a(\alpha)\right|>\frac{\varepsilon}{2}\right)+
P\left(\sup_{u\in[0,1]}\left|I_a(\alpha_0+u(\hat{\alpha}_{n}-\alpha_0))-I_a(\alpha_0)\right|>\frac{\varepsilon}{2}\right)\\
\quad&\le  P\left(\sup_{\theta\in\Theta}\left|\frac{1}{n}\partial_{{\alpha}}^2 l_n(\theta)+I_a(\alpha)\right|>\frac{\varepsilon}{2}\right)+
P\left(\sup_{\alpha:|\alpha-\alpha_0|<\varepsilon_n}\left|I_a(\alpha)-I_a(\alpha_0)\right|>\frac{\varepsilon}{2}\right)+P(\hat{A}_n^c)\\
&\to 0\quad(n\to\infty).
\end{align*}
This implies \eqref{col2:goal1-jump}.

{\bf Proof of \eqref{col2:goal2-jump}.}
 Since $\partial_{\beta}^2 l_n(\theta)= \partial_{\beta}^2 l_n^{(2)}(\beta|\alpha)$, we see from \eqref{b-2der} that 
\begin{align}\label{b-2der:col}
\sup_{\theta\in\Theta}\left|\frac{1}{nh_n}\partial_{\beta}^2l_n(\theta)+I_{b,c}(\theta)\right|\overset{P}{\to}0.
\end{align}
Since for all $u\in[0,1]$, $\theta_0+u(\hat{\theta}_{n}-\theta_0)\in\{\theta\in\Theta\ |\ |\theta-\theta_0|<\varepsilon_n\}$ on the set $\hat{A}_n$, it follows from \eqref{b-2der:col}, continuity of 
$I_{b,c}(\theta)$ and consistency of $\hat{\theta}_n$ that for all $\varepsilon>0$, 
\begin{align*}
&P\left(\sup_{u\in[0,1]}\left|\frac{1}{nh_n}\partial_{{\beta}}^2 l_n(\theta_0+u(\hat{\theta}_{n}-\theta_0))+I_{b,c}(\theta_0)\right|>\varepsilon\right)\\
\quad&\le  P\left(\sup_{\theta\in\Theta}\left|\frac{1}{nh_n}\partial_{{\beta}}^2 l_n(\theta)+I_{b,c}(\theta)\right|>\frac{\varepsilon}{2}\right)+
P\left(\sup_{u\in[0,1]}\left|I_{b,c}(\theta_0+u(\hat{\theta}_{n}-\theta_0))-I_{b,c}(\theta_0)\right|>\frac{\varepsilon}{2}\right)\\
&\leq P\left(\sup_{\theta\in\Theta}\left|\frac{1}{nh_n}\partial_{{\beta}}^2 l_n(\theta)+I_{b,c}(\theta)\right|>\frac{\varepsilon}{2}\right)+
P\left(\sup_{\theta:|\theta-\theta_0|<\varepsilon_n}\left|I_{b,c}(\theta)-I_{b,c}(\theta_0)\right|>\frac{\varepsilon}{2}\right)+P(\hat{A}_n^c)\\
&\to 0\quad(n\to\infty).
\end{align*}
This implies \eqref{col2:goal2-jump}.

{\bf Proof of \eqref{col2:goal3-jump}.}
Since $\partial_{\alpha\beta}^2 l_n(\theta)= \partial_{\alpha\beta}^2 l_n^{(2)}(\beta|\alpha)$,
 it holds from \eqref{ab-2der} that
\begin{align}\label{ab-2der:col}
\sup_{\theta\in\Theta}\left|\frac{1}{n\sqrt{h_n}}\partial_{\alpha\beta}^2 l_n(\theta)\right|&\overset{P}{\to}0.
\end{align}
Therefore,  one has
\begin{align*}
&P\left(\sup_{u\in[0,1]}\left|\frac{1}{n\sqrt{h_n}}\partial_{{\alpha\beta}}^2 l_n(\theta_0+u(\hat{\theta}_{n}-\theta_0))\right|>\varepsilon\right)\le
P\left(\sup_{\theta\in\Theta}\left|\frac{1}{n\sqrt{h_n}}\partial_{{\alpha\beta}}^2 l_n(\theta)\right|>\varepsilon\right)\overset{P}{\to}0.
\end{align*}
This implies \eqref{col2:goal3-jump}.

{\bf Proof of \eqref{col2:goal4-jump}.}
For $1\le m_1\le p$, we can calculate  
\begin{align*}
\partial_{\alpha_{m_1}}l_n(\theta)&= \partial_{\alpha_{m_1}}l_n^{(1)}(\alpha)\\
&\quad -\frac{1}{2h_n}\sum_{i=1}^n\left\{\left(\bar{X}_{i,n}(\beta)\right)^\top\partial_{\alpha_{m_1}}S_{i-1}^{-1}(\alpha)\bar{X}_{i,n}(\beta)\right.\\
&\qquad\qquad\qquad\left.-\left(\Delta X_i^n\right)^\top\partial_{\alpha_{m_1}}S_{i-1}^{-1}(\alpha)\Delta X_i^n\right\}\boldsymbol{1}_{\{|\Delta X_i^n|\le D_1h_n^{\rho_1}\}},
\end{align*}
and for $1\le m_2\le q$, 
$\partial_{\beta_{m_2}}l_n(\theta)=\partial_{\beta_{m_2}}l_n^{(2)}(\beta|\alpha)$. 
Let
\begin{align*}
\psi_i^{m_1}(\theta_0)&:=(b_{i-1}(\beta_0))^\top\partial_{\alpha_{m_1}}S_{i-1}^{-1}(\alpha_0)\bar{X}_{i,n}(\beta_0)\boldsymbol{1}_{\{|\Delta X_i^n|\le D_1h_n^{\rho_1}\}}\\
&\qquad\qquad+\frac{h_n}{2}(b_{i-1}(\beta_0))^\top\partial_{\alpha_{m_1}}S_{i-1}^{-1}(\alpha_0)b_{i-1}(\beta_0)\boldsymbol{1}_{\{|\Delta X_i^n|\le D_1h_n^{\rho_1}\}},
\end{align*}
then, using the same notations in the proof of Theorem \ref{thm2-2:jump}, we can express 
\begin{align*}
\partial_{\alpha_{m_1}}l_n(\theta_0)&= \sum_{i=1}^n\left(\xi_{i}^{m_1}(\alpha_0)+\psi_i^{m_1}(\theta_0)\right)\\
\partial_{\beta_{m_2}}l_n(\theta_0)&=\sum_{i=1}^n\left(\eta_{i,1}^{m_2}(\beta_0|\alpha_0)+\eta_{i,2}^{m_2}(\beta_0)\right).
\end{align*}
From  \citet{Hall-Heyde}, the following types of convergence are sufficient for \eqref{col2:goal4-jump}:
for $1\le m_1,m_1'\le p_1$, $1\le m_2,m_2'\le p_2$ and some $\nu_1,\nu_2>0$,
\begin{align}
&\sum_{i=1}^{n}{\mathbb{E}\left[\frac{1}{\sqrt{n}}\left(\xi_i^{m_1}(\alpha_0)+\psi_i^{m_1}(\theta_0)\right)\ |\mathcal{F}_{i-1}^n\right]}\overset{P}{\to}0\label{MCLT1-jump:col},\\
&\sum_{i=1}^{n}{\mathbb{E}\left[\frac{1}{\sqrt{nh_n}}\left(\eta_{i,1}^{m_2}(\beta_0|\alpha_0)+\eta_{i,2}^{m_2}(\beta_0)\right)\ |\mathcal{F}_{i-1}^n\right]}\overset{P}{\to}0\label{MCLT2-jump:col},\\
&\sum_{i=1}^{n}{\mathbb{E}\left[\frac{1}{n}\left(\xi_i^{m_1}(\alpha_0)+\psi_i^{m_1}(\theta_0)\right)\left(\xi_i^{m_1'}(\alpha_0)+\psi_i^{m_1'}(\theta_0)\right)\ |\mathcal{F}_{i-1}^n\right]}\overset{P}{\to}I_a^{m_1,m_1'}(\alpha_0)\label{MCLT3-jump:col},\\
&\sum_{i=1}^{n}{\mathbb{E}\left[\frac{1}{\sqrt{n}}\left(\xi_i^{m_1}(\alpha_0)+\psi_i^{m_1}(\theta_0)\right)\ |\mathcal{F}_{i-1}^n\right]\mathbb{E}\left[\frac{1}{\sqrt{n}}\left(\xi_i^{m_1'}(\alpha_0)+\psi_i^{m_1'}(\theta_0)\right)\ |\mathcal{F}_{i-1}^n\right]}\overset{P}{\to}0\label{MCLT4-jump:col},\\
&\sum_{i=1}^{n}{\mathbb{E}\left[\frac{1}{nh_n}\left(\eta_{i,1}^{m_2}(\beta_0|\alpha_0)+\eta_{i,2}^{m_2}(\beta_0)\right)\left(\eta_{i,1}^{m_2'}(\beta_0|\alpha_0)+\eta_{i,2}^{m_2'}(\beta_0)\right)\ |\mathcal{F}_{i-1}^n\right]}\overset{P}{\to}I_{b,c}^{m_2,m_2'}(\theta_0)\label{MCLT5-jump:col},\\
&\sum_{i=1}^{n}{\mathbb{E}\left[\frac{1}{\sqrt{nh_n}}\left(\eta_{i,1}^{m_2}(\beta_0|\alpha_0)+\eta_{i,2}^{m_2}(\beta_0)\right)\ |\mathcal{F}_{i-1}^n\right]\mathbb{E}\left[\frac{1}{\sqrt{nh_n}}\left(\eta_{i,1}^{m_2'}(\beta_0|\alpha_0)+\eta_{i,2}^{m_2'}(\beta_0)\right)\ |\mathcal{F}_{i-1}^n\right]}\overset{P}{\to}0\label{MCLT6-jump:col},\\
&\sum_{i=1}^{n}{\mathbb{E}\left[\frac{1}{n\sqrt{h_n}}\left(\xi_i^{m_1}(\alpha_0)+\psi_i^{m_1}(\theta_0)\right)\left(\eta_{i,1}^{m_2}(\beta_0|\alpha_0)+\eta_{i,2}^{m_2}(\beta_0)\right)\ |\mathcal{F}_{i-1}^n\right]}\overset{P}{\to}0\label{MCLT7-jump:col},\\
&\sum_{i=1}^{n}{\mathbb{E}\left[\frac{1}{\sqrt{n}}\left(\xi_i^{m_1}(\alpha_0)+\psi_i^{m_1}(\theta_0)\right)\ |\mathcal{F}_{i-1}^n\right]\mathbb{E}\left[\frac{1}{\sqrt{nh_n}}\left(\eta_{i,1}^{m_2}(\beta_0|\alpha_0)+\eta_{i,2}^{m_2}(\beta_0)\right)\ |\mathcal{F}_{i-1}^n\right]}\overset{P}{\to}0\label{MCLT8-jump:col},\\
&\sum_{i=1}^{n}{\mathbb{E}\left[\left|\frac{1}{\sqrt{n}}\left(\xi_i^{m_1}(\alpha_0)+\psi_i^{m_1}(\theta_0)\right)\right|^{2+\nu_1}|\mathcal{F}_{i-1}^n\right]}\overset{P}{\to}0\label{MCLT9-jump:col},\\
&\sum_{i=1}^{n}{\mathbb{E}\left[\left|\frac{1}{\sqrt{nh_n}}\left(\eta_{i,1}^{m_2}(\beta_0|\alpha_0)+\eta_{i,2}^{m_2}(\beta_0)\right)\right|^{2+\nu_2}|\mathcal{F}_{i-1}^n\right]}\overset{P}{\to}0\label{MCLT10-jump:col}.
\end{align}
By an analogous argument to \eqref{MCLT2-jump}, \eqref{MCLT5-jump}, \eqref{MCLT6-jump} and \eqref{MCLT10-jump}, it is obvious that
\eqref{MCLT2-jump:col}, \eqref{MCLT5-jump:col}, \eqref{MCLT6-jump:col} and \eqref{MCLT10-jump:col} hold. 

{\bf Proof of \eqref{MCLT1-jump:col}.}
From \eqref{MCLT1-jump}, it is sufficient to show that 
\begin{align*}
\sum_{i=1}^{n}{\mathbb{E}\left[\frac{1}{\sqrt{n}}\psi_i^{m_1}(\theta_0)\ |\mathcal{F}_{i-1}^n\right]}\overset{P}{\to}0.
\end{align*}
By using \eqref{under:rep1}, we have
\begin{align*}
\left|\sum_{i=1}^{n}{\mathbb{E}\left[\frac{1}{\sqrt{n}}\psi_i^{m_1}(\theta_0)\ |\mathcal{F}_{i-1}^n\right]}\right|&= \left|\frac{1}{\sqrt{n}}\sum_{i=1}^{n}\mathbb{E}\left[(b_{i-1})^\top\partial_{\alpha_{m_1}}S_{i-1}^{-1}\bar{X}_{i,n}\boldsymbol{1}_{\{|\Delta X_i^n|\le D_1h_n^{\rho_1}\}}\ |\mathcal{F}_{i-1}^n\right]\right.\\
&\quad +\left. \frac{h_n}{2\sqrt{n}}\sum_{i=1}^{n}\mathbb{E}\left[(b_{i-1})^\top\partial_{\alpha_{m_1}}S_{i-1}^{-1}b_{i-1}\boldsymbol{1}_{\{|\Delta X_i^n|\le D_1h_n^{\rho_1}\}}\ |\mathcal{F}_{i-1}^n\right]\right|\\
&\le \frac{1}{\sqrt{n}}\sum_{i=1}^n\sum_{k_1,k_2=1}^d|b_{i-1}^{(k_1)}||{\partial_{\alpha_{m_1}}S_{i-1}^{-1}}^{(k_1,k_2)}|\left|\mathbb{E}\left[{\bar{X}_{i,n}}^{(k_2)}\boldsymbol{1}_{\{|\Delta X_i^n|\le D_1h_n^{\rho_1}\}}\ |\mathcal{F}_{i-1}^n\right]\right|\\
&\quad +\frac{h_n}{2\sqrt{n}}\sum_{i=1}^{n}|(b_{i-1})^\top\partial_{\alpha_{m_1}}S_{i-1}^{-1}b_{i-1}|P\left(|\Delta X_i^n|\le D_1 h_n^{\rho_1}\ |\mathcal{F}_{i-1}^n\right)\\
&\le \frac{1}{n}\sum_{i=1}^n R(\theta,\sqrt{n}h_n^{1+2\rho_1},X_{t_{i-1}^n})+\frac{1}{n}\sum_{i=1}^n R(\theta,\sqrt{nh_n^2},X_{t_{i-1}^n})\\
&\overset{P}{\to} 0.
\end{align*}
This implies \eqref{MCLT1-jump:col}. In a similar way, we obtain
\begin{align}\label{MCLT1-jump:col:eq1}
{\mathbb{E}\left[\psi_i^{m_1}(\theta_0)\ |\mathcal{F}_{i-1}^n\right]}&= R\left(\theta,h_n^{1+2\rho_1},X_{t_{i-1}^n}\right)+R\left(\theta,h_n,X_{t_{i-1}^n}\right)\notag\\
&= R\left(\theta,h_n,X_{t_{i-1}^n}\right).
\end{align}

{\bf Proof of \eqref{MCLT3-jump:col}.}
By simple computation, 
\begin{align*}
&\sum_{i=1}^{n}{\mathbb{E}\left[\frac{1}{n}\left(\xi_i^{m_1}(\alpha_0)+\psi_i^{m_1}(\theta_0)\right)\left(\xi_i^{m_1'}(\alpha_0)+\psi_i^{m_1'}(\theta_0)\right)\ |\mathcal{F}_{i-1}^n\right]}\\
&= \sum_{i=1}^{n}{\mathbb{E}\left[\frac{1}{n}\xi_i^{m_1}(\alpha_0)\xi_i^{m_1'}(\alpha_0)\ |\mathcal{F}_{i-1}^n\right]}+\sum_{i=1}^{n}{\mathbb{E}\left[\frac{1}{n}\xi_i^{m_1}(\alpha_0)\psi_i^{m_1'}(\theta_0)\ |\mathcal{F}_{i-1}^n\right]}\\
&\quad+\sum_{i=1}^{n}{\mathbb{E}\left[\frac{1}{n}\xi_i^{m_1'}(\alpha_0)\psi_i^{m_1}(\theta_0)\ |\mathcal{F}_{i-1}^n\right]}+\sum_{i=1}^{n}{\mathbb{E}\left[\frac{1}{n}\psi_i^{m_1}(\theta_0)\psi_i^{m_1'}(\theta_0)\ |\mathcal{F}_{i-1}^n\right]}.
\end{align*}
By using \eqref{MCLT3-jump}, the first term on the right-hand side converges to $I_a^{m_1,m_1'}(\alpha_0)$ in probability.
Since by the definition of $\psi_i^{m_1}(\theta_0)$, $\psi_i^{m_1}(\theta_0)\boldsymbol{1}_{\{|\Delta X_i^n|\le D_1h_n^{\rho_1}\}}=\psi_i^{m_1}(\theta_0)$,  it follows from \eqref{MCLT1-jump:col:eq1} that 
\begin{align*}
&\sum_{i=1}^{n}{\mathbb{E}\left[\frac{1}{n}\xi_i^{m_1}(\alpha_0)\psi_i^{m_1'}(\theta_0)\ |\mathcal{F}_{i-1}^n\right]}\\
&= \left|\frac{1}{n}\sum_{i=1}^{n}{\mathbb{E}\left[-\frac{1}{2}\left\{h_n^{-1}(\Delta X_i^n)^\top\partial_{{\alpha_{m_1}}}S_{i-1}^{-1}(\alpha_0)\Delta X_i^n+\partial_{{\alpha}_{m_1}}\log\det S_{i-1}(\alpha_0)\right\}\boldsymbol{1}_{\{|\Delta X_i^n|\le D_1h_n^{\rho_1}\}}\psi_i^{m_1'}(\theta_0)\ |\mathcal{F}_{i-1}^n\right]}\right|\\
&\le \frac{1}{2nh_n}\sum_{i=1}^{n}\left|{\mathbb{E}\left[(\Delta X_i^n)^\top\partial_{{\alpha_{m_1}}}S_{i-1}^{-1}(\alpha_0)\Delta X_i^n\boldsymbol{1}_{\{|\Delta X_i^n|\le D_1h_n^{\rho_1}\}}\psi_i^{m_1'}(\theta_0)\ |\mathcal{F}_{i-1}^n\right]}\right|\\
&\quad +\frac{1}{2n}\sum_{i=1}^{n}\left|\partial_{{\alpha}_{m_1}}\log\det S_{i-1}(\alpha_0)\right|\left|\mathbb{E}\left[\psi_i^{m_1'}(\theta_0)\boldsymbol{1}_{\{|\Delta X_i^n|\le D_1h_n^{\rho_1}\}}\ |\mathcal{F}_{i-1}^n\right]\right|\\
&\le \frac{\left(D_1h_n^{\rho_1}\right)^2}{2nh_n}\sum_{i=1}^n\sum_{k_1,k_2=1}^d\left|\partial_{\alpha_{m_1}}{S_{i-1}^{-1}}^{(k_1,k_2)}\right|\left|\mathbb{E}\left[\psi_i^{m_1}(\theta_0)\ |\mathcal{F}_{i-1}^n\right]\right|+\frac{1}{2n}\sum_{i=1}^{n}\left|\partial_{{\alpha}_{m_1}}\log\det S_{i-1}(\alpha_0)\right|\left|\mathbb{E}\left[\psi_i^{m_1'}(\theta_0)\ |\mathcal{F}_{i-1}^n\right]\right|\\
&\le \frac{1}{n}\sum_{i=1}^n R(\theta,h_n^{2\rho_1},X_{t_{i-1}^n})+\frac{1}{n}\sum_{i=1}^n R(\theta,h_n,X_{t_{i-1}^n})\\
&\overset{P}{\to} 0.
\end{align*}
This implies the second term converges to 0 in probability.
The third term is the same. 
For the fourth term,  we see from \eqref{under:rep2} and \eqref{under:rep1} that
\begin{align*}
&\left|\sum_{i=1}^{n}{\mathbb{E}\left[\frac{1}{n}\psi_i^{m_1}(\theta_0)\psi_i^{m_1'}(\theta_0)\ |\mathcal{F}_{i-1}^n\right]}\right|\\
&\le \frac{1}{n}\sum_{i=1}^n\left|\mathbb{E}\left[(b_{i-1})^\top\partial_{{\alpha_{m_1}}}S_{i-1}^{-1}\bar{X}_{i,n}\cdot (b_{i-1})^\top\partial_{{\alpha_{m_1'}}}S_{i-1}^{-1}\bar{X}_{i,n}\boldsymbol{1}_{\{|\Delta X_i^n|\le D_1h_n^{\rho_1}\}}\ |\mathcal{F}_{i-1}^n\right]\right.\\
&\quad +\mathbb{E}\left[(b_{i-1})^\top\partial_{{\alpha_{m_1}}}S_{i-1}^{-1}\bar{X}_{i,n}\cdot \frac{h_n}{2}(b_{i-1})^\top\partial_{{\alpha_{m_1'}}}S_{i-1}^{-1}b_{i-1}\boldsymbol{1}_{\{|\Delta X_i^n|\le D_1h_n^{\rho_1}\}}\ |\mathcal{F}_{i-1}^n\right]\\
&\quad +\mathbb{E}\left[\frac{h_n}{2}(b_{i-1})^\top\partial_{{\alpha_{m_1}}}S_{i-1}^{-1}b_{i-1}\cdot (b_{i-1})^\top\partial_{{\alpha_{m_1'}}}S_{i-1}^{-1}\bar{X}_{i,n}\boldsymbol{1}_{\{|\Delta X_i^n|\le D_1h_n^{\rho_1}\}}\ |\mathcal{F}_{i-1}^n\right]\\
&\quad +\left.\mathbb{E}\left[\frac{h_n}{2}(b_{i-1})^\top\partial_{{\alpha_{m_1}}}S_{i-1}^{-1}b_{i-1}\cdot \frac{h_n}{2}(b_{i-1})^\top\partial_{{\alpha_{m_1'}}}S_{i-1}^{-1}b_{i-1}\boldsymbol{1}_{\{|\Delta X_i^n|\le D_1h_n^{\rho_1}\}}\ |\mathcal{F}_{i-1}^n\right]\right|\\
&\le \frac{1}{n}\sum_{i=1}^n\sum_{k_1,k_2=1}^d\sum_{k_3,k_4=1}^d R(\theta,1,X_{t_{i-1}}^n)\left|\mathbb{E}\left[\bar{X}_{i,n}^{(k_2)}\bar{X}_{i,n}^{(k_4)}\boldsymbol{1}_{\{|\Delta X_i^n|\le D_1h_n^{\rho_1}\}} \ |\mathcal{F}_{i-1}^n\right]\right|\\
&\quad +\frac{h_n}{2n}\sum_{i=1}^n\sum_{k_1,k_2=1}^d\sum_{k_3,k_4=1}^d R(\theta,1,X_{t_{i-1}}^n)\left|\mathbb{E}\left[\bar{X}_{i,n}^{(k_2)}\boldsymbol{1}_{\{|\Delta X_i^n|\le D_1h_n^{\rho_1}\}} \ |\mathcal{F}_{i-1}^n\right]\right|\\
&\quad +\frac{h_n}{2n}\sum_{i=1}^n\sum_{k_1,k_2=1}^d\sum_{k_3,k_4=1}^d R(\theta,1,X_{t_{i-1}}^n)\left|\mathbb{E}\left[\bar{X}_{i,n}^{(k_4)}\boldsymbol{1}_{\{|\Delta X_i^n|\le D_1h_n^{\rho_1}\}} \ |\mathcal{F}_{i-1}^n\right]\right|\\
&\quad +\frac{h_n^2}{4n}\sum_{i=1}^n\sum_{k_1,k_2=1}^d\sum_{k_3,k_4=1}^d R(\theta,1,X_{t_{i-1}}^n)P(|\Delta X_i^n|\le D_1h_n^{\rho_1}\ |\mathcal{F}_{i-1}^n)\\
&\le \frac{1}{n}\sum_{i=1}^n R(\theta,h_n,X_{t_{i-1}^n})+\frac{1}{n}\sum_{i=1}^n R(\theta,h_n^{2+2\rho_1},X_{t_{i-1}^n})+\frac{1}{n}\sum_{i=1}^n R(\theta,h_n^{2+2\rho_1},X_{t_{i-1}^n})++\frac{1}{n}\sum_{i=1}^n R(\theta,h_n^{2},X_{t_{i-1}^n})\\
&\overset{P}{\to} 0.
\end{align*}
This implies the fourth term converges to 0 in probability, and we obtain \eqref{MCLT3-jump:col}.

{\bf Proof of \eqref{MCLT4-jump:col}.}
From \eqref{MCLT4-jump}, \eqref{MCLT1-jump:col:eq1} and the proof of \eqref{MCLT1-jump}, it is easy to show that
\begin{align*}
&\sum_{i=1}^{n}{\mathbb{E}\left[\frac{1}{\sqrt{n}}\left(\xi_i^{m_1}(\alpha_0)+\psi_i^{m_1}(\theta_0)\right)\ |\mathcal{F}_{i-1}^n\right]\mathbb{E}\left[\frac{1}{\sqrt{n}}\left(\xi_i^{m_1'}(\alpha_0)+\psi_i^{m_1'}(\theta_0)\right)\ |\mathcal{F}_{i-1}^n\right]}\\
&=\sum_{i=1}^{n}{\mathbb{E}\left[\frac{1}{\sqrt{n}}\xi_i^{m_1}(\alpha_0)\ |\mathcal{F}_{i-1}^n\right]\mathbb{E}\left[\frac{1}{\sqrt{n}}\xi_i^{m_1'}(\alpha_0)\ |\mathcal{F}_{i-1}^n\right]}\\
&\quad +\sum_{i=1}^{n}{\mathbb{E}\left[\frac{1}{\sqrt{n}}\xi_i^{m_1}(\alpha_0)\ |\mathcal{F}_{i-1}^n\right]\mathbb{E}\left[\frac{1}{\sqrt{n}}\psi_i^{m_1'}(\alpha_0)\ |\mathcal{F}_{i-1}^n\right]}\\
&\quad +\sum_{i=1}^{n}{\mathbb{E}\left[\frac{1}{\sqrt{n}}\psi_i^{m_1}(\alpha_0)\ |\mathcal{F}_{i-1}^n\right]\mathbb{E}\left[\frac{1}{\sqrt{n}}\xi_i^{m_1'}(\alpha_0)\ |\mathcal{F}_{i-1}^n\right]}\\
&\quad +\sum_{i=1}^{n}{\mathbb{E}\left[\frac{1}{\sqrt{n}}\psi_i^{m_1}(\alpha_0)\ |\mathcal{F}_{i-1}^n\right]\mathbb{E}\left[\frac{1}{\sqrt{n}}\psi_i^{m_1'}(\alpha_0)\ |\mathcal{F}_{i-1}^n\right]}\\
&= o_p(1)+\frac{1}{n}\sum_{i=1}^n R\left(\theta,\sqrt{nh_n^{3+\delta}},X_{t_{i-1}^n}\right)+\frac{1}{n}\sum_{i=1}^n R\left(\theta,\sqrt{nh_n^{3+\delta}},X_{t_{i-1}^n}\right)+\frac{1}{n}\sum_{i=1}^n R(\theta,h_n^2,X_{t_{i-1}^n})\\
&\overset{P}{\to} 0.
\end{align*}

{\bf Proof of \eqref{MCLT7-jump:col}.}
By simple computation, 
\begin{align*}
&\sum_{i=1}^{n}{\mathbb{E}\left[\frac{1}{n\sqrt{h_n}}\left(\xi_i^{m_1}(\alpha_0)+\psi_i^{m_1}(\theta_0)\right)\left(\eta_{i,1}^{m_2}(\beta_0|\alpha_0)+\eta_{i,2}^{m_2}(\beta_0)\right)\ |\mathcal{F}_{i-1}^n\right]}\\
&= \sum_{i=1}^{n}{\mathbb{E}\left[\frac{1}{n\sqrt{h_n}}\xi_i^{m_1}(\alpha_0)\left(\eta_{i,1}^{m_2}(\beta_0|\alpha_0)+\eta_{i,2}^{m_2}(\beta_0)\right)\ |\mathcal{F}_{i-1}^n\right]}\\
&\quad +\sum_{i=1}^{n}{\mathbb{E}\left[\frac{1}{n\sqrt{h_n}}\psi_i^{m_1}(\theta_0)\left(\eta_{i,1}^{m_2}(\beta_0|\alpha_0)+\eta_{i,2}^{m_2}(\beta_0)\right)\ |\mathcal{F}_{i-1}^n\right]},
\end{align*}
and using \eqref{MCLT7-jump}, the right-hand side converges to 0 in probability. Therefore, it is sufficient to show that 
\begin{align*}
\sum_{i=1}^{n}{\mathbb{E}\left[\frac{1}{n\sqrt{h_n}}\psi_i^{m_1}(\theta_0)\left(\eta_{i,1}^{m_2}(\beta_0|\alpha_0)+\eta_{i,2}^{m_2}(\beta_0)\right)\ |\mathcal{F}_{i-1}^n\right]}\overset{P}{\to}0.
\end{align*}
First,
\begin{align*}
\left|\sum_{i=1}^{n}{\mathbb{E}\left[\frac{1}{n\sqrt{h_n}}\psi_i^{m_1}(\theta_0)\left(\eta_{i,1}^{m_2}(\beta_0|\alpha_0)+\eta_{i,2}^{m_2}(\beta_0)\right)\ |\mathcal{F}_{i-1}^n\right]}\right|\le \sum_{i=1}^3 I_i^n,
\end{align*}
where 
\begin{align*}
I_1^n&:= \frac{1}{n\sqrt{h_n}}\sum_{i=1}^n\left|\mathbb{E}\left[\psi_i^{m_1}(\theta_0)(\partial_{\beta_{m_2}}b_{i-1})^\top S_{i-1}^{-1}\bar{X}_{i,n}\boldsymbol{1}_{\left\{|\Delta X_i^n|\le D_1h_n^{\rho_1}\right\}}\ |\mathcal{F}_{i-1}^n\right]\right|,\\
I_2^n&:= \frac{1}{n\sqrt{h_n}}\sum_{i=1}^n\left|\mathbb{E}\left[\psi_i^{m_1}(\theta_0)\left(\partial_{\beta_{m_2}}\log \Psi_{\beta_0}(\Delta X_i^n,X_{t_{i-1}^n})\right)\varphi_n(X_{t_{i-1}^n},\Delta X_i^n)
\boldsymbol{1}_{\left\{|\Delta X_i^n|> D_2h_n^{\rho_2}\right\}}\ |\mathcal{F}_{i-1}^n\right]\right|,\\
I_3^n&:= \frac{\sqrt{h_n}}{n}\sum_{i=1}^n\int_{B}\left|\partial_{{\beta_{m_2}}}\Psi_{\beta_0}(y,X_{t_{i-1}^n})\right|\varphi_n(X_{t_{i-1}^n},y)dy\left|\mathbb{E}\left[\psi_i^{m_1}(\theta_0)\ |\mathcal{F}_{i-1}^n\right]\right|.
\end{align*}
For $I_1^n$,  it holds from \eqref{under:rep1}, \eqref{under:rep2} that
\begin{align*}
I_1^n&=\frac{1}{n\sqrt{h_n}}\sum_{i=1}^n\left|\mathbb{E}\left[\psi_i^{m_1}(\theta_0)(\partial_{\beta_{m_2}}b_{i-1})^\top S_{i-1}^{-1}\bar{X}_{i,n}\boldsymbol{1}_{\left\{|\Delta X_i^n|\le D_1h_n^{\rho_1}\right\}}\ |\mathcal{F}_{i-1}^n\right]\right|\\
&\le \frac{1}{n\sqrt{h_n}}\sum_{i=1}^n\left|\mathbb{E}\left[(b_{i-1})^\top\partial_{\alpha_{m_1}}S_{i-1}^{-1}\bar{X}_{i,n}(\partial_{\beta_{m_2}}b_{i-1})^\top S_{i-1}^{-1}\bar{X}_{i,n}\boldsymbol{1}_{\left\{|\Delta X_i^n|\le D_1h_n^{\rho_1}\right\}}\ |\mathcal{F}_{i-1}^n\right]\right|\\
&\quad+\frac{\sqrt{h_n}}{2n}\sum_{i=1}^n\left|\mathbb{E}\left[(b_{i-1})^\top\partial_{\alpha_{m_1}}S_{i-1}^{-1}b_{i-1}(\partial_{\beta_{m_2}}b_{i-1})^\top S_{i-1}^{-1}\bar{X}_{i,n}\boldsymbol{1}_{\left\{|\Delta X_i^n|\le D_1h_n^{\rho_1}\right\}}\ |\mathcal{F}_{i-1}^n\right]\right|\\
&\le \frac{1}{n\sqrt{h_n}}\sum_{i=1}^n\sum_{k_1,k_2}^d\sum_{k_3,k_4}^d\left|b_{i-1}^{(k_1)}\right|\left|\partial_{\alpha_{m_1}}{S_{i-1}^{-1}}^{(k_1,k_2)}\right|\left|\partial_{\beta_{m_2}}b_{i-1}^{(k_3)}\right|\left|{S_{i-1}^{-1}}^{(k_3,k_4)}\right| \\
&\qquad\qquad\qquad\times\left|\mathbb{E}\left[\bar{X}_{i,n}^{(k_2)}\bar{X}_{i,n}^{(k_4)}\boldsymbol{1}_{\left\{|\Delta X_i^n|\le D_1h_n^{\rho_1}\right\}}\ |\mathcal{F}_{i-1}^n\right]\right|\\
&\quad+ \frac{\sqrt{h_n}}{2n}\sum_{i=1}^n\sum_{k_1,k_2}^d\sum_{k_3,k_4}^d\left|b_{i-1}^{(k_1)}\right|\left|\partial_{\alpha_{m_1}}{S_{i-1}^{-1}}^{(k_1,k_2)}\right|\left|b_{i-1}^{(k_2)}\right|\left|\partial_{\beta_{m_2}}b_{i-1}^{(k_3)}\right|\left|{S_{i-1}^{-1}}^{(k_3,k_4)}\right|\\ &\qquad\qquad\qquad\qquad\times\left|\mathbb{E}\left[\bar{X}_{i,n}^{(k_4)}\boldsymbol{1}_{\left\{|\Delta X_i^n|\le D_1h_n^{\rho_1}\right\}}\ |\mathcal{F}_{i-1}^n\right]\right|\\
&\le \frac{1}{n\sqrt{h_n}}\sum_{i=1}^n R(\theta,h_n,X_{t_{i-1}^n})+\frac{\sqrt{h_n}}{2n}\sum_{i=1}^n R(\theta,h_n^{1+2\rho_1},X_{t_{i-1}^n})\\
&\le O_p\left(\sqrt{h_n}\right)+O_p\left(h_n^{\frac{3}{2}+2\rho_1}\right).
\end{align*}
For $I_2^n$, we can calculate
\begin{align*}
I_2^n&=\frac{1}{n\sqrt{h_n}}\sum_{i=1}^n\left|\mathbb{E}\left[\psi_i^{m_1}(\theta_0)\left(\partial_{\beta_{m_2}}\log \Psi_{\beta_0}(\Delta X_i^n,X_{t_{i-1}^n})\right)\varphi_n(X_{t_{i-1}^n},\Delta X_i^n)
\boldsymbol{1}_{\left\{|\Delta X_i^n|> D_2h_n^{\rho_2}\right\}}\ |\mathcal{F}_{i-1}^n\right]\right|\\
&\le \frac{1}{n\sqrt{h_n}}\sum_{i=1}^n\left|\mathbb{E}\left[(b_{i-1})^\top\partial_{\alpha_{m_1}}S_{i-1}^{-1}\bar{X}_{i,n}\left(\partial_{\beta_{m_2}}\log \Psi_{\beta_0}(\Delta X_i^n,X_{t_{i-1}^n})\right)\varphi_n(X_{t_{i-1}^n},\Delta X_i^n)\right.\right.\\
&\qquad\qquad\qquad\times\left.\left.\boldsymbol{1}_{\left\{|\Delta X_i^n|\le D_1h_n^{\rho_1}\right\}}\boldsymbol{1}_{\left\{|\Delta X_i^n|> D_2h_n^{\rho_2}\right\}}\ |\mathcal{F}_{i-1}^n\right]\right|\\
&\quad+\frac{\sqrt{h_n}}{2n}\sum_{i=1}^n\left|\mathbb{E}\left[(b_{i-1})^\top\partial_{\alpha_{m_1}}S_{i-1}^{-1}b_{i-1}\left(\partial_{\beta_{m_2}}\log \Psi_{\beta_0}(\Delta X_i^n,X_{t_{i-1}^n})\right)\varphi_n(X_{t_{i-1}^n},\Delta X_i^n)\right.\right.\\
&\qquad\qquad\qquad\times\left.\left.\boldsymbol{1}_{\left\{|\Delta X_i^n|\le D_1h_n^{\rho_1}\right\}}\boldsymbol{1}_{\left\{|\Delta X_i^n|> D_2h_n^{\rho_2}\right\}}\ |\mathcal{F}_{i-1}^n\right]\right|\\
&\le \frac{1}{n\sqrt{h_n}}\sum_{i=1}^n\sum_{k_1,k_2=1}^d\left|b_{i-1}^{(k_1)}\right|\left|\partial_{\alpha_{m_1}}{S_{i-1}^{-1}}^{(k_1,k_2)}\right|
\mathbb{E}\left[\left|\bar{X}_{i,n}^{(k_2)}\right|\left|\partial_{\beta_{m_2}}\log \Psi_{\beta_0}(\Delta X_i^n,X_{t_{i-1}^n})\right|\varphi_n(X_{t_{i-1}^n},\Delta X_i^n)\right.\\
&\qquad\qquad\qquad\times\left.\boldsymbol{1}_{\left\{|\Delta X_i^n|\le D_1h_n^{\rho_1}\right\}}\boldsymbol{1}_{\left\{|\Delta X_i^n|> D_2h_n^{\rho_2}\right\}}\ |\mathcal{F}_{i-1}^n\right]\\
&\quad+\frac{\sqrt{h_n}}{2n}\sum_{i=1}^n\sum_{k_1,k_2=1}^d\left|b_{i-1}^{(k_1)}\right|\left|\partial_{\alpha_{m_1}}{S_{i-1}^{-1}}^{(k_1,k_2)}\right|\left|b_{i-1}^{(k_2)}\right|\mathbb{E}\left[\left|\partial_{\beta_{m_2}}\log \Psi_{\beta_0}(\Delta X_i^n,X_{t_{i-1}^n})\right|\varphi_n(X_{t_{i-1}^n},\Delta X_i^n)\right.\\
&\qquad\qquad\qquad\times\left.\boldsymbol{1}_{\left\{|\Delta X_i^n|\le D_1h_n^{\rho_1}\right\}}\boldsymbol{1}_{\left\{|\Delta X_i^n|> D_2h_n^{\rho_2}\right\}}\ |\mathcal{F}_{i-1}^n\right].
\end{align*}
For the first term on the right-hand side, by noticing that
$\left|\bar{X}_{i,n}(\beta_0)\right|\boldsymbol{1}_{\left\{|\Delta X_i^n|\le D_1h_n^{\rho_1}\right\}}=R(\theta,h_n^{\rho_1},X_{t_{i-1}^n})\boldsymbol{1}_{\left\{|\Delta X_i^n|\le D_1h_n^{\rho_1}\right\}}$ and using Proposition \ref{prop2:jump}, it follows that
\begin{align*}
&\frac{1}{n\sqrt{h_n}}\sum_{i=1}^n\sum_{k_1,k_2=1}^d\left|b_{i-1}^{(k_1)}\right|\left|\partial_{\alpha_{m_1}}{S_{i-1}^{-1}}^{(k_1,k_2)}\right|
\mathbb{E}\left[\left|\bar{X}_{i,n}^{(k_2)}\right|\left|\partial_{\beta_{m_2}}\log \Psi_{\beta_0}(\Delta X_i^n,X_{t_{i-1}^n})\right|\varphi_n(X_{t_{i-1}^n},\Delta X_i^n)\right.\\
&\qquad\qquad\qquad\times\left.\boldsymbol{1}_{\left\{|\Delta X_i^n|\le D_1h_n^{\rho_1}\right\}}\boldsymbol{1}_{\left\{|\Delta X_i^n|> D_2h_n^{\rho_2}\right\}}\ |\mathcal{F}_{i-1}^n\right]\\
&\quad\le\begin{cases}
\displaystyle
\frac{1}{n\sqrt{h_n}}\sum_{i=1}^{n}\sum_{k_1,k_2=1}^d R(\theta,1,X_{t_{i-1}^n}) \\
\qquad\qquad\qquad\times\mathbb{E}\left[\left|\bar{X}_{i,n}^{k_2}\right|(1+|\Delta X_i^n|)^C\boldsymbol{1}_{\{|\Delta X_i^n|\le D_1h_n^{\rho_1}\}}\boldsymbol{1}_{\{|\Delta X_i^n|> D_2h_n^{\rho_2}\}} \ |\mathcal{F}_{i-1}^n\right]   &(\text{under}\ [\textbf{C}_2\textbf{1}])\\
\displaystyle
\frac{1}{n\sqrt{h_n}}\sum_{i=1}^{n}\sum_{k_1,k_2=1}^d R(\theta,\varepsilon_n^{-1},X_{t_{i-1}^n}) \mathbb{E}\left[\left|\bar{X}_{i,n}^{k_2}\right|\boldsymbol{1}_{\{|\Delta X_i^n|\le D_1h_n^{\rho_1}\}}\boldsymbol{1}_{\{|\Delta X_i^n|> D_2h_n^{\rho_2}\}} \ |\mathcal{F}_{i-1}^n\right]    &(\text{under}\ [\textbf{C}_2\textbf{2}])
\end{cases}\\
&\quad\le\begin{cases}
\displaystyle
\frac{1}{n\sqrt{h_n}}\sum_{i=1}^{n}\sum_{k_1,k_2=1}^d R(\theta,h_n^{\rho_1},X_{t_{i-1}^n}) \\
\qquad\qquad\qquad\times\mathbb{E}\left[(1+D_1h_n^{\rho_1})^C\boldsymbol{1}_{\{|\Delta X_i^n|\le D_1h_n^{\rho_1}\}}\boldsymbol{1}_{\{|\Delta X_i^n|> D_2h_n^{\rho_2}\}} \ |\mathcal{F}_{i-1}^n\right]   &(\text{under}\ [\textbf{C}_2\textbf{1}])\\
\displaystyle
\frac{1}{n\sqrt{h_n}}\sum_{i=1}^{n}\sum_{k_1,k_2=1}^d R(\theta,h_n^{\rho_1}\varepsilon_n^{-1},X_{t_{i-1}^n})\\
\qquad\qquad\qquad\times P\left(\{|\Delta X_i^n|\le D_1h_n^{\rho_1}\}\cap\{|\Delta X_i^n|> D_2h_n^{\rho_2}\}\ |\mathcal{F}_{i-1}^n\right)    &(\text{under}\ [\textbf{C}_2\textbf{2}])
\end{cases}\\
&\quad\le\begin{cases}
\displaystyle
\frac{1}{n\sqrt{h_n}}\sum_{i=1}^{n}\sum_{k_1,k_2=1}^d R(\theta,h_n^{\rho_1},X_{t_{i-1}^n}) \\
\qquad\qquad\qquad\times P\left(\{|\Delta X_i^n|\le D_1h_n^{\rho_1}\}\cap\{|\Delta X_i^n|> D_2h_n^{\rho_2}\}\ |\mathcal{F}_{i-1}^n\right)   &(\text{under}\ [\textbf{C}_2\textbf{1}])\\
\displaystyle
\frac{1}{n\sqrt{h_n}}\sum_{i=1}^{n}\sum_{k_1,k_2=1}^d R(\theta,h_n^{\rho_1}\varepsilon_n^{-1},X_{t_{i-1}^n})\\
\qquad\qquad\qquad\times P\left(\{|\Delta X_i^n|\le D_1h_n^{\rho_1}\}\cap\{|\Delta X_i^n|> D_2h_n^{\rho_2}\}\ |\mathcal{F}_{i-1}^n\right)    &(\text{under}\ [\textbf{C}_2\textbf{2}])
\end{cases}\\
&\quad\le\begin{cases}
\displaystyle
\frac{1}{n\sqrt{h_n}}\sum_{i=1}^{n}\sum_{k_1,k_2=1}^d R\left(\theta,h_n^{1+2\rho_1},X_{t_{i-1}^n}\right) & (\text{under}\ [\textbf{C}_2\textbf{1}])\\
\displaystyle
\frac{1}{n\sqrt{h_n}}\sum_{i=1}^{n}\sum_{k_1,k_2=1}^d R\left(\theta,h_n^{1+2\rho_1}\varepsilon_n^{-1},X_{t_{i-1}^n}\right)  & (\text{under}\ [\textbf{C}_2\textbf{2}])
\end{cases}\\
&\quad\le\begin{cases}
O_p\left(h_n^{\frac{1}{2}+2\rho_1}\right)& (\text{under}\ [\textbf{C}_2\textbf{1}]),\\
O_p\left(h_n^{\frac{1}{2}+2\rho_1}\varepsilon_n^{-1}\right)& (\text{under}\ [\textbf{C}_2\textbf{2}]).
\end{cases}
\end{align*}
For the second term,  we see from \ref{prop2:jump} that 
\begin{align*}
&\frac{\sqrt{h_n}}{2n}\sum_{i=1}^n\sum_{k_1,k_2=1}^d\left|b_{i-1}^{(k_1)}\right|\left|\partial_{\alpha_{m_1}}{S_{i-1}^{-1}}^{(k_1,k_2)}\right|\left|b_{i-1}^{(k_2)}\right|\mathbb{E}\left[\left|\partial_{\beta_{m_2}}\log \Psi_{\beta_0}(\Delta X_i^n,X_{t_{i-1}^n})\right|\varphi_n(X_{t_{i-1}^n},\Delta X_i^n)\right.\\
&\qquad\qquad\qquad\qquad\times\left.\boldsymbol{1}_{\left\{|\Delta X_i^n|\le D_1h_n^{\rho_1}\right\}}\boldsymbol{1}_{\left\{|\Delta X_i^n|> D_2h_n^{\rho_2}\right\}}\ |\mathcal{F}_{i-1}^n\right]\\
&\quad\le\begin{cases}
\displaystyle
\frac{\sqrt{h_n}}{2n}\sum_{i=1}^{n}\sum_{k_1,k_2=1}^d R(\theta,1,X_{t_{i-1}^n}) \\
\qquad\qquad\qquad\times\mathbb{E}\left[(1+|\Delta X_i^n|)^C\boldsymbol{1}_{\{|\Delta X_i^n|\le D_1h_n^{\rho_1}\}}\boldsymbol{1}_{\{|\Delta X_i^n|> D_2h_n^{\rho_2}\}} \ |\mathcal{F}_{i-1}^n\right]   &(\text{under}\ [\textbf{C}_2\textbf{1}])\\
\displaystyle
\frac{\sqrt{h_n}}{2n}\sum_{i=1}^{n}\sum_{k_1,k_2=1}^d R(\theta,\varepsilon_n^{-1},X_{t_{i-1}^n}) P\left(\{|\Delta X_i^n|\le D_1h_n^{\rho_1}\}\cap\{|\Delta X_i^n|> D_2h_n^{\rho_2}\}\ |\mathcal{F}_{i-1}^n\right)    &(\text{under}\ [\textbf{C}_2\textbf{2}])
\end{cases}\\
&\quad\le\begin{cases}
\displaystyle
\frac{\sqrt{h_n}}{2n}\sum_{i=1}^{n}\sum_{k_1,k_2=1}^d R(\theta,1,X_{t_{i-1}^n}) \\
\qquad\qquad\qquad\times\mathbb{E}\left[(1+D_1h_n^{\rho_1})^C\boldsymbol{1}_{\{|\Delta X_i^n|\le D_1h_n^{\rho_1}\}}\boldsymbol{1}_{\{|\Delta X_i^n|> D_2h_n^{\rho_2}\}} \ |\mathcal{F}_{i-1}^n\right]   &(\text{under}\ [\textbf{C}_2\textbf{1}])\\
\displaystyle
\frac{\sqrt{h_n}}{2n}\sum_{i=1}^{n}\sum_{k_1,k_2=1}^d R(\theta,\varepsilon_n^{-1},X_{t_{i-1}^n}) P\left(\{|\Delta X_i^n|\le D_1h_n^{\rho_1}\}\cap\{|\Delta X_i^n|> D_2h_n^{\rho_2}\}\ |\mathcal{F}_{i-1}^n\right)    &(\text{under}\ [\textbf{C}_2\textbf{2}])
\end{cases}\\
&\quad\le\begin{cases}
\displaystyle
\frac{\sqrt{h_n}}{2n}\sum_{i=1}^{n}\sum_{k_1,k_2=1}^d R(\theta,1,X_{t_{i-1}^n})P\left(\{|\Delta X_i^n|\le D_1h_n^{\rho_1}\}\cap\{|\Delta X_i^n|> D_2h_n^{\rho_2}\}\ |\mathcal{F}_{i-1}^n\right)   &(\text{under}\ [\textbf{C}_2\textbf{1}])\\
\displaystyle
\frac{\sqrt{h_n}}{2n}\sum_{i=1}^{n}\sum_{k_1,k_2=1}^d R(\theta,\varepsilon_n^{-1},X_{t_{i-1}^n}) P\left(\{|\Delta X_i^n|\le D_1h_n^{\rho_1}\}\cap\{|\Delta X_i^n|> D_2h_n^{\rho_2}\}\ |\mathcal{F}_{i-1}^n\right)    &(\text{under}\ [\textbf{C}_2\textbf{2}])
\end{cases}\\
&\quad\le\begin{cases}
\displaystyle
\frac{\sqrt{h_n}}{2n}\sum_{i=1}^{n}\sum_{k_1,k_2=1}^d R(\theta,h_n^{1+\rho_1},X_{t_{i-1}^n})   &(\text{under}\ [\textbf{C}_2\textbf{1}])\\
\displaystyle
\frac{\sqrt{h_n}}{2n}\sum_{i=1}^{n}\sum_{k_1,k_2=1}^d R(\theta,h_n^{1+\rho_1}\varepsilon_n^{-1},X_{t_{i-1}^n})     &(\text{under}\ [\textbf{C}_2\textbf{2}])
\end{cases}\\
&\quad\le\begin{cases}
\displaystyle
O_p\left(h_n^{\frac{3}{2}+\rho_1}\right)   &(\text{under}\ [\textbf{C}_2\textbf{1}]),\\
\displaystyle
O_p\left(h_n^{\frac{3}{2}+\rho_1}\varepsilon_n^{-1}\right)    &(\text{under}\ [\textbf{C}_2\textbf{2}]).
\end{cases}
\end{align*}
For $I_3^n$, it is obvious from \eqref{MCLT1-jump:col:eq1} that 
\begin{align*}
I_3^n&=\frac{\sqrt{h_n}}{n}\sum_{i=1}^n\int_{B}\left|\partial_{{\beta_{m_2}}}\Psi_{\beta_0}(y,X_{t_{i-1}^n})\right|\varphi_n(X_{t_{i-1}^n},y)dy\left|\mathbb{E}\left[\psi_i^{m_1}(\theta_0)\ |\mathcal{F}_{i-1}^n\right]\right|\\
&\le \frac{1}{n}\sum_{i=1}^nR(\theta,h_n^{\frac{3}{2}},X_{t_{i-1}^n})\\
&\le O_p\left(h_n^{\frac{3}{2}}\right).
\end{align*}
Finally, 
\begin{align*}
&\sum_{i=1}^{n}{\mathbb{E}\left[\frac{1}{n\sqrt{h_n}}\psi_i^{m_1}(\theta_0)\left(\eta_{i,1}^{m_2}(\beta_0|\alpha_0)+\eta_{i,2}^{m_2}(\beta_0)\right)\ |\mathcal{F}_{i-1}^n\right]}\le \sum_{i=1}^3 I_i^n\\
&\quad\le\begin{cases}
\displaystyle
O_p\left(\sqrt{h_n}\right)+O_p\left(h_n^{\frac{3}{2}+2\rho_1}\right)+O_p\left(h_n^{\frac{1}{2}+2\rho_1}\right)+O_p\left(h_n^{\frac{3}{2}+\rho_1}\right)+O_p\left(h_n^{\frac{3}{2}}\right)   &(\text{under}\ [\textbf{C}_2\textbf{1}])\\
\displaystyle
O_p\left(\sqrt{h_n}\right)+O_p\left(h_n^{\frac{3}{2}+2\rho_1}\right)+O_p\left(h_n^{\frac{1}{2}+2\rho_1}\varepsilon_n^{-1}\right)+O_p\left(h_n^{\frac{3}{2}+\rho_1}\varepsilon_n^{-1}\right)+O_p\left(h_n^{\frac{3}{2}}\right)
    &(\text{under}\ [\textbf{C}_2\textbf{2}])
\end{cases}\\
&\quad\overset{P}{\to}0.
\end{align*}
This implies \eqref{MCLT7-jump:col}.

{\bf Proof of \eqref{MCLT8-jump:col}.}
From \eqref{MCLT8-jump}, \eqref{MCLT1-jump:col:eq1} and the proof of \eqref{MCLT2-jump}, it is easy to show that
\begin{align*}
&\sum_{i=1}^{n}{\mathbb{E}\left[\frac{1}{\sqrt{n}}\left(\xi_i^{m_1}(\alpha_0)+\psi_i^{m_1}(\theta_0)\right)\ |\mathcal{F}_{i-1}^n\right]\mathbb{E}\left[\frac{1}{\sqrt{nh_n}}\left(\eta_{i,1}^{m_2}(\beta_0|\alpha_0)+\eta_{i,2}^{m_2}(\beta_0)\right)\ |\mathcal{F}_{i-1}^n\right]}\\
&= \sum_{i=1}^{n}{\mathbb{E}\left[\frac{1}{\sqrt{n}}\xi_i^{m_1}(\alpha_0)\ |\mathcal{F}_{i-1}^n\right]\mathbb{E}\left[\frac{1}{\sqrt{nh_n}}\left(\eta_{i,1}^{m_2}(\beta_0|\alpha_0)+\eta_{i,2}^{m_2}(\beta_0)\right)\ |\mathcal{F}_{i-1}^n\right]}\\
&\quad +\sum_{i=1}^{n}{\mathbb{E}\left[\frac{1}{\sqrt{n}}\psi_i^{m_1}(\theta_0)\ |\mathcal{F}_{i-1}^n\right]\mathbb{E}\left[\frac{1}{\sqrt{nh_n}}\left(\eta_{i,1}^{m_2}(\beta_0|\alpha_0)+\eta_{i,2}^{m_2}(\beta_0)\right)\ |\mathcal{F}_{i-1}^n\right]}\\
&= o_p(1)+\sum_{i=1}^n R\left(\theta,\frac{h_n}{\sqrt{n}},X_{t_{i-1}^n}\right)\left\{R\left(\theta,\frac{h_n\varepsilon_n^{-2}}{\sqrt{n}},X_{t_{i-1}^n}\right)+R\left(\theta,\sqrt{\frac{h_n^{1+4\rho_1}}{n}},X_{t_{i-1}^n}\right)\right\}+o_P(1)\\
&\overset{P}{\to} 0.
\end{align*}

{\bf Proof of \eqref{MCLT9-jump:col}.}
Let us note that
\begin{align*}
&\sum_{i=1}^{n}{\mathbb{E}\left[\left|\frac{1}{\sqrt{n}}\left(\xi_i^{m_1}(\alpha_0)+\psi_i^{m_1}(\theta_0)\right)\right|^{2+\nu_1}|\mathcal{F}_{i-1}^n\right]}\\
&\le C_{\nu_1}\sum_{i=1}^{n}{\mathbb{E}\left[\left|\frac{1}{\sqrt{n}}\xi_i^{m_1}(\alpha_0)\right|^{2+\nu_1}|\mathcal{F}_{i-1}^n\right]}+C_{\nu_1}\sum_{i=1}^{n}{\mathbb{E}\left[\left|\frac{1}{\sqrt{n}}\psi_i^{m_1}(\theta_0)\right|^{2+\nu_1}|\mathcal{F}_{i-1}^n\right]}.
\end{align*}
By using \eqref{MCLT9-jump}, the first term on the right-hand side converges to 0 in probability if $\frac{5\rho_1-1}{1-2\rho_1}\ge \nu_1>0$. Therefore, it is sufficient to show that for $\frac{5\rho_1-1}{1-2\rho_1}\ge \nu_1'>0$, 
\begin{align*}
\sum_{i=1}^{n}{\mathbb{E}\left[\left|\frac{1}{\sqrt{n}}\psi_i^{m_1}(\theta_0)\right|^{2+\nu_1'}|\mathcal{F}_{i-1}^n\right]}\overset{P}{\to}0.
\end{align*}
Since $|\bar{X}_{i,n}(\beta_0)|^{2+\nu_1'}\boldsymbol{1}_{\left\{|\Delta X_i^n|\le D_1h_n^{\rho_1}\right\}}=R(\theta,h_n^{(2+\nu_1')\rho_1},X_{t_{i-1}^n})\boldsymbol{1}_{\left\{|\Delta X_i^n|\le D_1h_n^{\rho_1}\right\}}$, it follows that
\begin{align*}
&\sum_{i=1}^{n}{\mathbb{E}\left[\left|\frac{1}{\sqrt{n}}\psi_i^{m_1}(\theta_0)\right|^{2+\nu_1'}|\mathcal{F}_{i-1}^n\right]}\\
&\le n^{-\frac{\nu_1'}{2}}\cdot\frac{1}{n}\sum_{i=1}^n\mathbb{E}\left[\left|(b_{i-1})^\top\partial_{\alpha_{m_1}}{S_{i-1}^{-1}}\bar{X}_{i,n}\right|^{2+\nu_1'}\boldsymbol{1}_{\{|\Delta X_i^n|\le D_1h_n^{\rho_1}\}} \ |\mathcal{F}_{i-1}^n\right]\\
&\quad +n^{-\frac{\nu_1'}{2}}\cdot\frac{h_n^{2+\nu_1'}}{2^{2+\nu_1'}n}\sum_{i=1}^n\left|(b_{i-1})^\top\partial_{\alpha_{m_1}}{S_{i-1}^{-1}}b_{i-1}\right|^{2+\nu_1'}P\left(|\Delta X_i^n|\le D_1h_n^{\rho_1}\} \ |\mathcal{F}_{i-1}^n\right)\\
&\le n^{-\frac{\nu_1'}{2}}\cdot\frac{1}{n}\sum_{i=1}^n\sum_{k_1,k_2=1}^d R(\theta,1,X_{t_{i-1}^n})\mathbb{E}\left[\left|\bar{X}_{i,n}^{(k_2)}\right|^{2+\nu_1'}\boldsymbol{1}_{\{|\Delta X_i^n|\le D_1h_n^{\rho_1}\}} \ |\mathcal{F}_{i-1}^n\right]\\
&\quad + \frac{1}{n}\sum_{i=1}^n R(\theta,n^{-\frac{\nu_1'}{2}}h_n^{2+\nu_1'},X_{t_{i-1}^n})\\
&\le \frac{1}{n}\sum_{i=1}^n\sum_{k_1,k_2=1}^d R(\theta,n^{-\frac{\nu_1'}{2}}h_n^{(2+\nu_1')\rho_1},X_{t_{i-1}^n})P(|\Delta X_i^n|\le D_1h_n^{\rho_1}\ |\mathcal{F}_{i-1}^n)\\
&\quad + \frac{1}{n}\sum_{i=1}^n R(\theta,n^{-\frac{\nu_1'}{2}}h_n^{2+\nu_1'},X_{t_{i-1}^n})\\
&\le \frac{1}{n}\sum_{i=1}^n\sum_{k_1,k_2=1}^d R(\theta,n^{-\frac{\nu_1'}{2}}h_n^{(2+\nu_1')\rho_1},X_{t_{i-1}^n}) + \frac{1}{n}\sum_{i=1}^n R(\theta,n^{-\frac{\nu_1'}{2}}h_n^{2+\nu_1'},X_{t_{i-1}^n}).
\end{align*}
Since for $\nu_1'>0$,
\begin{align*}
\sum_{i=1}^{n}{\mathbb{E}\left[\left|\frac{1}{\sqrt{n}}\psi_i^{m_1}(\theta_0)\right|^{2+\nu_1'}|\mathcal{F}_{i-1}^n\right]}\overset{P}{\to}0,
\end{align*}
 one has \eqref{MCLT9-jump:col} for $\frac{5\rho_1-1}{1-2\rho_1}\ge \nu_1>0$. This ends the proof of \eqref{col2:goal4-jump}.

In a similar way to the proof of Theorem \ref{thm2-2:jump}, it holds from \eqref{col2:goal1-jump}, \eqref{col2:goal2-jump}, \eqref{col2:goal3-jump}, \eqref{col2:goal4-jump} that $\hat{S}_n\overset{d}{\to}N_{p+q}(0,I(\theta_0)^{-1})$, and this completes the proof.
\end{proof}
\subsection{Proof of Chapter \ref{sec:test}}
Let  $D_n$  be a $(p+q)\times(p+q)$-matrix defined as
\begin{align*}
D_n:=\begin{pmatrix}
    \sqrt{n}E_p&O\\
    O&\sqrt{nh_n}E_q
\end{pmatrix},
\end{align*}
where $E_k$ denotes the $k\times k$ identity matrix.
Partition $I_a(\alpha;\alpha_0)$ and $I_{b,c}(\theta;\theta_0)$ into  four matrices as follows:
\begin{align*}
&I_a(\alpha;\alpha_0)=\begin{pmatrix}
    I_{a,1}(\alpha;\alpha_0)&I_{a,2}(\alpha;\alpha_0)\\
    I_{a,2}(\alpha;\alpha_0)^\top&I_{a,3}(\alpha;\alpha_0)
\end{pmatrix}, & I_{b,c}(\theta;\theta_0)=\begin{pmatrix}
    I_{b,c,1}(\theta;\theta_0)&I_{b,c,2}(\theta;\theta_0)\\
    I_{b,c,2}(\theta;\theta_0)^\top&I_{b,c,3}(\theta;\theta_0)
\end{pmatrix},
\end{align*}
where $I_{a,1}(\alpha;\alpha_0), I_{a,2}(\alpha;\alpha_0)$ and $I_{a,3}(\alpha;\alpha_0)$ are the $k\times k$, the $k\times(p-k)$ and the $(p-k)\times(p-k)$ matrices, respectively, and $I_{b,c,1}(\theta;\theta_0), I_{b,c,2}(\theta;\theta_0)$ and $I_{b,c,3}(\theta;\theta_0)$ are the $l\times l$, the $l\times(q-l)$ and the $(q-l)\times(q-l)$ matrices, respectively.
Moreover, let $H$ be a $(p+q)\times(p+q)$-matrix as follows:
\begin{align*}
H=\begin{pmatrix}
    O&O&O&O\\
    O&I_{a,3}(\theta_0;\theta_0)^{-1}&O&O\\
    O&O&O&O\\
    O&O&O&I_{b,c,3}(\theta_0;\theta_0)^{-1}
\end{pmatrix}
\end{align*}
\subsubsection{Preliminaries}

\ \  If a random variable $Y$ has a $p$-dimensional normal distribution with mean $0$ and covariance matrix $I$,
we write $Y \sim N_{p}(0,I)$. 

\begin{lemma}\label{lemma:H0test}
Assume {\bf{[A1]}}-{\bf{[A12]}}, {\bf{[B2]}}, {\bf{[B3]}}, {[{\bf T1}]} and either {[$\textbf{D}_2 \textbf{1}$]} or {[$\textbf{D}_2 \textbf{2}$]} of Corollary \ref{thm2-1:jump}. 
Then, under $H_0$, $D_n^{\frac{1}{2}}(\hat{\theta}_n-\hat{\theta}^*_n)\overset{d}{\to}\left(I(\theta_0;\theta_0)^{-1}-H\right)Y$,
where 
$Y \sim N_{p+q}(0,I(\theta_0;\theta_0))$.
\end{lemma}
\begin{proof}
We discuss  $\hat{\theta}_n$ and  $\hat{\theta}_n^*$ under $H_0$. 
By using  Corollaries \ref{thm1-1:jump} and \ref{thm2-1:jump}, 
\begin{align}\label{lemma2:theta-pconv}
\hat{\theta}_n\overset{P}{\to}\theta_0,
\end{align}
\begin{align}\label{lemma2:theta-dconv}
D_n^{\frac{1}{2}}(\hat{\theta}_n-\theta_0)=O_p(1).
\end{align}
Next, we show that $\hat{\theta}_n^*\overset{P}{\to}\theta_0$. 
We define $\bar{\hat{\theta}}_n^*$ and $\bar{\theta}_0$ as follows:
\begin{align*}
\bar{\hat{\theta}}_n^* &= ({\hat{\alpha}_n}^{*(k+1)},\cdots,{\hat{\alpha}_n}^{*(p)},{\hat{\beta}_n}^{*(l+1)},\cdots,{\hat{\beta}_n}^{*(q)})^\top,\\
\bar{\theta}_0&= (\alpha_0^{(k+1)},\cdots,\alpha_0^{(p)},\beta_0^{(l+1)},\cdots,\beta_0^{(q)})^\top,
\end{align*}
where
\begin{align*}
\hat{\theta}_n^* &= (0,\cdots,0,{\hat{\alpha}_n}^{*(k+1)},\cdots,{\hat{\alpha}_n}^{*(p)},0,\cdots,0,{\hat{\beta}_n}^{*(l+1)},\cdots,{\hat{\beta}_n}^{*(q)})^\top,\\
\theta_0&= (0,\cdots,0,\alpha_0^{(k+1)},\cdots,\alpha_0^{(p)},0,\cdots,0,\beta_0^{(l+1)},\cdots,\beta_0^{(q)})^\top.
\end{align*}
Let $\bar{\Theta}_0:=\{\bar{\theta}\in\mathbb{R}^{(p+q)-(k+l)}\ |\ \exists\theta\in\Theta_0,\ \bar{\theta}= (\theta^{(k+1)},\cdots,\theta^{(p)},\theta^{(l+1)},\cdots,\theta^{(q)})^\top\}$ and define $U_n(\bar{\theta})$ as follows:
\begin{align*}
U_n\left((\alpha^{(k+1)},\cdots,\alpha^{(p)},\beta^{(l+1)},\cdots,\beta^{(q)})^\top\right):=l_n\left((0,\cdots,0,\alpha^{(k+1)},\cdots,\alpha^{(p)},0,\cdots,0,\beta^{(l+1)},\cdots,\beta^{(q)})^\top\right).
\end{align*}  
Then $U_n(\bar{\theta})$ can be regarded as a quasi-log likelihood function in $(p+q)-(k+l)$ dimensions. By the definition of $\bar{\hat{\theta}}_n^*, U_n(\bar{\theta})$, we have $\argsup_{\bar{\theta}\in\bar{\Theta}_0} U_n(\bar{\theta})=\bar{\hat{\theta}}_n^*$.
Therefore, by using Corollary \ref{thm1-1:jump}, it holds that $\bar{\hat{\theta}}_n^*\overset{P}{\to}\bar{\theta}_0$. This implies 
\begin{align*}
\hat{\alpha}_n^{*(i)}-\alpha_0^{(i)}\overset{P}{\to}0\ \ \ \ (k+1\le i\le p),\\
\hat{\beta}_n^{*(j)}-\beta_0^{(j)}\overset{P}{\to}0\ \ \ \ (l+1\le j\le q).
\end{align*}
Hence, we obtain
\begin{align}\label{lemma2:theta*-pconv}
\hat{\theta}_n^*\overset{P}{\to}\theta_0
\end{align}
since
\begin{align*}
\hat{\theta}_n^*-\theta_0=(0,\cdots,0,\hat{\alpha}_n^{*(k+1)}-\alpha_0^{(k+1)},\cdots,\hat{\alpha}_n^{*(p)}-\alpha_0^{(p)},0,\cdots,0,\hat{\beta}_n^{*(l+1)}-\beta_0^{(l+1)},\cdots,\hat{\beta}_n^{*(q)}-\beta_0^{(q)})^\top.
\end{align*}
Let $\bar{I}(\bar{\theta};\theta_0)$ be a $(p+q-k-l)\times(p+q-k-l)$-matrix as follows:
\begin{align*}
\bar{I}(\bar{\theta};\theta_0)=\begin{pmatrix}
    \bar{I}_{a,3}(\bar{\alpha};\alpha_0)&O\\
    O&\bar{I}_{b,c,3}(\bar{\theta};\theta_0)
\end{pmatrix},
\end{align*}
where
\begin{align*}
\bar{I}_{a,3}(\bar{\alpha};\alpha_0)&:=I_{a,3}((0,\cdots,0,\alpha^{(k+1)},\cdots,\alpha^{(p)})^\top;\alpha_0), \\
\bar{I}_{b,c,3}(\bar{\theta};\theta_0)&:=I_{b,c,3}((0,\cdots,0,\alpha^{(k+1)},\cdots,\alpha^{(p)},0,\cdots,0,\beta^{(l+1)},\cdots,\beta^{(q)})^\top;\theta_0).
\end{align*}
Since, from {\bf [A12]}, 
$I(\theta_0;\theta_0)$ is non-singular,  $\bar{I}(\bar{\theta}_0;\theta_0)$ is non-singular, too.
Then it follows from Corollary \ref{thm2-1:jump} that
\begin{align*}
    \begin{pmatrix}
        \sqrt{n}(\bar{\hat{\alpha}}_n^*-\bar{\alpha}_0)\\
        \sqrt{nh_n}(\bar{\hat{\beta}}_n^*-\bar{\beta}_0)
    \end{pmatrix}\overset{d}{\to}N_{p+q-k-l}(0,\bar{I}(\bar{\theta}_0;\theta_0)^{-1}).
\end{align*}
Since 
\begin{align*}
    \begin{pmatrix}
        \sqrt{n}(\bar{\hat{\alpha}}_n^*-\bar{\alpha}_0)\\
        \sqrt{nh_n}(\bar{\hat{\beta}}_n^*-\bar{\beta}_0)
    \end{pmatrix}=O_p(1),
\end{align*}
 one has
\begin{align}\label{lemma2:theta*-dconv}
    D_n^{\frac{1}{2}}(\hat{\theta}_n^*-\theta_0)=O_p(1).
\end{align}
Let  $\varepsilon_0$ be a positive constant such that $\{\theta\in\Theta_0\ ;\ |\theta-\theta_0|<\varepsilon_0\}\subset\mathrm{Int}(\Theta_0)$, then, from \eqref{lemma2:theta-pconv} and \eqref{lemma2:theta*-pconv}, there exists a real valued sequence $\varepsilon_n<\varepsilon_0$ which satisfies $P(\hat{A}_n\cap\hat{A}_n^*)\to 1$, where
$\hat{A}_n:=\{\omega\in\Omega\ |\ |\hat{\theta}_n(\omega)-\theta_0|<\varepsilon_n\}$ and $\hat{A}_n^*:=\{\omega\in\Omega\ |\ |\hat{\theta}_n^*(\omega)-\theta_0|<\varepsilon_n\}$. 
On the set $\hat{A}_n^*$, it follows from Taylor's theorem that 
\begin{align}\label{lemma2:eq1}
D_n^{-\frac{1}{2}}\partial_{\theta}l_n(\hat{\theta}_n^*)-D_n^{-\frac{1}{2}}\partial_\theta l_n(\theta_0)=D_n^{-\frac{1}{2}}\int_0^1 \partial_\theta^2 l_n(\theta_0+u(\hat{\theta}_n^*-\theta_0))duD_n^{-\frac{1}{2}} D_n^{\frac{1}{2}}(\hat{\theta}_n^*-\theta_0)\ \ \ \ (\omega\in \hat{A}_n^*).
\end{align}
We define 
\begin{align*}
I_n^{(1)}(\theta,\theta_0)&:=\begin{pmatrix}\displaystyle
\int_{0}^{1}{\frac{1}{n}\partial_{{\alpha}}^2 l_n(\theta_0+u(\theta-\theta_0))du}&\displaystyle\int_{0}^{1}{\frac{1}{n\sqrt{h_n}}\partial_{{\alpha\beta}}^2 l_n(\theta_0+u(\theta-\theta_0))du}\\
\displaystyle\int_{0}^{1}{\frac{1}{n\sqrt{h_n}}\partial_{{\beta\alpha}}^2 l_n(\theta_0+u(\theta-\theta_0))du}&\displaystyle
\int_{0}^{1}{\frac{1}{nh_n}\partial_{{\beta}}^2 l_n(\theta_0+u(\theta-\theta_0))du}
\end{pmatrix},
\end{align*}
then since $|\{\theta_0+u(\hat{\theta}_n^*-\theta_0)\}-\theta_0|< \varepsilon_n$ on the set $\hat{A}_n^*$,  one has
\begin{align}\label{lemma2:eq:In}
\left|I_n^{(1)}(\hat{\theta}_n^*,\theta_0)+I(\theta_0;\theta_0)\right|&\le \left|I_n^{(1)}(\hat{\theta}_n^*,\theta_0)+I(\theta_0+u(\hat{\theta}_n^*-\theta_0);\theta_0)\right|+\left|I(\theta_0+u(\hat{\theta}_n^*-\theta_0);\theta_0)-I(\theta_0;\theta_0)\right|\nonumber\\
&\le \sup_{\theta\in\Theta_0}|I_n^{(1)}(\theta,\theta)+I(\theta;\theta_0)|+\sup_{\theta:|\theta-\theta_0|<\varepsilon_n}\left|I(\theta;\theta_0)-I(\theta_0;\theta_0)\right|.
\end{align}
By using \eqref{a-2der:col}, \eqref{b-2der:col} and \eqref{ab-2der:col}, we have
\begin{align}\label{lemma2:Cn-pconv}
\sup_{\theta\in\Theta_0}|I_n^{(1)}(\theta,\theta)+I(\theta;\theta_0)|\overset{P}{\to}0,
\end{align}
and for all $\varepsilon>0$, it follows from \eqref{lemma2:theta*-pconv}, \eqref{lemma2:eq:In}, \eqref{lemma2:Cn-pconv} and continuity of $I(\theta;\theta_0)$ that
\begin{align*}
P\left(\left|{I}_n^{(1)}(\hat{\theta}_n^*,\theta_0)+I(\theta_0;\theta_0)\right|>\varepsilon\right)&\le P\left(\left\{\left|{I}_n^{(1)}(\hat{\theta}_n^*,\theta_0)+I(\theta_0;\theta_0)\right|>\varepsilon\right\}\cap\hat{A}_n^*\right)+P(\hat{A}_n^{*c})\\
&\le P\left(\left\{\sup_{\theta\in\Theta_0}|I_n^{(1)}(\theta,\theta)+I(\theta;\theta_0)|>\frac{\varepsilon}{2}\right\}\cap\hat{A}_n^*\right)\\
&\quad +P\left(\left\{\sup_{\theta:|\theta-\theta_0|<\varepsilon_n}\left|I(\theta;\theta_0)-I(\theta_0;\theta_0)\right|>\frac{\varepsilon}{2}\right\}\cap\hat{A}_n^*\right)+P(\hat{A}_n^{*c})\\
&\le P\left(\sup_{\theta\in\Theta_0}|I_n^{(1)}(\theta,\theta)+I(\theta;\theta_0)|>\frac{\varepsilon}{2}\right)\\
&\quad +P\left(\sup_{\theta:|\theta-\theta_0|<\varepsilon_n}\left|I(\theta;\theta_0)-I(\theta_0;\theta_0)\right|>\frac{\varepsilon}{2}\right)+P(\hat{A}_n^{*c})\\
&\to 0.
\end{align*}
This implies 
\begin{align}\label{lemma2:In-pconv}
{I}_n^{(1)}(\hat{\theta}_n^*,\theta_0)\overset{P}{\to}-I(\theta_0;\theta_0).
\end{align}
We can rewrite \eqref{lemma2:eq1} as follows:
\begin{align*}
D_n^{-\frac{1}{2}}\partial_{\theta}l_n(\hat{\theta}_n^*)-D_n^{-\frac{1}{2}}\partial_\theta l_n(\theta_0)&=-I(\theta_0;\theta_0)D_n^{\frac{1}{2}}(\hat{\theta}_n^*-\theta_0)+({I}_n^{(1)}(\hat{\theta}_n^*,\theta_0)+I(\theta_0;\theta_0))D_n^{\frac{1}{2}}(\hat{\theta}_n^*-\theta_0)\ \ \ \ (\omega\in \hat{A}_n^*).
\end{align*}
Let us call the second term on the right-hand side $\hat{T}_n^*$.  It follows from \eqref{lemma2:In-pconv} and \eqref{lemma2:theta*-dconv} that
\begin{align}\label{lemma2:eq2}
&D_n^{-\frac{1}{2}}\partial_{\theta}l_n(\hat{\theta}_n^*)-D_n^{-\frac{1}{2}}\partial_\theta l_n(\theta_0)=-I(\theta_0;\theta_0)D_n^{\frac{1}{2}}(\hat{\theta}_n^*-\theta_0)+\hat{T}_n^*\ \ \ \ (\omega\in\hat{A}_n^*)
\end{align}
and
\begin{align}\label{lemma2:Tn*-pconv}
\hat{T}_n^*\overset{P}{\to}0.
\end{align}
Set $H=\begin{pmatrix}
    H_a&O\\
    O&H_{b,c}
\end{pmatrix}$, where
$H_a=\begin{pmatrix}
    O&O\\
    O&I_{a,3}(\alpha_0;\alpha_0)^{-1}
\end{pmatrix}$, $H_{b,c}=\begin{pmatrix}
    O&O\\
    O&I_{b,c,3}(\theta_0;\theta_0)^{-1}
\end{pmatrix}$.
Since
\begin{align*}
D_n^{-\frac{1}{2}}\partial_\alpha l_n(\hat{\alpha}_n^*)=\begin{pmatrix}
    \frac{1}{\sqrt{n}}\partial_{\alpha_{1}}l_n(\hat{\alpha}_n^*)\\
    \vdots\\
    \frac{1}{\sqrt{n}}\partial_{\alpha_{k}}l_n(\hat{\alpha}_n^*)\\
    0\\
    \vdots\\
    0
\end{pmatrix},\ 
D_n^{-\frac{1}{2}}\partial_\beta l_n(\hat{\beta}_n^*)=\begin{pmatrix}
    \frac{1}{\sqrt{nh_n}}\partial_{\beta_{1}}l_n(\hat{\beta}_n^*)\\
    \vdots\\
    \frac{1}{\sqrt{nh_n}}\partial_{\beta_{k}}l_n(\hat{\beta}_n^*)\\
    0\\
    \vdots\\
    0
\end{pmatrix},
\end{align*}
we have
\begin{align}\label{lemma2:eq3}
HD_n^{-\frac{1}{2}}\partial_{\theta}l_n(\hat{\theta}_n^*)=\begin{pmatrix}
    \displaystyle H_a\partial_\alpha l_n(\hat{\alpha}_n^*)\\
    \displaystyle H_{b,c}\partial_\beta l_n(\hat{\beta}_n^*)
\end{pmatrix}=0\ \ \ \ (\omega\in\hat{A}_n^*).
\end{align}
Moreover, since  
\begin{align}\label{lemma2:eq4}
HI(\theta_0;\theta_0)&=\begin{pmatrix}
    O&O&O&O\\
    I_{a,3}(\alpha_0;\alpha_0)^{-1}I_{a,2}(\alpha_0;\alpha_0)^\top&E_{p-k}&O&O\\
    O&O&O&O\\
    O&O&I_{b,c,3}(\theta_0;\theta_0)^{-1}I_{b,c,2}(\theta_0;\theta_0)^\top&E_{q-l}
\end{pmatrix},
\end{align}
by simple computation,  one has
\begin{align}\label{lemma2:eq5}
HI(\theta_0;\theta_0)D_n^{\frac{1}{2}}(\hat{\theta}_n^*-\theta_0)=D_n^{\frac{1}{2}}(\hat{\theta}_n^*-\theta_0).
\end{align}
By using \eqref{lemma2:eq3} and \eqref{lemma2:eq5}, and rearranging
\eqref{lemma2:eq2}, we obtain
\begin{align}\label{lemma2:eq6}
D_n^{\frac{1}{2}}(\hat{\theta}_n^*-\theta_0)=HD_n^{-\frac{1}{2}}\partial_\theta l_n(\theta_0)+H\hat{T}_n^*\ \ \ (\omega\in\hat{A}_n^*).
\end{align}
On the set $\hat{A}_n$, since it follows from $\hat{\theta}_n\in\mathrm{Int}(\Theta)$ that $\partial_\theta l_n(\hat{\theta}_n)=0$, it holds from Taylor's theorem that 
\begin{align}\label{lemma2:eq7}
D_n^{\frac{1}{2}}(\hat{\theta}_n-\theta_0)=I(\theta_0;\theta_0)^{-1}D_n^{-\frac{1}{2}}\partial_\theta l_n(\theta_0)+I(\theta_0;\theta_0)^{-1}\hat{T}_n\ \ \ (\omega\in\hat{A}_n),
\end{align}
where $\hat{T}_n:=({I}_n^{(1)}(\hat{\theta}_n,\theta_0)+I(\theta_0;\theta_0))D_n^{\frac{1}{2}}(\hat{\theta}_n-\theta_0)$.
In an analogous manner to \eqref{lemma2:In-pconv},  one has
\begin{align*}
{I}_n^{(1)}(\hat{\theta}_n,\theta_0)\overset{P}{\to}-I(\theta_0;\theta_0).
\end{align*}
Hence, 
\begin{align}\label{lemma2:Tn-pconv}
\hat{T}_n\overset{P}{\to}0.    
\end{align}
By using \eqref{lemma2:eq6} and \eqref{lemma2:eq7}, it follows on the set $\hat{A}_n\cap\hat{A}_n^*$ that 
\begin{align*}
D_n^{\frac{1}{2}}(\hat{\theta}_n-\hat{\theta}_n^*)&= D_n^{\frac{1}{2}}(\hat{\theta}_n-\theta_0)+D_n^{\frac{1}{2}}(\hat{\theta}_n^*-\theta_0)\\
&=I(\theta_0;\theta_0)^{-1}D_n^{-\frac{1}{2}}\partial_\theta l_n(\theta_0)+I(\theta_0;\theta_0)^{-1}\hat{T}_n\\
&\quad-HD_n^{-\frac{1}{2}}\partial_\theta l_n(\theta_0)-H\hat{T}_n^*\\
&= \left(I(\theta_0;\theta_0)^{-1}-H\right)D_n^{-\frac{1}{2}}\partial_\theta l_n(\theta_0)+I(\theta_0;\theta_0)^{-1}\hat{T}_n-H\hat{T}_n^*.
\end{align*}
For the right-hand side, it holds, by using  \eqref{lemma2:Tn*-pconv}, \eqref{lemma2:Tn-pconv}, 
Slutsky's theorem and  the continuous mapping theorem, that
\begin{align}\label{lemma2:rightside-dconv}
\left(I(\theta_0;\theta_0)^{-1}-H\right)D_n^{-\frac{1}{2}}\partial_\theta l_n(\theta_0)+I(\theta_0;\theta_0)^{-1}\hat{T}_n-H\hat{T}_n^*&\overset{d}{\to}\left(I(\theta_0;\theta_0)^{-1}-H\right)Y
\end{align}
since it follows from \eqref{col2:goal4-jump} that
\begin{align*}
D_n^{-\frac{1}{2}}\partial_\theta l_n(\theta_0)\overset{d}{\to}Y.
\end{align*}
Thus, for any closed set $F\subset\mathbb{R}^{p+q}$,  we see from \eqref{lemma2:rightside-dconv} that 
\begin{align*}
\limsup_{n\to\infty}P\left(D_n^{\frac{1}{2}}(\hat{\theta}_n-\hat{\theta}_n^*)\in F\right)&\le \limsup_{n\to\infty}P\left(\left\{D_n^{\frac{1}{2}}(\hat{\theta}_n-\hat{\theta}_n^*)\in F\right\}\cap\hat{A}_n\cap\hat{A}_n^*\right)+\limsup_{n\to\infty}P\left(\hat{A}_n^c\cup\hat{A}_n^{*c}\right)\\
&= \limsup_{n\to\infty}P\left(\left\{\left(I(\theta_0;\theta_0)^{-1}-H\right)D_n^{-\frac{1}{2}}\partial_\theta l_n(\theta_0)\right.\right.\\
&\qquad\qquad\qquad\qquad\qquad\qquad\left.\left.+I(\theta_0;\theta_0)^{-1}\hat{T}_n-H\hat{T}_n^*\in F\right\}\cap\hat{A}_n\cap\hat{A}_n^*\right)\\
&\le \limsup_{n\to\infty}P\left(\left(I(\theta_0;\theta_0)^{-1}-H\right)D_n^{-\frac{1}{2}}\partial_\theta l_n(\theta_0)+I(\theta_0;\theta_0)^{-1}\hat{T}_n-H\hat{T}_n^*\in F\right)\\
&\le P\left(\left(I(\theta_0;\theta_0)^{-1}-H\right)Y\in F\right).
\end{align*}
This implies
\begin{align}\label{lemma2:goal}
D_n^{\frac{1}{2}}(\hat{\theta}_n-\hat{\theta}_n^*)\overset{d}{\to}\left(I(\theta_0;\theta_0)^{-1}-H\right)Y.
\end{align}
\end{proof}
\begin{lemma}\label{lemma:H1}
 Let $\{ X_n \}_{n=1,2,\ldots}$ be a sequence of real valued random variables such that
 $X_n\overset{P}{\to}c$ for a positive constant $c$. 
Then 
for any real valued sequence $\{a_n\}_{n=1,2,\ldots}$ which satisfies $a_n\to 0$,
\[P(X_n\le a_n)\to 0.\]
\end{lemma}
\begin{proof}
    Since $a_n\to 0$, there exists a natural number $N\in\mathbb{N}$ such that for all $n\ge N$, $c-a_n\ge \frac{c}{2}$.  For $n\ge N$, we have 
    \begin{align*}
        0\le P(X_n\le a_n)&\le P(|X_n-c|\ge c-a_n)\\
        &\le P\left(|X_n-c|\ge \frac{c}{2}\right).
    \end{align*}
    Since $X_n\overset{P}{\to}c$, taking the limit as $n\to \infty$ on both sides leads to the conclusion.
\end{proof}
\subsubsection{Proof of Theorem \ref{thm5:H0test}}
\begin{proof}
\begin{enumerate}
    \item {Proof of the case that $\tilde{\theta}_n=\hat{\theta}_n$ and $\tilde{\theta}_n^*=\hat{\theta}_n^*$.}
We discuss  $\hat{\theta}_n$ and $\hat{\theta}_n^*$ under $H_0$.
Let  $\varepsilon_0$ be a positive constant such that $\{\theta\in\Theta_0\ ;\ |\theta-\theta_0|<3\varepsilon_0\}\subset\mathrm{Int}(\Theta_0)$, then, from \eqref{lemma2:theta-pconv} and \eqref{lemma2:theta*-pconv}, there exists a real valued sequence $\varepsilon_n<\varepsilon_0$ which satisfies $P(\hat{A}_n\cap\hat{A}_n^*)\to 1$, where
$\hat{A}_n:=\{\omega\in\Omega\ |\ |\hat{\theta}_n(\omega)-\theta_0|<\varepsilon_n\}$ and $\hat{A}_n^*:=\{\omega\in\Omega\ |\ |\hat{\theta}_n^*(\omega)-\theta_0|<\varepsilon_n\}$. 
 Let $B_n =\{\theta\in\Theta\ ;\ |\theta-\hat{\theta}_n|<2\varepsilon_n\}$ on the set $\hat{A}_n\cap\hat{A}_n^*$. 
Since   on the set $\hat{A}_n\cap\hat{A}_n^*$, for all $\theta\in B_n$, 
$|\theta-\theta_0|\le |\theta-\hat{\theta}_n|+|\hat{\theta}_n-\theta_0|< 2\varepsilon_n+\varepsilon_n=3\varepsilon_n$, 
we have 
 $B_n\subset\{\theta\in\Theta\ ;\ |\theta-\theta_0|<3\varepsilon_n\}\subset\mathrm{Int}(\Theta_0)$
  on the set $\hat{A}_n\cap\hat{A}_n^*$. 
 Therefore, on the set $\hat{A}_n\cap\hat{A}_n^*$, it follows from Taylor's theorem that 
 for all  $\theta\in B_n$,  
\begin{align*}
l_n(\theta)&=l_n(\hat{\theta}_n)+\partial_\theta l_n(\hat{\theta}_n)^\top(\theta-\hat{\theta}_n)\\
&\quad+(\theta-\hat{\theta}_n)^\top\int_0^1(1-u)\partial_\theta^2 l_n(\hat{\theta}_n+u(\theta-\hat{\theta}_n))du(\theta-\hat{\theta}_n).
\end{align*}
On the set $\hat{A}_n\cap\hat{A}_n^*$, 
 since $\hat{\theta}_n\in\mathrm{Int}(\Theta_0)$, 
it holds that $\partial_\theta l_n(\hat{\theta}_n)=0$ and 
 $\hat{\theta}_n^*\in B_n$. 
Hence,
\begin{align}\label{thm5:eq1}
l_n(\hat{\theta}_n^*)-l_n(\hat{\theta}_n)=\left(D_n^{\frac{1}{2}}(\hat{\theta}_n^*-\hat{\theta}_n)\right)^\top I_n^{(2)}(\hat{\theta}_n^*,\hat{\theta}_n)\left(D_n^{\frac{1}{2}}(\hat{\theta}_n^*-\hat{\theta}_n)\right)\ \ (\omega\in \hat{A}_n\cap\hat{A}_n^*),
\end{align}
where
\begin{align*}
    I_n^{(2)}(\theta^*,\theta):=D_n^{-\frac{1}{2}}\int_0^1(1-u)\partial_\theta^2 l_n(\theta+u(\theta^*-\theta))duD_n^{-\frac{1}{2}}.
\end{align*}
Set
\begin{align*}
C_n^{(2)}(\theta)&:=\begin{pmatrix}
\displaystyle
\frac{1}{n}\partial_{{\alpha}}^2 l_n(\theta)&\displaystyle\frac{1}{n\sqrt{h_n}}\partial_{{\alpha\beta}}^2 l_n(\theta)\\
\displaystyle\frac{1}{n\sqrt{h_n}}\partial_{{\beta\alpha}}^2 l_n(\theta)&\displaystyle
\frac{1}{nh_n}\partial_{{\beta}}^2 l_n(\theta)
\end{pmatrix}.
\end{align*}
On the set $\hat{A}_n\cap\hat{A}_n^*$, since $|\{\hat{\theta}_n+u(\hat{\theta}_n^*-\hat{\theta}_n)\}-\theta_0|\le 2\varepsilon_n$, 
 one has
\begin{align*}
    \left|{I}_n^{(2)}(\hat{\theta}_n^*,\hat{\theta}_n)+\frac{1}{2}I(\theta_0;\theta_0)\right|&=\left|D_n^{-\frac{1}{2}}\int_0^1(1-u)\partial_\theta^2 l_n(\hat{\theta}_n+u(\hat{\theta}_n^*-\hat{\theta}_n))duD_n^{-\frac{1}{2}}+\frac{1}{2}I(\theta_0;\theta_0)\right|\\
    &\le \left|\int_0^1(1-u)C_n^{(2)}(\hat{\theta}_n+u(\hat{\theta}_n^*-\hat{\theta}_n))du+\int_0^1(1-u)I(\hat{\theta}_n+u(\hat{\theta}_n^*-\hat{\theta}_n);\theta_0)du\right|\\
    &\quad+\left|-\int_0^1(1-u)I(\hat{\theta}_n+u(\hat{\theta}_n^*-\hat{\theta}_n);\theta_0)du+\int_0^1(1-u)I(\theta_0;\theta_0)du\right|\\
    &\le \int_0^1(1-u)\sup_{\theta\in\Theta}\left|C_n^{(2)}(\theta)+I(\theta;\theta_0)\right|du\\
    &\quad+\int_0^1(1-u)\left|I(\hat{\theta}_n+u(\hat{\theta}_n^*-\hat{\theta}_n);\theta_0)-I(\theta_0;\theta_0)\right|du\\
    &\le\int_0^1(1-u)du\left\{\sup_{\theta\in\Theta}\left|C_n^{(2)}(\theta)+I(\theta;\theta_0)\right|+\sup_{\theta:|\theta-\theta_0|<2\varepsilon_n}\left|I(\theta;\theta_0)-I(\theta_0;\theta_0)\right|\right\}\\
    &= 2\left\{\sup_{\theta\in\Theta}\left|C_n^{(2)}(\theta)+I(\theta;\theta_0)\right|+\sup_{\theta:|\theta-\theta_0|<2\varepsilon_n}\left|I(\theta;\theta_0)-I(\theta_0;\theta_0)\right|\right\}.
\end{align*}
By using \eqref{a-2der:col}, \eqref{b-2der:col} and \eqref{ab-2der:col}, we obtain 
\begin{align}\label{thm5:Cn-pconv}
\sup_{\theta\in\Theta}|C_n^{(2)}(\theta)+I(\theta;\theta_0)|\overset{P}{\to}0.
\end{align}
It follows from \eqref{lemma2:theta-pconv},
\eqref{lemma2:theta*-pconv}, \eqref{thm5:Cn-pconv} and continuity of $I(\theta;\theta_0)$ that for all $\varepsilon>0$, 
\begin{align*}
P\left(\left|{I}_n^{(2)}(\hat{\theta}_n^*,\hat{\theta}_n)+\frac{1}{2}I(\theta_0;\theta_0)\right|>\varepsilon\right)&\le P\left(\left\{\left|{I}_n^{(2)}(\hat{\theta}_n^*,\hat{\theta}_n)+\frac{1}{2}I(\theta_0;\theta_0)\right|>\varepsilon\right\}\cap\hat{A}_n\cap\hat{A}_n^*\right)+P(\hat{A}_n^c\cup\hat{A}_n^{*c})\\
&\le P\left(\left\{\sup_{\theta\in\Theta}|C_n^{(2)}(\theta)+I(\theta;\theta_0)|>\frac{\varepsilon}{4}\right\}\cap\hat{A}_n\cap\hat{A}_n^*\right)\\
&\quad +P\left(\left\{\sup_{\theta:|\theta-\theta_0|<2\varepsilon_n}\left|I(\theta;\theta_0)-I(\theta_0;\theta_0)\right|>\frac{\varepsilon}{4}\right\}\cap\hat{A}_n\cap\hat{A}_n^*\right)\\
&\qquad+P(\hat{A}_n^c\cup\hat{A}_n^{*c})\\
&\le P\left(\sup_{\theta\in\Theta}|C_n^{(2)}(\theta)+I(\theta;\theta_0)|>\frac{\varepsilon}{4}\right)\\
&\quad +P\left(\sup_{\theta:|\theta-\theta_0|<2\varepsilon_n}\left|I(\theta;\theta_0)-I(\theta_0;\theta_0)\right|>\frac{\varepsilon}{4}\right)+P(\hat{A}_n^c\cup\hat{A}_n^{*c})\\
&\to 0.
\end{align*}
This implies  
\begin{align}\label{thm5:Jn-pconv}
{I}_n^{(2)}(\hat{\theta}_n^*,\hat{\theta}_n)\overset{P}{\to}-\frac{1}{2}I(\theta_0;\theta_0).
\end{align}
We can express \eqref{thm5:eq1} as follows:
\begin{align}\label{thm5:eq2}
\Lambda_n(\hat{\theta}_n,\hat{\theta}_n^*)=-2\left(D_n^{\frac{1}{2}}(\hat{\theta}_n^*-\hat{\theta}_n)\right)^\top {I}_n^{(2)}(\hat{\theta}_n^*,\hat{\theta}_n) \left(D_n^{\frac{1}{2}}(\hat{\theta}_n^*-\hat{\theta}_n)\right)\ \ (\omega\in \hat{A}_n\cap\hat{A}_n^*).
\end{align}
Under $H_0$, it holds from Lemma \ref{lemma:H0test} and \eqref{thm5:Jn-pconv} that
\begin{align*}
    \left(D_n^{\frac{1}{2}}(\hat{\theta}_n-\hat{\theta}_n^*),{I}_n^{(2)}(\hat{\theta}_n^*,\hat{\theta}_n)\right)\overset{d}{\to}\left(\left(I(\theta_0;\theta_0)^{-1}-H\right)Y,-\frac{1}{2}I(\theta_0;\theta_0)\right),
\end{align*}
where $Y\sim N_{p+q}(0,I(\theta_0;\theta_0)).$ Thus, for the right-hand side of \eqref{thm5:eq2}, 
 it follows from the continuous mapping theorem that
\begin{align}\label{thm5-rightside-dconv}
-2\left(D_n^{\frac{1}{2}}(\hat{\theta}_n^*-\hat{\theta}_n)\right)^\top {I}_n^{(2)}(\hat{\theta}_n^*,\hat{\theta}_n) \left(D_n^{\frac{1}{2}}(\hat{\theta}_n^*-\hat{\theta}_n)\right)\overset{d}{\to} Y^\top\left(I(\theta_0;\theta_0)^{-1}-H\right)^\top I(\theta_0;\theta_0)\left(I(\theta_0;\theta_0)^{-1}-H\right)Y.
\end{align}
By using Lemma 3 in Chapter 9 of \citet{Ferguson}, we have
\begin{align*}
    Y^\top\left(I(\theta_0;\theta_0)^{-1}-H\right)^\top I(\theta_0;\theta_0)\left(I(\theta_0;\theta_0)^{-1}-H\right)Y\sim \chi_{k+l}^2.
\end{align*}
Therefore,  it follows from \eqref{thm5:eq2} and \eqref{thm5-rightside-dconv} that for any closed set $F\subset\mathbb{R}^{p+q}$, under $H_0$, 
\begin{align*}
&\limsup_{n\to\infty}P\left(\Lambda_n(\hat{\theta}_n,\hat{\theta}_n^*)\in F\right)\\
&\quad\le \limsup_{n\to\infty}P\left(\left\{\Lambda_n(\hat{\theta}_n,\hat{\theta}_n^*)\in F\right\}\cap\hat{A}_n\cap\hat{A}_n^*\right)+\limsup_{n\to\infty}P\left(\hat{A}_n^c\cup\hat{A}_n^{*c}\right)\\
&\quad=\limsup_{n\to\infty}P\left(\left\{-2\left(D_n^{\frac{1}{2}}(\hat{\theta}_n^*-\hat{\theta}_n)\right)^\top {I}_n^{(2)}(\hat{\theta}_n^*,\hat{\theta}_n) \left(D_n^{\frac{1}{2}}(\hat{\theta}_n^*-\hat{\theta}_n)\right)\in F\right\}\cap\hat{A}_n\cap\hat{A}_n^*\right)\\
&\quad\le \limsup_{n\to\infty}P\left(-2\left(D_n^{\frac{1}{2}}(\hat{\theta}_n^*-\hat{\theta}_n)\right)^\top {I}_n^{(2)}(\hat{\theta}_n^*,\hat{\theta}_n) \left(D_n^{\frac{1}{2}}(\hat{\theta}_n^*-\hat{\theta}_n)\right)\in F\right)\\
&\quad\le P(\chi_{k+l}^2\in F).
\end{align*}
This implies, under $H_0$, 
\begin{align}\label{thm5:goal1}
    \Lambda_n(\hat{\theta}_n,\hat{\theta}_n^*)\overset{d}{\to}\chi_{k+l}^2.
\end{align}
    \item {Proof of the case where the estimator $\tilde{\theta}_n$ and $\tilde{\theta}_n^*$ satisfy {\bf [T1]}.}
 One has
\begin{align*}
\Lambda_n(\tilde{\theta}_n,\tilde{\theta}_n^*)&=-2(l_n(\tilde{\theta}_n^*)-l_n(\tilde{\theta}_n))\\
&= -2(l_n(\tilde{\theta}_n^*)-l_n(\hat{\theta}_n^*))-2(l_n(\hat{\theta}_n^*)-l_n(\hat{\theta}_n))-2(l_n(\tilde{\theta}_n)-l_n(\hat{\theta}_n)).
\end{align*}
Since, from the proof of case 1, the second term on the right-hand side converges to $\chi_{k+l}^2$ in distribution under $H_0$, it is sufficient to show the following types of convergence under $H_0$ for the proof:
\begin{align}
    &-2(l_n(\tilde{\theta}_n^*)-l_n(\hat{\theta}_n^*))=o_p(1),  \label{thm5-2:eq1}\\
    &-2(l_n(\tilde{\theta}_n)-l_n(\hat{\theta}_n))=o_p(1).\label{thm5-2:eq2}
\end{align}

{\bf Proof of \eqref{thm5-2:eq1}.}
Let  $\varepsilon_0$ be a positive constant such that $\{\theta\in\Theta_0\ ;\ |\theta-\theta_0|<3\varepsilon_0\}\subset\mathrm{Int}(\Theta_0)$. It holds from {\bf [T1]} that $\tilde{\theta}_n\overset{P}{\to}\theta_0$ and  $\tilde{\theta}_n^*\overset{P}{\to}\theta_0$ under $H_0$. Moreover, from \eqref{lemma2:theta-pconv} and  \eqref{lemma2:theta*-pconv}, 
there exists a real valued sequence $\varepsilon_n<\varepsilon_0$ which satisfies $P(\tilde{A}_n\cap\tilde{A}_n^*\cap\hat{A}_n\cap\hat{A}_n^*)\to 1$, where
$\tilde{A}_n:=\{\omega\in\Omega\ |\ |\tilde{\theta}_n(\omega)-\theta_0|<\varepsilon_n\}, \tilde{A}_n^*:=\{\omega\in\Omega\ |\ |\tilde{\theta}_n^*(\omega)-\theta_0|<\varepsilon_n\}, \hat{A}_n:=\{\omega\in\Omega\ |\ |\hat{\theta}_n(\omega)-\theta_0|<\varepsilon_n\}$ and $\hat{A}_n^*:=\{\omega\in\Omega\ |\ |\hat{\theta}_n^*(\omega)-\theta_0|<\varepsilon_n\}$.
In a similar way to the proof of case 1,  we see from Taylor's theorem that
\begin{align}\label{thm5-2:eq3}
l_n(\tilde{\theta}_n^*)-l_n(\hat{\theta}_n^*)&= \left(D_n^{-\frac{1}{2}}\partial_\theta l_n(\hat{\theta}_n^*)\right)^\top D_n^{\frac{1}{2}}(\tilde{\theta}_n^*-\hat{\theta}_n^*)+\left(D_n^{\frac{1}{2}}(\tilde{\theta}_n^*-\hat{\theta}_n^*)\right){I}_n^{(2)}(\tilde{\theta}_n^*,\hat{\theta}_n^*)\left(D_n^{\frac{1}{2}}(\tilde{\theta}_n^*-\hat{\theta}_n^*)\right)
\end{align}
on the set $\tilde{A}_n^*\cap\hat{A}_n^*$.
Furthermore, by using Taylor's theorem, on the set $\hat{A}_n\cap\hat{A}_n^*$, it follows from $\partial_\theta l_n(\hat{\theta}_n)=0$ that
\begin{align}\label{thm5-2:eq4}
    D_n^{-\frac{1}{2}}\partial_\theta l_n(\hat{\theta}_n^*)&= \bar{I}_n^{(1)}(\hat{\theta}_n^*,\theta_0)D_n^{\frac{1}{2}}(\hat{\theta}_n^*-\hat{\theta}_n),
\end{align}
where
\begin{align*}
    \bar{I}_n^{(1)}(\hat{\theta}_n^*,\theta_0)&=\begin{cases}
        I_n^{(1)}(\hat{\theta}_n^*,\theta_0),&(\omega\in\hat{A}_n\cap\hat{A}_n^*\cap\tilde{A}_n^*)\\
        0&(\omega\notin\hat{A}_n\cap\hat{A}_n^*\cap\tilde{A}_n^*).
    \end{cases}.
\end{align*}
We set
\begin{align*}
C_n^{(1)}(\theta)&:=\begin{pmatrix}
\displaystyle
\frac{1}{\sqrt{n}}\partial_{{\alpha}} l_n(\theta)\\
\displaystyle
\frac{1}{\sqrt{nh_n}}\partial_{{\beta}} l_n(\theta)
\end{pmatrix}
\end{align*}
and define
\begin{align*}
    \bar{C}_n^{(1)}(\hat{\theta}_n^*)&=\begin{cases}
        C_n^{(1)}(\hat{\theta}_n^*)&(\omega\in\hat{A}_n\cap\hat{A}_n^*\cap\tilde{A}_n^*),\\
        0&(\omega\notin\hat{A}_n\cap\hat{A}_n^*\cap\tilde{A}_n^*).
    \end{cases}
\end{align*}
Then it holds from \eqref{thm5-2:eq3} that 
\begin{align}\label{thm5-2:eq5}
l_n(\tilde{\theta}_n^*)-l_n(\hat{\theta}_n^*)&= {\bar{C}_n^{(1)}(\hat{\theta}_n^*)}^\top D_n^{\frac{1}{2}}(\tilde{\theta}_n^*-\hat{\theta}_n^*)\nonumber\\
&\quad+\left(D_n^{\frac{1}{2}}(\tilde{\theta}_n^*-\hat{\theta}_n^*)\right){I}_n^{(2)}(\tilde{\theta}_n^*,\hat{\theta}_n^*)\left(D_n^{\frac{1}{2}}(\tilde{\theta}_n^*-\hat{\theta}_n^*)\right)\ \ (\omega\in\hat{A}_n\cap\hat{A}_n^*\cap\tilde{A}_n^*),
\end{align}
and from \eqref{thm5-2:eq4} that
\begin{align}\label{thm5-2:eq6}
\bar{C}_n^{(1)}(\hat{\theta}_n^*)=\bar{I}_n^{(1)}(\hat{\theta}_n,\hat{\theta}_n)D_n^{\frac{1}{2}}(\hat{\theta}_n^*-\hat{\theta}_n^*).
\end{align}
In an analogous manner to the proof of \eqref{lemma2:In-pconv}, we have $\bar{I}_n^{(1)}(\hat{\theta}_n,\hat{\theta}_n)\overset{P}{\to}-I(\theta_0;\theta_0)$ under $H_0$. Moreover, by using Lemma \ref{lemma:H0test}, it follows under $H_0$ that $D_n^{\frac{1}{2}}(\hat{\theta}_n^*-\hat{\theta}_n^*)=O_p(1)$. 
Hence,  it holds from \eqref{thm5-2:eq6} that $\bar{C}_n^{(1)}(\hat{\theta}_n^*)=O_p(1)$ under $H_0$. Furthermore, in a similar manner to the proof of \eqref{thm5:Jn-pconv}, we obtain ${I}_n^{(2)}(\tilde{\theta}_n^*,\hat{\theta}_n^*)\overset{P}{\to}-\frac{1}{2}I(\theta_0;\theta_0)$ under $H_0$.  Since, by {\bf [T1]}, $D_n^{\frac{1}{2}}(\tilde{\theta}_n^*-\hat{\theta}_n^*)=o_p(1)$ under
$H_0$, it follows for the right-hand side of \eqref{thm5-2:eq5} that under $H_0$, 
\begin{align*}
{\bar{C}_n^{(1)}(\hat{\theta}_n^*)}^\top D_n^{\frac{1}{2}}(\tilde{\theta}_n^*-\hat{\theta}_n^*)+\left(D_n^{\frac{1}{2}}(\tilde{\theta}_n^*-\hat{\theta}_n^*)\right){I}_n^{(2)}(\tilde{\theta}_n^*,\hat{\theta}_n^*)\left(D_n^{\frac{1}{2}}(\tilde{\theta}_n^*-\hat{\theta}_n^*)\right)=o_p(1).    
\end{align*}
Hence, for all $\varepsilon>0$,  we see from \eqref{thm5-2:eq5} that 
\begin{align*}
P\left(|l_n(\tilde{\theta}_n^*)-l_n(\hat{\theta}_n^*)|>\varepsilon\right)&\le P\left(\{|l_n(\tilde{\theta}_n^*)-l_n(\hat{\theta}_n^*)|>\varepsilon\}\cap\hat{A}_n\cap\hat{A}_n^*\cap\tilde{A}_n^*\right)+P(\hat{A}_n^c\cup\hat{A}_n^{*c}\cup\tilde{A}_n^{*c})\\
&\quad+P\left(\left\{\left|{\bar{C}_n^{(1)}(\hat{\theta}_n^*)}^\top D_n^{\frac{1}{2}}(\tilde{\theta}_n^*-\hat{\theta}_n^*)\right.\right.\right.\\
&\qquad\qquad+\left.\left.\left.\left(D_n^{\frac{1}{2}}(\tilde{\theta}_n^*-\hat{\theta}_n^*)\right){I}_n^{(2)}(\tilde{\theta}_n^*,\hat{\theta}_n^*)\left(D_n^{\frac{1}{2}}(\tilde{\theta}_n^*-\hat{\theta}_n^*)\right)\right|>\varepsilon\right\}\cap\hat{A}_n\cap\hat{A}_n^*\cap\tilde{A}_n^*\right)\\
&\quad + P(\hat{A}_n^c\cup\hat{A}_n^{*c}\cup\tilde{A}_n^{*c})\\
&\le P\left(\left|{\bar{C}_n^{(1)}(\hat{\theta}_n^*)}^\top D_n^{\frac{1}{2}}(\tilde{\theta}_n^*-\hat{\theta}_n^*)\right.\right.\\
&\qquad\qquad\left.\left.+\left(D_n^{\frac{1}{2}}(\tilde{\theta}_n^*-\hat{\theta}_n^*)\right){I}_n^{(2)}(\tilde{\theta}_n^*,\hat{\theta}_n^*)\left(D_n^{\frac{1}{2}}(\tilde{\theta}_n^*-\hat{\theta}_n^*)\right)\right|>\varepsilon\right)\\
&\quad+P(\hat{A}_n^c\cup\hat{A}_n^{*c}\cup\tilde{A}_n^{*c})\\
&\to 0.
\end{align*}
This implies \eqref{thm5-2:eq1}.

{\bf Proof of \eqref{thm5-2:eq2}.}
We can  show \eqref{thm5-2:eq2} in an analogous manner to the proof of \eqref{thm5-2:eq1}.
\end{enumerate}
\end{proof}
\subsubsection{Proof of Proposition \ref{prop:test}}
\begin{proof}
\begin{enumerate}
    \item{Proof of $D_n^{\frac{1}{2}}(\hat{\theta}_n-\check{\theta}_n)=o_p(1)$.}
The following types of convergence are sufficient to show the proof:
\begin{align}
    &\sqrt{n}(\hat{\alpha}_n-\check{\alpha}_n)=o_p(1), \label{propH0:goal1}\\
    &\sqrt{nh_n}(\hat{\beta}_n-\check{\beta}_n)=o_p(1).\label{propH0:goal2}
\end{align}
Let  $\varepsilon_0$ be a positive constant such that $\{\theta\in\Theta_0\ ;\ |\theta-\theta_0|<\varepsilon_0\}\subset\mathrm{Int}(\Theta_0)$. By using Theorem \ref{thm1-2:jump}, we have $\check{\theta}_n\overset{P}{\to}\theta_0$. Moreover, from \eqref{lemma2:theta-pconv}, there exists a real valued sequence $\varepsilon_n<\varepsilon_0$ which satisfies $P(\check{A}_n\cap\hat{A}_n)\to 1$, where
$\check{A}_n:=\{\omega\in\Omega\ |\ |\check{\alpha}_n(\omega)-\alpha_0|<\varepsilon_n,\ |\check{\beta}_n(\omega)-\beta_0|<\varepsilon_n\}$ and $\hat{A}_n:=\{\omega\in\Omega\ |\ |\hat{\theta}_n(\omega)-\theta_0|<\varepsilon_n\}$. 
 It follows from Taylor's theorem that on the set $\check{A}_n$,
\begin{align}
-\frac{1}{\sqrt{n}}\partial_{{\alpha}}l_n^{(1)}(\alpha_0)&=\left(\int_{0}^{1}{\frac{1}{n}\partial_{{\alpha}}^2 l_n^{(1)}(\alpha_0+u(\check{\alpha}_{n}-\alpha_0))du}\right)\sqrt{n}(\check{\alpha}_{n}-\alpha_0),\label{propH0:eq1}\\
-\frac{1}{\sqrt{nh_n}}\partial_{{\beta}}l_n^{(2)}(\beta_0|\check{\alpha}_{n})&=\left(\int_{0}^{1}{\frac{1}{nh_n}\partial_{{\beta}}^2 l_n^{(2)}(\beta_0+u(\check{\beta}_{n}-\beta_0)|\check{\alpha}_{n})du}\right)\sqrt{nh_n}(\check{\beta}_{n}-\beta_0).\label{propH0:eq2}
\end{align}
It holds from \eqref{thm2:goal1-jump} that  
\begin{align*}
    \int_{0}^{1}{\frac{1}{n}\partial_{{\alpha}}^2 l_n^{(1)}(\alpha_0+u(\check{\alpha}_{n}-\alpha_0))du}\overset{P}{\to}-I_a(\alpha_0;\alpha_0).
\end{align*}
Since $-I_a(\alpha_0)$ is non-singular, we obtain $P(\check{J}_n^{(1)})\to 1$, where 
\begin{align*}
\check{J}_n^{(1)}:=\left\{\omega\in\Omega\ |\ \int_{0}^{1}{\frac{1}{n}\partial_{{\alpha}}^2 l_n^{(1)}(\alpha_0+u(\check{\alpha}_{n}-\alpha_0))du}\ \text{is non-singular}\right\}.
\end{align*}
Thus, we have
\begin{align}\label{propH0:Yn1-pconv}
(\check{Y}_n^{(1)})^{-1}\overset{P}{\to}-I_a^{-1}(\alpha_0;\alpha_0),
\end{align}
where
\begin{align*}
    \check{Y}_n^{(1)}:=\begin{cases}\displaystyle
        \int_{0}^{1}{\frac{1}{n}\partial_{{\alpha}}^2 l_n^{(1)}(\alpha_0+u(\check{\alpha}_{n}-\alpha_0))du}&\omega\in \check{J}_n^{(1)}\\
        I_p&\omega\notin \check{J}_n^{(1)}
    \end{cases}.
\end{align*}
From \eqref{propH0:eq1}, it follows on the set $\check{A}_n\cap \check{J}_n^{(1)}$ that 
\begin{align*}
    \sqrt{n}(\check{\alpha}_n-\alpha_0)&= -(\check{Y}_n^{(1)})^{-1}\frac{1}{\sqrt{n}}\partial_\alpha l_n^{(1)}(\alpha_0)\\
    &= I_a^{-1}(\alpha_0;\alpha_0)\frac{1}{\sqrt{n}}\partial_\alpha l_n^{(1)}(\alpha_0)+\left((\check{Y}_n^{(1)})^{-1}+I_a^{-1}(\alpha_0;\alpha_0)\right)\left(-\frac{1}{\sqrt{n}}\partial_\alpha l_n^{(1)}(\alpha_0)\right).
\end{align*}
Let us call the second term on the right-hand side $\check{R}_n$, then it follows that
\begin{align}\label{propH0:essentialeq1}
    \sqrt{n}(\check{\alpha}_n-\alpha_0)&= I_a^{-1}(\alpha_0;\alpha_0)\frac{1}{\sqrt{n}}\partial_\alpha l_n^{(1)}(\alpha_0)+\check{R}_n\ \ \ (\omega\in \check{A}_n\cap \check{J}_n^{(1)}).
\end{align}
 Using \eqref{propH0:Yn1-pconv} and \eqref{thm2:goal4-jump}, we have
\begin{align}\label{propH0:Rn-pconv}
    \check{R}_n\overset{P}{\to}0.
\end{align}
 It holds from  \eqref{thm2:goal2-jump} that
\begin{align*}
    \int_{0}^{1}{\frac{1}{nh_n}\partial_{{\beta}}^2 l_n^{(2)}(\beta_0+u(\check{\beta}_{n}-\beta_0)|\check{\alpha}_{n})du}\overset{P}{\to}-I_{b,c}(\theta_0;\theta_0).
\end{align*}
Since $-I_{b,c}(\theta_0;\theta_0)$ is non-singular, we obtain $P(\check{J}_n^{(2)})\to 1$, where 
\begin{align*}
\check{J}_n^{(2)}:=\left\{\omega\in\Omega\ |\ \int_{0}^{1}{\frac{1}{nh_n}\partial_{{\beta}}^2 l_n^{(2)}(\beta_0+u(\check{\beta}_{n}-\beta_0)|\check{\alpha}_n)du}\ \text{is non-singular}\right\}
\end{align*}
Therefore, we have
\begin{align}\label{propH0:Yn2-pconv}
    (\check{Y}_n^{(2)})^{-1}\overset{P}{\to}-I_{b,c}^{-1}(\theta_0;\theta_0),
\end{align}
where
\begin{align*}
    \check{Y}_n^{(2)}&=\begin{cases}\displaystyle
        \int_{0}^{1}{\frac{1}{nh_n}\partial_{{\beta}}^2 l_n^{(2)}(\beta_0+u(\check{\beta}_{n}-\beta_0)|\check{\alpha}_n)du}& (\omega\in \check{J}_n^{(2)}), \\
        I_q&(\omega\notin \check{J}_n^{(2)}).
    \end{cases}
\end{align*}
On the set $\check{A}_n\cap \check{J}_n^{(2)}$, it follows from 
\eqref{propH0:eq2} that
\begin{align*}
    \sqrt{nh_n}(\check{\beta}_n-\beta_0)&= -(\check{Y}_n^{(2)})^{-1}\frac{1}{\sqrt{nh_n}}\partial_{\beta}l_n^{(2)}(\beta_0|\check{\alpha}_n)\\
    &= I_{b,c}^{-1}(\theta_0;\theta_0)\frac{1}{\sqrt{nh_n}}\partial_{\beta}l_n^{(2)}(\beta_0|\alpha_0)-\left((\check{Y}_n^{(2)})^{-1}+I_{b,c}^{-1}(\theta_0;\theta_0)\right)\frac{1}{\sqrt{nh_n}}\partial_{\beta}l_n^{(2)}(\beta_0|\alpha_0)\\
    &\quad +(\check{Y}_n^{(2)})^{-1}\frac{1}{\sqrt{nh_n}}\left(\partial_{\beta}l_n^{(2)}(\beta_0|\alpha_0)-\partial_{\beta}l_n^{(2)}(\beta_0|\check{\alpha}_n)\right).
\end{align*}
Let us call the second and third term on the right-hand side $\check{S}_n^{(1)}$ and $\check{S}_n^{(2)}$, respectively, and define $\check{S}_n:=\check{S}_n^{(1)}+\check{S}_n^{(2)}$. Then  one has
\begin{align}\label{propH0:essentialeq2}
    \sqrt{nh_n}(\check{\beta}_n-\beta_0)&=I_{b,c}^{-1}(\theta_0;\theta_0)\frac{1}{\sqrt{nh_n}}\partial_{\beta}l_n^{(2)}(\beta_0|\alpha_0)+\check{S}_n\ \ \ (\omega\in \check{A}_n\cap \check{J}_n^{(2)}).
\end{align}
By using \eqref{propH0:Yn2-pconv} and \eqref{thm2:goal4-jump}, we obtain  
\begin{align}\label{propH0:Sn1-pconv}
    \check{S}_n^{(1)}&\overset{P}{\to}0.
\end{align}
It follows from Taylor's theorem that 
\begin{align*}
    \frac{1}{\sqrt{nh_n}}\partial_{{\beta}}l_n^{(2)}(\beta_0|\check{\alpha}_{n})&=\frac{1}{\sqrt{nh_n}}\partial_{{\beta}}l_n^{(2)}(\beta_0|\alpha_0)+
\left(\int_{0}^{1}{\frac{1}{n\sqrt{h_n}}\partial_{{\alpha\beta}}^2 l_n^{(2)}(\beta_0|\alpha_0+u(\check{\alpha}_{n}-\alpha_0))du}\right)\sqrt{n}(\check{\alpha}_{n}-\alpha_0).
\end{align*}
By using \eqref{thm2:goal3-jump} and Theorem \ref{thm2-2:jump}, the second term on the right-hand side converges to 0 in probability.
Thus, 
\begin{align*}
    \frac{1}{\sqrt{nh_n}}\left(\partial_{\beta}l_n^{(2)}(\beta_0|\alpha_0)-\partial_{\beta}l_n^{(2)}(\beta_0|\check{\alpha}_n)\right)\overset{P}{\to}0,
\end{align*}
and since  we see from \eqref{propH0:Yn2-pconv} that  $(\check{Y}_n^{(2)})^{-1}=O_p(1)$, we obtain 
\begin{align}\label{propH0:Sn2-pconv}
    \check{S}_n^{(2)}\overset{P}{\to}0.
\end{align}
By using \eqref{propH0:Sn1-pconv} and \eqref{propH0:Sn2-pconv}, we have 
\begin{align}\label{propH0:Sn-pconv}
    \check{S}_n\overset{P}{\to}0.
\end{align}
Similarly, it follows from Taylor's theorem that
\begin{align}
-\frac{1}{\sqrt{n}}\partial_{{\alpha}}l_n(\theta_0)&=\left(\int_{0}^{1}{\frac{1}{n}\partial_{{\alpha}}^2 l_n(\theta_0+u(\hat{\theta}_{n}-\theta_0))du}\right)\sqrt{n}(\hat{\alpha}_{n}-\alpha_0),\label{propH0:eq3}\\
-\frac{1}{\sqrt{nh_n}}\partial_{{\beta}}l_n(\theta_0)&=\left(\int_{0}^{1}{\frac{1}{nh_n}\partial_{{\beta}}^2 l_n(\theta_0+u(\hat{\theta}_{n}-\theta_0))du}\right)\sqrt{nh_n}(\hat{\beta}_{n}-\beta_0)\label{propH0:eq4}
\end{align}
on the set $\hat{A}_n$.
By using \eqref{col2:goal1-jump} and \eqref{col2:goal2-jump},  one has
\begin{align*}
    &\int_{0}^{1}{\frac{1}{n}\partial_{{\alpha}}^2 l_n(\theta_0+u(\hat{\theta}_{n}-\theta_0))du}\overset{P}{\to}-I_a(\alpha_0;\alpha_0),\\
    &\int_{0}^{1}{\frac{1}{nh_n}\partial_{{\beta}}^2 l_n(\theta_0+u(\hat{\theta}_{n}-\theta_0))du}\overset{P}{\to}-I_{b,c}(\theta_0;\theta_0).
\end{align*}
Since $-I_a(\alpha_0;\alpha_0)$ and $-I_{b,c}(\theta_0;\theta_0)$  are non-singular, it holds that $P(\hat{J}_n^{(1)})\overset{P}{\to}1$ and $P(\hat{J}_n^{(2)})\overset{P}{\to}1$, where 
\begin{align*}
&\hat{J}_n^{(1)}:=\left\{\omega\in\Omega\ |\ \int_{0}^{1}{\frac{1}{n}\partial_{{\alpha}}^2 l_n(\theta_0+u(\hat{\theta}_{n}-\theta_0))du}\ \text{is non-singular}\right\},\\
&\hat{J}_n^{(2)}:=\left\{\omega\in\Omega\ |\ \int_{0}^{1}{\frac{1}{nh_n}\partial_{{\beta}}^2 l_n(\theta_0+u(\hat{\theta}_{n}-\theta_0))du}\ \text{is non-singular}\right\}.
\end{align*}
Therefore,  one has
\begin{align}
    &(\hat{Y}_n^{(1)})^{-1}\overset{P}{\to}-I_{a}^{-1}(\alpha_0;\alpha_0),\label{propH0:Yn1.j-pconv}\\
    &(\hat{Y}_n^{(2)})^{-1}\overset{P}{\to}-I_{b,c}^{-1}(\theta_0;\theta_0),\label{propH0:Yn2.j-pconv}
\end{align}
where
\begin{align*}
    \hat{Y}_n^{(1)}&=\begin{cases}\displaystyle
        \int_{0}^{1}{\frac{1}{n}\partial_{{\alpha}}^2 l_n(\theta_0+u(\hat{\theta}_{n}-\theta_0))du}& (\omega\in \hat{J}_n^{(1)}),\\
        I_p&(\omega\notin \hat{J}_n^{(1)}),
    \end{cases}\\
    \hat{Y}_n^{(2)}&=\begin{cases}\displaystyle
        \int_{0}^{1}{\frac{1}{nh_n}\partial_{{\beta}}^2 l_n(\theta_0+u(\hat{\theta}_{n}-\theta_0))du}& (\omega\in \hat{J}_n^{(2)}),\\
        I_p&(\omega\notin \hat{J}_n^{(2)}).
    \end{cases}
\end{align*}
 It follows from \eqref{propH0:eq3} that on the set $\hat{A}_n\cap\hat{J}_n^{(1)}$, 
\begin{align*}
    \sqrt{n}(\hat{\alpha}_n-\alpha_0)&= -(\hat{Y}_n^{(1)})^{-1}\frac{1}{\sqrt{n}}\partial_\alpha l_n(\theta_0)\\
    &=I_a^{-1}(\alpha_0;\alpha_0)\frac{1}{\sqrt{n}}\partial_\alpha l_n(\theta_0)-\left((\hat{Y}_n^{(1)})^{-1}+I_a^{-1}(\alpha_0;\alpha_0)\right)\frac{1}{\sqrt{n}}\partial_\alpha l_n(\theta_0).
\end{align*}
Let us call the second term on the right-hand side $\hat{R}_n$.  One has
\begin{align}\label{propH0:essentialeq3}
    \sqrt{n}(\hat{\alpha}_n-\alpha_0)&=I_a^{-1}(\alpha_0;\alpha_0)\frac{1}{\sqrt{n}}\partial_\alpha l_n(\theta_0)+\hat{R}_n\ \ \ (\omega\in\hat{A}_n\cap \hat{J}_n^{(1)}),
\end{align}
and we see from \eqref{propH0:Yn1.j-pconv} and \eqref{col2:goal4-jump} that 
\begin{align}\label{propH0:Rn.j-pconv}
    \hat{R}_n\overset{P}{\to}0.
\end{align}
From \eqref{propH0:eq4}, it holds that 
\begin{align*}
    \sqrt{nh_n}(\hat{\beta}_n-\beta_0)&=-(\hat{Y}_n^{(2)})^{-1}\frac{1}{\sqrt{nh_n}}\partial_\beta l_n(\theta_0)\\
    &=I_{b,c}^{-1}(\theta_0;\theta_0)\frac{1}{\sqrt{nh_n}}\partial_\beta l_n(\theta_0)-\left((\hat{Y}_n^{(2)})^{-1}+I_{b,c}^{-1}(\theta_0;\theta_0)\right)\frac{1}{\sqrt{nh_n}}\partial_\beta l_n(\theta_0)
\end{align*}
on the set $\hat{A}_n\cap\hat{J}_n^{(2)}$.
We call the the second term on the right-hand side $\hat{S}_n$.  We have
\begin{align}\label{propH0:essentialeq4}
    \sqrt{nh_n}(\hat{\beta}_n-\beta_0)&=I_{b,c}^{-1}(\theta_0;\theta_0)\frac{1}{\sqrt{nh_n}}\partial_\beta l_n(\theta_0)+\hat{S}_n\ \ \ (\omega\in\hat{A}_n\cap \hat{J}_n^{(2)}).
\end{align}
 Using \eqref{propH0:Yn2.j-pconv} and \eqref{col2:goal4-jump}, we obtain
\begin{align}\label{propH0:Sn.j-pconv}
    \hat{S}_n\overset{P}{\to}0.
\end{align}
 It holds from \eqref{propH0:essentialeq1} and \eqref{propH0:essentialeq3} that 
on the set $\check{A}_n\cap\hat{A}_n\cap\check{J}_n^{(1)}\cap\hat{J}_n^{(1)}$,
\begin{align}
    \sqrt{n}(\hat{\alpha}_n-\check{\alpha}_n)&=\sqrt{n}(\hat{\alpha}_n-\alpha_0)-\sqrt{n}(\check{\alpha}_n-\alpha_0)\nonumber\\
    &=I_a^{-1}(\alpha_0;\alpha_0)\frac{1}{\sqrt{n}}\left(\partial_\alpha l_n(\theta_0)-\partial_\alpha l_n^{(1)}(\alpha_0)\right)+\hat{R}_n-\check{R}_n.\label{propH0:essentialeq5}
\end{align}
By using $P(\check{A}_n\cap\hat{A}_n\cap\check{J}_n^{(1)}\cap\hat{J}_n^{(1)})\to 1$, \eqref{propH0:Rn-pconv} and \eqref{propH0:Rn.j-pconv}, it is sufficient to show the following convergence for the proof of \eqref{propH0:goal1}:
\begin{align}\label{propH0:ada.joint.diff-alpha}
    \frac{1}{\sqrt{n}}\left(\partial_\alpha l_n(\theta_0)-\partial_\alpha l_n^{(1)}(\alpha_0)\right)=o_p(1).
\end{align}

{\bf Proof of \eqref{propH0:ada.joint.diff-alpha}.}
In order to distinguish the thresholds contained in the joint quasi-log likelihood function and in the adaptive quasi-log likelihood function, 
 we express the thresholds in joint function as follows: $\{|\Delta X_i^n|\le \bar{D}_1h_n^{\bar{\rho}_1}\}$ and $\{|\Delta X_i^n|> \bar{D}_2h_n^{\bar{\rho}_2}\}$.
First, we show the case  where $D_1=D_2=D_3=\bar{D}_1=\bar{D}_2=1$.
Since $\Delta X_i^n=\bar{X}_{i,n}(\beta_0)-h_nb_{i-1}(\beta_0)$, for $1\le m_1\le p$, we can calculate 
\begin{align*}
    &\frac{1}{\sqrt{n}}\left(\partial_{\alpha_{m_1}} l_n(\theta_0)-\partial_{\alpha_{m_1}} l_n^{(1)}(\alpha_0)\right)\\
    &=-\frac{1}{2\sqrt{n}h_n}\sum_{i=1}^n\left\{\left(\bar{X}_{i,n}\right)^\top\partial_{\alpha_{m_1}}S_{i-1}^{-1}\bar{X}_{i,n}\boldsymbol{1}_{\{|\Delta X_i^n|\le h_n^{\bar{\rho}_1}\}}-\left(\Delta X_i^n\right)^\top\partial_{\alpha_{m_1}}S_{i-1}^{-1}\Delta X_i^n\boldsymbol{1}_{\{|\Delta X_i^n|\le h_n^{\rho_1}\}}\right\}\\
    &\quad-\frac{1}{2\sqrt{n}}\sum_{i=1}^n\partial_{\alpha_{m_1}}\log\det S_{i-1}\left(\boldsymbol{1}_{\{|\Delta X_i^n|\le h_n^{\bar{\rho}_1}\}}-\boldsymbol{1}_{\{|\Delta X_i^n|\le h_n^{\rho_1}\}}\right)\\
    &=-\frac{1}{2\sqrt{n}h_n}\sum_{i=1}^n\left(\bar{X}_{i,n}\right)^\top\partial_{\alpha_{m_1}}S_{i-1}^{-1}\bar{X}_{i,n}\left(\boldsymbol{1}_{\{|\Delta X_i^n|\le h_n^{\bar{\rho}_1}\}}-\boldsymbol{1}_{\{|\Delta X_i^n|\le h_n^{\rho_1}\}}\right)\\
    &\quad+\sqrt{nh_n^2}\frac{1}{nh_n}\sum_{i=1}^n(b_{i-1})^\top\partial_{\alpha_{m_1}}S_{i-1}^{-1}\bar{X}_{i,n}\boldsymbol{1}_{\{|\Delta X_i^n|\le h_n^{\rho_1}\}}\\
    &\quad+\sqrt{nh_n^2}\frac{1}{2n}\sum_{i=1}^n(b_{i-1})^\top\partial_{\alpha_{m_1}}S_{i-1}^{-1}b_{i-1}\boldsymbol{1}_{\{|\Delta X_i^n|\le h_n^{\rho_1}\}}\\
    &\quad-\frac{1}{2\sqrt{n}}\sum_{i=1}^n\partial_{\alpha_{m_1}}\log\det S_{i-1}\left(\boldsymbol{1}_{\{|\Delta X_i^n|\le h_n^{\bar{\rho}_1}\}}-\boldsymbol{1}_{\{|\Delta X_i^n|\le h_n^{\rho_1}\}}\right)\\
    &=-\frac{1}{2\sqrt{n}h_n}\sum_{i=1}^n\left(\bar{X}_{i,n}\right)^\top\partial_{\alpha_{m_1}}S_{i-1}^{-1}\bar{X}_{i,n}\left(\boldsymbol{1}_{\{|\Delta X_i^n|\le h_n^{\bar{\rho}_1}\}}-\boldsymbol{1}_{\{|\Delta X_i^n|\le h_n^{\rho_1}\}}\right)\\
    &\quad+\sqrt{nh_n^2}\frac{1}{nh_n}\sum_{i=1}^n(b_{i-1})^\top\partial_{\alpha_{m_1}}S_{i-1}^{-1}\bar{X}_{i,n}\boldsymbol{1}_{\{|\Delta X_i^n|\le h_n^{\rho_1}\}}\\
    &\quad+\frac{1}{n}\sum_{i=1}^n R(\theta,\sqrt{nh_n^2},X_{t_{i-1}^n})\\
    &\quad-\frac{1}{2\sqrt{n}}\sum_{i=1}^n\partial_{\alpha_{m_1}}\log\det S_{i-1}\left(\boldsymbol{1}_{\{|\Delta X_i^n|\le h_n^{\bar{\rho}_1}\}}-\boldsymbol{1}_{\{|\Delta X_i^n|\le h_n^{\rho_1}\}}\right).
\end{align*}
By using Proposition \ref{prop6:jump}, the second term on the right-hand side converges to 0 in probability, and it is obvious that the third term converges to 0 in probability. Therefore, we show the first and fourth  terms converge to 0 in probability.
If $\bar{\rho}_1=\rho_1$, then both terms equal to 0. We evaluate the case  where $\bar{\rho}_1\neq\rho_1$. First, we discuss the case   where  $\bar{\rho}_1>\rho_1$. Since $\{|\Delta X_i^n|\le h_n^{\bar{\rho}_1}\}\subset \{|\Delta X_i^n|\le h_n^{\rho_1}\}$, we have 
\begin{align}\label{propH0:rho1-diff}
    \boldsymbol{1}_{\{|\Delta X_i^n|\le h_n^{\bar{\rho}_1}\}}-\boldsymbol{1}_{\{|\Delta X_i^n|\le h_n^{\rho_1}\}}=-\boldsymbol{1}_{\{|\Delta X_i^n|\le h_n^{\rho_1}\}}\boldsymbol{1}_{\{|\Delta X_i^n|> h_n^{\bar{\rho}_1}\}}.
\end{align}
Hence, since $|\bar{X}_{i,n}(\beta_0)|^2\boldsymbol{1}_{\{|\Delta X_i^n|\le h_n^{\rho_1}\}}=R(\theta,h_n^{2\rho_1},X_{t_{i-1}^n})$, it follows from \eqref{propH0:rho1-diff}, Markov's inequality, and Proposition \ref{prop2:jump} that, for the first term, for all $\varepsilon>0$, 
\begin{align*}
    &P\left(\left|\frac{1}{2\sqrt{n}h_n}\sum_{i=1}^n\left(\bar{X}_{i,n}\right)^\top\partial_{\alpha_{m_1}}S_{i-1}^{-1}\bar{X}_{i,n}\boldsymbol{1}_{\{|\Delta X_i^n|\le h_n^{\rho_1}\}}\boldsymbol{1}_{\{|\Delta X_i^n|> h_n^{\bar{\rho}_1}\}}\right|>\varepsilon\right)\\
    &\le \frac{1}{2\varepsilon\sqrt{n}h_n}\sum_{i=1}^n\sum_{k_1,k_2=1}^d\mathbb{E}\left[\left|\bar{X}_{i,n}^{(k_1)}{\partial_{\alpha_{m_1}}S_{i-1}^{-1}}^{(k_1,k_2)}\bar{X}_{i,n}^{(k_2)}\boldsymbol{1}_{\{|\Delta X_i^n|\le h_n^{\rho_1}\}}\boldsymbol{1}_{\{|\Delta X_i^n|> h_n^{\bar{\rho}_1}\}}\right|\right]\\
    &\le \frac{1}{2\varepsilon\sqrt{n}h_n}\sum_{i=1}^n\sum_{k_1,k_2=1}^d\mathbb{E}\left[R(\theta,h_n^{2\rho_1},X_{t_{i-1}^n})P\left(\{|\Delta X_i^n|\le h_n^{\rho_1}\}\cap\{|\Delta X_i^n|> h_n^{\bar{\rho}_1}\}\ |\mathcal{F}_{i-1}^n\right)\right]\\
    &\le C\sqrt{n}h_n^{3\rho_1}\\
    &= O(\sqrt{nh_n^{6\rho_1}}).
\end{align*}
By $\rho_1\ge \frac{1+\delta}{6}$, the first term converges to 0.
In a similar way, for the fourth term,  it holds from \eqref{propH0:rho1-diff}, Markov's inequality, and Proposition \ref{prop2:jump} that for all $\varepsilon>0$, 
\begin{align*}
    &P\left(\left|\frac{1}{2\sqrt{n}}\sum_{i=1}^n\partial_{\alpha_{m_1}}\log\det S_{i-1}\boldsymbol{1}_{\{|\Delta X_i^n|\le h_n^{\rho_1}\}}\boldsymbol{1}_{\{|\Delta X_i^n|> h_n^{\bar{\rho}_1}\}}\right|>\varepsilon\right)\\
    &\le \frac{1}{2\varepsilon\sqrt{n}}\sum_{i=1}^n\mathbb{E}\left[|\partial_{\alpha_{m_1}}\log\det S_{i-1}|P\left(\{|\Delta X_i^n|\le h_n^{\rho_1}\}\cap\{|\Delta X_i^n|> h_n^{\bar{\rho}_1}\}\ |\mathcal{F}_{i-1}^n\right)\right]\\
    &\le \frac{1}{2\varepsilon\sqrt{n}}\sum_{i=1}^n\mathbb{E}\left[R(\theta,h_n^{1+\rho_1},X_{t_{i-1}^n})\right]\\
    &\le C\sqrt{n}h_n^{1+\rho_1}\\
    &= O(\sqrt{nh_n^2}\cdot h_n^{\rho_1})\\
    &\to 0.
\end{align*}
This implies the fourth term converges to 0, and \eqref{propH0:ada.joint.diff-alpha} holds if $D_1=D_2=D_3=\bar{D}_1=\bar{D}_2=1$ and $\bar{\rho}_1>\rho_1$. 
Note that, by the statement of Theorem \ref{thm2-2:jump} and Corollary \ref{thm2-1:jump}, the two sets including $\rho_1$ and  $\bar{\rho}_1$ are the same.
Hence, we can show the case  where $D_1=D_2=D_3=\bar{D}_1=\bar{D}_2=1$ and $\bar{\rho}_1<\rho_1$ in a similar way. After all, \eqref{propH0:ada.joint.diff-alpha} holds if $D_1=D_2=D_3=\bar{D}_1=\bar{D}_2=1$.
In order to evaluate more general case, we take $D,\bar{D}>0,\ \rho,\bar{\rho}\in(0,\frac{1}{2})$, and discuss $\{|\Delta X_i^n|\le Dh_n^{\rho}\}$ and $\{|\Delta X_i^n|\le \bar{D}h_n^{\bar{\rho}}\}$. 
Since these both sets are related to the upper bound of $|\Delta X_i^n|$, the inclusion relationship holds for all $n\in\mathbb{N}$:
\begin{enumerate}
    \item[(i)] If $Dh_n^{\rho}\le \bar{D}h_n^{\bar{\rho}}$, we have
    $\{|\Delta X_i^n|\le Dh_n^{\rho}\}\subset\{|\Delta X_i^n|\le \bar{D}h_n^{\bar{\rho}}\}$. Therefore,
    \begin{align}\label{propH0:boldsymbol:eq1}
        \boldsymbol{1}_{\{|\Delta X_i^n|\le Dh_n^{\rho}\}}-\boldsymbol{1}_{\{|\Delta X_i^n|\le \bar{D}h_n^{\bar{\rho}}\}}&= -\boldsymbol{1}_{\{|\Delta X_i^n|\le \bar{D}h_n^{\bar{\rho}}\}}\boldsymbol{1}_{\{|\Delta X_i^n|> Dh_n^{\rho}\}},
    \end{align}
    and by $Dh_n^{\rho}\le \bar{D}h_n^{\bar{\rho}}$,
    \begin{align}\label{propH0:boldsymbol:eq2}
        \boldsymbol{1}_{\{|\Delta X_i^n|\le Dh_n^{\rho}\}}\boldsymbol{1}_{\{|\Delta X_i^n|> \bar{D}h_n^{\bar{\rho}}\}}&=0.
    \end{align}
    \item [(ii)] If $Dh_n^{\rho}> \bar{D}h_n^{\bar{\rho}}$,  one has $\{|\Delta X_i^n|\le Dh_n^{\rho}\}\supset\{|\Delta X_i^n|\le \bar{D}h_n^{\bar{\rho}}\}$. Hence, 
    \begin{align}\label{propH0:boldsymbol:eq3}
        \boldsymbol{1}_{\{|\Delta X_i^n|\le Dh_n^{\rho}\}}-\boldsymbol{1}_{\{|\Delta X_i^n|\le \bar{D}h_n^{\bar{\rho}}\}}&= \boldsymbol{1}_{\{|\Delta X_i^n|\le Dh_n^{\rho}\}}\boldsymbol{1}_{\{|\Delta X_i^n|> \bar{D}h_n^{\bar{\rho}}\}},
    \end{align}
    and by $Dh_n^{\rho}> \bar{D}h_n^{\bar{\rho}}$,
    \begin{align}\label{propH0:boldsymbol:eq4}
        -\boldsymbol{1}_{\{|\Delta X_i^n|\le \bar{D}h_n^{\bar{\rho}}\}}\boldsymbol{1}_{\{|\Delta X_i^n|> Dh_n^{\rho}\}}&=0.
    \end{align}
\end{enumerate}
By using \eqref{propH0:boldsymbol:eq1}, \eqref{propH0:boldsymbol:eq2}, \eqref{propH0:boldsymbol:eq3} and \eqref{propH0:boldsymbol:eq4}, it follows for $D,\bar{D}>0$ and $\rho,\bar{\rho}\in(0,\frac{1}{2})$ that for all $n\in\mathbb{N}$,
\begin{align}\label{propH0:boldsymbol:result1}
        \boldsymbol{1}_{\{|\Delta X_i^n|\le Dh_n^{\rho}\}}-\boldsymbol{1}_{\{|\Delta X_i^n|\le \bar{D}h_n^{\bar{\rho}}\}}&= \boldsymbol{1}_{\{|\Delta X_i^n|\le Dh_n^{\rho}\}}\boldsymbol{1}_{\{|\Delta X_i^n|> \bar{D}h_n^{\bar{\rho}}\}}-\boldsymbol{1}_{\{|\Delta X_i^n|\le \bar{D}h_n^{\bar{\rho}}\}}\boldsymbol{1}_{\{|\Delta X_i^n|> Dh_n^{\rho}\}}.
\end{align}
Thus, we can show the general case by replacing \eqref{propH0:rho1-diff} with \eqref{propH0:boldsymbol:result1}. Therefore, it holds that \eqref{propH0:ada.joint.diff-alpha}, and we obtain \eqref{propH0:goal1}.
Next, by using \eqref{propH0:essentialeq2} and \eqref{propH0:essentialeq4},  we see that on the set $\check{A}_n\cap\hat{A}_n\cap\check{J}_n^{(2)}\cap\hat{J}_n^{(2)}$, 
\begin{align}
    \sqrt{nh_n}(\hat{\beta}_n-\check{\beta}_n)&=\sqrt{nh_n}(\hat{\beta}_n-\beta_0)-\sqrt{nh_n}(\check{\beta}_n-\beta_0)\nonumber\\
    &=I_{b,c}^{-1}(\theta_0;\theta_0)\frac{1}{\sqrt{nh_n}}\left(\partial_\beta l_n(\theta_0)-\partial_\beta l_n^{(2)}(\beta_0|\alpha_0)\right)+\hat{S}_n-\check{S}_n.\label{propH0:essentialeq6}
\end{align}
Since $P(\check{A}_n\cap\hat{A}_n\cap\check{J}_n^{(2)}\cap\hat{J}_n^{(2)})\to 1$, \eqref{propH0:Sn-pconv} and \eqref{propH0:Sn.j-pconv}, it is sufficient to show the following convergence for the proof of \eqref{propH0:goal2}: 
\begin{align}\label{propH0:ada.joint.diff-beta}
    \frac{1}{\sqrt{nh_n}}\left(\partial_\beta l_n(\theta_0)-\partial_\beta l_n^{(2)}(\beta_0|\alpha_0)\right)=o_p(1).
\end{align}

{\bf Proof of \eqref{propH0:ada.joint.diff-beta}.}
First, we show the case  where $D_1=D_2=D_3=\bar{D}_1=\bar{D}_2=1$. For $1\le m_2\le q$, 
\begin{align*}
    &\frac{1}{\sqrt{nh_n}}\left(\partial_\beta l_n(\theta_0)-\partial_\beta l_n^{(2)}(\beta_0|\alpha_0)\right)\\
    &= \frac{1}{\sqrt{nh_n}}\sum_{i=1}^n    \left(\partial_{\beta_{m_2}}b_{i-1}\right)^\top S_{i-1}^{-1}\bar{X}_{i,n}\left(\boldsymbol{1}_{\{|\Delta X_i^n|\le h_n^{\bar{\rho}_1}\}}-\boldsymbol{1}_{\{|\Delta X_i^n|\le h_n^{\rho_3}\}}\right)\\
    &\quad+\frac{1}{\sqrt{nh_n}}\sum_{i=1}^n\partial_{\beta_{m_2}}\log\Psi_{\beta_0}(\Delta X_i^n,X_{t_{i-1}^n})\varphi_n(X_{t_{i-1}^n},\Delta X_i^n)\left(\boldsymbol{1}_{\{|\Delta X_i^n|> h_n^{\bar{\rho}_2}\}}-\boldsymbol{1}_{\{|\Delta X_i^n|> h_n^{\rho_2}\}}\right).
\end{align*}
We show  that both terms on the right-hand side converge to 0 in probability. If $\bar{\rho}_1=\rho_3$, the first term equals to 0.
Hence, we evaluate the case  where $\bar{\rho}_1\neq\rho_3$. First, we discuss the case  where $\bar{\rho}_1>\rho_3$. From \eqref{propH0:rho1-diff}, it holds that 
\begin{align}\label{propH0:rho1-diff2}
    \boldsymbol{1}_{\{|\Delta X_i^n|\le h_n^{\bar{\rho}_1}\}}-\boldsymbol{1}_{\{|\Delta X_i^n|\le h_n^{\rho_3}\}}=-\boldsymbol{1}_{\{|\Delta X_i^n|\le h_n^{\rho_3}\}}\boldsymbol{1}_{\{|\Delta X_i^n|> h_n^{\bar{\rho}_1}\}}.
\end{align}
Since $|\bar{X}_{i,n}(\beta_0)|\boldsymbol{1}_{\{|\Delta X_i^n|\le h_n^{\rho_3}\}}=R(\theta,h_n^{\rho_3},X_{t_{i-1}^n})$, by using \eqref{propH0:rho1-diff2}, Markov's inequality, and Proposition \ref{prop2:jump},  it holds that for all $\varepsilon>0$, 
\begin{align*}
    &P\left(\left|-\frac{1}{\sqrt{nh_n}}\sum_{i=1}^n    \left(\partial_{\beta_{m_2}}b_{i-1}\right)^\top S_{i-1}^{-1}\bar{X}_{i,n}\boldsymbol{1}_{\{|\Delta X_i^n|\le h_n^{\rho_3}\}}\boldsymbol{1}_{\{|\Delta X_i^n|> h_n^{\bar{\rho}_1}\}}\right|>\varepsilon\right)\\
    &\le \frac{1}{\varepsilon\sqrt{nh_n}}\sum_{i=1}^n\sum_{k_1,k_2=1}^d\left|\mathbb{E}\left[\partial_{\beta_{m_2}}b_{i-1}^{(k_1)} {S_{i-1}^{-1}}^{(k_1,k_2)}\bar{X}_{i,n}^{(k_2)}\boldsymbol{1}_{\{|\Delta X_i^n|\le h_n^{\rho_3}\}}\boldsymbol{1}_{\{|\Delta X_i^n|> h_n^{\bar{\rho}_1}\}}\right]\right|\\
    &\le \frac{1}{\varepsilon\sqrt{nh_n}}\sum_{i=1}^n\sum_{k_1,k_2=1}^d\left|\mathbb{E}\left[R(\theta,h_n^{\rho_3},X_{t_{i-1}^n})P\left(\{|\Delta X_i^n|\le h_n^{\rho_3}\}\cap\{|\Delta X_i^n|> h_n^{\bar{\rho}_1}\}\ |\mathcal{F}_{i-1}^n\right)\right]\right|\\
    &\le C\sqrt{n}h_n^{\frac{1}{2}+2\rho_3}\\
    &= O\left(\sqrt{nh_n^{1+4\rho_3}}\right).
\end{align*}
Since $\rho_3\ge\frac{\delta}{4}$, the first term converges to 0 in probability if $D_1=D_2=D_3=\bar{D}_1=\bar{D}_2=1$ and $\bar{\rho}_1>\rho_3$.
In a similar way, if $D_1=D_2=D_3=\bar{D}_1=\bar{D}_2=1$ and $\bar{\rho}_1<\rho_3$, we have 
\begin{align*}
    &P\left(\left|-\frac{1}{\sqrt{nh_n}}\sum_{i=1}^n    \left(\partial_{\beta_{m_2}}b_{i-1}\right)^\top S_{i-1}^{-1}\bar{X}_{i,n}\boldsymbol{1}_{\{|\Delta X_i^n|\le h_n^{\rho_3}\}}\boldsymbol{1}_{\{|\Delta X_i^n|> h_n^{\bar{\rho}_1}\}}\right|>\varepsilon\right)\\
    &\le O\left(\sqrt{nh_n^{1+4\bar{\rho}_1}}\right).
\end{align*}
Since $\bar{\rho}_1\ge\frac{1+\delta}{6}\ge\frac{\delta}{4}$ by $\delta\in(0,\frac{1}{2})$, this converges to 0 in probability, too.
Next, If $\bar{\rho}_2=\rho_2$, the second term equals to 0. Hence we evaluate the case  where $\bar{\rho}_2\neq\rho_2$. First, we discuss the case  where $\bar{\rho}_2>\rho_2$. Then, since $\{|\Delta X_i^n|> h_n^{\bar{\rho}_2}\}\supset \{|\Delta X_i^n|> h_n^{\rho_2}\}$, we have
\begin{align}\label{propH0:rho2-diff}
    \boldsymbol{1}_{\{|\Delta X_i^n|> h_n^{\bar{\rho}_2}\}}-\boldsymbol{1}_{\{|\Delta X_i^n|> h_n^{\rho_2}\}}=\boldsymbol{1}_{\{|\Delta X_i^n|\le h_n^{\rho_2}\}}\boldsymbol{1}_{\{|\Delta X_i^n|> h_n^{\bar{\rho}_2}\}}.
\end{align}
Since $(1+|\Delta X_i^n|)^C\boldsymbol{1}_{\{|\Delta X_i^n|\le h_n^{\rho_2}\}}=2C\boldsymbol{1}_{\{|\Delta X_i^n|\le h_n^{\rho_2}\}}$, it follows from  Proposition \ref{prop2:jump} that for $\varepsilon>0$, 
\begin{align*}
&P\left(\left|\frac{1}{\sqrt{nh_n}}\sum_{i=1}^n\partial_{\beta_{m_2}}\log\Psi_{\beta_0}(\Delta X_i^n,X_{t_{i-1}^n})\varphi_n(X_{t_{i-1}^n},\Delta X_i^n)\boldsymbol{1}_{\{|\Delta X_i^n|\le h_n^{\rho_2}\}}\boldsymbol{1}_{\{|\Delta X_i^n|> h_n^{\bar{\rho}_2}\}}\right|>\varepsilon\right)\\
&\le\begin{cases}
\displaystyle
\frac{C}{\sqrt{nh_n}}\sum_{i=1}^{n}\mathbb{E}\left[(1+|X_{t_{i-1}^n}|)^C\mathbb{E}\left[(1+|\Delta X_i^n|)^C\boldsymbol{1}_{\{|\Delta X_i^n|\le h_n^{\rho_2}\}}\boldsymbol{1}_{\{|\Delta X_i^n|> h_n^{\bar{\rho}_2}\}} \ |\mathcal{F}_{i-1}^n\right]\right]   &(\text{under}\ [\textbf{C}_2\textbf{1}])\\
\displaystyle
\frac{C}{\sqrt{nh_n}\varepsilon_n}\sum_{i=1}^{n} \mathbb{E}\left[(1+|X_{t_{i-1}^n}|)^C P\left({\{|\Delta X_i^n|\le h_n^{\rho_2}\}}\cap{\{|\Delta X_i^n|> h_n^{\bar{\rho}_2}\}}\ |\mathcal{F}_{i-1}^n\right)\right]    &(\text{under}\ [\textbf{C}_2\textbf{2}])
\end{cases}\\
&\le\begin{cases}
\displaystyle
\frac{C}{\sqrt{nh_n}}\sum_{i=1}^{n}\mathbb{E}\left[2C(1+|X_{t_{i-1}^n}|)^CP\left({\{|\Delta X_i^n|\le h_n^{\rho_2}\}}\cap{\{|\Delta X_i^n|> h_n^{\bar{\rho}_2}\}}\ |\mathcal{F}_{i-1}^n\right)\right]   &(\text{under}\ [\textbf{C}_2\textbf{1}])\\
\displaystyle
C\left(nh_n^{1+2\rho_2}\varepsilon_n^{-2}\right)^\frac{1}{2}    &(\text{under}\ [\textbf{C}_2\textbf{2}])
\end{cases}\\
&\le\begin{cases}
\displaystyle
C\left(nh_n^{1+2\rho_2}\right)^\frac{1}{2}   &(\text{under}\ [\textbf{C}_2\textbf{1}]),\\
\displaystyle
C\left(nh_n^{1+2\rho_2}\varepsilon_n^{-2}\right)^\frac{1}{2}    &(\text{under}\ [\textbf{C}_2\textbf{2}]).
\end{cases}
\end{align*}
Since $\rho_2\ge \frac{\delta}{2}$ under $[\textbf{C}_2\textbf{1}]$
and $nh_n^{1+2\rho_2}\varepsilon_n^{-2}\to 0$ under $[\textbf{C}_2\textbf{2}]$, the second term converges to 0 in probability if $D_1=D_2=D_3=\bar{D}_1=\bar{D}_2=1$ and $\bar{\rho}_2>\rho_2$. Moreover, since the two sets including $\rho_2$ and $\bar{\rho}_2$ are the same, we can show the case of $D_1=D_2=D_3=\bar{D}_1=\bar{D}_2=1$ and $\bar{\rho}_2<\rho_2$ in an analogous manner. After all, \eqref{propH0:ada.joint.diff-beta} holds if $D_1=D_2=D_3=\bar{D}_1=\bar{D}_2=1$. Next, we evaluate the more general case. By \eqref{propH0:boldsymbol:result1}, for
$D,\bar{D}>0$ and $\rho,\bar{\rho}\in(0,\frac{1}{2})$, it holds that for all $n\in\mathbb{N}$,
\begin{align}\label{propH0:boldsymbol:result2}
        \boldsymbol{1}_{\{|\Delta X_i^n|> Dh_n^{\rho}\}}-\boldsymbol{1}_{\{|\Delta X_i^n|> \bar{D}h_n^{\bar{\rho}}\}}&= \left(1-\boldsymbol{1}_{\{|\Delta X_i^n|\le Dh_n^{\rho}\}}\right)-\left(1-\boldsymbol{1}_{\{|\Delta X_i^n|\le \bar{D}h_n^{\bar{\rho}}\}}\right)\nonumber\\
        &= -\left(\boldsymbol{1}_{\{|\Delta X_i^n|\le Dh_n^{\rho}\}}-\boldsymbol{1}_{\{|\Delta X_i^n|\le \bar{D}h_n^{\bar{\rho}}\}}\right)\nonumber\\
        &= -\boldsymbol{1}_{\{|\Delta X_i^n|\le Dh_n^{\rho}\}}\boldsymbol{1}_{\{|\Delta X_i^n|> \bar{D}h_n^{\bar{\rho}}\}}+\boldsymbol{1}_{\{|\Delta X_i^n|\le \bar{D}h_n^{\bar{\rho}}\}}\boldsymbol{1}_{\{|\Delta X_i^n|> Dh_n^{\rho}\}}.
\end{align}
Therefore, we can show the more general case by replacing \eqref{propH0:rho1-diff2} and \eqref{propH0:rho2-diff} with \eqref{propH0:boldsymbol:result1} and \eqref{propH0:boldsymbol:result2}, respectively.
Thus, \eqref{propH0:ada.joint.diff-beta} holds and we obtain \eqref{propH0:goal2}. This implies $D_n^{\frac{1}{2}}(\hat{\theta}_n-\check{\theta}_n)=o_p(1)$.

\item{Proof of $D_n^{\frac{1}{2}}(\hat{\theta}_n^*-\check{\theta}_n^*)=o_p(1)$ under $H_0$.}
    In an analogous manner to the evaluation of \eqref{lemma2:theta*-pconv} and \eqref{lemma2:theta*-dconv} in the proof of Lemma \ref{lemma:H0test}, we can show   the desired result in a similar way to case 1 under $H_0$.
Let $\bar{\check{\theta}}_n^*=(\bar{\check{\alpha}}_n^*, \bar{\check{\beta}}_n^*)$ as follows:
\begin{align*}
\bar{\check{\alpha}}_n^* &= ({\check{\alpha}_n}^{*(k+1)},\cdots,{\check{\alpha}_n}^{*(p)})^\top,\\
\bar{\check{\beta}}_n^* &= ({\check{\beta}_n}^{*(l+1)},\cdots,{\check{\beta}_n}^{*(q)})^\top,
\end{align*}
where
\begin{align*}
\check{\alpha}_n^* &= (0,\cdots,0,{\check{\alpha}_n}^{*(k+1)},\cdots,{\check{\alpha}_n}^{*(p)},0)^\top,\\
\check{\beta}_n^* &= (0,\cdots,0,{\check{\beta}_n}^{*(l+1)},\cdots,{\check{\beta}_n}^{*(q)})^\top,
\end{align*}
and let $\bar{\Theta}_{\alpha_0}:=\{\bar{\alpha}\in\mathbb{R}^{p-k}\ |\ \exists\alpha\in\Theta_{\alpha_0},\ \bar{\alpha}= (\alpha^{(k+1)},\cdots,\alpha^{(p)})^\top\}$ and $\bar{\Theta}_{\beta_0}:=\{\bar{\beta}\in\mathbb{R}^{q-l}\ |\ \exists\beta\in\Theta_{\beta_0},\ \bar{\beta}= (\beta^{(l+1)},\cdots,\beta^{(q)})^\top\}$. We define $U_n^{(1)}(\bar{\alpha})$ and $U_n^{(2)}(\bar{\beta}|\bar{\alpha})$ as follows with $l_n^{(1)}(\alpha)$ and $l_n^{(2)}(\beta|\alpha)$:
\begin{align*}
&U_n^{(1)}\left((\alpha^{(k+1)},\cdots,\alpha^{(p)})^\top\right):=l_n^{(1)}\left((0,\cdots,0,\alpha^{(k+1)},\cdots,\alpha^{(p)})^\top\right), \\
&U_n^{(2)}\left((\beta^{(l+1)},\cdots,\beta^{(q)})^\top|(\alpha^{(k+1)},\cdots,\alpha^{(p)})^\top\right)\\
&\quad:=l_n^{(2)}\left((0,\cdots,0,\beta^{(l+1)},\cdots,\beta^{(q)})^\top|(0,\cdots,0,\alpha^{(k+1)},\cdots,\alpha^{(p)})^\top\right).
\end{align*}  
Then $U_n^{(1)}(\bar{\alpha})$ and $U_n^{(2)}(\bar{\beta}|\bar{\alpha})$ can be regarded as an adaptive quasi-log likelihood function in $(p+q)-(k+l)$ dimensions. In particular, by the definition of $\bar{\check{\alpha}}_n^*,\bar{\check{\beta}}_n^*, U_n^{(1)}(\bar{\alpha})$ and $U_n^{(2)}(\bar{\beta}|\bar{\alpha})$, we have $\argsup_{\bar{\alpha}\in\bar{\Theta}_{\alpha_0}} U_n^{(1)}(\bar{\alpha})=\bar{\hat{\alpha}}_n^*$ and $\argsup_{\bar{\beta}\in\bar{\Theta}_{\beta_0}} U_n^{(2)}(\bar{\beta})=\bar{\hat{\beta}}_n^*$.
Since
\begin{align*}
\hat{\theta}_n^*-\check{\theta}_n^*=(0,\cdots,0,\hat{\alpha}_n^{*(k+1)}-\check{\alpha}_n^{*(k+1)},\cdots,\hat{\alpha}_n^{*(p)}-\check{\alpha}_n^{*(p)},0,\cdots,0,\hat{\beta}_n^{*(l+1)}-\check{\beta}_n^{*(l+1)},\cdots,\hat{\beta}_n^{*(q)}-\check{\beta}_n^{*(q)})^\top,
\end{align*}
the following types of convergence are sufficient for the proof of $D_n^{\frac{1}{2}}(\hat{\theta}_n^*-\check{\theta}_n^*)=o_p(1)$ under $H_0$:
\begin{align*}
&\sqrt{n}(\bar{\hat{\alpha}}_n^*-\bar{\check{\alpha}}_n^*)=o_p(1)\\
&\sqrt{nh_n}(\bar{\hat{\beta}}_n^*-\bar{\check{\beta}}_n^*)=o_p(1).
\end{align*}
We can prove this in an analogous manner to the proof of $D_n^{\frac{1}{2}}(\hat{\theta}_n-\check{\theta}_n)=o_p(1)$.
This completes the proof.
\end{enumerate}
\end{proof}
\subsubsection{Proof of Theorem \ref{thm6:H1test}}
\begin{proof}
Since
\begin{align*}
    H_1:\ \alpha^{(1)}\neq 0\ \mathrm{or}\ \ldots\ \alpha^{(k)}\neq 0\ \ \mathrm{or}\ \ \beta^{(1)}\neq 0\ \ldots\ \beta^{(l)}\neq 0,
\end{align*} 
we divide $H_1$ into 
\begin{align*}
    &H_1^{(1)}:\ \alpha^{(i)}\neq 0\ \mathrm{for\ some}\ i\in\{1,\ldots,k\},\\
    &H_1^{(2)}:\ \alpha^{(1)}=\cdots=\alpha^{(k)}=0\ \ \mathrm{and}\ \ \beta^{(j)}\neq 0\ \mathrm{for\ some}\ j\in\{1,\ldots,l\}.
\end{align*}
\begin{enumerate}
    \item For the case of $H_1^{(1)}$, 
     it follows from \eqref{a-cons:col1:uni} and \eqref{b-cons:col1:uni_jump} that
    \begin{align*}
        \sup_{\theta\in\Theta}\left|\frac{1}{n}l_n(\theta)-U_1^*(\alpha,\alpha_1)\right|&\le \sup_{\theta\in\Theta}\left|\frac{1}{n}\bar{l}_n(\theta)-U_1^*(\alpha,\alpha_0)\right|+h_n\sup_{\beta\in\Theta_\beta}\left|\frac{1}{nh_n}\tilde{l}_n(\beta)-\tilde{U}_{\beta_1}^{(2)^*}(\beta)\right|+h_n\tilde{U}_{\beta_1}^{(2)^*}(\beta)\\
        &\overset{P}{\to}0,
    \end{align*}
    and since $\tilde{\theta}_n\overset{P}{\to}\theta_1$, $\tilde{\theta}_n^*\overset{P}{\to}\theta^*$ under $H_1^{(1)}$,
     we see from continuity of $U_1^*(\alpha,\alpha_1)$ that under $H_1^{(1)}$, 
    \begin{align*}
        &\left|\frac{1}{n}\Lambda_n(\tilde{\theta}_n,\tilde{\theta}_n^*)+2\left(U_1^*(\alpha^*,\alpha_1)-U_1^*(\alpha_1,\alpha_1)\right)\right|\\
        &=2\left|\frac{1}{n}l_n(\tilde{\theta}_n)-U_1^*(\tilde{\alpha}_n,\alpha_1)\right|+2|U_1^*(\tilde{\alpha}_n,\alpha_1)-U_1^*(\alpha_1,\alpha_1)|\\
        &\quad+2\left|\frac{1}{n}l_n(\tilde{\theta}_n^*)-U_1^*(\tilde{\alpha}_n^*,\alpha_1)\right|+2|U_1^*(\tilde{\alpha}_n^*,\alpha_1)-U_1^*(\alpha^*,\alpha_1)|\\
        &\le 4\sup_{\theta\in\Theta}\left|\frac{1}{n}l_n(\theta)-U_1^*(\alpha,\alpha_1)\right|+2|U_1^*(\tilde{\alpha}_n,\alpha_1)-U_1^*(\alpha_1,\alpha_1)|+2|U_1^*(\tilde{\alpha}_n^*,\alpha_1)-U_1^*(\alpha^*,\alpha_1)|\\
        &\overset{P}{\to}0.
    \end{align*}
    Since $\alpha^*\neq \alpha_1$ under $H_1^{(1)}$, it follows from the identifiability condition that $U_1^*(\alpha_1,\alpha_1)-U(\alpha^*,\alpha_1)>0$. Then by using Lemma \ref{lemma:H1} with $X_n=\frac{1}{n}\Lambda_n(\tilde{\theta}_n,\tilde{\theta}_n^*)$ and $a_n=\frac{1}{n}\chi_{k+l,\varepsilon}^2$, we obtain 
    \begin{align*}
        0\le P(\Lambda_n(\tilde{\theta}_n,\tilde{\theta}_n^*)\le \chi_{k+l,\varepsilon}^2)&=P\left(\frac{1}{n}\Lambda_n(\tilde{\theta}_n,\tilde{\theta}_n^*)\le \frac{1}{n}\chi_{k+l,\varepsilon}^2\right)\\
        &\to 0
    \end{align*}
    under $H_1^{(1)}$.
    Therefore,  one has 
    \begin{align*}        P(\Lambda_n(\tilde{\theta}_n,\tilde{\theta}_n^*)>\chi_{k+l,\varepsilon}^2)\to 1
    \end{align*}
    under $H_1^{(1)}$.
    
    \item For the case of $H_1^{(2)}$, since $l_n(\hat{\theta})=\sup_{\theta\in\Theta}l_n(\theta)\ge l_n(\tilde{\alpha}_n^*,\beta_1)$, it follows that
    \begin{align*}
    \Lambda_n(\tilde{\theta}_n,\tilde{\theta}_n^*)&=2(l_n(\tilde{\theta}_n)-l_n(\tilde{\theta}_n^*))\\
        &=2\left\{(l_n(\tilde{\theta}_n)-l_n(\hat{\theta}_n))+(l_n(\hat{\theta}_n)-l_n(\tilde{\alpha}_n^*,\beta_1))+(l_n(\tilde{\alpha}_n^*,\beta_1)-l_n(\tilde{\theta}_n^*))\right\}\\
        &\ge 2\left\{(l_n(\tilde{\theta}_n)-l_n(\hat{\theta}_n))+(l_n(\tilde{\alpha}_n^*,\beta_1)-l_n(\tilde{\theta}_n^*))\right\}\\
        &=:\bar{\Lambda}_n(\tilde{\theta}_n,\tilde{\theta}_n^*,\hat{\theta}_n).
    \end{align*}
    We discuss the behavior of $\frac{1}{nh_n}\bar{\Lambda}_n(\tilde{\theta}_n,\tilde{\theta}_n^*,\hat{\theta}_n)$. In a similar way to the proof of \eqref{thm5-2:eq2},  one has $\frac{2}{nh_n}(l_n(\tilde{\theta}_n)-l_n(\hat{\theta}_n))=o_p(1)$ under $H_1^{(2)}$. By using 
    Theorem \ref{thm1-2:jump}, $\tilde{\theta}_n^*\overset{P}{\to}\theta^*$ under $H_1^{(2)}$, \eqref{b-cons:col1:uni_jump}, \eqref{b-cons:col1:uni} and continuity of $V_{\beta_1}^*(\alpha,\beta)$, it holds under $H_1^{(2)}$ that 
    \begin{align*}
        &\left|\frac{2}{nh_n}(l_n(\tilde{\alpha}_n^*,\beta_1)-l_n(\tilde{\theta}_n^*))+2V_{\beta_1}^*(\alpha^*,\beta^*)\right|\\
        &\quad\le 2\left|\frac{1}{nh_n}(l_n(\tilde{\theta}_n^*)-l_n(\tilde{\alpha}_n^*,\beta_1))-V_{\beta_1}^*(\tilde{\alpha}_n^*,\tilde{\beta}_n^*)\right|+2\left|V_{\beta_1}^*(\tilde{\alpha}_n^*,\tilde{\beta}_n^*)-V_{\beta_1}^*(\alpha^*,\beta^*)\right|\\
        &\quad\le 2\sup_{\theta\in\Theta}\left|\frac{1}{nh_n}(l_n(\theta)-l_n(\alpha,\beta_1))-V_{\beta_1}^*(\alpha,\beta)\right|+2\left|V_{\beta_1}^*(\tilde{\alpha}_n^*,\tilde{\beta}_n^*)-V_{\beta_1}^*(\alpha^*,\beta^*)\right|\\
        &\quad\le 2\sup_{\theta\in\Theta}\left|\frac{1}{nh_n}(\bar{l}_n(\theta)-\bar{l}_n(\alpha,\beta_1))-\bar{U}_{\beta_1}^{(2)^*}(\alpha,\beta)\right|+4\sup_{\theta\in\Theta}\left|\frac{1}{nh_n}\tilde{l}_n(\theta)-\tilde{U}_{\beta_1}^{(2)^*}(\beta)\right|\\
        &\qquad+2\left|V_{\beta_1}^*(\tilde{\alpha}_n^*,\tilde{\beta}_n^*)-V_{\beta_1}^*(\alpha^*,\beta^*)\right|\\
        &\overset{P}{\to}0.
    \end{align*}
    Hence, we obtain  
    \begin{align*}
        \frac{1}{nh_n}\bar{\Lambda}_n(\tilde{\theta}_n,\tilde{\theta}_n^*,\hat{\theta}_n)\overset{P}{\to}-2V_{\beta_1}^*(\alpha^*,\beta^*)
    \end{align*}
    under $H_1^{(2)}$, and since $\alpha_1=\alpha^*,\ \beta_1\neq \beta^*$, it follows from the  identifiability condition that $-2V_{\beta_1}^*(\alpha^*,\beta^*)=-2V_{\beta_1}^*(\alpha_1,\beta^*)>-2V_{\beta_1}^*(\alpha_1,\beta_1)=0$ under $H_1^{(2)}$. By using Lemma \ref{lemma:H1} with $X_n=\frac{1}{nh_n}\bar{\Lambda}_n(\tilde{\theta}_n,\tilde{\theta}_n^*,\hat{\theta}_n)$ and $a_n=-\frac{2}{nh_n}V_{\beta_1}^*(\alpha^*,\beta^*)$, it holds under $H_1^{(2)}$ that 
    \begin{align*}
        0\le P(\Lambda_n(\tilde{\theta}_n,\tilde{\theta}_n^*)\le \chi_{k+l,\varepsilon}^2)&\le P(\bar{\Lambda}_n(\tilde{\theta}_n,\tilde{\theta}_n^*,\hat{\theta}_n)\le \chi_{k+l,\varepsilon}^2)\\
        &=P\left(\frac{1}{nh_n}\bar{\Lambda}_n(\tilde{\theta}_n,\tilde{\theta}_n^*,\hat{\theta}_n)\le \frac{1}{nh_n}\chi_{k+l,\varepsilon}^2\right)\\
        &\overset{P}{\to}0.
    \end{align*}
    This implies, under $H_1^{(2)}$, 
    \begin{align*}
        P(\Lambda_n(\tilde{\theta}_n,\tilde{\theta}_n^*)>\chi_{k+l,\varepsilon}^2)\to 1.
    \end{align*}
    This completes the proof.
\end{enumerate}
\end{proof}
\subsubsection{Proof of Proposition \ref{prop:H1test}}
\begin{proof}
     It follows from the proof of Proposition \ref{prop:test} that $D_n^{\frac{1}{2}}(\hat{\theta}_n-\check{\theta}_n)=o_p(1)$. Moreover, it holds from Theorem \ref{thm1-2:jump} that $\check{\theta}_n\overset{P}{\to}\theta_1$. Hence, we show that $\check{\theta}_n^*\overset{P}{\to}\theta^*$ under $H_1$. 
     Set $\bar{\theta}^*$ as follows:
\begin{align*}
\bar{\theta}^*&= (\alpha^{*(k+1)},\cdots,\alpha^{*(p)},\beta^{*(l+1)},\cdots,\beta^{*(q)})^\top.
\end{align*}
Then  by using {\bf{[E1]}}, the identifiability condition for $\theta^*$ holds. Thus, similarly to the construction of $U_n^{(1)}(\bar{\alpha})$ and $\bar{U}_n^{(2)}(\bar{\beta}|\bar{\alpha})$, by redefining the domain of $U_1^*(\alpha,\alpha_1),V_{\beta_1}(\alpha,\beta)$ with  the reduced dimension, 
the identifiability conditions for $\bar{\alpha}^*$ and $\bar{\beta}^*$ hold. Therefore, in an analogous manner to the proof of Theorem \ref{thm1-2:jump}, we have $\bar{\check{\theta}}_n^*\overset{P}{\to}\bar{\theta}^*$ under $H_1$. This implies ${\check{\theta}}_n^*\overset{P}{\to}{\theta}^*$.
\end{proof}

%
\bibliographystyle{apalike}
\nocite{*}

\begin{thebibliography}{}

\bibitem[Amorino and Gloter, 2018]{Gloter}
Amorino, C. and Gloter, A. (2018).
\newblock Contrast function estimation for the drift parameter of ergodic jump diffusion process.
\newblock {\em Scandinavian Journal of Statistics}, 47:279 -- 346.

\bibitem[Ferguson, 1996]{Ferguson}
Ferguson, T.~S. (1996).
\newblock A course in large sample theory.

\bibitem[Genon-Catalot and Jacod, 1993]{Genon-Catalot}
Genon-Catalot, V. and Jacod, J. (1993).
\newblock On the estimation of the diffusion coefficient for multi-dimensional diffusion processes.
\newblock {\em Annales de l'I.H.P. Probabilit\'es et statistiques}, 29(1):119--151.

\bibitem[Hall and Heyde, 1980]{Hall-Heyde}
Hall, P. and Heyde, C. (1980).
\newblock Martingale limit theory and its application.
\newblock Academic Press.

\bibitem[Inatsugu and Yoshida, 2021]{Inatsugu-Yoshida}
Inatsugu, H. and Yoshida, N. (2021).
\newblock Global jump filters and quasi-likelihood analysis for volatility.
\newblock {\em Annals of the Institute of Statistical Mathematics}, 73:555 -- 598.

\bibitem[Kawai and Uchida, 2022]{Kawai-Uchida}
Kawai, T. and Uchida, M. (2022).
\newblock Adaptive testing method for ergodic diffusion processes based on high frequency data.
\newblock {\em Journal of Statistical Planning and Inference}, 217:241--278.

\bibitem[Kessler, 1997]{Kessler}
Kessler, M. (1997).
\newblock Estimation of an ergodic diffusion from discrete observations.
\newblock {\em Scandinavian Journal of Statistics}, 24.

\bibitem[Kitagawa and Uchida, 2014]{Kitagawa-Uchida}
Kitagawa, H. and Uchida, M. (2014).
\newblock Adaptive test statistics for ergodic diffusion processes sampled at discrete times.
\newblock {\em Journal of Statistical Planning and Inference}, 150:84--110.

\bibitem[Mancini, 2004]{Mancini}
Mancini, C. (2004).
\newblock Estimation of the characteristics of the jumps of a general poisson-diffusion model.
\newblock {\em Scandinavian Actuarial Journal}, 2004:42 -- 52.

\bibitem[Masuda, 2007]{Masuda}
Masuda, H. (2007).
\newblock Ergodicity and exponential $\beta$-mixing bounds for multidimensional diffusions with jumps.
\newblock {\em Stochastic Processes and their Applications}, 117(1):35--56.

\bibitem[Nakakita and Uchida, 2019]{Nakakita-Uchida}
Nakakita, S.~H. and Uchida, M. (2019).
\newblock Adaptive test for ergodic diffusions plus noise.
\newblock {\em Journal of Statistical Planning and Inference}, 203:131--150.

\bibitem[Ogihara and Uehara, 2023]{Ogihara-Uehara}
Ogihara, T. and Uehara, Y. (2023).
\newblock {Local asymptotic normality for ergodic jump-diffusion processes via transition density approximation}.
\newblock {\em Bernoulli}, 29(3):2342 -- 2366.

\bibitem[Ogihara and Yoshida, 2011]{Ogihara-Yoshida}
Ogihara, T. and Yoshida, N. (2011).
\newblock Quasi-likelihood analysis for the stochastic differential equation with jumps.
\newblock {\em Statistical Inference for Stochastic Processes}, 14:189--229.

\bibitem[Shimizu and Yoshida, 2003]{Shimizu-Yoshida_JP}
Shimizu, Y. and Yoshida, N. (2003).
\newblock Janpugata kakusankatei no risankansokukarano suiteinitsuite.
\newblock {\em The 2003 Japanese Joint Statistical Meeting}.

\bibitem[Shimizu and Yoshida, 2006]{Shimizu-Yoshida}
Shimizu, Y. and Yoshida, N. (2006).
\newblock Estimation of parameters for diffusion processes with jumps from discrete observations.
\newblock {\em Statistical Inference for Stochastic Processes}, 9:227--277.

\end{thebibliography}

\end{document}